\documentclass[12pt,letterpaper,reqno]{amsart}

\usepackage{amssymb,amsmath}
\usepackage{amsthm}
\usepackage{color,graphicx}
\usepackage{eucal}

\usepackage{comment}
\usepackage{caption}
\usepackage{subcaption}

\usepackage{wrapfig}
\usepackage{floatflt}

\newcommand{\R}{\mathbb R}

\newcommand{\sech}{\textmd{sech}}

\newcommand{\ep}{\epsilon}

\newcommand{\ds}{\displaystyle}

\newtheorem{theorem}{Theorem}[section]

\newtheorem{lemma}[theorem]{Lemma}
\newtheorem{corollary}[theorem]{Corollary}
\newtheorem{definition}[theorem]{Definition}

\newtheorem*{theoremA}{Theorem A}

\theoremstyle{remark}
\newtheorem{remark}[theorem]{Remark}

\numberwithin{equation}{section}
\numberwithin{table}{section}
\numberwithin{figure}{section}
\numberwithin{theorem}{section}

\theoremstyle{remark}

\makeatletter
\newcommand{\setTheoremAlabel}{\def\@currentlabel{A}}
\makeatother

\title[ground states in biharmonic NLS]{Ground states to bi-harmonic 
nonlinear Schr\"odinger equations  with competing dispersion}

\author[C. Klein]{Christian Klein}
\address{Université Bourgogne Europe, CNRS, IMB UMR 5584, F-21000 Dijon, France\\
Institut Universitaire de France} 
\email{Christian.Klein@u-bourgogne.fr}

\author[S. Roudenko]{Svetlana Roudenko}
\address{Department of Mathematics \& Statistics\\
Florida International University,  Miami, FL 33199, USA}
\curraddr{}
\email{sroudenko@fiu.edu}

\date{}

\subjclass[2020]{35J91, 35B40, 35Q55, 65N35, 35B06.}

\keywords{Fourth-order nonlinear Schr\"odinger equation, 2d biharmonic NLS, mixed dispersion,
nonradial ground states, symmetry breaking, least-action minimizers, normalized ground states, Weinstein quotient, asymptotic behavior, branching.}

\begin{document}

\begin{abstract}
We study ground-state solutions of the fourth-order nonlinear Schr\"odinger equation with mixed dispersion via the elliptic profile equation
$$
\Delta^2 Q + 2a \Delta Q + bQ - |Q|^\alpha Q=0.
$$
The main focus of this paper is to investigate symmetry breaking of least-action ground states in 2D in the regime $a=1$, $b=1+\epsilon$, $0<\epsilon\ll1$, where nonradiality was proved by Lenzmann and Weth in \cite{LW2021}. 
We establish a quotient-to-mass relation converting small-$\ep$ asymptotics of the
Weinstein-type quotient into asymptotics of the mass, in any dimension and symmetry class.
In 1D we prove the sharp small-$\ep$ rate for the quotient for every $\alpha>0$ and
deduce the rates for the mass, action and potential norm. In two and higher dimensions the same relation, combined with the quotient asymptotics of Lenzmann-Weth \cite{LW2021} and
Mandel-Oliveira e Silva \cite{MO2023}, gives the mass rates for the Knapp-type nonradial and
the radial ground states. The resulting thresholds $\alpha_K(d)=\frac{8}{d+1}$ coincide for
$d \leq 3$ with the existence thresholds for normalized minimizers of \cite{FJMM2022}, connecting the least-action and mass-constrained formulations.

We complement these analytical results with a detailed numerical 
study. For the {\it cubic} nonlinearity we construct a nonradial branch relying on the Knapp example, compare it with radial and angular-mode branches, and verify the predicted rates. For the {\it quartic} nonlinearity the computed ground state families exhibit turning points and branching in the mass-energy diagrams, and we find two- and four-peak nonradial branches. Our computations for this nonlinearity indicate that, although the least-action ground states are nonradial for small $\ep$, the normalized ground states of the same mass are radial, a consequence of branching. For $\alpha \geq 4$ we obtain only radial ground states, consistent with the endpoint of the symmetry-breaking range and with the conjectured sharp form
of the Stein-Tomas inequality on $\mathbb S^1$, which we show is equivalent to the radial and
nonradial masses having the same leading coefficient.

\end{abstract}

\maketitle

\tableofcontents

\section{Introduction}

We consider the 4th order or bi-harmonic NLS equation: 
\begin{equation}\label{biNLS}
\text{(bi-NLS)} 
\qquad \qquad 
i\, u_t - \Delta^2 u - 2a \Delta u + |u|^{\alpha} u = 0, \qquad x \in \R^d, \quad t \in \R, \qquad
\end{equation}
where $u(t,x)$ is a complex-valued function, $a \in \mathbb R$, $\Delta$ stands for the standard Laplacian, and the nonlinearity power $\alpha>0$. While our analytical results hold in any dimension, we focus our numerical study on two dimensions ($d=2$).
This equation is a higher-dispersion generalization of the standard nonlinear Schr\"odinger equation
\begin{equation}
\label{NLS}
\text{(NLS)} \qquad \qquad
i\, u_t + \Delta u + |u|^{\alpha} u = 0, \qquad x \in \R^d, \qquad t \in \R, \qquad
\end{equation}
whose coherent structures, such as solitons, have been discovered 
over 50 years ago in the context of the cubic NLS as an integrable 
system (Zakharov \& Shabat \cite{ZS}). Their experimental 
confirmations also have a long history, for instance, first such observations in optics were done over 40 years ago, as optical solitary waves propagating in fibers, in \cite{MSG}. 
In the standard NLS equation \eqref{NLS}, the soliton, a nonlinear object, balances the second order dispersion and the self-focusing (e.g., Kerr/cubic) nonlinearity. This has been fundamental in physical applications and in mathematical studies. 

More recently, new technologies and fiber designs in modern optics led to the discovery of the photonic crystal fiber (PCF) and waveguides on a chip, which enabled the experimentalists to produce silicon PCF waveguides known now as {\it pure quartic} optical solitons on a chip \cite{Nature}, which have higher order dispersion, namely, the 4th order. Further investigation showed how Gaussian  initial data can evolve with the 4th order NLS into such quartic solitons \cite{Tam2019}, and a more general model such as \eqref{biNLS} was suggested, for example, in \cite{Tam2020} to contain both fourth and second order dispersions, since obtaining purely quartic solitons (or even purely quadratic solitons) is challenging experimentally.  

Before the recent experimental discoveries of the quartic solitons, 
the bi-harmonic NLS model was proposed in the early 90s by Karpman 
\cite{K1991} and Karpman \& Shagalov \cite{KS1991}, exactly for the 
purpose of accounting for the influence of higher order dispersion into the NLS equation to model intense laser beam propagation. 
Numerical simulations of solutions to the 4th order NLS equation were also initially done in \cite{KS1991}, later more thorough investigations were performed by Fibich, Ilan \& Papanicolaou in  \cite{Fibich2002} and their follow-up work \cite{BFM2010}, \cite{BFM2010b}, \cite{BF2011}, especially, the collapsing or blow-up solutions in critical and supercritical cases. 
Analytically, showing existence of finite time blow-up was a challenging task for some time and the breakthrough for proving the blow-up in the bi-harmonic NLS equation ($d \geq 2$, $\alpha \geq 8/d$) was by Boulenger \& Lenzmann in \cite{BL2017}, see further progress, for example, in \cite{BCGJ2019b, Dinh2021, Gou2024}. The one-dimensional mixed-dispersion model we investigated numerically in \cite{KPRS}, where profiles of solitary waves and their properties were studied as well as several dynamical features; we also noted some changes (compared to the standard NLS) in sharp thresholds and dichotomy, especially, when the second order dispersion works against the fourth order, which is influencing the long term behavior and finite time blow-up. Local well-posedness of the Cauchy problem in the energy $H^2$-space follows from the dispersive estimates established by Ben-Artzi, Koch \& Saut in \cite{BAKS2000}, which is sufficient for this paper; for references to follow-up work, for example, see \cite{PRR, KPRS}.  


During their lifespan, the solutions $u(t, x)$ to \eqref{biNLS} conserve several quantities, including mass and energy (or Hamiltonian):
\begin{equation}\label{E:M}
M(u(t))=\int_{\R^d} |u|^2(t)\, dx = M(u(0))
\end{equation}
and
\begin{equation}\label{E:E}
E(u(t))=\dfrac{1}{2}\int_{\R^d} |\Delta u(t)|^2 \, dx - a \int_{\R^d} |\nabla u(t)|^2 \; dx - \dfrac1{\alpha+2} \int_{\R^d} |u(t)|^{\alpha+2} \, dx = E(u(0)).
\end{equation}

Similar to the NLS equation, the biharmonic NLS has time, space and phase invariances, however, the usual Galilean symmetry is lost due to the presence of the fourth-order dispersion. The one symmetry, which is especially useful in the evolution equations, is {\it the scaling invariance}, which states that an appropriately rescaled version of the original solution is also a solution of the equation, though for the equation \eqref{biNLS} due to the different dispersion terms, there is no simple suitable symmetry like that, however, if one considers the higher dispersion term dominant ($a=0$), then it is given by
\begin{equation}\label{scaling}
u_\lambda(t, x)=\lambda^{\frac{4}{\alpha}} u(\lambda^4 t, \lambda x), ~~~ x\in \R^d.
\end{equation}
This symmetry makes a specific Sobolev norm $\dot{H}^s$ invariant, i.e.,
\begin{equation*}
\|u_\lambda(0,\cdot,\cdot) \|_{\dot{H}^s(\mathbb R^d)}=\lambda^{\frac{4}{\alpha}+s-\frac{d}{2}} \|u_0\|_{\dot{H}^s(\mathbb R^d)},
\end{equation*}
and the index $s$ gives rise to the critical-type classification of equations.
For the biharmonic NLS equation \eqref{biNLS} (with $a=0$) the critical index is 
$$
s=\frac{d}2-\frac4{\alpha},
$$
and when $s=0$ ($\alpha=8/d$) the equation \eqref{biNLS} is $L^2$-critical, when $s<0$ ($\alpha<8/d$) it is subcritical and $s>0$ ($\alpha>8/d$) is supercritical. When $a \neq 0$, this scaling is no longer valid, but the exponent $8/d$ is still useful as a reference point. In other regimes (such as the small-$\epsilon$ regime considered, for instance, in Section \ref{S:asym}) other effective scalings appear, for example, from the concentration on the Fourier side near the minimum set of the considered Fourier multiplier or symbol.

The biharmonic NLS equation \eqref{biNLS} has a family of (standing) solitary waves
\begin{equation}\label{Eq:SW}
u(t,x) = e^{ibt} \, Q(x), ~~~b>0,
\end{equation}
with $Q \in H^2(\mathbb R^d)$ and $Q(x) \to 0$ as $|x| \to + \infty$. 
Substituting \eqref{Eq:SW} into \eqref{biNLS} gives the
nonlinear elliptic equation
\begin{equation}\label{E:groundstate}
\Delta^2 Q + 2a\,\Delta Q +  b\, Q  - |Q|^{\alpha} Q = 0,
\end{equation}
solutions of which we investigate in this paper. In particular, we are interested in {\it ground state} solutions of \eqref{E:groundstate}, which can be defined in several closely related ways:
as least-action critical points, as Weinstein quotient minimizers, as 
minimizers of the energy functional under a mass constraint or potential energy; we recall and compare these notions in
Section~\ref{S:Def-GS}.

Because of the positivity of the associated quadratic form (see \eqref{E:q-ab-new}), variational ground states exist only when $a < \sqrt b$ (see \cite{BCSN2018}, \cite{LW2021} and \S \ref{S:positive}). 
The shape of ground states depends strongly on the sign and size of the mixed-dispersion parameter $a$. We discuss this in more detail in Section \ref{S:positive}. For now we mention that ground states can be positive (if $a \leq -\sqrt b$) or sign-changing with decaying oscillatory tails (if $-\sqrt b < a < \sqrt b$).
The case $a=0$ corresponds to the pure biharmonic NLS equation and has a simple scaling in $b$,
\begin{equation}\label{Qscaling}
Q_b(x) = b^{1/\alpha} Q(b^{1/4} x),
\end{equation}
which produces a family of solutions ${Q_b}$, provided $Q$ is a solution of \eqref{E:groundstate} (with $a=0$ and $b=1$).

The fact that the ground-state profiles in the biharmonic, or fourth-order, NLS equation,
need not be positive and may have oscillatory tails as $|x|\to\infty$, distinguishes this dispersive equation from the standard NLS equation (as well as KdV/ZK-type equations).  Even more intriguing, the fact that the loss of positivity can be accompanied by a loss of symmetry: in higher dimensions, ground states may be {\it nonradial}.  In particular, taking in equation  \eqref{E:groundstate} $a=1$ and $b=1+\epsilon$, with $0<\epsilon\ll1$, Lenzmann and Weth proved in \cite{LW2021} the following symmetry-breaking result.

\begin{theorem}[\cite{LW2021}, Theorem 1.2]\label{Thm1}
Let $d \geq 2$ and $0<\alpha < \frac4{d-1}$. Then there exists $\epsilon_0 = \epsilon_0(\alpha,d) > 0$ such that every ground state solution\footnote{in the sense of Definition \ref{D:BCo}} $u \in H^2(\mathbb R^d)$ of \eqref{E:groundstate} is a nonradial function if $0<\epsilon \leq \epsilon_0$.
\end{theorem}
  

A recent paper of Mandel and Oliveira e Silva \cite{MO2023} takes a step further in considering symmetries of solutions, the so-called $G_k$-symmetric groups, which are {\it block-radial}, and show that there exist radial and block-radial solutions of \eqref{E:groundstate} with action larger than the unrestricted ground states. This implies that ground states may fail to be radial, but also block-radial. In 2D, the only nontrivial block-radial class is $G_1 = O(1) \times O(1)$, which means that $G_1$-symmetry is coordinate-wise evenness:
$$
u(x,y) = u(-x,y) = u(x,-y),
$$
and thus, we have $H^2_{rad}(\R^2) \subset H^2_{G_1}(\R^2) \subset H^2(\R^2)$.
One may expect that the unrestricted space could produce the true 
ground state (the smallest minimizer in a certain sense, see \S \ref{S:Def-GS}), in 2D it is not known (see \cite[Appendix C]{MO2023}). 

\smallskip

The aim of this paper is to investigate the two-dimensional ground state solutions of \eqref{E:groundstate} in more detail. 
For that, we first derive small-$\ep$ asymptotics for the mass, with a special emphasis on one and two dimensions,  building on 

(i) our one-dimensional numerical study in \cite{KPRS} and the results in \cite{FJMM2022}, and 

(ii) the two (and higher)-dimensional asymptotic results in \cite{LW2021}, \cite{MO2023}, \cite{FJMM2022},  
and references therein. 

In particular, we establish the quotient-to-mass relation (Lemma~ \ref{L:quotient-to-mass}), which allows us to transfer rates from the minimizing quotient 
to the rates for mass. This helps us examine the mass behavior, including turning points, or minima in the mass curves, which relate to branching phenomenon that we established in \cite{KPRS}. Small-$\ep$ asymptotics allows us to distinguish the behavior of different branches of ground states, which we perform next.
Besides providing a generalization to the relation lemma (Theorem~\ref{T:quotient-mass-general}),
we derive the sharp unrestricted quotient lower bound from the non-homogeneous Gagliardo-Nirenberg inequality of \cite{FJMM2022} in any dimension, in the ranges stated below. In $d \geq 2$, this recovers the lower bound exponent of \cite{LW2021}.
We also give the reader a sense of what functions (we call them trial functions in appendix) yield such asymptotics. These examples become very important in the numerical part of the paper, which involves delicate attenuation for finding specific ground states. 

\smallskip

More precisely, our main analytical results are the following. 
Consider the $H^2$-subcritical range: 
$\alpha>0$ if $d \leq 4$ and $0 < \alpha < 8/(d-4)$ if $d \geq 5$. 
For $u \in H^2(\R^d)$ set
$q_{1,1+\ep}(u) = \int_{\R^d} ( (|\xi|^2-1)^2 + \ep )|\hat{u}(\xi)|^2 \, d\xi$, 
and for a symmetry class $\mathcal X\subset H^2(\R^d)$ (for instance all functions or the radial ones) define
$$
\mathcal R_\ep^{\mathcal X}(\alpha) = \inf_{u \in \mathcal X \setminus \{0\} } 
\frac{q_{1,1+\ep}(u)}{\| u \|_{L^{\alpha+2}(\R^d)}^2},
$$
the Weinstein quotient on $\mathcal X$ (we drop the superscript when
$\mathcal X=H^2(\R^d)$), see further details in \S \ref{S:Def-GS} and \ref{S:R-asym}. When the infimum is attained, $Q_\ep^{\mathcal X}$ denotes the
corresponding least-action ground state of \eqref{E:groundstate} in the class $\mathcal X$, with $a=1$, $b=1+\ep$. Here, $\approx$ denotes two-sided bounds with positive constants independent of $\ep$ as described in Remark \ref{R:approx}. 

\begin{theoremA}\setTheoremAlabel\label{ThmA}
Let $\alpha>0$ be in the $H^2$-subcritical range and $\ep \to 0^+$.
\begin{itemize}
\item[\underline{Part 1:}] {\rm (Quotient-to-mass relation)} 
Assume that the infimum defining $\mathcal R_\ep^{\mathcal X}(\alpha)$ is attained for every sufficiently small $\ep>0$. If
$\mathcal R_\ep^{\mathcal X}(\alpha) \approx \ep^{\gamma}L(\ep)$ for some $0< \gamma<1$ and some
positive slowly varying $L$ at zero, then
$$
M\big(Q_\ep^{\mathcal X}\big) 
\approx \frac{ (\mathcal R_\ep^{\mathcal X}(\alpha) )^{\frac{\alpha+2}{\alpha}}}{\ep}
\approx\ \ep^{\gamma\frac{\alpha+2}{\alpha}-1} 
L(\ep)^{\frac{\alpha+2}{\alpha}}.
$$
Moreover, if ~~ $\mathcal R_{\lambda\ep}^{\mathcal X}(\alpha)/\mathcal R_\ep^{\mathcal X}(\alpha)
\to\lambda^{\gamma}$~ for every $\lambda>0$, then 
$$
\ep \,M(Q_\ep^{\mathcal X}) = \gamma\, (\mathcal R_\ep^{\mathcal X}(\alpha))^{\frac{\alpha+2}{\alpha}}\,(1+o(1)).
$$

\item[\underline{Part 2:}] {\rm (Sharp unrestricted quotient and mass rates)} 
For $d=1$ and every $\alpha>0$,
\begin{equation}\label{E:rates-1D}
\mathcal R_\ep(\alpha) \approx \ep^{\frac{\alpha+4}{2(\alpha+2)}},
\qquad \mbox{and thus,} \qquad
M(Q_\ep) \approx \ep^{\frac{2}{\alpha}-\frac12}.
\end{equation}
For $d \geq 2$ and $0 < \alpha \leq 4/(d-1)$,
\begin{equation}\label{E:rates-2D+}
\mathcal R_\ep(\alpha) \approx \ep^{\frac{8+(3-d)\alpha}{4(\alpha+2)}}, 
\qquad \mbox{and thus,} \qquad M(Q_\ep) \approx \ep^{\frac2{\alpha}-\frac{d+1}{4}}.
\end{equation}
\end{itemize}
\end{theoremA}

We prove Part 1 in Lemma~\ref{L:quotient-to-mass}, and more generally, in Theorem \ref{T:quotient-mass-general}. This part converts two-sided asymptotics of the quotient into asymptotics of the mass, in any dimension and symmetry class, where the quotient infimum is attained. When the quotient admits an expansion with a leading constant, as in \cite[Theorems 1.3(i), 2.1(i)]{LW2021} for $\alpha>\frac{2}{d-1}$ in the radial class and $\alpha\ge\frac{4}{d-1}$ in general, the second statement of Part 1 relates that constant to the mass, providing sharp asymptotics with explicit constants, see Corollary \ref{C:sharp-mass} and Remark \ref{R:constants-radiality}. 
The 1D statement in Part 2 is proved in \S \ref{S:rate-consequences}, where we improve the bound $\mathcal R_\ep \lesssim \ep^{1/2}$ of \cite[Prop. 3.7]{FJMM2022}, and the matching lower bound is obtained by either a Hausdorff-Young estimate or the non-homogeneous Gagliardo-Nirenberg inequality from \cite[Theorem 1.1]{FJMM2022}. 
To our knowledge, the sharp {\it one-dimensional} rates were not previously known. 
Both approaches for obtaining lower bounds extend to higher dimensions, but only the non-homogeneous Gagliardo-Nirenberg argument gives  the sharp unrestricted lower bound of \cite{LW2021}, we prove it in Theorem~\ref{Thm:GN-lower}. 
Combining this lower bound with the matching upper bound (from \cite{LW2021}) and then applying Part 1 gives the higher-dimensional statement in Part 2. The
higher-dimensional quotient rate was already known, while the
quotient-to-mass relation gives new information about the corresponding mass rate.

The unrestricted rates in \eqref{E:rates-1D} and \eqref{E:rates-2D+} give the Knapp mass thresholds
\begin{equation}\label{E:alpha-K}
\alpha_K=4~~\mbox{(in 1D)}\quad\mbox{and}\quad \alpha_K=\frac83~~\mbox{(in 2D)},
\end{equation}
which are the cases $d=1,2$ of the threshold $\alpha_K(d)=\frac{8}{d+1}$. For $d \leq 3$ this coincides with the existence threshold for mass-constrained minimizers in \cite[Theorem 1.2]{FJMM2022} (see further discussion in Remark \ref{R:higher-dimensions}). They are consistent with our 1D numerics from
\cite{KPRS} and related to the value $p_*$ ($p=\alpha+1$) studied in \cite{SP2020,STCK2022}, see
discussion in Remarks \ref{R:normGS} (1D) and \ref{R:Knapp-threshold} (2D). 

We remark further about the 2D rates, since they become essential for our numerical study. The unrestricted rate in 2D from Part 2 of Theorem \ref{ThmA} gives 
\begin{equation}\label{E:rates-2D-nr}
\qquad M(Q_\epsilon^{\rm nr}) \approx \epsilon^{\frac{2}{\alpha} - \frac{3}{4}}
\qquad (\text{2D, nonradial}, ~0<\alpha<4).
\end{equation}
For the radial class, Part 1 of Theorem \ref{ThmA} combined with the radial quotient asymptotics of Lenzmann-Weth \cite{LW2021} and Mandel-Oliveira~e~Silva \cite{MO2023} gives 
\begin{equation}\label{E:rates-2D-r}
\qquad M(Q_\epsilon^{\rm rad})\approx \left\{
\begin{array}{ll}
\epsilon^{\frac2{\alpha}-1}, &  0<\alpha<2,\\
|\ln\epsilon|^{-1},      &  ~\alpha=2, \\
\epsilon^{\frac1{\alpha}-\frac12}, & ~ \alpha>2.
\end{array}
\right.
\quad (\text{2D, radial}, ~ \alpha>0)
\end{equation}
Note that the radial mass threshold in 2D occurs at $\alpha=2$ (unlike $\alpha_K=\frac83$ in \eqref{E:alpha-K}), with the mass still tending to zero logarithmically at that endpoint.  
We use these rates in our numerical study, to which we turn next, to distinguish different 
types of ground state solutions to \eqref{E:groundstate}, as well as to interpret possible branching in one and two dimensions.
\smallskip

We search numerically for {\it nonradial} ground state solutions in the {\it cubic} two-dimensional problem, craftily designing the initializations (e.g., the Knapp-type concentration caps) and carefully tracking various branches, then  compare them with radial and symmetry-constrained branches, and study their characteristics. We pay special attention to the conserved quantities such as mass and energy as functions of $b$, and more so, studying dependence of energy on mass (mass-energy diagrams), which later becomes essential for higher nonlinearities. 
We then examine {\it quartic} nonlinearity in 2D, investigate radial and nonradial solutions in that setting, and observe that {\it branching} behavior appears in mass-energy diagrams, similar to the ones we found in the 1D setting in \cite{KPRS}. After that, for higher nonlinearities $\alpha \geq 4$, we compute the ground state solutions, which in all our simulations turn out to be  radial (note that for $\alpha \geq 4$ in 2D the radial and unrestricted quotients have the same $\ep$-rate, \cite[Theorem 1.3]{LW2021}, thus, the rate comparison that produces symmetry breaking for $\alpha<4$ no longer distinguishes the two branches) and further examine the branching phenomenon.
\smallskip

Our detailed numerical findings are as follows:

\begin{itemize}
\item[(i)] In the cubic case ($\alpha=2$) in 2D, using a Knapp-type cap-concentrated initial guess and delicate tracing, we numerically construct a {\it nonradial ground state} branch near $b = 1.001$ (see Figure \ref{figflat1001}) 
and also find that it has lower energy than the computed radial solutions branch with the same mass. To our knowledge, this is the first numerical construction of ground states without radial symmetry, which confirms the radial symmetry-breaking result of Lenzmann \& Weth \cite[Theorem 1.2]{LW2021} (see Theorem \ref{Thm1} stated above). We note that examples that we consider in 2D have $G_1$ symmetry (even coordinate-wise), thus, our numerical findings do not address the distinction between the $G_1$-restricted vs. unrestricted ground states or minimization problem (see  Remark \ref{R:attainment} and \cite[Appendix C]{MO2023}).

\item[(ii)] We compute angular-mode branches with cosine-type symmetries $m=0,1,2,3,4,6,8$, and observe that within the angular-mode class, the radial branch ($m=0$) has the lowest energy. The higher angular modes could be considered as excited or symmetry-constrained branches.

\item[(iii)] In the quartic case ($\alpha=3$) in 2D, we find that all our computed solutions have a turning point and develop branching in mass-energy diagrams, indicating that there is a lower-energy ({`stable'}) branch for sufficiently large values of $b$ and an upper-energy ({`unstable'}) branch for $b$ decreasing down to $1$. (While we do not investigate any stability of branches in this work, for consistency, we use the terminology of `stable' and `unstable'  branches motivated by our 1D paper \cite{KPRS}.)  

\item[(iv)] In the same quartic case ($\alpha=3$) in 2D, for very small positive $\ep$, we find two-peak and four-peak nonradial branches of solutions. In this regime when $\ep \to 0^+$, the two-peak solutions have lower energy than the radial solution (that have also very small $\ep$), which lie on the {\it unstable} radial branch of solutions. This also confirms the radial symmetry-breaking result for small $\ep$ in \cite{LW2021}. However, comparing with the stable radial branch of ground state solutions, which has larger values of $b=1+\ep$, we find that the radial ground states on the stable branch with the same mass as the nonradial (but different $b$) have lower energy. Thus, in the language of mass-constrained ground states, the normalized ground state solutions appear to be radial in this case of nonlinearity (exactly due to branching). We also find other nonradial and angular mode solutions, which indicate possible nonradial excited states.

\item[(v)] For higher nonlinearities ($\alpha=4,5,6$) in 2D, we only find radial solutions with our approach considering various initializations and tracing, which is consistent with the analytical results on ground state solutions being nonradial in the range $0 < \alpha <  4$ in 2D in Theorem \ref{Thm1} (or \cite[Theorem 1.2]{LW2021}). It is also consistent with the leading constants of the radial and unrestricted mass asymptotics coinciding for $\alpha \geq 4$, which by Corollary \ref{C:sharp-mass} is equivalent to constant functions being extremizers of the Stein - Tomas inequality on $\mathbb S^1$, a well-known conjecture at $\alpha=4$ (e.g., \cite{Becker2025}), see Remark \ref{R:constants-radiality}. 
\end{itemize}


The structure of this paper is as follows: in Section~\ref{S:GS} we review the Pokhozhaev identities, the least-action, best-constant and mass-constrained notions of ground states, and the positive and oscillatory ground state regimes. 
In Section~\ref{S:asym} we discuss the 
quotient asymptotics, prove a quotient-to-mass relation Lemma~ \ref{L:quotient-to-mass} together with its generalization (Theorem \ref{T:quotient-mass-general}), and give the proofs for lower bounds via Hausdorff-Young (Theorem \ref{Thm:B1}) and Gagliardo-Nirenberg (Theorem \ref{Thm:GN-lower}) inequalities. We then combine these bounds with the trial-function constructions to prove the quotient and mass rates in Theorem \ref{ThmA}. 
In Section~\ref{S:1dGS} we review the 1D ground states and obtain their mass asymptotics, including the threshold case $\alpha=4$.
In Section~\ref{S:2dGS} we adapt these facts to two dimensions, show the resulting mass rates, and also describe the initial approximations used in our computations with the motivation for such initializations.
In Section ~\ref{S:2d-cubic} we construct and compare radial, angular-mode and nonradial branches for the {\it cubic} case. Section \ref{S:quartic} considers the {\it quartic} case,  where we obtain two-peak and four-peak nonradial branches of solutions and then investigate their relation to the radial branch; we also find angular-mode branches, which could represent excited states or a specific-symmetry constrained ground states. 
In Section ~\ref{S:higher} we discuss higher nonlinearities, for which our computations only find radial solutions, we also observe branching in mass-energy diagrams in these cases. In Section \ref{S:branching} we discuss {\it branching} and normalized minimization.  
In the appendix we give the detailed trial-function calculations for the quotient upper bounds, using the Knapp-type wave packets in 1D, Knapp caps in 2D, and a Bessel construction in the radial 2D case. We conclude with a table collecting all the small-$\ep$ rates, in the unrestricted and radial classes, in all dimensions.
\smallskip

{\bf Acknowledgments.} 
The  work of C.K. was partially supported by 
 the ANR-17-EURE-0002 EIPHI and by the ANR project 
ISAAC-ANR-23-CE40-0015-01.
The research of S.R. was partially supported by the NSF grant 
DMS-2452782. She would like to thank the hospitality of the IMB, 
Universit\'e Bourgogne Europe, where a significant part of the research for this paper was carried out.  


\section{Notions of ground states and loss of positivity}\label{S:GS}

We first discuss ground state properties in any dimension, including 
Pokhozhaev identities, equivalent definitions, and positivity vs. sign-changing regimes. 

\subsection{Pokhozhaev identities}
Recalling that for a standing wave $u(t,x)=e^{ibt}Q(x)$, $x \in \mathbb R^d$, the profile $Q$ solves the elliptic equation \eqref{E:groundstate}, we multiply \eqref{E:groundstate} by $Q$ or $x\cdot\nabla Q$ and integrate to obtain two Pokhozhaev identities,  
\begin{equation}\label{E:Pokh1-d}
\|\Delta Q\|_{L^2(\R^d)}^2
-2a\|\nabla Q\|_{L^2(\R^d)}^2 + b\|Q\|_{L^2(\R^d)}^2
-\|Q\|_{L^{\alpha+2}(\R^d)}^{\alpha+2}=0,
\end{equation}
\begin{equation}\label{E:Pokh2-d}
\frac{4-d}{2}\|\Delta Q\|_{L^2(\R^d)}^2 + a(d-2)\|\nabla Q\|_{L^2(\R^d)}^2
-\frac{d}2 b\|Q\|_{L^2(\R^d)}^2
+\frac{d}{\alpha + 2}\|Q\|_{L^{\alpha+2}(\R^d)}^{\alpha + 2}=0.
\end{equation}
Solving for $\|\Delta Q\|^2_{L^2}$  
from \eqref{E:Pokh1-d} and \eqref{E:Pokh2-d} and recalling the energy \eqref{E:E}, we obtain  
\begin{equation}\label{E:energy-grad}
E(Q)=\frac{b}{2}\,\frac{d\alpha-8}{8+(4-d)\alpha}\,
\|Q\|_{L^2(\R^d)}^2 - \frac{2a\alpha}{8+(4-d)\alpha}\,
\|\nabla Q\|_{L^2(\R^d)}^2.
\end{equation}
We note that the denominator is positive in the (scale-invariant) $H^2$-subcritical range; for $d\leq4$ this is automatic for every $\alpha>0$, while for $d \geq 5$ it corresponds to $0<\alpha < 8/(d-4)$. In the scale-invariant $L^2$-critical  case ($\alpha=8/d$) 
formula \eqref{E:energy-grad} reduces to
\begin{equation}\label{E:critical-energy}
E(Q)=-\frac a2 \|\nabla Q\|_{L^2(\R^d)}^2,
\end{equation}
and thus, the energy of a nontrivial ground state is positive if $a<0$ and negative if $a>0$; in the scaling-invariant case $a=0$, the ground state has zero energy.

From the same Pokhozhaev identities, we also have
\begin{equation}\label{EM-relation}
E(Q) =- \frac{b}2 \, M(Q) + \frac{\alpha}{2(\alpha+2)} \|Q\|_{L^{\alpha+2}(\R^d)}^{\alpha+2}. 
\end{equation}

Thus, a trivial consequence from \eqref{EM-relation} is
$$
\qquad E(Q) + \frac{b}2 M(Q) > 0, \quad \mbox{or ~~equivalently,} \quad E(Q) >  -\frac{b}2 M(Q)
$$
for any non-trivial solution of \eqref{E:groundstate}. 
This is one of the first numerical checks that can be done when searching for the ground state solutions of \eqref{E:groundstate}. 
In our computations below this holds, which can be seen from the provided figures. 

Secondly, in our computations, we also check the error $\mathcal E(Q)$ for any solution $Q$ of \eqref{E:groundstate} by showing that 
\begin{equation}\label{E:Error}
\mathcal E (Q) := E(Q)+\frac{b}{2}\|Q\|_{L^2(\R^d)}^2
-\frac{\alpha}{2(\alpha+2)}\|Q\|_{L^{\alpha+2}(\R^d)}^{\alpha+2}.
\end{equation}
In the 2D case considered here, the absolute value of the error 
\eqref{E:Error} is typically on the order of $10^{-10}$ or less for 
$\epsilon \gg 1$ (for $\epsilon \sim 0$, both the residual of the 
elliptic equation and the Pokhozhaev identities drop). Previously, in \cite{KPRS} we also computed the same error in the 1D case and confirmed a similar accuracy.

\subsection{Definitions of ground states}\label{S:Def-GS}

For $b>0$ we define an action (or Weinstein functional) 
\begin{equation}\label{E:action}
S_b(Q) =E(Q)+\frac b2\|Q\|_{L^2(\R^d)}^2
\end{equation}
or, substituting energy, 
\begin{equation*}\label{E:action-expanded}
S_b(Q) = \frac12\|\Delta Q\|_{L^2(\R^d)}^2 -a\|\nabla Q\|_{L^2(\R^d)}^2 +\frac b2\|Q\|_{L^2(\R^d)}^2
-\frac{1}{\alpha+2}\|Q\|_{L^{\alpha+2}(\R^d)}^{\alpha+2}.
\end{equation*}

Note that if $Q$ is a solution of \eqref{E:groundstate}, then using the Pokhozhaev identity \eqref{E:Pokh1-d}, the first three terms in $S_b$ could be substituted by the potential norm to get
\begin{equation}\label{E:S-Potential}
S_b(Q)= \frac{1}{2}\|Q\|_{L^{\alpha+2}(\R^d)}^{\alpha+2} -\frac{1}{\alpha+2}\|Q\|_{L^{\alpha+2}(\R^d)}^{\alpha+2} \equiv \frac{\alpha}{2(\alpha+2)}\|Q\|_{L^{\alpha+2}(\R^d)}^{\alpha+2}.
\end{equation}
This is useful in the action minimization or definition of ground states, which we discuss next. 
\smallskip

There are several definitions of ground state solutions in the literature, so we briefly review them here and make some comparison. 

\begin{definition}[{Least-action critical point}]\label{D:LA}
A nonzero solution $Q\in H^2(\R^d)$ of \eqref{E:groundstate} is called a {\rm least-action} ground state if
\begin{equation}\label{E:least-action-def}
S_b(Q) = \inf \left\{S_b(u):~ u\in H^2(\R^d)\setminus\{0\},~ S_b'(u)=0\right\},
\end{equation}
where the derivative is understood as the Fr\'echet derivative. 

Equivalently, a nonzero solution $Q\in H^2(\mathbb R^d)$ is a least-action ground state, 
if 
$$
S_b(Q)\leq S_b(\widetilde Q)
$$
for every nonzero solution $\widetilde Q\in H^2(\mathbb R^d)$ of the equation \eqref{E:groundstate}.
\end{definition}

For the next definition, we introduce some notation. 
Following \cite{LW2021}, we define the quadratic form
\begin{equation}\label{E:q-ab-new}
q_{a,b}(u):= \|\Delta u\|_{L^2(\R^d)}^2
-2a\|\nabla u\|_{L^2(\R^d)}^2 + b\|u\|_{L^2(\R^d)}^2.
\end{equation}
The corresponding Fourier multiplier, or symbol of the linear part in the equation \eqref{E:groundstate}, is 
\begin{equation}\label{E:symbol-m}
m_{a,b}(\xi)=|\xi|^4-2a|\xi|^2+b.
\end{equation}
We mention that the quadratic form $q_{a,b}$ is positive-definite when
the multiplier is positive, i.e., when  $b>0$ and $a<\sqrt b$, see further discussion on positivity in \S \ref{S:positive}.

For a given nonlinearity $\alpha$, we consider the Weinstein quotient ${q_{a,b}(u)}/{~\|u\|_{L^{\alpha+2}(\R^d)}^2}$, sometimes also referred to as the (nonlinear) Rayleigh quotient, for example, in \cite{MO2023} (the classical linear Rayleigh quotient would have an $L^2$ norm in the denominator), and define the infimum value of it as
\begin{equation}\label{E:R-ab}
R_{a,b}(\alpha):= \inf_{u \in H^2(\R^d)\setminus\{0\}} \frac{q_{a,b}(u)}{~\|u\|_{L^{\alpha+2}(\R^d)}^2}.
\end{equation}

\begin{definition}[{Best-constant optimizer}]\label{D:BCo}
A non-trivial function $v \in  H^2(\R^d)$ 
is an optimizer for \eqref{E:R-ab} if the infimum in \eqref{E:R-ab} is attained, that is, 
$$
R_{a,b}(\alpha) \equiv \frac{q_{a,b}(v)}{~\|v\|_{L^{\alpha+2}(\R^d)}^2}.
$$
\end{definition}

We show the equivalence of Definitions \ref{D:LA} and \ref{D:BCo}, whenever the infimum in \eqref{E:R-ab} is attained.  

Fix the power $\alpha$. Suppose first that $v_0$ is an optimizer for \eqref{E:R-ab} as in Definition \ref{D:BCo}, normalized by
$\|v_0\|_{L^{\alpha+2}(\R^d)}=1$. Then $R_{a,b} = q_{a,b}(v_0)$. 
Considering the constrained minimization problem 
$$
q_{a,b}(v_0) = \inf \{ q_{a,b}(v) :~ \|v\|_{L^{\alpha+2}(\mathbb R^d)}=1 \} = R_{a,b},
$$
and the first variation of the quadratic form and the constraint 
(with the Lagrange multiplier $\lambda = R_{a,b}$), 
we obtain the Euler-Lagrange equation for $v_0$,
\begin{equation}\label{E:v-Euler-new}
\Delta^2 v_0 + 2a\Delta v_0 + bv_0 =R_{a,b}|v_0|^\alpha v_0,
\end{equation}
where $v_0$ solves the equation weakly, then the standard elliptic regularity upgrades it to the required smoothness.  
Defining
\begin{equation}\label{E:Q-v}
Q_0=R_{a,b}^{1/\alpha}v_0,
\end{equation}
we obtain that $Q_0$ solves \eqref{E:groundstate}, and hence,  
by \eqref{E:S-Potential}, 
\begin{equation}\label{E:Q_0}
S_b(Q_0) = \frac{\alpha}{2(\alpha+2)} \|Q_0\|^{\alpha+2}_{L^{\alpha+2}(\mathbb R^d)} = 
\frac{\alpha}{2(\alpha+2)} \| R^{1/{\alpha}}_{a,b} v_0\|^{\alpha+2}_{L^{\alpha+2}(\mathbb R^d)}
= \frac{\alpha}{2(\alpha+2)} \,R^{(\alpha+2)/{\alpha}}_{a,b}.
\end{equation}

On the other hand, every nonzero critical point $Q$ of $S_b$ is a solution of \eqref{E:groundstate}, and by the first Pokhozhaev identity \eqref{E:Pokh1-d}, we have
\begin{equation}\label{E:q-equals-Lp-new}
q_{a,b}(Q)=\|Q\|_{L^{\alpha+2}(\R^d)}^{\alpha+2},
\end{equation}
and hence, 
$$
R_{a,b} \leq \frac{q_{a,b}(Q)} {\|Q\|_{L^{\alpha+2}(\mathbb R^d)}^2}
= \|Q\|_{L^{\alpha+2}(\mathbb R^d)}^\alpha.
$$
Raising to the power $\frac{\alpha+2}{\alpha}$, we get
\begin{equation}\label{E:Lp-R}
R_{a,b}^{(\alpha+2)/\alpha} \leq \|Q\|_{L^{\alpha+2}(\mathbb R^d)}^{\alpha+2}.
\end{equation}
Thus, combining with \eqref{E:Q_0} and \eqref{E:S-Potential}, we obtain
\begin{equation}\label{E:S_b-lower}
S_b(Q_0) \leq S_b(Q),
\end{equation}
implying that $Q_0$ is a least-action critical point. 

To the opposite, if $Q$ is a least-action critical point, then $S_b(Q) \leq S_b(Q_0)$, where $Q_0$ is defined by \eqref{E:Q-v} with $v_0$ being the optimizer of the quotient \eqref{E:R-ab}. Then the previous lower bound \eqref{E:S_b-lower} holds also, and we obtain equality $S_b(Q) = S_b(Q_0)$, which also implies that  
\begin{equation}\label{E:potential-R}
\|Q\|_{L^{\alpha+2}(\R^d)}^{\alpha+2}=R_{a,b}^{(\alpha+2)/\alpha}.
\end{equation}
It follows that
$$
\frac{q_{a,b}(Q)} {\|Q\|_{L^{\alpha+2}(\R^d)}^2}
= R_{a,b},
$$
so the normalized function
$$
v =\frac{Q}{\|Q\|_{L^{\alpha+2}(\R^d)}} = R_{a,b}^{-1/\alpha}Q
$$
is an optimizer for \eqref{E:R-ab}. 
Therefore, we conclude that the least-action and best-constant optimizer (on which the quotient infimum is attained) definitions are equivalent.
(In a similar fashion one can also show that $R_{a,b}(\alpha)^{(\alpha+2)/\alpha} = C_\alpha \, \inf_{u \in H^{2}(\mathbb R^d), \|u\|_{L^{\alpha+2}}=1} \sup_{t \geq 0} S_b(tu)$, $C_\alpha = 2(\alpha+2)/\alpha$, with $t = R_{a,b}^{1/\alpha}$ taken at the quotient optimizer $v_0$ (also $L^{\alpha+2}$ normalized), see \cite{LW2021}.)

We record that for ground states defined as above, the action and the best constant are connected as 
\begin{equation}\label{E:least-action-value}
S_b(Q) =\frac{\alpha}{2(\alpha+2)}R_{a,b}(\alpha)^{(\alpha+2)/\alpha}.
\end{equation}

A third common definition of a ground state is that of a {\it normalized} ground state.

\begin{definition}[{Normalized or mass-constrained minimizer}]\label{D:norm-GS}
For a prescribed mass $m>0$, consider
\begin{equation}\label{E:mass-constrained-problem}
I(m) = \inf\left\{E(u):\; u\in H^2(\R^d),\; \|u\|_{L^2(\R^d)}^2=m\right\}.
\end{equation}
If, for a given $m$, the infimum is attained at $Q_m \in H^2(\R^d)$  
for some Lagrange multiplier $b_m\in\R$, then $Q_m$ is called a {\rm normalized} minimizer and 
$Q_m$ satisfies an Euler-Lagrange equation of the form
\begin{equation}\label{E:normalized-EL}
\Delta^2Q_m+2a\Delta Q_m+b_mQ_m-|Q_m|^\alpha Q_m=0.
\end{equation}
\end{definition}

The relation between \eqref{E:least-action-def} and \eqref{E:mass-constrained-problem} is more delicate than the equivalence with the best-constant and least-action optimizers.  In the mass-subcritical range 
$0<\alpha<\frac8{d}$,
the energy is bounded from below on fixed-mass sets $S_\mu = \{u \in H^2(\R^d), ~ \|u\|_{L^2(\R^d)}^2=\mu \}$. 
Indeed, by interpolation,
$$
\|\nabla u\|_{L^2(\mathbb R^d)}^2 \leq
\|u\|_{L^2(\mathbb R^d)} \|\Delta u\|_{L^2(\mathbb R^d)}
= \mu^{1/2}\|\Delta u\|_{L^2(\mathbb R^d)}.
$$
Using the Gagliardo-Nirenberg inequality (for the Laplacian) 
$$
\|u\|_{L^{\alpha+2}(\mathbb R^d)}^{\alpha+2}
\leq C_{GN} \|\Delta u\|_{L^2(\mathbb R^d)}^{d\alpha/4}
\|u\|_{L^2(\mathbb R^d)}^{\alpha+2-d\alpha/4},
$$
we get 
$$
\|u\|_{L^{\alpha+2}(\mathbb R^d)}^{\alpha+2}
\leq C_{GN} \mu^{(\alpha+2-d\alpha/4)/2} \|\Delta u\|_{L^2(\mathbb R^d)}^{d\alpha/4}.
$$
Then 
\begin{equation}\label{E:Ebelow}
E(u) \geq \frac12\|\Delta u\|_{L^2(\mathbb R^d)}^2
-|a|\mu^{1/2}\|\Delta u\|_{L^2(\mathbb R^d)}
-\frac{C_{GN}}{\alpha+2} \mu^{(\alpha+ 2 - d\alpha/4)/2}
\|\Delta u\|_{L^2(\mathbb R^d)}^{d\alpha/4},
\end{equation}
and since $d\alpha/4<2$, we obtain
$$
\inf_{u\in S_\mu}E(u)>-\infty.
$$

To show that the infimum is attained, one typically applies the concentration-compactness method:
starting with a minimizing sequence, one rules out vanishing and dichotomy, and then recovers compactness up to translations.  This is the classical strategy of Lions, see also the exposition in
\cite{Caz-book2003}. In the fourth-order and mixed-dispersion setting, this compactness scheme has been used in several related normalized variational problems; see, for instance,
\cite{BCSN2018,BCGJ2019,FJMM2022,SP2020,STCK2022}. 
Thus, whenever the fixed-mass minimization problem is compact, one obtains a normalized minimizer as in Definition \ref{D:norm-GS}. When compactness fails or when the mass lies below the threshold appearing in the normalized theory, such a minimizer may not exist.

At the $L^2$-critical exponent $\alpha=8/d$, a similar reasoning to \eqref{E:Ebelow} shows that there is a threshold for the mass, 
while for $\alpha>8/d$ the fixed-mass energy as in \eqref{E:Ebelow} can be unbounded from below.

Reiterating, we note that it might happen that $Q_m$ may not exist 
for a given $m$, since all solutions may have higher mass than the prescribed value $m$. However, if it does exist, then the normalized minimizer specifies the value of $b=b_m$, for which the energy is minimized. This becomes important in view of our results in 1D in \cite{KPRS} (see also below similar behavior in 2D), where we observed branching of ground states, using the first definition. The normalized minimizers would be located only on the lower energy branch, if at all (see further discussion on this in Section \ref{S:branching}) and for sufficiently large values of $b$.

\subsection{Positive and sign-changing regimes}\label{S:positive}
We next discuss the positivity or sign-changing and symmetry properties of ground states, which depend strongly on the parameters $a$ and $b$ in equation \eqref{E:groundstate}. Recalling the multiplier $m_{a,b}$ from \eqref{E:symbol-m}, and writing it as 
\begin{equation}\label{E:m-ab-factor}
m_{a,b}(\xi) = (|\xi|^2-a)^2 +(b-a^2),
\end{equation}
ensures its positivity when $b>0$ and $a<\sqrt b$ as well as the 
coercivity of the quadratic form $q_{a,b}$. We note that positivity 
of the symbol does not guarantee the positivity-preserving property of the corresponding Green's function, which is responsible for the positive ground state solutions. We review that next. 

$\bullet$ \underline{Positive ground states} (or the positivity-preserving regime)
Let
\begin{equation}\label{E:positive-regime}
a\leq -\sqrt b < 0.
\end{equation}
Then we have the following factorization
\begin{equation*}
\qquad \Delta^2+2a\Delta+b = (-\Delta+\lambda_+)(-\Delta+\lambda_-),
\quad
\lambda_\pm=-a \pm\sqrt{a^2-b}.
\end{equation*}
Note that for $a< -\sqrt b$, both roots $\lambda_+$ and $\lambda_-$ are strictly positive (since $a<0$) and at $a=-\sqrt b$ the two roots coincide. Therefore, if $f \geq 0$ and not identically zero ($f \not\equiv0$), then $(-\Delta+\lambda_\pm)^{-1}f > 0$,  
and so their composition is positivity preserving: $(\Delta^2+2a\Delta+b)^{-1}f>0$.  
In a sense, this gives a replacement for the maximum principle, which is not generally available for fourth-order equations (see, e.g., \cite{BerchioGazzolaMitidieri2006,GazzolaGrunauSweers2010}). 
This positivity argument can also be viewed through the cooperative-system reformulation used by Bonheure-Nascimento \cite{BN2015}, see their later work \cite{BCSN2018}. Thus, in this regime, variational ground state could be chosen to be positive (up to a constant phase), and standard elliptic arguments give smoothness and exponential decay. Symmetry, uniqueness, and non-degeneracy require additional hypotheses and are treated in several related works; see, for example,
\cite{BCSN2018, LS2021}. 

Since the symmetry is the main theme of this paper, we mention the result of Lenzmann-Sok \cite{LS2021} on the Fourier rearrangement, which gives a
useful symmetry statement for least-action (or best-constant) ground states in a class of higher-order NLS equations.  In particular, \cite[Theorem~3]{LS2021}
shows that, for operators with {\it radial} and {\it increasing} Fourier multipliers and for {\it even} nonlinear powers, every least-action ground state is equal to its Fourier rearrangement (up to translation and multiplication by a constant phase).
Hence, such ground states are radial and real-valued, up to these symmetries. Note that this includes a larger class of Fourier multipliers than described in \eqref{E:positive-regime}, it would include any $a \leq 0$ (and $b>0$), since $m_{a,b}$ is monotone and increasing in that regime. It is only when $a>0$, the function $m_{a,b}$ loses its single minimum, or monotonicity in $|\xi|$.  

A related notion about symmetries is {\it evenness} of ground states, which can hold not only for the positive ground states, but for sign-changing as well. Before addressing that, we discuss the regime where the ground states are sign-changing. 
\smallskip

$\bullet$ \underline{Sign-changing ground states} (oscillatory regime). ~~ In the range
\begin{equation}\label{E:oscillatory-regime}
-\sqrt b<a<\sqrt b,
\end{equation}
the quadratic form $q_{a,b}$ remains coercive, since
$m_{a,b}(\xi)=(|\xi|^2-a)^2+(b-a^2)>0$, however, the Green's function 
is no longer positivity preserving and will produce oscillatory 
sign-changing behavior. This can be seen when $|x|$ becomes large and goes to infinity, and hence, $Q(x) \approx 0 $, thus, the nonlinearity does not influence as much as the linear part of the equation, and we can consider $Q^{(4)}+2a \Delta Q + bQ \approx 0$. In one dimension, for simplicity, substituting $Q(x) \sim e^{\lambda x}$, we obtain a characteristic polynomial in $\lambda$, with roots $\lambda^2=-a\pm i\sqrt{b-a^2}$, or more precisely, 
$$
\qquad \lambda=\pm \kappa\pm i\nu, \qquad 
\kappa=\left(\tfrac{\sqrt b-a}{2}\right)^{1/2},
\quad
\nu=\left(\tfrac{\sqrt b+a}{2}\right)^{1/2}.
$$
Therefore, the ground state solutions have exponentially decaying but oscillatory tails of the form $e^{-\kappa |x|} \cos(\nu 
|x|+\theta)$, and the corresponding Green's function changes sign. In higher dimensions, the same complex characteristic roots produce the exponential and oscillatory factors, with an additional dimension-dependent algebraic multiple. Hence, the variational ground states would be oscillatory and sign-changing (even though the quadratic form is positive and coercive).

We now examine the regime $0<a<\sqrt b$ more closely, since the symbol $m_{a,b}$ changes its single minimum shape and becomes non-monotone for $|\xi |$. From the representation $m_{a,b}(\xi) = (|\xi|^2-a)^2+(b-a^2)$, it is easily seen that the {\it minimum} of this symbol occurs on the sphere
\begin{equation*}\label{E:minimizing-sphere}
|\xi|^2=a.
\end{equation*}
Thus, the least contributing Fourier modes will be the nonzero-frequency modes located at $|\xi|=\sqrt a$.  
The distribution of locations of modes in dimension $d \geq 2$ gives a possibility of symmetry breaking via cap-concentrating modes (i.e., they concentrate only around certain parts of the circle $|\xi|=\sqrt a$, which is more favorable for the least action than uniform radial representation) as in the famous Knapp example used to prove the existence of non-radial ground states in \cite{LW2021}. 

One useful renormalization for $Q$ and the ground state equation \eqref{E:groundstate} when $a>0$ (see also \cite[Remark 2.1]{KPRS}) is  
\begin{equation}\label{E:ab-rescale}
Q(x) = a^{\frac2{\alpha}} \tilde Q(\sqrt a \, x).
\end{equation}
Then $\tilde Q$ solves 
$$
\Delta^2 \tilde Q + 2 \Delta \tilde Q + \frac{b}{a^2} \tilde Q - |\tilde Q|^{\alpha} \tilde Q = 0.
$$
Thus, the relevant parameter is ${b}/{a^2}$. For simplicity and following the notation from Lenzmann and Weth \cite{LW2021}, we set 
\begin{equation}\label{E:LW-regime}
a=1, \qquad b=1+\epsilon, \quad \epsilon>0. 
\end{equation}
Then the Fourier multiplier for the quadratic term becomes
\begin{equation}\label{E:LW-symbol}
m_{1,1+\epsilon} = |\xi|^4-2|\xi|^2+1+\epsilon =(|\xi|^2-1)^2+\epsilon. 
\end{equation}

Thus, at least intuitively, to minimize the action, the Fourier modes should concentrate near the unit sphere: 
$$
|\xi|=1.
$$
When $\epsilon>0$ is small, a ground state profile 
should have most of its Fourier contribution or Fourier mass 
in a thin neighborhood of this unit sphere.  There are, however, different ways to do this. A {\it radial} profile distributes its Fourier mass around the full annulus $|\xi|\approx1$, so then in physical space this produces an oscillatory Bessel-type profile (and can have different symmetry modes as we show below in \S \ref{S:angular-modes}).  A {\it non-radial} Knapp-type profile instead concentrates its Fourier mass on a small cap of the unit sphere, together with the opposite cap in order to produce a real-valued function (e.g., $e^{i x}+e^{-ix} = 2 \cos (x)$). So, in two dimensions, caps near $(1,0)$ and $(-1,0)$ correspond to a physical-space wave packet of the form
\begin{equation}\label{E:Knapp-1}
Q_\epsilon(x,y) \sim \epsilon^{1/\alpha} W(\epsilon^{1/2} x,\epsilon^{1/4} y )\cos x,
\end{equation}
where (see details in the appendix \ref{A:sub-formal-2d-nr}) the scaling amplitude $\epsilon^{1/\alpha}$ comes from the nonlinear terms, the slowly varying localized profile is denoted by $W$, with different scaling in $x$ and $y$ to accommodate Knapp cap geometry: the normal width of the cap is of order $\epsilon^{1/2}$, while the tangential width is of order $\epsilon^{1/4}$, then in the physical space it becomes a long tube along the $x$-axis with the corresponding scaling.   
This cap concentration is the classical Knapp example in harmonic analysis (Fourier restriction theory), originating from unpublished work of A. W. Knapp and presented (with the permission) by Strichartz in his famous 1977 paper \cite{Strichartz1977}, see also \cite[Chapter~7]{Wolff2003}. This is exactly what Lenzmann and Weth use to construct nonradial optimizers for the quotient \eqref{E:R-ab} in \cite[Proposition~3.2]{LW2021}. 
In \cite[Theorem~1.2]{LW2021} (and what we restated in the introduction in Theorem \ref{Thm1}), they prove that for $d \geq 2$ and nonlinearities with 
$$
0< \alpha< \frac{4}{d-1},
$$
all ground states are nonradial for $\epsilon>0$ sufficiently small. In particular, in two dimensions this range is $0<\alpha<4$, which includes the cubic case $\alpha=2$. We also use the quotient asymptotics from \cite[Theorem~1.3]{LW2021} and refinements from \cite{MO2023} to further confirm the non-radial or radial branches of the profiles to the equation \eqref{E:groundstate} and derive the mass asymptotics.

\section{Asymptotics for small $\ep$}
\label{S:asym}
To distinguish the different types (or branches) of ground states such as non-radial, radial, $G_k$-symmetric, we need information about the behavior of certain quantities (such as mass or action) for small-$\ep$ asymptotics.  
For that, we first discuss the known asymptotic rates for the Weinstein quotient 
\eqref{E:R-ab}, for the corresponding action $S_b$ \eqref{E:least-action-def}, as well as the potential $L^{\alpha+2}$ norm. We then prove the quotient-to-mass relation in Section 
\ref{S:M-to-R}, establish the lower bounds by Hausdorff-Young and Gagliardo-Nirenberg in
Sections \ref{S:HY-lower} and \ref{S:GN-lower}, and combine the bounds in
Section \ref{S:rate-consequences} to obtain the mass rates in
Theorem \ref{ThmA}. 
(The examples of trial-function which provide the upper-bounds, with detailed calculations, are in Appendix \ref{S:appendix-2}.)

\subsection{Quotient asymptotics}\label{S:R-asym}
Recalling $a=1$ and $b=1+\ep$, we set
\begin{equation}\label{E:R-epsilon}
\mathcal R_\epsilon(\alpha):=R_{1,1+\epsilon}(\alpha)
= \inf_{u\in H^2(\mathbb R^d)\setminus \{0\}}
\frac{\int_{\mathbb R^d}
((|\xi|^2 - 1)^2 + \ep )|\widehat u(\xi)|^2\,d\xi}{\|u\|_{L^{\alpha+2}(\mathbb R^d)}^2}.
\end{equation}

We denote by $\mathcal R_\epsilon^{\rm rad}(\alpha)$ the same quotient restricted to {\it radial} functions, 
$$
u \in H^2_{\rm rad}(\mathbb R^d)\setminus\{0\}.
$$

Similarly, we write $\mathcal R_\ep^{\mathcal X}(\alpha)$ for the quotient restricted to a block symmetry subspace $\mathcal X$ such as $G_1$-symmetric subspace mentioned in the introduction.
\smallskip


\begin{remark}\label{R:approx}
In what follows, the sign $\approx$ in $A_\ep \approx B_\ep$ means that there exist lower and upper bounds with the same rate $\gamma$, i.e., there exist constants $0< c < C <\infty$, independent of $\ep$, such that $cB_\ep \leq A_\ep \leq C B_\ep$ for all sufficiently small $\ep>0$. The constants $c,C$ may depend on the dimension, the nonlinearity power $\alpha$, and the symmetry class. 
\end{remark}

{\bf The one-dimensional quotient.}
The result of Fern\'andez et al. in \cite[Proposition~3.7]{FJMM2022} gives the following upper bound in 1D:  
\begin{equation}\label{E:FJMM-upper}
\mathcal R_\epsilon(\alpha)
\leq C_\alpha \epsilon^{1/2} \quad \mbox{as } \ep \to 0^+.
\end{equation}
Then for the corresponding least-action ground states by \eqref{E:least-action-value} we have 
\begin{equation}\label{E:S-Q-FJMM}
S_{1 + \epsilon}(Q_\epsilon)
\leq C_\alpha \epsilon^{\frac{\alpha+2}{2\alpha}} 
\quad \mbox{and} \quad 
\|Q_\epsilon\|_{L^{\alpha+2}(\mathbb R)}^{\alpha+2}
\leq C_\alpha \epsilon^{\frac{\alpha+2}{2\alpha}}.
\end{equation}

We further refine these asymptotics, obtaining a sharper upper bound than \eqref{E:FJMM-upper}, which explains the mass dependence on $b$ (equivalently, on $\ep$) that we obtained numerically in 1D in \cite{KPRS}. 
For that recall that in dimension one, the minimum set of the symbol $m_{1,1+\ep}$ consists of the two points $\xi=\pm1$. Incorporating this into a wave-packet similarly as in \eqref{E:Knapp-1} (see also appendix \ref{A:sub-formal-1d} for the construction), we write
$$
\tilde Q_\ep(x)= \ep^{1/\alpha} W(\ep^{1/2}x)\cos x.
$$
This trial function gives the upper bound $\mathcal R_\ep(\alpha)\lesssim\ep^{\frac{\alpha+4}{2(\alpha+2)}}$, and a matching lower bound follows either from the Hausdorff-Young inequality (see Section \ref{S:HY-lower}) or, alternatively, from the non-homogeneous Gagliardo-Nirenberg inequality of \cite[Theorem~1.1]{FJMM2022} (see Section \ref{S:GN-lower}). Together they give
\begin{equation}\label{E:R_sharp}
\mathcal R_\ep(\alpha) \approx  
\ep^{\frac{\alpha+4}{2(\alpha + 2)}}
\quad \mbox{as } \ep \to 0^+.
\end{equation}
Since $\ep^{\frac{\alpha+4}{2(\alpha+2)}} = \epsilon^{\frac12+\frac1{\alpha+2}}  \ll \epsilon^{1/2}$ as $\epsilon \to 0^+$,  
the estimate \eqref{E:R_sharp} gives a sharper one-dimensional bound than \eqref{E:FJMM-upper}, which helps us to explain the results we obtained for the behavior of the mass in the 1D bi-NLS in \cite[Section 2.4.2]{KPRS}, see discussion in Section \ref{S:M-1D}. 
From \eqref{E:R_sharp} and \eqref{E:least-action-value}, we get refinements for \eqref{E:S-Q-FJMM} as
$$
S_{1+\epsilon}(Q_\epsilon)
=\frac{\alpha}{2(\alpha+2)} \|Q_\epsilon\|_{L^{\alpha+2}(\mathbb R)}^{\alpha+2}
\approx \epsilon^{\frac{\alpha+4}{2\alpha}}.
$$

{\bf The quotient in dimensions $d \geq 2$.}
From \cite[Theorem 1.3]{LW2021}, it follows that for $d \geq 2$
the asymptotic behavior of $\mathcal R_\epsilon (\alpha)$ as $\epsilon \to 0^+$ is
\begin{equation}\label{E:R-asym}
\hspace{2cm} \mathcal R_\epsilon(\alpha) \approx \epsilon^{\gamma_{d,\alpha}},
\quad \gamma_{d,\alpha}
:= \frac{8+(3-d)\alpha}{4(\alpha+2)}, 
\quad 0<\alpha \leq \frac{4}{d-1}.
\end{equation}

If we restrict the quotient only to radial functions, then \cite[Theorem 2.1]{LW2021} gives
\begin{equation}\label{E:R-asym-rad}
\mathcal R_\epsilon^{\rm rad}(\alpha)
\left\{
\begin{array}{ll}
\gtrsim \epsilon^{\beta_{d,\alpha}} \quad \mbox{for any} ~~
\beta_{d,\alpha}> r_{d,\alpha}, 
& \mbox{if} \quad 0< \alpha \leq \frac{2}{d-1}, \\
\approx \epsilon^{1/2}, 
& \mbox{if} \quad \frac{2}{d-1}< \alpha < \frac{4}{d-1},
\end{array}
\right. ~~
r_{d,\alpha} = \frac{4+(2-d)\alpha}{2(\alpha+2)}
\end{equation}
Here, ``for any $\beta_{d,\alpha}$'' means that there is a constant $C = C(\alpha, \beta)$ that for any $\beta$ greater than a given value $r_{d,\alpha}$ in \eqref{E:R-asym-rad}, there is a lower bound on $\mathcal R_\ep^{\rm rad} \geq C \ep^\beta$ as $\ep \to 0^+$. Thus, for nonlinearities $\alpha \in (0,\frac{2}{d-1}]$ the result in \cite{LW2021} provided only the lower bound. While the follow-up work in \cite{MO2023} gives the upper bound, which we discuss further below, nevertheless, this lower bound provides us the idea of rate comparison, so we could classify different branches of ground states in our numerical simulations, for example, as described in Remark \ref{R:rates}.

We note that the rates in \eqref{E:R-asym-rad} compare as 
\begin{align*}
r_{d,\alpha} & > \tfrac12
\quad \mbox{when} \quad 0<\alpha < \tfrac{2}{d-1},\\
r_{d,\alpha} & = \tfrac12
\quad \mbox{when} \quad \alpha = \tfrac{2}{d-1}.
\end{align*}

Moreover, comparing the unrestricted power in \eqref{E:R-asym} with the radial one in \eqref{E:R-asym-rad}, we have
\begin{align*}
\qquad \gamma_{d,\alpha} & > r_{d,\alpha} 
\quad  \mbox{if} \quad 0<\alpha\leq \tfrac{2}{d-1},\\
\qquad \gamma_{d,\alpha} & >\tfrac12
\qquad  \mbox{if} \quad \tfrac{2}{d-1}<\alpha<\tfrac{4}{d-1}, \quad \mbox{and}\\
\qquad \gamma_{d,\alpha} & = \tfrac12
\qquad  \mbox{if} \quad \alpha = \tfrac{4}{d-1},
\end{align*}
thus, the radial and unrestricted rates match at the endpoint $4/(d-1)$, and hence, the endpoint is not included in Theorem \ref{Thm1} or \cite[Theorem  1.2]{LW2021}.
Regarding the top line of \eqref{E:R-asym-rad}, we can always find a rate $\beta_{d,\alpha}$ such that 
$$
\gamma_{d,\alpha} > \beta_{d,\alpha} > r_{d,\alpha},
$$ 
and hence, for $\ep \to 0^+$ one has $\ep^{\gamma_{d,\alpha}} < \ep^{\beta_{d,\alpha}}$ implying that 
the best-constant quotient attains a lower value at nonradial functions for sufficiently small $\epsilon$ than at radial functions. 

The bottom line of  \eqref{E:R-asym-rad} from the result of \cite[Theorem 2.1 (i)]{LW2021} is actually stronger, it gives an explicit constant,
\begin{equation}\label{E:R-rad-sharp}
\mathcal R_\ep^{\rm rad}(\alpha) = \frac{2\,\mathsf C^{\rm rad}_{ST}(\alpha+2)}{\pi}\,\sqrt\ep \,(1 + o(1)),
\quad \alpha > \frac{2}{d-1} ~~ (H^2\mbox{-subcritical}),
\end{equation}
where the constant $\mathsf C^{\rm rad}_{ST}(\alpha+2)=|\mathbb S^{d-1}|/\|\check{\mathbf 1}_{\mathbb S^{d-1}}\|_{L^{\alpha+2}(\R^d)}^2$
is the radial Stein - Tomas constant from \cite[(2.3)]{LW2021}. In 2D, 
$\check{\mathbf 1}_{\mathbb S^{1}}(x)=J_0(|x|)$ (see Appendix \ref{A:sub-formal-rad} for a Bessel trial-function \eqref{E:Bessel}). 
For the unrestricted class, the result from \cite[Theorem 1.3(i)]{LW2021} gives
\begin{equation}\label{E:R-unrestr-sharp}
\mathcal R_\ep(\alpha)=\frac{2\,\mathsf C_{ST}(\alpha+2)}{\pi}\sqrt\ep \,(1+o(1)),
\quad \alpha \geq \frac{4}{d-1},
\end{equation} 
with the Stein - Tomas constant $\mathsf C_{ST}$ from \cite[(1.13)]{LW2021}. We use these in Corollary \ref{C:sharp-mass} below.

In the subsequent work \cite{MO2023}, Mandel and Oliveira e Silva prove in Theorem 1.3 matching upper and lower bounds for the radial class in dimension $d \geq 2$, including a refinement with the logarithmic correction in the threshold case. They show that
\begin{equation}\label{E:R-radial-d}
\mathcal R_\ep^{\rm rad} \approx 
\left\{
\begin{array}{ll}
\ep^{r_{d,\alpha}} & \mbox{if} \quad 0<\alpha < \frac2{d-1}, \quad r_{d,\alpha}~\mbox{as ~in~~} \eqref{E:R-asym-rad},\\
\ep^{\frac{1}{2}}\,|\ln\ep|^{-\frac{d-1}{d}} & \mbox{if} \quad \alpha=\frac{2}{d-1}, \\
\ep^{\frac12} & \mbox{if} \quad \alpha > \frac{2}{d-1},
\end{array}
\right.
\end{equation} 
where in the last case $\alpha$ is restricted to the $H^2$-subcritical range, i.e., $\alpha<\frac{8}{d-4}$ when $d \geq 5$. Note that the last case is exactly the same as in \eqref{E:R-asym-rad} (or \cite[Theorem 2.1]{LW2021}). We also mention  Appendix \S \ref{A:sub-formal-rad} for the Bessel trial-function computation that shows how the logarithmic correction appears in the threshold. 

In dimension two, the main purpose of this paper, it states
\begin{equation}\label{E:R-asym-rad-MO}
\mathcal R_\epsilon^{\rm rad}(\alpha)\approx
\left\{
\begin{array}{ll}
\epsilon^{\frac{2}{\alpha+2}}, & 0<\alpha<2,\\
\epsilon^{1/2}|\ln\epsilon|^{-1/2}, & ~\alpha=2,\\
\epsilon^{1/2}, & ~\alpha>2.
\end{array}
\right.
\end{equation}

In terms of the action $S_b$ from \eqref{E:least-action-def}, the properties \eqref{E:least-action-value} and \eqref{E:S-Potential} imply
$$
S_{1+\epsilon}(Q_\epsilon) 
=\frac{\alpha}{2(\alpha+2)} \left(\mathcal R_\epsilon(\alpha)\right)^{\frac{\alpha+2}{\alpha}} \quad \mbox{and} \quad S_{1+\epsilon}(Q_\epsilon)= \frac{\alpha}{2(\alpha+2)}
\|Q_\epsilon\|_{L^{\alpha+2}(\mathbb R^d)}^{\alpha+2},
$$
and thus, for the unrestricted minimizers, we deduce
\begin{equation}\label{E:S-asym-general}
S_{1+\epsilon}(Q_\epsilon)
\approx \|Q_\epsilon\|_{L^{\alpha+2}(\R^d)}^{\alpha+2}
\approx \epsilon^{\sigma_{d,\alpha}},
\qquad \sigma_{d,\alpha} = \frac{8+ (3-d)\alpha}{4\alpha},
\end{equation}
and for radial minimizers $Q_\epsilon^{\rm rad}$ in two dimensions, we get
\begin{equation}\label{E:S-asym-rad-2d}
S_{1+\epsilon}(Q_\epsilon^{\rm rad})
\approx \|Q_\epsilon^{\rm rad}\|_{L^{\alpha+2}(\R^2)}^{\alpha+2}
\approx \left\{
\begin{array}{ll}
\epsilon^{\frac2\alpha}, & 0<\alpha<2,\\
\epsilon|\ln\epsilon|^{-1}, & ~\alpha=2,\\
\epsilon^{\frac{\alpha+2}{2\alpha}}, & ~\alpha>2.
\end{array}
\right.
\end{equation}

To obtain the respective mass dependence on $\ep$, we derive a quotient-to-mass relation in Lemma \ref{L:quotient-to-mass} and its generalization Theorem \ref{T:quotient-mass-general}, this proves Part 1 of Theorem \ref{ThmA}.


\subsection{Quotient-to-mass relation}\label{S:M-to-R}
We now relate the asymptotics of the best-constant quotient with that of mass of the corresponding least-action ground states. For that we rewrite the quadratic form $q_{1,1+\ep}$ 
as 
\begin{equation}\label{E:q-split}
q_{1,1+\epsilon}(u)= \int_{\mathbb R^d}(|\xi|^2-1)^2 |\widehat u(\xi)|^2\,d\xi
+\ep \int_{\R^d} |\widehat u(\xi)|^2\,d\xi =
q_{1,1}(u)+\epsilon M(u). 
\end{equation}

The next simple lemma shows how to obtain the asymptotics on the mass for small $\ep$, provided the bounds for $\mathcal R_\ep$  are known.

\begin{lemma}\label{L:quotient-to-mass}
Assume that the infimum in \eqref{E:R-epsilon} is attained by $v_\epsilon \in H^2(\mathbb R^d)$, which is normalized as $\|v_\epsilon\|_{L^{\alpha+2}(\mathbb R^d)}=1$. 

Suppose that the asymptotic rate for $\mathcal R_\ep(\alpha)$ is known, that is, for some $0<\gamma<1$, there exist positive constants $c$ and $C$ (with $c<C$) such that
$$
c\,\epsilon^\gamma \leq \mathcal R_\epsilon(\alpha)
\leq C\epsilon^\gamma
$$
for all sufficiently small $\epsilon >0$.  
If
$$
Q_\epsilon = \mathcal R_\epsilon(\alpha)^{1/\alpha}v_\epsilon,
$$
then the lower and upper bounds for mass of $Q_\ep$ coincide up to a constant, namely, 
\begin{equation}\label{E:quotient-mass}
M(Q_\epsilon) \approx \epsilon^{\gamma\frac{\alpha+2}{\alpha}-1}.
\end{equation}

\end{lemma}

\begin{proof}
Since $q_{1,1+\ep}(v_\epsilon)=\mathcal R_\epsilon(\alpha)$, and $q_{1,1}(v_\ep)$ is nonnegative, we have the following upper bound
$$
\epsilon\|v_\epsilon\|_{L^2(\R^d)}^2 \leq
\mathcal R_\epsilon(\alpha) \leq C \epsilon^\gamma,
$$
and therefore, 
$$
\|v_\epsilon\|_{L^2(\R^d)}^2 \leq C \epsilon^{\gamma-1}.
$$ 
On the other hand, for a fixed $\lambda>1$, we can dilate $\ep$ by $\lambda \ep$ and consider $\mathcal R_{\lambda \epsilon} (\alpha)$, yielding
\begin{align}
\mathcal R_{\lambda\epsilon}(\alpha) 
& \leq q_{1,1+\lambda\epsilon}(v_\epsilon)
= q_{1,1}(v_\ep)+\lambda\epsilon\|v_\ep\|_{L^2(\mathbb R^d)}^2 \notag \\
& = q_{1,1}(v_\ep)+\epsilon\|v_\ep\|_{L^2(\mathbb R^d)}^2 + (\lambda - 1)\ep \|v_\ep\|_{L^2(\mathbb R^d)}^2 \notag \\
&=
\mathcal R_\epsilon(\alpha)
+(\lambda-1)\epsilon\|v_\epsilon\|_{L^2}^2. \label{E:lower-bound-1}
\end{align}
Since $\lambda > 1$, we can choose $\lambda$ so that $c\lambda^\gamma>C$.  Then, for all sufficiently small $\epsilon$, we have a lower bound
$$
\mathcal R_{\lambda \ep} (\alpha) \geq c (\lambda \ep)^\gamma > C \ep^\gamma,
$$
implying together with \eqref{E:lower-bound-1}
$$
\|v_\epsilon\|_{L^2}^2 \geq \tilde C \epsilon^{\gamma-1}.
$$
Thus, combining the lower and upper bounds on  $\|v_\epsilon\|_{L^2}^2$, we deduce
$$
\|v_\epsilon\|_{L^2}^2 \approx \epsilon^{\gamma-1}.
$$  
Since $Q_\epsilon=\mathcal R_\epsilon(\alpha)^{1/\alpha}v_\epsilon$, we have
\begin{equation}\label{E:quotient-mass-lemma}
M(Q_\epsilon)=\|Q_\epsilon\|_{L^2(\mathbb R^d)}^2  =  
\mathcal R_\epsilon(\alpha)^{2/\alpha} \|v_\epsilon\|_{L^2}^2
\approx
\epsilon^{\gamma(\frac2{\alpha} +1)-1},
\end{equation}
and hence, \eqref{E:quotient-mass} follows.
\end{proof}

\begin{remark}\label{R:quotient-to-mass-log}
The same argument applies when the quotient contains a logarithmic factor, for example, as in the two-dimensional radial endpoint $\alpha=2$, 
$$
\mathcal R_\epsilon^{\rm rad}(2)
\approx \epsilon^{1/2}|\ln\epsilon|^{-1/2}.
$$
Then for every fixed $\lambda>1$, we have
$$
\frac{|\ln(\lambda\epsilon)|^{-1/2}} {|\ln\epsilon|^{-1/2}} \rightarrow 1
\quad \mbox{as } \ep \to 0^+.
$$
Therefore, comparing $\mathcal R_{\lambda\epsilon}^{\rm rad}(2)$ and
$\mathcal R_\epsilon^{\rm rad}(2)$ as in the proof of Lemma~ \ref{L:quotient-to-mass}, we obtain
$$
\|v_\epsilon^{\rm rad}\|_{L^2(\R^2)}^2
\approx
\epsilon^{-1/2}|\ln\epsilon|^{-1/2}.
$$
Since $Q_\epsilon^{\rm rad}=\left(\mathcal R_\epsilon^{\rm rad}(2)\right)^{1/2}
v_\epsilon^{\rm rad}$,
we obtain
$$
M(Q_\epsilon^{\rm rad}) \approx
\frac{\left(\mathcal R_\epsilon^{\rm rad}(2)\right)^2}{\epsilon} 
\approx 
|\ln\epsilon|^{-1},
$$
thus showing that the quotient-to-mass relation applies when a logarithmic correction is present. A more general correction is also allowed, which we show in the next theorem.
\end{remark}


For reader's convenience we show the summary of rates in Table \ref{T:rates-summary} in Appendix \ref{A:table}.
\smallskip

We next state a generalization of the previous lemma, which includes the above remark about the logarithmic factor. As in Section \ref{S:asym} (after \eqref{E:R-epsilon}), let $\mathcal X \subset H^2(\R^d)$ be either a whole space, or a block-symmetry subspace such as radial functions $H^2_{rad}(\R^d)$ or $G_k$-symmetric space, and recall the definition of $\mathcal R_\ep^{\mathcal X}(\alpha)$ as the quotient restricted to functions on $\mathcal X$, namely, 
\begin{equation}\label{E:R-class}
\mathcal R_\ep^{\mathcal X}(\alpha) = \inf_{u \in \mathcal X \setminus \{0\} } \frac{q_{1,1}(u)+\ep M(u)}{\|u\|^2_{L^{\alpha+2}(\R^d)}} \equiv \inf_{u \in \mathcal X \setminus \{0\} } \Big( \frac{q_{1,1}(u)}{\|u\|^2_{L^{\alpha+2}(\R^d)}} + \ep \frac{M(u)}{\|u\|^2_{L^{\alpha+2}(\R^d)}} \Big). 
\end{equation}

\begin{theorem}\label{T:quotient-mass-general}
Assume that the infimum in \eqref{E:R-class} is attained for any $\ep\in(0,\ep_0)$, for some $\ep_0>0$.  
For each such $\ep$, let $v_\ep \in \mathcal X$ be an optimizer normalized
by $\|v_\ep\|_{L^{\alpha+2}(\R^d)} = 1$, and define 
$$
Q_\ep^{\mathcal X} = \big(\mathcal R_\ep^{\mathcal X}(\alpha)\big)^{1/\alpha}v_\ep.
$$ 
Then the following properties hold:
\begin{itemize}
\item[(i)]  The quotient $\mathcal R_\ep^{\mathcal X}(\alpha)$, viewed as a
function of $\ep$ on $(0,\ep_0)$, is increasing and concave, and for
$0<\ep_1<\ep_2<\ep_0$ we have bounds on the difference quotient
\begin{equation}\label{E:secant-mass}
M(v_{\ep_2}) \leq \frac{\mathcal R_{\ep_2}^{\mathcal X}(\alpha)-\mathcal R_{\ep_1}^{\mathcal X}(\alpha)}
{\ep_2-\ep_1}  \leq M(v_{\ep_1}).
\end{equation}
In particular, $M(v_\ep)$ is nonincreasing in $\ep$, and (at every point where
$\mathcal R_\ep^{\mathcal X}(\alpha)$ is differentiable), we have
\begin{equation}\label{E:slope-identity}
\frac{d}{d\ep}\,\mathcal R_\ep^{\mathcal X}(\alpha)=M(v_\ep),
\quad \mbox{and} \quad
\frac{d}{d\ep}\,S_{1+\ep}\big(Q_\ep^{\mathcal X}\big)
=\frac12\,M\big(Q_\ep^{\mathcal X}\big).
\end{equation}

\item[(ii)]  If for some $0<\gamma<1$, 
\begin{equation}\label{E:R-rate-general}
\mathcal R_\ep^{\mathcal X}(\alpha)\approx\ep^{\gamma}L(\ep),
\end{equation}
where $L>0$ is slowly varying at $0$, i.e.,  $L(\lambda\ep)/L(\ep)\to1$ as
$\ep\to0^+$ for each fixed $\lambda>0$, then
\begin{equation}\label{E:M-transfer-general}
M\big(Q_\ep^{\mathcal X}\big)
\approx \frac{\big(\mathcal R_\ep^{\mathcal X}(\alpha)\big)^{\frac{\alpha+2}{\alpha}}}{\ep}
\approx \ep^{\gamma\frac{\alpha+2}{\alpha}-1}L(\ep)^{\frac{\alpha+2}{\alpha}} .
\end{equation}

\item[(iii)]  
If, in addition,
$\mathcal R_{\lambda\ep}^{\mathcal X}(\alpha)/\mathcal R_\ep^{\mathcal X}(\alpha)
\to\lambda^{\gamma}$ as $\ep\to0^+$ for each fixed $\lambda>0$, 
then the relation \eqref{E:M-transfer-general} holds with the explicit
constant $\gamma$, namely, as $\ep\to0^+$,
\begin{equation}\label{E:M-sharp-constant}
\ep\,M\big(Q_\ep^{\mathcal X}\big)
= \gamma\,\big(\mathcal R_\ep^{\mathcal X}(\alpha)\big)^{\frac{\alpha+2}{\alpha}}\big(1+o(1)\big),
\end{equation}
and each part of $q_{1,1+\ep}$ as written in 
\eqref{E:q-split} gets a fixed proportion, 
\begin{equation}\label{E:q-splitting}
\quad \ep\,M\big(Q_\ep^{\mathcal X}\big)
= \gamma\,\big\|Q_\ep^{\mathcal X}\big\|_{L^{\alpha+2}}^{\alpha+2}\,\big(1+o(1)\big),
\quad 
q_{1,1}\big(Q_\ep^{\mathcal X}\big)
=(1-\gamma)\,\big\|Q_\ep^{\mathcal X}\big\|_{L^{\alpha+2}}^{\alpha+2}\,\big(1+o(1)\big).
\end{equation}

\end{itemize}

\end{theorem}

\begin{remark}
Taking $L \equiv 1$ in part (ii) gives Lemma~\ref{L:quotient-to-mass}, and taking
$L(\ep)=|\ln\ep|^{-1/2}$, $\gamma=\frac12$, $\alpha=2$, recovers
Remark~\ref{R:quotient-to-mass-log}.

\end{remark}

\begin{proof}
(i) For each $u \in \mathcal X\setminus\{0\}$, the expression in parentheses in \eqref{E:R-class}, $
\frac{q_{1,1}(u)}{\|u\|_{L^{\alpha+2}}^2} + \ep\,\frac{M(u)}{\|u\|_{L^{\alpha+2}}^2}$, 
as a function of $\ep$, is a straight line with a positive slope
$M(u)/\|u\|_{L^{\alpha+2}}^2$, thus, the quotient
$\mathcal R_\ep^{\mathcal X}(\alpha)$ is increasing and concave in $\ep$.

Substituting $v_{\ep_1}$ into \eqref{E:R-class} with $\ep_2$ gives
$$
\mathcal R_{\ep_2}^{\mathcal X}(\alpha)
\leq q_{1,1}(v_{\ep_1})+\ep_2 M(v_{\ep_1})
=\mathcal R_{\ep_1}^{\mathcal X}(\alpha)+(\ep_2-\ep_1)M(v_{\ep_1}),
$$
which is the right inequality in \eqref{E:secant-mass}. Exchanging $\ep_1$ with $\ep_2$ gives the left   inequality in \eqref{E:secant-mass}. Now, take $\ep_2 \to \ep_1$ or $\ep_1 \to \ep_2$ in \eqref{E:secant-mass}, then we get the first identity in \eqref{E:slope-identity}. 
We obtain the second identity by differentiating
$S_{1+\ep}(Q_\ep^{\mathcal X})=\frac{\alpha}{2(\alpha+2)}\big(\mathcal R_\ep^{\mathcal X}(\alpha)\big)^{\frac{\alpha+2}{\alpha}}$,
see \eqref{E:least-action-value}, and using
$M(Q_\ep^{\mathcal X})=\big(\mathcal R_\ep^{\mathcal X}(\alpha)\big)^{2/\alpha}M(v_\ep)$.
\smallskip

(ii) The same argument as in Lemma~\ref{L:quotient-to-mass} yields $\ep M(v_\ep)\leq\mathcal R_\ep^{\mathcal X}(\alpha)$, which gives the upper bound. For the lower bound, we rewrite \eqref{E:R-rate-general} as
$c\,\ep^\gamma L(\ep)\leq\mathcal R_\ep^{\mathcal X}(\alpha)\leq C\,\ep^\gamma L(\ep)$
and apply \eqref{E:secant-mass} with $\ep_1=\ep$ and $\ep_2=\lambda\ep$ to obtain
$$
M(v_\ep)\geq \frac{\mathcal R_{\lambda\ep}^{\mathcal X}(\alpha)-\mathcal R_\ep^{\mathcal X}(\alpha)}
{(\lambda-1)\ep}.
$$
With a similar argument as in the lemma, we fix $\lambda>1$ with $c\lambda^\gamma>2C$.  Since $L(\lambda\ep)/L(\ep)\to1$,
for all sufficiently small $\ep$ the numerator is at least
$\frac12(c\lambda^\gamma-C)\ep^\gamma L(\ep)$, so
$M(v_\ep)\gtrsim\ep^{\gamma-1}L(\ep)$, and hence,
$M(v_\ep)\approx\mathcal R_\ep^{\mathcal X}(\alpha)/\ep$. Multiplying by
$\big(\mathcal R_\ep^{\mathcal X}(\alpha)\big)^{2/\alpha}$ gives
\eqref{E:M-transfer-general}.
\smallskip

(iii)  Fix  $\lambda>1$. Using \eqref{E:secant-mass} with
$\ep_1=\ep$, $\ep_2=\lambda\ep$ and dividing by
$\mathcal R_\ep^{\mathcal X}(\alpha)/\ep$, we obtain
$$
\frac{\ep\,M(v_\ep)}{\mathcal R_\ep^{\mathcal X}(\alpha)} \geq 
\frac{1}{\lambda-1} \Big(\frac{\mathcal R_{\lambda\ep}^{\mathcal X}(\alpha)}{\mathcal R_\ep^{\mathcal X}(\alpha)}-1\Big) \rightarrow \frac{\lambda^\gamma-1}{\lambda-1}
\quad \mbox{as } ~ \ep \to 0^+.
$$
Similarly, for $0<\mu<1$, using \eqref{E:secant-mass} with $\ep_1=\mu\ep$,
$\ep_2=\ep$ gives
$$
\frac{\ep\,M(v_\ep)}{\mathcal R_\ep^{\mathcal X}(\alpha)} \leq \frac{1}{1-\mu}
\Big(1-\frac{\mathcal R_{\mu\ep}^{\mathcal X}(\alpha)}{\mathcal R_\ep^{\mathcal X}(\alpha)} \Big) \to  \frac{1-\mu^\gamma}{1-\mu} \quad \mbox{as } ~ \ep \to 0^+ .
$$
Thus, for every $\lambda>1$ and every $0<\mu<1$,
$$
\frac{\lambda^\gamma-1}{\lambda-1} 
\leq \liminf_{\ep \to 0^+} \frac{\ep\,M(v_\ep)}{\mathcal R_\ep^{\mathcal X}(\alpha)}
\leq \limsup_{\ep \to 0^+}\frac{\ep\,M(v_\ep)}{\mathcal R_\ep^{\mathcal X}(\alpha)} \leq \frac{1-\mu^\gamma}{1-\mu}.
$$

Both of the sides are difference quotients of the function $f(t)= t^\gamma$ at $t=1$, so they tend to $\gamma$ as
$\lambda \to 1$ and $\mu \to 1$, and hence, we get
$$
\frac{\ep\,M(v_\ep)}{\mathcal R_\ep^{\mathcal X}(\alpha)} \to \gamma.
$$
Since $Q_\ep^{\mathcal X} = \left(\mathcal R_\ep^{\mathcal X}(\alpha)\right)^{1/\alpha}v_\ep$,
we have
$$
\frac{\ep\,M(Q_\ep^{\mathcal X})}
{\left(\mathcal R_\ep^{\mathcal X}(\alpha)\right)^{\frac{\alpha+2}{\alpha}}}
= \frac{\ep\,M(v_\ep)}{\mathcal R_\ep^{\mathcal X}(\alpha)}
\rightarrow \gamma,
$$
which gives \eqref{E:M-sharp-constant}. 
Moreover, since $\|v_\ep\|_{L^{\alpha+2}}=1$, we have 
$\|Q_\ep^{\mathcal X}\|_{L^{\alpha+2}}^{\alpha+2}
= (\mathcal R_\ep^{\mathcal X}(\alpha))^{(\alpha+2)/\alpha}$.
By the first Pokhozhaev identity, we have 
$q_{1,1} (Q_\ep^{\mathcal X} ) + \ep\, M (Q_\ep^{\mathcal X} )
= \|Q_\ep^{\mathcal X} \|_{L^{\alpha+2}}^{\alpha+2},$
which together with \eqref{E:M-sharp-constant} yields \eqref{E:q-splitting}. 
\end{proof}
The hypothesis of part (iii) holds whenever the quotient admits an expansion with a leading constant, for example as in \eqref{E:R-rad-sharp} and \eqref{E:R-unrestr-sharp}, see Corollary \ref{C:sharp-mass} at the end of this section.

\begin{remark}\label{R:concavity-branching}
Note that by \eqref{E:secant-mass} the mass $M(v_\ep)$ is nonincreasing in $\ep$. Since the mass of the ground state has a factor with $\mathcal R_\ep$, i.e., 
$M(Q_\ep^{\mathcal X}) = (\mathcal R_\ep^{\mathcal X}(\alpha))^{2/\alpha}M(v_\ep)$, which increases with $\ep$, the mass $M(Q_\ep^{\mathcal X})$ does not need to be monotone and may have a turning point.  By the second identity in \eqref{E:slope-identity}, such a turning point would be exactly an inflection point of the least-action  $S_{1+\ep}(Q_\ep^{\mathcal X})$ (provided a sufficiently smooth optimizer branch and an actual change of sign of $M'(Q_\ep^{\mathcal X})$), which can be thought of as a variational interpretation of branching that we observe in the mass-energy diagrams in \S\ref{S:branching}.
\end{remark}


\subsection{Lower bounds for the quotient}\label{A:sub-rigorous}

In 1D we give two independent proofs of the lower bound on $\mathcal R_\ep$: 
one via Hausdorff-Young and a Fourier multiplier estimate, 
and one via the non-homogeneous Gagliardo-Nirenberg inequality of \cite{FJMM2022}. We show that the Hausdorff-Young argument extends to the $H^2$-subcritical range in higher dimensions 
but it is sharp only in 1D. 
On the other hand, with the Gagliardo-Nirenberg argument not only do we give an alternative proof in 1D, but we also recover the sharp unrestricted exponent $\gamma_{d,\alpha}$ from \eqref{E:R-asym} (i.e., from \cite[Theorem 1.3]{LW2021})
for $d \geq 2$ in the range $0<\alpha \leq 4/(d-1)$. Note that at the endpoint $\alpha = 4/(d-1)$ the rate agrees with the radial, and thus, no longer yields the symmetry-breaking, so the range in Theorem \ref{Thm1} does not include the endpoint).

\subsubsection{Lower bound on $\mathcal R_\ep$ via Hausdorff-Young}\label{S:HY-lower}

\begin{theorem}\label{Thm:B1}
Let $\alpha>0$ if $d \leq 4$ and $0<\alpha < \frac8{d-4}$ if $d>4$ (i.e., $H^2$-subcritical).  Then as $\ep \to 0^+$
\begin{equation}\label{E:Rep-lower-bd}
\mathcal R_\epsilon(\alpha) \gtrsim
\epsilon^{\frac{\alpha+4}{2(\alpha+2)}}.
\end{equation}
\end{theorem}

\begin{proof}
Let $u \in H^2(\R^d)\setminus\{0\}$. We show that
\begin{equation}\label{E:HY-key}
\|u\|_{L^{\alpha+2}(\R^d)}^2 \leq C(\alpha,d) \, \epsilon^{-\frac{\alpha+4}{2(\alpha+2)}} \, q_{1,1+\epsilon}(u),
\end{equation}
and then the conclusion $q_{1,1+\epsilon}(u)/\|u\|_{L^{\alpha+2}}^2 \geq c \,\epsilon^{\frac{\alpha+4}{2(\alpha+2)}}$ follows by taking the infimum.

First, we set $p = (\alpha+2)/(\alpha+1) 
\in [1, 2]$, so $p' = \alpha+2$. By the Hausdorff-Young inequality (e.g., \cite[Theorem 5.7]{Wolff2003} or \cite[Section 1.2]{Grafakos2014})), we obtain
$$
\|u\|_{L^{\alpha+2}(\R^d)} =  \|u\|_{L^{p'}(\R^d)} \leq  C_d \, \|\widehat u\|_{L^p(\R^d)}.
$$

For brevity, we write $m=m_{1,1+\ep}$. We split $\widehat u = m^{-1/2} (m^{1/2} \widehat u)$~ and apply H\"older's inequality with  
$\frac{1}{p} = \frac{1}{s} + \frac{1}{2}$.
This implies that $s = \frac{2(\alpha+2)}{\alpha}$, and thus, 
$$
\|\widehat u\|_{L^p(\mathbb R^d)} \leq  \|m^{-1/2}\|_{L^s(\R^d)} \, \|m^{1/2}\widehat u\|_{L^2(\R^d)},
$$
where
$$
\|m^{1/2}\widehat u\|_{L^2(\R^d)}^2 = \int_{\R^d} m(\xi) |\widehat u(\xi)|^2 d\xi = q_{1,1+\epsilon}(u).
$$
Thus, we are left with estimating $\|m^{-1/2}\|_{L^s}$.
We decompose the integral into two, splitting the domain into 
$$
\big| |\xi|-1 \big| \leq \tfrac12
$$
and its complement.  In the first region, polar coordinates and $t=|\xi|-1$ give, with constants independent of $\epsilon$,
$$
\int_{||\xi|-1| \leq 1/2} m (\xi)^{-s/2}\,d\xi
\approx
\int_{-1/2}^{1/2}(t^2+\epsilon)^{-s/2}\,dt
\approx
\epsilon^{\frac{1-s}{2}}.
$$

On the complement, the integral converges provided $2s>d$ (or $\frac{4(\alpha+2)}{\alpha} >d)$, because $m(\xi)^{-s/2} \lesssim (|\xi|^4)^{-s/2} = |\xi|^{-2s}$ as $|\xi| \to \infty$ (that's where the fourth power shows up).  Since $s>1$, the contribution near the sphere dominates, and therefore,
$$
\|m^{-1/2}\|_{L^s(\R^d)}
\lesssim \ep^{\frac{1-s}{2s}}
= \epsilon^{-\frac{\alpha+4}{4(\alpha+2)}}.
$$
Squaring, we deduce
$$
\|u\|_{L^{\alpha+2}(\R^d)}^2
\lesssim \epsilon^{-\frac{\alpha+4}{2(\alpha+2)}}q_{1,1+\ep}(u),
$$
which gives \eqref{E:HY-key} and finishes the proof.  
\end{proof}

\begin{remark}
We emphasize that dimension $d$ enters the proof only at one point, via the convergence of $\int m_{1,1+\ep}^{-s/2}$ away from the sphere (and thus, the requirement $2s>d$). 
\end{remark}

\begin{remark}(non-sharpness in $d \geq 2$)\label{R:d>2}
Comparing $\frac{\alpha+4}{2(\alpha+2)}$ with the exponent 
$\gamma_{d,\alpha}=\frac{8+(3-d)\alpha}{4(\alpha+2)}$ 
from \eqref{E:R-asym} (or \cite[Theorem 1.3]{LW2021}), we deduce 
$$
\frac{\alpha+4}{2(\alpha+2)}-\gamma_{d,\alpha} = \frac{(d-1)\alpha}{4(\alpha+2)}\geq 0,
$$
with equality in $d=1$. Thus, the Hausdorff-Young lower bound is sharp only in $d=1$. 
This is because in one dimension the minimum set consists of the two points $\{\pm 1\}$ and there is no tangential direction, so Hausdorff-Young captures the sharp scale.  In $d\geq 2$, the above argument treats the full singular neighborhood of the sphere only through the scalar weight $m_{1,1+\ep}^{-1/2}$, it does not use Fourier-extension cancellation or cap geometry along the $(d-1)$ tangential directions.  To get the sharp estimates, it is precisely that information needed.
\end{remark}

\subsubsection{Lower bound via a non-homogeneous Gagliardo-Nirenberg}\label{S:GN-lower}

Here we use the non-homogeneous biharmonic Gagliardo-Nirenberg inequality of Fern\'andez-Jeanjean-Mandel-Mari\c{s} \cite[Theorem~1.1]{FJMM2022} to obtain a lower bound on a quotient in any dimension. This gives an alternative proof of the lower bound in 1D on the quotient, that is the same as in Theorem \ref{Thm:B1}, but in dimensions $d \geq 2$ this approach gives a sharper lower bound.

\begin{theorem}\label{Thm:GN-lower}
For $d=1$ let $\alpha>0$, and for $d\geq2$ let $ 0<\alpha\leq\frac4{d-1}$.
Then
\begin{equation}\label{E:R-GN-lower}
\mathcal R_\ep(\alpha) \gtrsim \ep^{\gamma_{d,\alpha}},
\qquad \gamma_{d,\alpha}
= \frac{8+(3-d)\alpha}{4(\alpha+2)}.
\end{equation}
In particular, for $d=1$,
$$
\gamma_{1,\alpha} = \frac{\alpha+4}{2(\alpha+2)},
$$
while for $d\geq2$ and
$0<\alpha \leq 4/(d-1)$ this is the sharp unrestricted exponent
in \eqref{E:R-asym}.
\end{theorem}

\begin{proof}
We write the quadratic form as
\begin{equation}\label{E:q-two}
q_{1,1+\ep}(u) = \|\Delta u + u \|_{L^2(\R^d)}^2 + \ep\|u\|_{L^2(\R^d)}^2.
\end{equation}

The non-homogeneous Gagliardo-Nirenberg inequality
from \cite[Theorem 1.1]{FJMM2022} states that, for $\alpha>0$ (in the range stated in the theorem),
\begin{equation}\label{E:GN-double}
\|u\|_{L^{\alpha+2}(\R^d)} \leq C_{\alpha,d} \|u\|_{L^2(\R^d)}^{\gamma_{d,\alpha}}
\|\Delta u + u\|_{L^2(\R^d)}^{1-\gamma_{d,\alpha}}.
\end{equation}
Note that $1-\gamma_{d,\alpha} = \frac{d+1}{2} (\frac12 - \frac1{\alpha+2})$ and $\gamma_{d,\alpha} \geq \frac12$ precisely when $\alpha \leq 4/(d-1)$ for $d \geq 2$.
Dropping either of the terms in \eqref{E:q-two}, we get trivial bounds
$$
\ep\|u\|_{L^2(\R^d)}^2 \leq q_{1,1+\ep}(u)\quad \mbox{and} \quad
\|\Delta u + u\|_{L^2(\R^d)}^2 \leq q_{1,1+\ep}(u).
$$
Substituting them into \eqref{E:GN-double}, we obtain
\begin{align*}
\|u\|_{L^{\alpha+2}(\R^d)}^{2}
\leq C_{\alpha,d} \left(q_{1,1+\ep}(u)/\ep \right)^{\gamma_{d, \alpha}}
\left(q_{1,1+\ep}(u)\right)^{1-\gamma_{d, \alpha} } =
C_{\alpha,d}\,  q_{1,1+\ep}(u)/ \ep^{\gamma_{d, \alpha} }.
\end{align*}
Therefore,
$$
\frac{q_{1,1+\ep}(u)}
{\|u\|_{L^{\alpha+2}(\R^d)}^2}
\gtrsim \ep^{\gamma_{d, \alpha}},
$$
and taking the infimum over all nonzero $u\in H^2(\R^d)$, we obtain the lower bound
$$
\mathcal R_\ep(\alpha) \gtrsim \ep^{\gamma_{d, \alpha} }.
$$
Note that for $d=1$, $\gamma_{1,\alpha} = \frac{\alpha+4}{2(\alpha+2)}$,
which recovers the lower bound of Theorem \ref{Thm:B1} in 1D. 
\end{proof}

\begin{remark}[Comparison for $d\geq 2$ bounds]\label{R:GN-duality}
The reason that in $d \geq 2$ the two arguments part company is because the Hausdorff-Young estimate  controls the neighborhood of the sphere via the scalar weight $m_\ep^{-1/2}$ alone and it can not sense the tangential width of $\ep^{1/4}$ from Knapp cap (and thus, the gap $\frac{(d-1)\alpha}{4(\alpha+2)}$ that we pointed out in Remark \ref{R:d>2}). The Gagliardo-Nirenberg route, on the other hand, does give sharp bounds, since for $d \geq 2$ the GN inequality from \cite[Theorem 2.6]{FJMM2022} relies on the Stein-Tomas restriction estimate instead of Hausdorff-Young, the borderline case of its admissible exponents gives precisely $\gamma_{d,\alpha}$, so it recovers the sharp lower bound in \eqref{E:R-asym}) from \cite{LW2021}.
\end{remark}

\subsection{From two-sided quotient bounds to the rates}\label{S:rate-consequences}

We now prove Part 2 of Theorem \ref{ThmA} by combining the trial-function upper bounds from \S\ref{S:appendix-2} with the matching lower bounds above and the quotient-to-mass results of Section \ref{S:M-to-R}. In 1D the sharp lower bound follows either from the 
Hausdorff-Young argument in Theorem\ref{Thm:B1} 
or from the Gagliardo-Nirenberg argument in Theorem \ref{Thm:GN-lower}. 
In 2D the sharp lower bound follows from Theorem \ref{Thm:GN-lower} 
(and agrees with \cite{LW2021}), while the sharp radial lower bounds are those from
\cite{LW2021, MO2023}. We then use Lemma \ref{L:quotient-to-mass}
to obtain the corresponding mass rates. 
\smallskip

{\it One dimension.} The trial function of \S\ref{A:sub-formal-1d} gives the upper bound \eqref{E:R-upper-1d}, while Theorem  \ref{Thm:B1} (or, alternatively, the Gagliardo-Nirenberg argument of \S\ref{S:GN-lower}) gives the matching lower bound. Thus, we obtain
\begin{equation}\label{E:R-1d-approx}
\mathcal R_\epsilon(\alpha) \approx \epsilon^{\frac{\alpha+4}{2(\alpha+2)}},
\qquad d=1, \alpha>0,
\end{equation}
which proves the quotient rate in \eqref{E:rates-1D}. The infimum defining $\mathcal R_\ep(\alpha)$ is attained in one dimension by \cite[Theorem 3.6]{FJMM2022}, so Lemma \ref{L:quotient-to-mass} applies with $\gamma=\frac{\alpha+4}{2(\alpha+2)}$ and yields
$$
M(Q_\ep) \approx \epsilon^{\frac2{\alpha}-\frac12},
\qquad
S_{1+\epsilon}(Q_\epsilon) \approx \|Q_\epsilon\|_{L^{\alpha+2}(\R)}^{\alpha+2}
 \approx  \ep^{\frac{\alpha+4}{2\alpha}},
$$
where the second relation follows from \eqref{E:least-action-value}. 
This proves the right part of \eqref{E:rates-1D}.

\smallskip

{\it Two dimensions.}\label{A:sub-rig-2d}
Here, Theorem \ref{Thm:B1} is no longer sharp, so for the lower bound we take  for  the unrestricted class the bound \eqref{E:R-asym}, proved in \cite{LW2021}, and for the radial class the bound \eqref{E:R-asym-rad-MO}, proved in \cite{MO2023}. 
The matching upper bounds are given by the Knapp and Bessel trial functions of \S\ref{A:sub-formal-2d-nr} and \S\ref{A:sub-formal-rad}. 
We obtain the two-sided quotient rates. 
Applying them into Lemma~\ref{L:quotient-to-mass}, and into Theorem~\ref{T:quotient-mass-general}(ii) at the endpoint $\alpha=2$, where a slowly varying factor is present, gives the mass rates collected in Table~\ref{T:rates-summary}.

\begin{remark}[Attainment]\label{R:attainment}
Lemma~\ref{L:quotient-to-mass} assumes that the quotient infimum is attained. For the unrestricted problem this follows from \cite[Theorem~3.6]{FJMM2022} in any dimension (in the $H^2$-subcritical range). For the problem restricted to radial functions, attainment in every dimension $d\geq2$ in
the $H^2$-subcritical range follows from \cite[Appendix A, Prop. 1]{MO2023} by coercitivy and compact embedding of radial functions. 
For the block-radial class $G_1$ in $d=2$ neither attainment nor the strict inequality $\mathcal R_\ep^\circ<\mathcal R_\ep^{G_1}$ is known, see \cite[Appendix~C]{MO2023}.
\end{remark}

\begin{remark}[Higher-dimensions]\label{R:higher-dimensions}
The quotient-to-mass  relation from Lemma \ref{L:quotient-to-mass} 
also gives mass rates for $d \geq 3$. 
For the unrestricted case, combining Theorem \ref{Thm:GN-lower} with the matching 
upper bound of \cite{LW2021}, we obtain for $d \geq 2$ and
$0<\alpha \leq 4/(d-1)$
$$
\qquad \mathcal R_\ep(\alpha) \approx \ep^{\gamma_{d,\alpha}}, \qquad \gamma_{d,\alpha} =\frac{8+(3-d)\alpha}{4(\alpha+2)}.
$$
Therefore, Lemma \ref{L:quotient-to-mass} yields
$$
M(Q_\ep) \approx \ep^{\frac2{\alpha} - \frac{d+1}{4}}.
$$
The corresponding mass-scaling threshold (zero of the exponent) is
$$
\alpha_K(d)=\frac8{d+1}.
$$
In $d=3$,
$$
\qquad M(Q_\ep)\approx\ep^{\frac2\alpha-1}, \quad 0<\alpha<2,
$$
so that $\alpha_K=2$ coincides with the endpoint
of the interval $0 < \alpha < 2$, the symmetry-breaking range from Theorem \ref{Thm1} (or \cite[Theorem 1.2]{LW2021}). 
For $d \geq 4$, the value of $\alpha_K(d)$ lies outside the range of $\alpha$, i.e., $\frac8{d+1}> \frac4{d-1}$, which in particular, gives for $\alpha \in (0,4/(d-1)]$ that $M(Q_\ep) \to 0^+$ as $\ep$ decreases down to 0. 
\end{remark}
We can also extend the unrestricted mass rate beyond the endpoint, namely, for $d\geq2$ and $\alpha\geq4/(d-1)$ in the
$H^2$-subcritical range. Combining \cite[Theorem 1.3(i)]{LW2021} with
Theorem \ref{T:quotient-mass-general}(iii) gives
$$
M(Q_\ep)=\frac12 C(\alpha)^{\frac{\alpha+2}{\alpha}}
\ep^{\frac1{\alpha}-\frac12}(1+o(1)),
$$
where $C(\alpha)>0$ is the constant in 
$\mathcal R_\ep(\alpha)=C(\alpha)\ep^{1/2}(1+o(1))$, see \eqref{E:R-unrestr-sharp}.
\smallskip

The same principle applies to the radial class for $d \geq 2$, and since the radial infimum is also attained (Remark \ref{R:attainment}), Lemma \ref{L:quotient-to-mass} and Theorem \ref{T:quotient-mass-general}(ii) applied to \eqref{E:R-radial-d} gives 
\begin{equation}\label{E:M-rad-d}
M(Q_\ep^{\rm rad}) \approx \left\{ 
\begin{array}{ll}
\ep^{\frac2{\alpha}-\frac{d}{2}} & 0<\alpha < \frac2{d-1},\\
\ep^{\frac{d-2}{2}}\,|\ln\ep|^{-(d-1)}, & \alpha=\frac{2}{d-1},\\
\ep^{\frac1{\alpha} - \frac12} & \alpha > \frac2{d-1}.
\end{array}
\right. 
\end{equation}

\begin{corollary}\label{C:sharp-mass}
Let $d \geq 2$ and $\alpha > \frac{2}{d-1}$ be $H^2$-subcritical. 
The constant in the last line of \eqref{E:M-rad-d} is 
\begin{equation}\label{E:M-rad-sharp}
M(Q_\ep^{\rm rad}) = \frac12 \Big(\frac{2\,\mathsf C^{\rm rad}_{ST}(\alpha+2)}{\pi}\Big)^{\frac{\alpha+2}{\alpha}}
\ep^{\frac1{\alpha}-\frac12}\,(1+o(1)),\quad \ep \to 0^+.
\end{equation}
For ~$\alpha \geq \frac{4}{d-1}$~ the same holds for the unrestricted mass $M(Q_\ep)$, with ~$\mathsf C^{\rm rad}_{ST}(\alpha+2)$~ replaced by
~$\mathsf C_{ST}(\alpha+2)$. 
\end{corollary}

\begin{proof}
By \eqref{E:R-rad-sharp}, $\mathcal R^{\rm rad}_{\lambda\ep}(\alpha)/\mathcal R^{\rm rad}_\ep(\alpha)\to\lambda^{1/2}$ for $\lambda>0$, 
so part (iii) of Theorem \ref{T:quotient-mass-general} applies to the
radial class with $\gamma = \frac12$.
Substituting \eqref{E:R-rad-sharp} into \eqref{E:M-sharp-constant} and dividing by $\ep$ gives \eqref{E:M-rad-sharp}. Similarly, by \eqref{E:R-unrestr-sharp} for the unrestricted mass. 
\end{proof}

\begin{remark}\label{R:constants-radiality}
The question of sharp constant is connected with the question  whether ground states are radial for $\alpha \geq \frac{4}{d-1}$, which is raised by our numerics in \S \ref{S:higher}.
In this  range both $\mathcal R_\ep$ and $\mathcal R_\ep^{\rm rad}$ are of order $\sqrt \ep$, and since
$\mathsf C_{ST}(\alpha+2) \leq \mathsf C^{\rm rad}_{ST}(\alpha+2)$, one has $M(Q_\ep) \leq M(Q_\ep^{\rm rad})$ from the leading constants in Corollary \ref{C:sharp-mass}. 
The equality holds if and only if constants are extremizers of the adjoint Stein-Tomas inequality on $\mathbb S^{d-1}$ at the power $p=\alpha+2$. 
If the inequality is strict, then it would prove that every ground state is nonradial for small $\ep$, which would extend Theorem \ref{Thm1} beyond its range of $\alpha$. 
We mention that in 3D for $\alpha=2$ (the endpoint) constants are the global extremizers by Foschi \cite{Foschi2015}, and thus, the two leading constants coincide. In 2D for $\alpha=4$ (also the endpoint) this is known only locally by Carneiro–Foschi–Oliveira e Silva–Thiele \cite{CFOT2017} and conjectured globally. Thus, the radial ground states that
we compute for $\alpha \geq 4$ in 2D in \S \ref{S:higher} are consistent with the
conjecture that constant functions optimize the Stein-Tomas inequality on $\mathbb S^1$ (see further details, e.g., in \cite{Becker2025}).
\end{remark}


Before we embark onto more thorough study in two dimensions, we review the one-dimensional ground states. 

\section{Ground states in 1D}\label{S:1dGS}
In 1D the equation \eqref{biNLS} becomes  
\begin{equation}\label{E:explicit1}
i \partial_{t} u -\partial_{x}^{4}u - 2a \partial_{x}^{2}u + |u|^{\alpha}u=0.
\end{equation}
Letting $u(x,t) = e^{i bt} Q(x)$, $b>0$, with $Q$ real, it follows that $Q$ satisfies
\begin{equation}\label{E:1dGS}
Q^{(4)} +2a\,Q^{\prime\prime}+b\,Q-|Q|^{\alpha} Q = 0.
\end{equation}

\subsection{Properties of 1D ground states}\label{S:1D-properties}
In \cite{KPRS} we studied this case in various regimes, here we point out a few other properties of ground states. 
Fixing $a=1$ and $b=1+\epsilon$ with $0< \epsilon \ll 1$, 
we look at the small range of $\epsilon$, or  as $\epsilon \to 0$. 

In Figure \ref{F:1d-zeros} we consider the 1D cubic equation \eqref{E:explicit1} 
and show how the profiles of ground states change as $\epsilon$ \scalebox{0.75}{ $\searrow$} $0$ (i.e., $b=1+\epsilon$ \scalebox{0.75}{$\searrow$} $1$). Note that the height (or the $L^\infty$ norm) of the profiles is decreasing, the oscillations going out are increasing in their height (compare $b=1.01$ tail with $b=1.1$ tail), but the zeros of the profile stay the same (such behavior would be similar for other values of positive $a$ via scaling \eqref{E:ab-rescale}). 
\begin{figure}[!htb]
\includegraphics[width=0.49\hsize,height=.35\hsize]{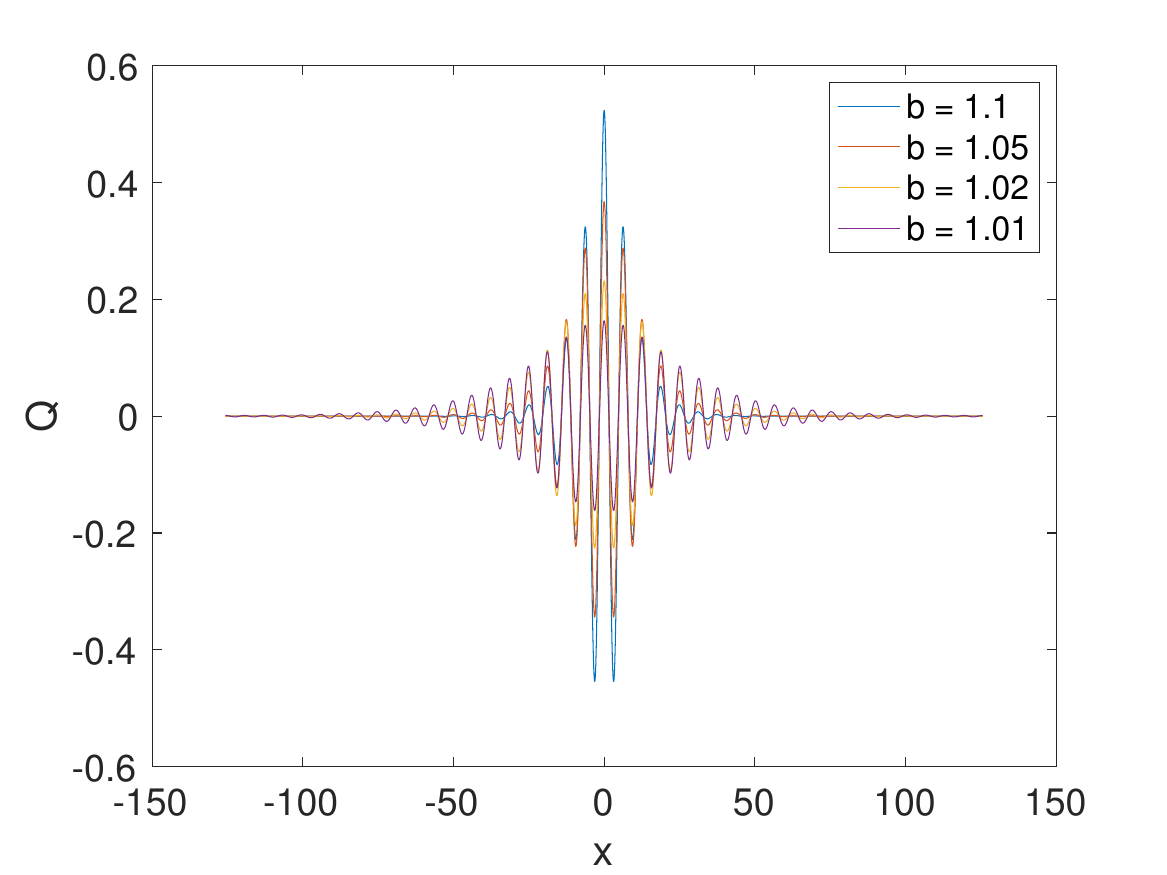}
\includegraphics[width=0.5\hsize,height=.35\hsize]{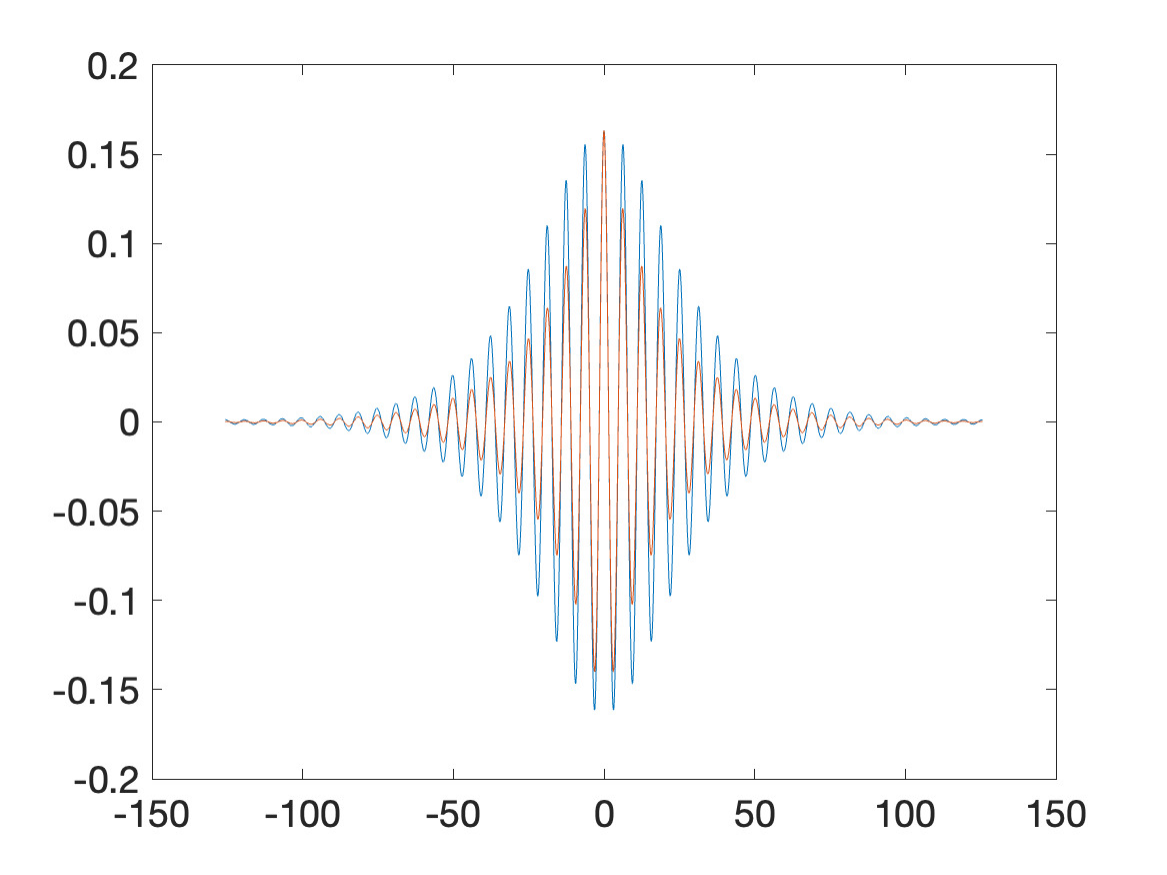}
\caption{\footnotesize Ground state profiles in the 1d cubic bi-NLS 
\eqref{biNLS} with $a=1$ and varying $b$ \scalebox{0.75}{$\searrow$} 
$1$ (left) and comparison of the ground state for $b=1.01$ (blue) with the Green's function (red) from \eqref{E:Green} (normalized as $\lambda G$ so the maximum heights coincide).}
\label{F:1d-zeros}
\end{figure}

As discussed in \S \ref{S:positive}, the Green's function in this regime would be oscillatory, and in 1D for $a=1$, $b=1+\epsilon$ it is given by (up to a constant and  phase shift, thus, dropping sine)
\begin{equation}\label{E:Green}
G(x) \sim \frac{e^{-k_2 |x|}}{4 k_2 \sqrt{1+\epsilon}} \cos(k_1 x), \qquad k_1 = (\tfrac{\sqrt{1+\epsilon}+1}{2})^{\frac12}, ~~k_2 = (\tfrac{\sqrt{1+\epsilon}-1}{2})^{\frac12}.
\end{equation}
To further confirm the Green's function influence on the solutions, we 
plot a normalized (in amplitude) version of \eqref{E:Green} on the 
right of Figure~\ref{F:1d-zeros} together with the ground state with 
$b=1.01$, showing that the zeros of both coincide and that there is similarity in the oscillatory decay. 

\noindent
\begin{minipage}[t]{0.48\textwidth}
\vspace{0pt}
In \cite{KPRS} we discussed, following \cite{LW2021} and \cite{FJMM2022} that 1D ground states (or energy minimizers) are even (with respect to a fixed point, which without loss of generality can be taken to be the origin). To confirm this numerically, we performed a search for non-even (odd) ground state solutions of \eqref{E:explicit1}, restricting to the odd functions space. After a meticulous search we were able to find an {\it odd} symmetry-constrained minimizer and obtain its energy, see red dot around mass $4$ in Figure \ref{F:1D-odd}, which is higher than that for the even minimizer - the blue line shows the even energy minimizers (ground states). 
\end{minipage}
\hfill
\begin{minipage}[t]{0.48\textwidth}
\vspace{0pt}
\centering
\includegraphics[width=\linewidth, height=0.66\linewidth]{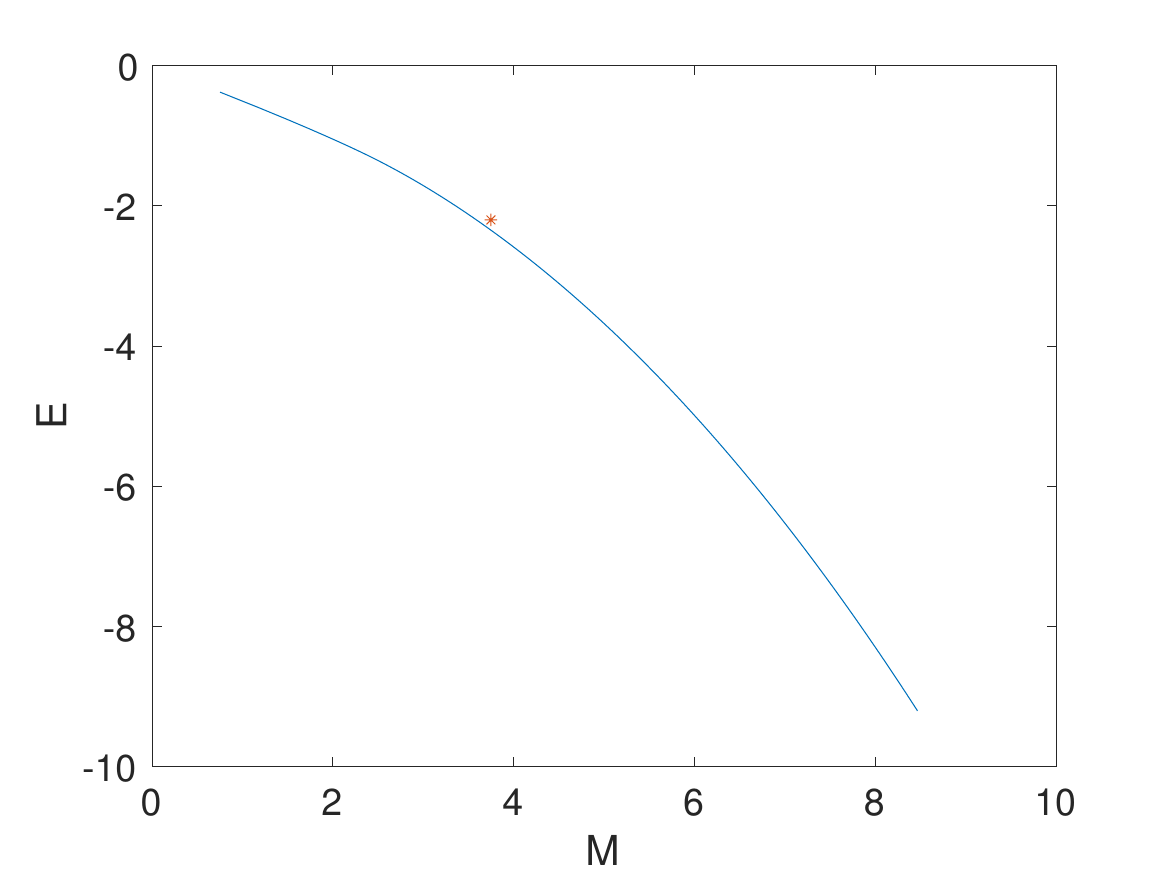}
\captionsetup{width=\linewidth}
\captionof{figure}{\small Energy minimizers of \eqref{E:1dGS}, $\alpha=2$: even (blue line) and odd (red dot).} 
\label{F:1D-odd}
\end{minipage}

\subsubsection{Mass asymptotics in 1D}\label{S:M-1D}

Using the decay rate of $\mathcal R_\ep$ in \eqref{E:R_sharp} and the relation with mass asymptotics (see \eqref{E:quotient-mass} in Lemma~\ref{L:quotient-to-mass}), the corresponding asymptotics for the ground state mass is
\begin{equation}\label{E:1d-rates}
M(Q_\epsilon) \approx 
\epsilon^{\frac{2}{\alpha}-{\frac12}} \quad \mbox{as} \quad \epsilon \to 0^+.
\end{equation}
Thus, in this regime the asymptotic behavior is 
\begin{equation}\label{E:alpha*}
\left\{
\begin{array}{ll}
\alpha<4 \quad & \quad
M(Q_\epsilon) \to 0,\\
\alpha=4 \quad & \quad
M(Q_\epsilon) \approx C <\infty,\\
\alpha>4 \quad & \quad
M(Q_\epsilon) \to +\infty,
\end{array}
\right.
\end{equation}
which explains the behavior of the ground state mass in dependence of $b$ in Figures 8 (A,D,G) and 9 (A, D) in \cite{KPRS}.

\begin{figure}[!htb]
\includegraphics[width=0.49\hsize,height=0.33\hsize]{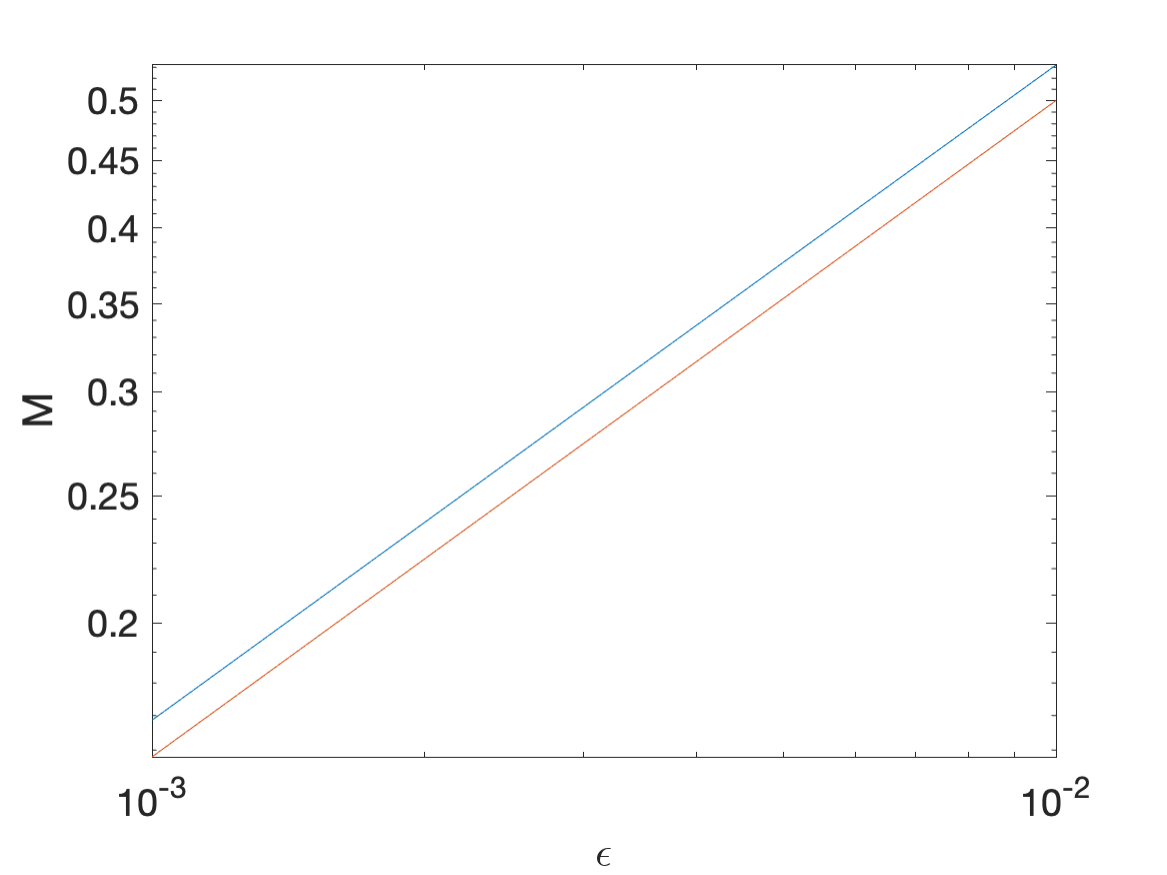}
\includegraphics[width=0.49\hsize,height=0.33\hsize]{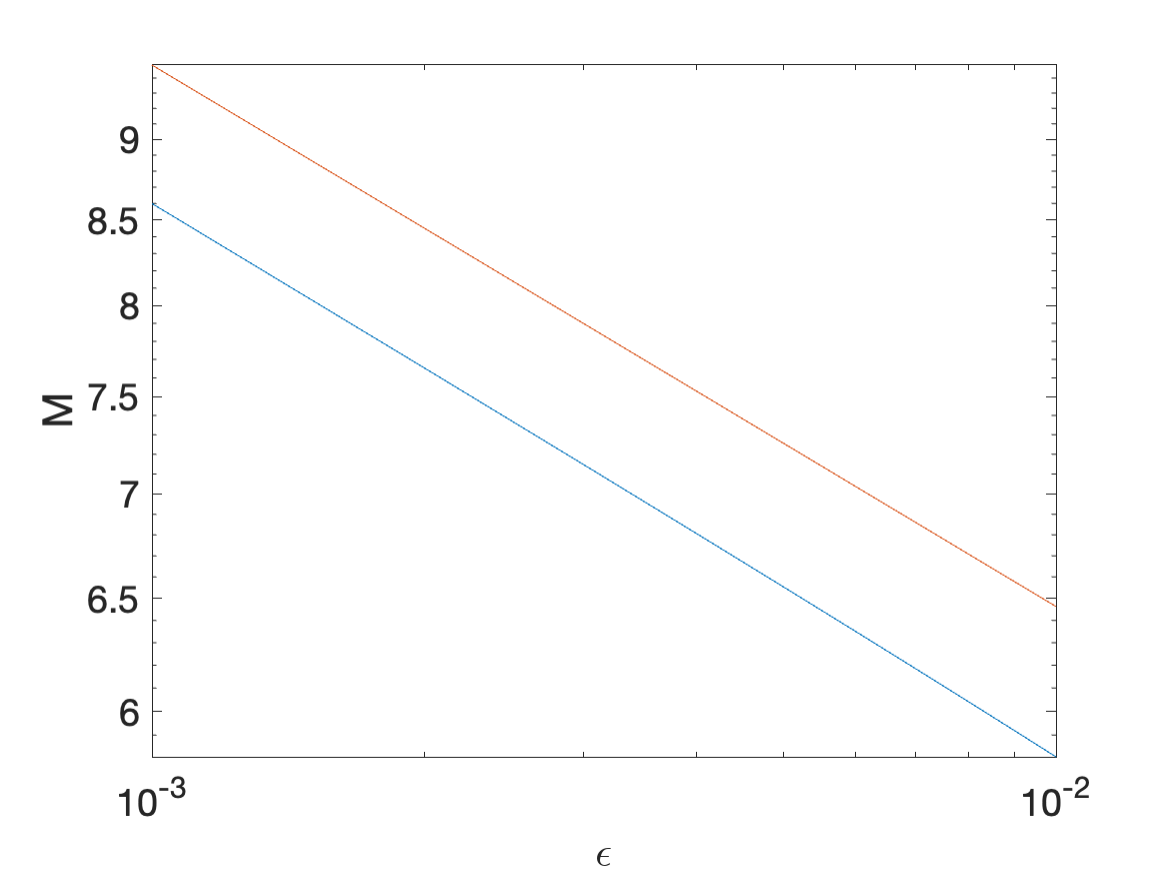}
\caption{\footnotesize Logarithmic scale dependence of the mass $M(Q_\ep)$ (blue) on $C\,\ep^{\frac2{\alpha}-\frac12}$ (red) as in \eqref{E:1d-rates}-\eqref{E:rates-sample}: $\alpha=2$ (left) and $\alpha=6$ (right), {$C=20$ is used only for visualization}.} 
\label{F:1d-rates}
\end{figure}
To confirm this further, we plot the logarithmic dependence of the mass $M(Q_\ep)$ on $\epsilon$ versus the rate given in \eqref{E:1d-rates} for cases $\alpha=2$ and $\alpha=6$ in Figure~\ref{F:1d-rates}, and for the borderline case $\alpha=4$ in Figure~\ref{F:1d-4}. In the first two cases, from the formula \eqref{E:1d-rates} we have 
\begin{equation}\label{E:rates-sample}
M(Q_\ep) \approx \ep^{\frac12} \quad \mbox{and} \quad M(Q_\ep) \approx \ep^{-\frac16},
\end{equation}
respectively, which is confirmed by parallel-type lines in each plot with respective slopes $\frac12$ and $-\frac16$.

\noindent
\begin{minipage}[t]{0.56\textwidth}
\vspace{0pt}
~~~In the case of $\alpha=4$, we trace a very slow change, note the scale on the $y$-axis in Figure~\ref{F:1d-4}, which provides some numerical evidence for stabilization of the $M(Q_\ep)$ towards a constant, as indicated in \eqref{E:alpha*}.  

~~We make an estimate of this constant via the 1D wave-packet similar to the 2D \eqref{E:Knapp-1}, namely, the example \eqref{E:ansatz-1d} (which is also the trial function in Appendix \S \ref{A:sub-formal-1d}) with the amplitude given in \eqref{E:1d-A}. With that we have 
$$
M(\tilde Q_\ep)=\frac{A^2}{2}\,\ep^{-1/2}\|W\|_{L^2(\R)}^2
$$
and
$$
A^\alpha=\frac{4\|W'\|_{L^2}^2+\|W\|_{L^2}^2}{2c_\alpha\,\|W\|_{L^{\alpha+2}(\R)}^{\alpha+2}} \, \ep,
$$
where $c_\alpha$ is given in \eqref{E:c-alpha}.
\end{minipage}%
\hfill
\begin{minipage}[t]{0.4\textwidth}
\vspace{0pt}
\centering
\includegraphics[width=\linewidth, height=0.72\linewidth]{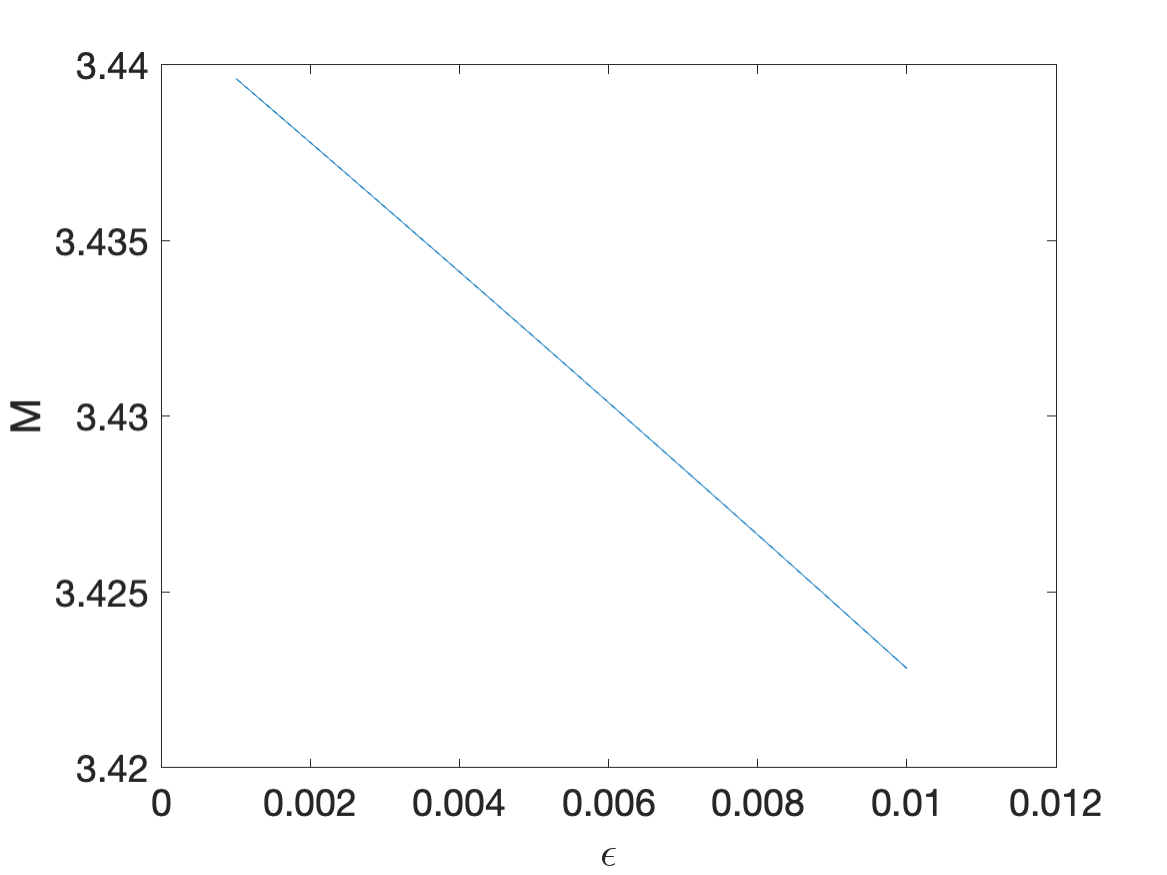}
\captionsetup{width=\linewidth}
\captionof{figure}{\footnotesize Dependence of the mass $M(Q_\ep)$ on $\ep$ as in \eqref{E:alpha*} for the threshold case $\alpha=4$, note that $3.44 \approx \sqrt{6/5}\, \pi$, the predicted value of $C(4)$ in \eqref{E:C4}.}
\label{F:1d-4}
\end{minipage}

For $\alpha=4$ this gives $A^2=\big( (4\|W'\|_{L^2}^2+\|W\|_{L^2}^2)/(2 c_4\|W\|_{L^6}^6)\big)^{1/2}\ep^{1/2}$, so the powers of $\ep$ cancel, and we obtain
\begin{equation}\label{E:C4-constant}
M(\tilde Q_\ep) \longrightarrow C(4)=\frac{\|W\|_{L^2(\R)}^2}{2}
\left(\frac{4\|W'\|_{L^2}^2+\|W\|_{L^2}^2}{2c_4\,\|W\|_{L^6(\R)}^6}\right)^{1/2}
\end{equation}
(which is invariant under $W\mapsto\lambda W$). 
The optimal profile is $W(x)=\sech^{1/2}(x)$ (as discussed in Appendix at the end of \S \ref{A:sub-formal-1d}), for which
$$
\|W\|_{L^2}^2=\int_\R\sech x\,dx=\pi,~~~~
\|W\|_{L^6}^6=\int_\R\sech^3x\,dx=\frac{\pi}2, ~~~~
\|W'\|_{L^2}^2=\frac14\int_\R \sech x\tanh^2x\,dx=\frac{\pi}8.
$$
Using \eqref{E:c-alpha} to compute $c_4=\Gamma(\frac72)/(\sqrt\pi\,\Gamma(4))=\frac5{16}$ and the above norm values, substituted into \eqref{E:C4-constant}, we obtain a prediction for the constant
\begin{equation}\label{E:C4}
C(4) = \sqrt{\tfrac65}\,\pi \approx 3.44,
\end{equation}
which is the value appearing in Figure \ref{F:1d-4}.

\begin{remark}\label{R:normGS}
The asymptotic behavior in \eqref{E:alpha*} can also be connected with the normalized ground states.
Assuming that the mass depends continuously on $b>1$ and as $b$ grows larger, it goes to infinity, top line of \eqref{E:alpha*} implies that for 
nonlinearities $0<\alpha <4$ the mass $M(Q_\epsilon)$ goes to zero as $\ep$ decreases to zero (e.g., see \cite[Fig. 8(A)]{KPRS}). In other words, values of $M(Q_{\ep})$ cover the entire interval $(0,\infty)$, so this branch contains a stationary solution of every prescribed mass $m>0$ (and thus, the energy minimizer could exist for any value of the constraint $m$ in \eqref{E:mass-constrained-problem}). On the other hand, for $4 < \alpha <8$, the mass diverges as $\ep \to 0^+$, and thus, a mass curve has a minimum (e.g., see the curves in \cite[Fig. 8 (D,G)]{KPRS}), and thus, a minimizer could exist only with a constraint above or equal to that minimum value. Such a change in the mass behavior corresponds to the value $\alpha_K=4$  in \eqref{E:alpha*}. (Note that in the $\alpha=4$ case, following our plots in \cite[Figure 8 (D)]{KPRS}, there is also a minimum value in the mass curve.)

The corresponding existence and nonexistence results for normalized minimizers are established in \cite[Theorem~1.2]{FJMM2022}: in 1D a minimizer of \eqref{E:mass-constrained-problem} exists for every $m>0$ when $\alpha<4$, whereas for $4 \leq \alpha < 8$ there is $m_0>0$ such that no minimizer exists for $m\in(0,m_0)$. The threshold from \cite{FJMM2022} is $\alpha < \max\{2,\frac{8}{d+1}\}$, which agrees
with our threshold $\alpha_K(d)=\frac{8}{d+1}$ when $d \leq 3$, 
in particular, in both dimensions considered here (for 2D see Remark~\ref{R:Knapp-threshold}). We mention that the corresponding 1D value $p_*=\alpha_K+1=5$ is also studied in \cite{SP2020}.
\end{remark}

\section{Ground states in 2D}\label{S:2dGS}

\subsection{Basic properties}
We re-write the equation \eqref{biNLS} in 2D with $\Delta = \partial_x^2 + \partial_y^2$, 
\begin{equation}
\Delta^2 = (\partial_x^2 + \partial_y^2)^2
\equiv  \partial_{x}^{4} + 2\partial_{x}^{2} \partial_y^2  + \partial_{y}^{4},
\end{equation}
and thus, 
\begin{equation}\label{E:2dbiNLS}
i \partial_{t} u -(\partial_{x}^{2} + \partial_y^2)^2 u  - 2a (\partial_{x}^{2} + \partial_{y}^{2})u +|u|^{\alpha}u=0,
\end{equation}
Substituting a standing solitary wave profile  
$$
u(x,y,t) = e^{i bt} Q(x,y), \quad b>0,
$$ 
with $Q$ real, vanishing at infinity, into \eqref{E:2dbiNLS}, we obtain \begin{equation}\label{E:2dGS}
(\partial_{x}^{2} + \partial_y^2)^2 Q + 2a (\partial_{x}^{2} + \partial_y^2) Q + b\,Q - |Q|^{\alpha} Q=0.
\end{equation}
We explicitly write the energy from \eqref{E:E}, 
$$
E(u) =  \int_{\mathbb R^2} \Big(\frac12 |u_{xx} + u_{yy}|^2 -a(|u_x|^2 +|u_y|^2) - \frac1{\alpha+2}|u|^{\alpha+2}  \Big) dx dy,
$$
and the two Pokhozhaev identities in 2D, 
\begin{equation}\label{E:P1}
\|\Delta Q\|^2_{L^2(\R^2)} -2a \|\nabla Q\|^2_{L^2(\R^2)} + b\|Q\|^2_{L^2(\R^2)} - \|Q\|_{L^{\alpha+2}(\R^2)}^{\alpha+2} = 0,
\end{equation}
\begin{equation}\label{E:P2}
\|\Delta Q\|^2_{L^2(\R^2)} - b\|Q\|^2_{L^2(\R^2)} + \frac{2}{\alpha+2} \|Q\|_{L^{\alpha+2}(\R^2)}^{\alpha+2} = 0.
\end{equation}
Solving for  $\|\Delta Q\|_{L^2}^2$ or $\|\nabla Q\|^2_{L^2}$ from \eqref{E:P1} and \eqref{E:P2}, 
we express the energy via mass and the $L^2$-norm of the gradient
\begin{equation}\label{E:E-viaM}
E(Q) = \frac{b(\alpha-4)}{2(\alpha+4)} \, M(Q) - \frac{\alpha \, a}{\alpha+4} \| \nabla Q\|_{L^{2}(\R^2)}^{2}, 
\end{equation}
which, for instance, shows that in the $L^2$-critical case ($\alpha=4$) the energy of the ground state is positive if $a<0$ and negative if $a>0$. Note that in the scaling-invariant case $a=0$, the ground state has zero energy. 
Another useful equality from Pokhozhaev identities is the connection of energy and mass via the potential norm, which we obtained in \eqref{EM-relation} and 
for $\alpha=2$ it becomes 
\begin{equation}\label{E:action-cubic}
E(Q)+\frac b2 M(Q) = \frac14\|Q\|_{L^4(\R^2)}^4.
\end{equation}
This identity is especially useful for verifying numerical computations.
One of the first numerical checks we do is to confirm the bound coming from \eqref{E:action-cubic}:
$$
E(Q) + \frac{b}2 M(Q) > 0, \quad \mbox{or} \quad E(Q) >  -\frac{b}2 M(Q).
$$
This holds in all our computations of solutions to \eqref{E:2dGS}.

Secondly, in our computations, we also check the error $\mathcal E(Q)$ for any solution $Q$ of \eqref{E:2dGS} by showing that, e.g., in the case of $\alpha=2$, 
\begin{equation}\label{E:Err}
\mathcal{E}(Q): = E(Q) + \frac{b}2 \, M(Q) - \frac{1}{4} 
\|Q\|_{L^{4}(\R^2)}^{4}.
\end{equation}

Finally, the values of $E(Q)+\frac b2 M(Q)$ and $\|Q\|_{L^4(\R^2)}^4$ 
play an important role in distinguishing radial and nonradial branches, which we discuss next.

\subsection{Remarks on 2D action asymptotics}
\label{S:2d-asym}
Recalling the asymptotics for the action (and hence for the potential norm) \eqref{E:S-asym-general} and \eqref{E:S-asym-rad-2d} in the case of the cubic nonlinearity $\alpha=2$, the nonradial ground states satisfy
\begin{equation}\label{E:nr-3}
S_{1+\epsilon}(Q_\epsilon^{\rm nr})
\approx \|Q_\epsilon^{\rm nr}\|_{L^4(\R^2)}^4
\approx \epsilon^{5/4},
\end{equation}
whereas the radial branch has
\begin{equation}\label{E:r-3}
S_{1+\epsilon}(Q_\epsilon^{\rm rad})
\approx \|Q_\epsilon^{\rm rad}\|_{L^4(\R^2)}^4
\approx \ep \,|\ln\epsilon|^{-1}.
\end{equation}
Coming back for a moment to the lower bound from \cite{LW2021}, we have the corresponding lower-bound exponent
\begin{equation}\label{E:r-3b}
S_{1+\epsilon}(Q_\epsilon^{\rm rad}) \gtrsim \epsilon^\sigma,
\quad \mbox{for every} \quad \sigma>1,
\end{equation}
and thus, showing the asymptotic behavior of the potential norm behaving as $\ep^\sigma$ with $\sigma$ just slightly greater than $1$ but smaller than $\frac54$, 
would let us separate the nonradial branch of ground states from the radial one 
(this would also follow from \eqref{E:r-3} bound $\epsilon^{5/4} \ll \ep |\ln \epsilon|^{-1}$ as $\epsilon\to 0^+$, see remark below).

\begin{remark}\label{R:rates}
In Figure~\ref{F:loglog-1} we show that indeed the nonradial branch confirms the asymptotic behavior rate with power $\frac54$, while the radial branch has power $\sigma$, which on the computed domain of $\ep$ is approximated by $1+\frac19$. Note that computationally this is consistent with the asymptotic rate \eqref{E:r-3b} or more precise \eqref{E:r-3}, which contains a logarithmic correction. To investigate it more accurately can be a future endeavour, for the current work, however, first of all, log corrections are computationally challenging to confirm, and secondly, obtaining ground state solutions for much smaller $\ep$ than $10^{-3}$ is a computationally arduous task. 
\end{remark}


\subsection{Initial approximations in 2D}\label{S:InAp}

In order to search for the nonradial ground states, one should have 
an appropriate initial guess. We discuss two such possibilities, based on the structure of the Laplacian and the Knapp example.

\subsubsection{Angular modes of the Laplacian}\label{S:InLap}

We recall the standard harmonic modes on a disk.  Let $B_R$ be the disk of radius $R$, and consider the Dirichlet problem
$$
\Delta u=0 \quad \mbox{in } B_R, \qquad  u(R,\theta) = f(\theta).
$$
In polar coordinates, the functions
\begin{equation}\label{E:cos}
u_m^c(r,\theta) = r^m \cos(m \theta), \qquad u_m^s(r,\theta) = r^m \sin(m\theta),
\end{equation}
solve
$$
\Delta u=0 \quad \mbox{in } B_R
$$
for $m=0,1,2,\ldots$ in the cosine case and $m=1,2,\ldots$ in the sine case.  Expanding the boundary data into a Fourier series, we have
$$
f(\theta)=\frac{a_0}{2} +\sum_{m=1}^{\infty}\left(a_m\cos(m\theta)+b_m\sin(m\theta)\right),
$$
where
$$
a_m=\frac1{\pi}\int_0^{2\pi} f(\theta) \cos(m\theta)\,d\theta,
\quad m\geq0,
\quad \mbox{and} \quad
b_m=\frac1{\pi}\int_0^{2\pi} f(\theta) \sin(m\theta)\,d\theta,
\quad m\geq1.
$$
Then the harmonic extension is given by
$$
u(r,\theta)=\frac{a_0}{2} +\sum_{m=1}^{\infty}
\left(\frac{r}{R}\right)^m
\left(a_m \cos(m \theta) + b_m \sin(m \theta)\right).
$$
Note that the case $m=0$ gives the radially symmetric harmonic mode, while in the cases when $m\geq1$ one obtains the $m$-th angular mode. We also point out that the angular function $\cos(m\theta)$ has $2m$ nodal rays (which we see below in computations).  Moreover, under the reflection $(x,y)\mapsto(-x,-y)$,
which corresponds to $\theta \mapsto \theta + \pi$, we have
$$
\cos(m(\theta+\pi))=(-1)^m \cos(m \theta).
$$
Thus, the modes with even $m$ are compatible with the property of ground states 
being even $u(x,y)=u(-x,-y)$, 
discussed earlier in \S \ref{S:positive}. For this reason, in some of our numerical simulations, we use angular expressions of type $r^m \cos(m\theta)$, typically with even $m$, for 
$Q^{(0)}$ (we also show comparison of energies for both odd and even $m$ in Figure~\ref{figME}).
For instance, to construct angular solutions we typically choose initial iterates of the 
form 
\begin{equation}\label{init}
Q^{(0)} = \lambda r^{m}e^{-r^{2}}\cos(m \theta),
\end{equation}
where $\lambda$ is a positive constant that helps fine-tune the amplitude of the solution and
the envelope function being a radial Gaussian to ensure sufficient decay at the boundary. 
(We choose the Gaussian $e^{-r^2}$ for the decay at infinity 
and smoothness at the origin, however, we have also tried other 
functions such as the super-Gaussian, $e^{-r^4}$, or $\sech~ {x} \sim 
e^{-|x|}$, as well as uncorrelated powers of $r$ and $\cos$ such as 
$r^{m_1} \cos(m_2 \theta)e^{-r^{2}}$, which did not make much difference in computing the $m$-th node ground state solutions.)

\subsubsection{Cap concentration}

To utilize the Knapp geometry \eqref{E:Knapp-1} of being concentrated on the opposite poles, we consider the following initialization ($b=1+\epsilon$):
$$
Q^{(0)}_\epsilon (x,y) \sim  \epsilon^{1/2} \,G(\epsilon^{1/2} x, \epsilon^{1/4}y) \cos x 
$$
with the envelope function $G$ being of Gaussian-type,
$$
G_\epsilon(x,y)=e^{ -\frac12 
\left( \frac{(\epsilon^{1/2} x)^{2}}{\sigma_x^{2}} 
+\frac{(\epsilon^{1/4}y)^{2}}{\sigma_y^{2}}
\right)}. 
$$
Explicitly, giving a parameter $\lambda$ to fine-tune the initial amplitude, we have 
\begin{equation}\label{E:cap1}
Q^{(0)}_\epsilon (x,y) =\lambda \, \epsilon^{1/2} 
e^{ -\frac12 \left( \frac{(\epsilon^{1/2} x)^{2}}{\sigma_x^{2}} 
+\frac{(\epsilon^{1/4}y)^{2}}{\sigma_y^{2}}
\right)} 
\cos x,
\end{equation}
with parameters $\lambda \in [0.5,5]$ and $\sigma_x, \sigma_y \in [0.5, 1.5]$.

Due to symmetry of cosine and Gaussian, our initial guess is even, $Q_\ep^{(0)} (-x,-y) = Q_\ep^{(0)} (x,y)$. Furthermore, it is even in each variable, so we expect that the solution has coordinate-wise symmetry $Q(-x,y) = Q(x,y)$ and similar in $y$. As we show in Section \ref{S:nonradial} this is indeed the case.

\section{Numerical construction of ground state solutions: cubic case}\label{S:2d-cubic}

We now numerically construct ground states candidates to the 
2D bi-harmonic cubic NLS equation \eqref{biNLS}, 
i.e., localized solutions as critical points of the elliptic equation \eqref{E:groundstate}. 
As in the previous sections we take $a=1$ and write $b=1+\epsilon$. In this 
section we consider the cubic nonlinearity $\alpha=2$. 

The numerical approach is a 2D version of the 1D approach in 
\cite{KPRS}, a 2D discrete Fourier transform (DFT) with a 
Newton-Krylov method where the action of the inverse of the Jacobian 
is computed with GMRES \cite{SS86}.

{\it Parameters.} We initially work with the following numerical parameters: $L_{x}=L_{y}=40$, 
$N_{x}=N_{y}=2^{11}$. Note that the energy or action identity \eqref{EM-relation}, coming from Pokhozhaev, is typically satisfied by the solutions constructed in this section to 
the order of the stopping criteria for the Newton iteration (the 
iterations stopped once the residual of the discretized equation 
\eqref{E:2dGS} has an $L^{\infty}$ norm smaller than a certain threshold, 
typically $10^{-6}$). 

{\it Branch tracing.} We typically start with a larger value for $b$, 
and then apply a tracing technique to reach values of $b \sim 1$ and smaller. This means we consider, for example, $100$ values $b_{n}$, $n=1,\ldots,100$, between the 
minimal and maximal considered values for $b$, and we always use the result for a previous 
value $b_{n+1}$ as the initial iterate for a slightly smaller $b_{n}$.

\subsection{Initializations with angular modes}\label{S:angular-modes}
We first utilize the properties of the Laplacian on a disk as discussed in \S \ref{S:InLap} as they provide a principal structure of the equation \eqref{E:groundstate} and we take \eqref{init} as the initial guess.

\subsubsection{Radially symmetric case, $m=0$}
We start with the case $m=0$ in \eqref{init}, and for $b=4$ we choose $\lambda=3$. The radial ground state solution in this case is shown in Figure~\ref{figm0}. By continuous tracing from $b=4$ we obtain the solution for $b=1.1$. Note that solutions become more oscillatory as $b \to 1$, similar to the 1D case (e.g., Figure 5 in \cite{KPRS}). 

To reach even smaller values of $b$, we introduce another 100 values 
for $b\in[1.01,1.1]$ and obtain the solution with the above tracing 
method. A finer discretization becomes necessary, since the solution 
becomes more and more oscillatory, and the peaks (maximum values) change strongly with the value of $b$ similar to the 1D case shown in Figure \ref{F:1d-zeros}. 
\begin{figure}[!htb]
\begin{subfigure}{.32\textwidth}
\includegraphics[width=1\linewidth,height=0.75\linewidth]{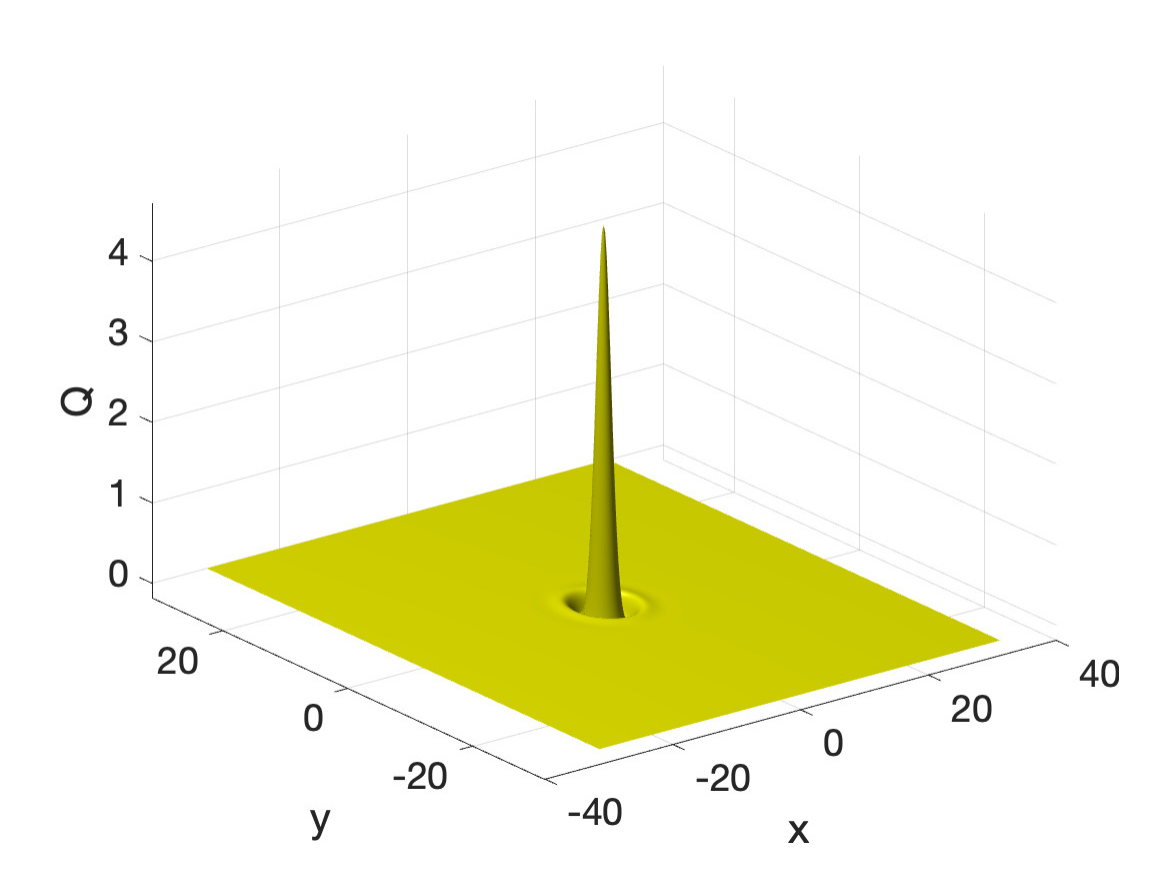}
\subcaption[]{{\footnotesize $b=4$}}
\end{subfigure}
\begin{subfigure}{.32\textwidth}
\includegraphics[width=1\linewidth,height=0.75\linewidth]{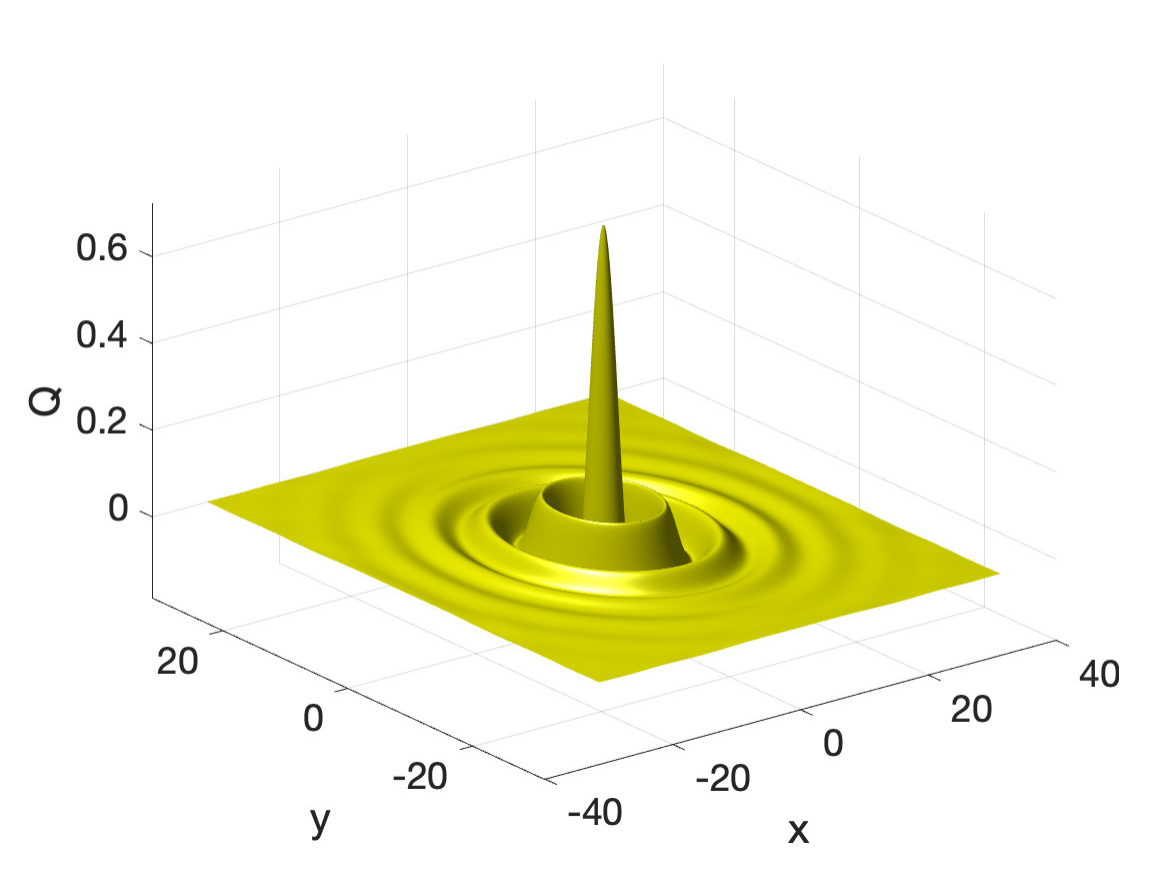}
\subcaption[]{{\footnotesize $b=1.1$}}
\end{subfigure}
\begin{subfigure}{.32\textwidth}
\includegraphics[width=1\linewidth,height=0.75\linewidth]{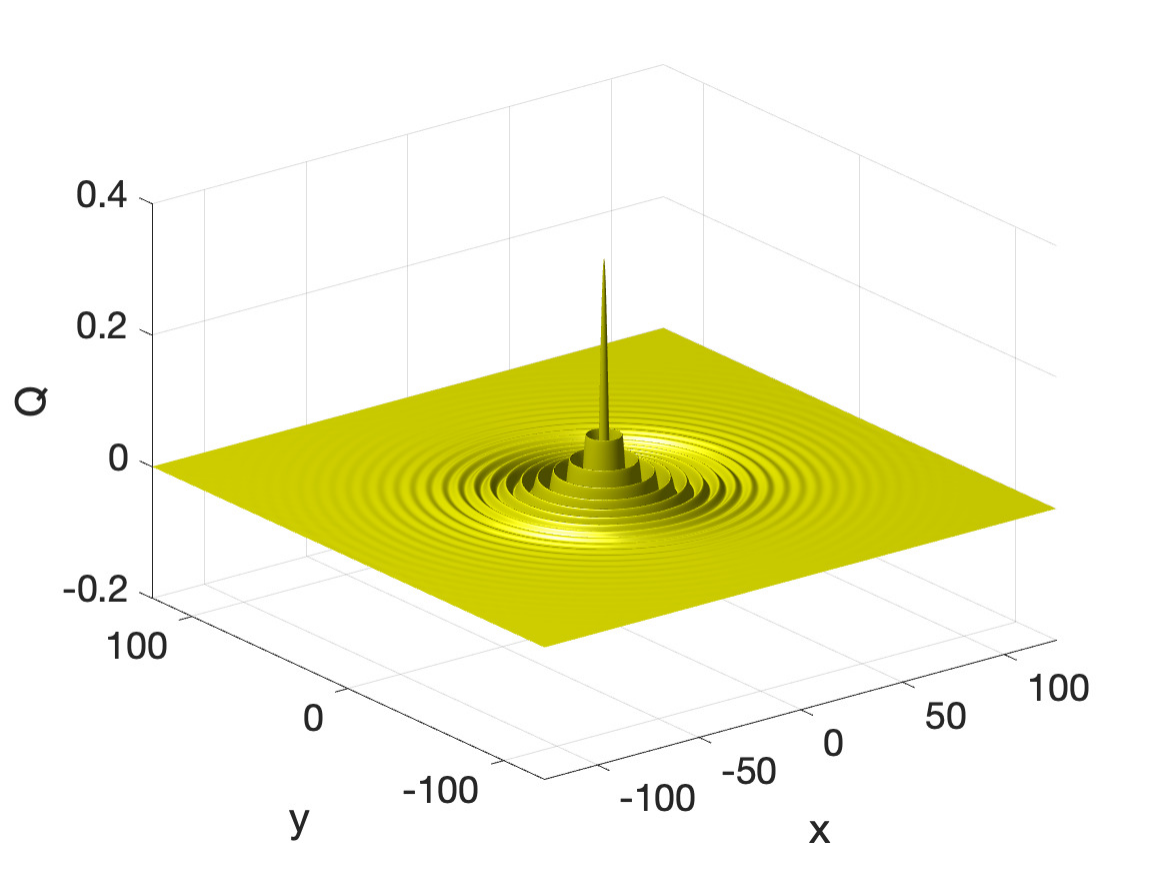}
\subcaption[]{{\footnotesize $b=1.01$}}
\end{subfigure}\\
\begin{subfigure}{.32\textwidth}
\includegraphics[width=1\linewidth,height=0.65\linewidth]{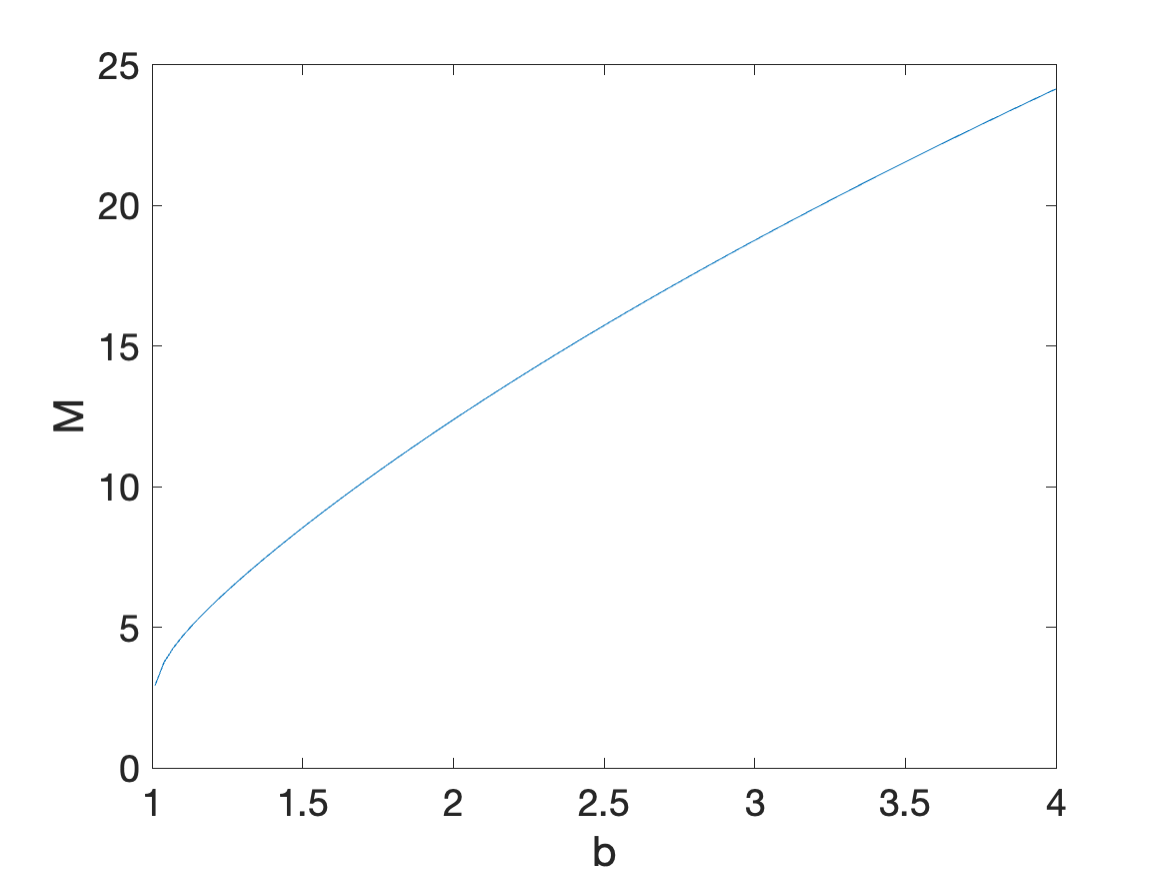}
\subcaption[]{{\footnotesize $M(Q) = M(b)$}}
\end{subfigure}
\begin{subfigure}{.32\textwidth}
\includegraphics[width=1\linewidth,height=0.65\linewidth]{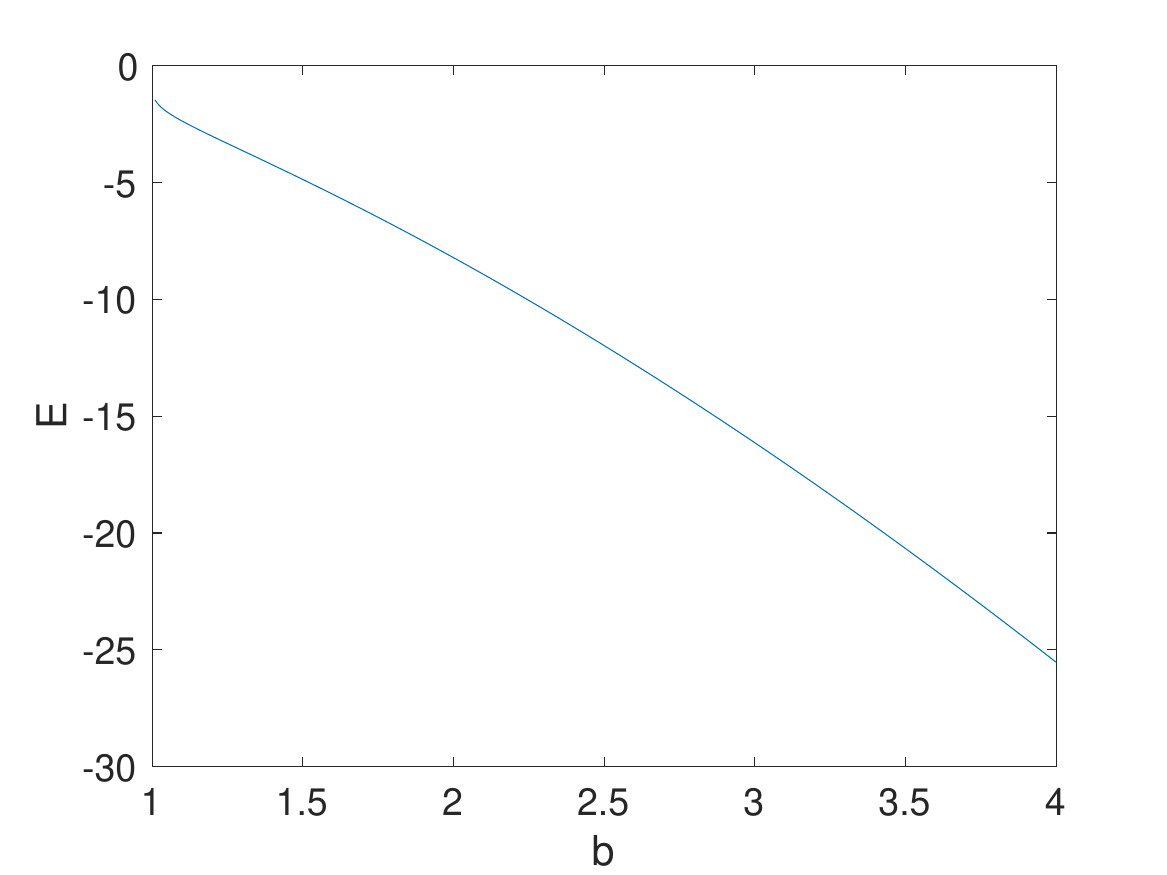}
\subcaption[]{{\footnotesize $E(Q)$ as function of $b$}}
\end{subfigure}
\begin{subfigure}{.32\textwidth}
\includegraphics[width=1\linewidth,height=0.65\linewidth]{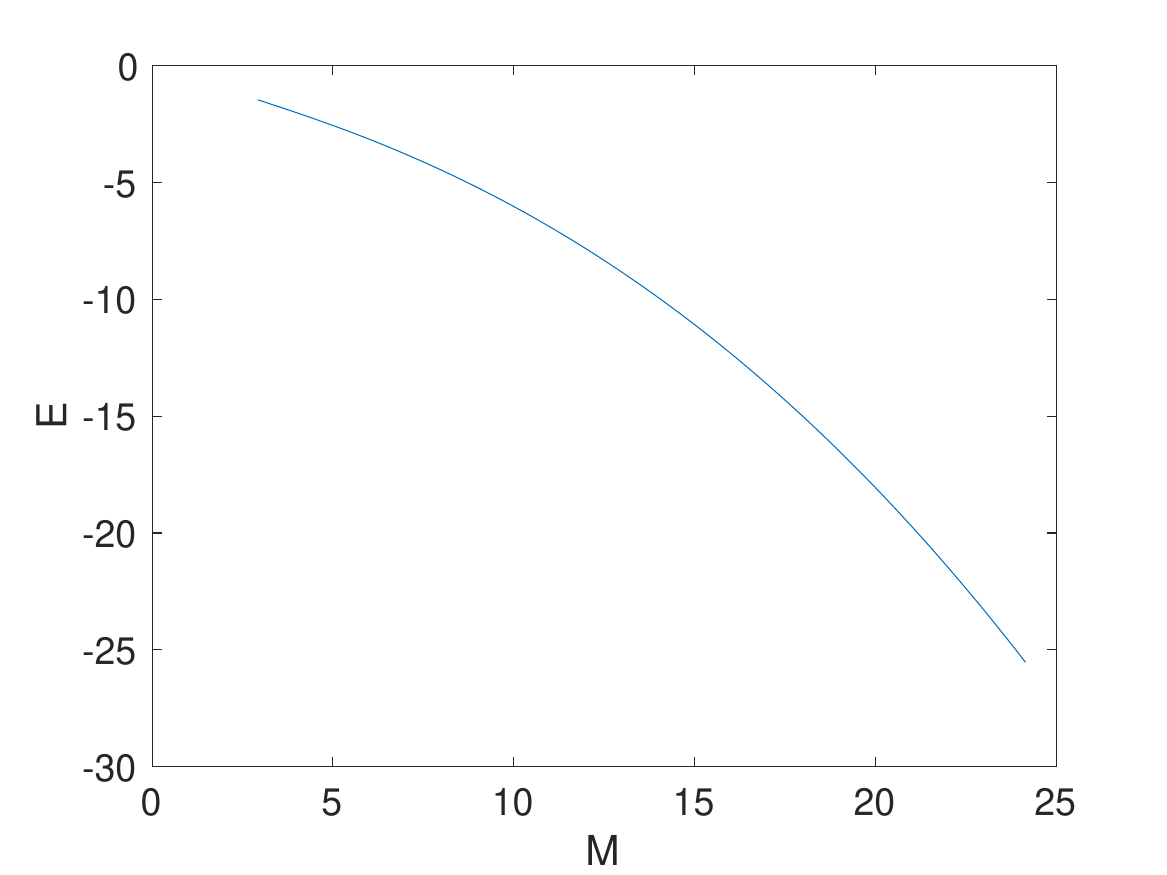}
\subcaption[]{{\footnotesize $E = E(M)$}}
\end{subfigure}
\caption{\footnotesize Radially symmetric ($m=0$) solutions to the cubic equation \eqref{E:2dGS}, $\alpha=2$, for different values of $b$ (top row). Dependence of mass $M(Q)$ and energy $E(Q)$ on $b$ (bottom left, middle), and energy as a function of mass $E=E(M)$ (bottom right).}
\label{figm0}
\end{figure}
\begin{figure}[!htb]
\begin{subfigure}{.32\textwidth}
\includegraphics[width=1\linewidth,height=0.75\linewidth]{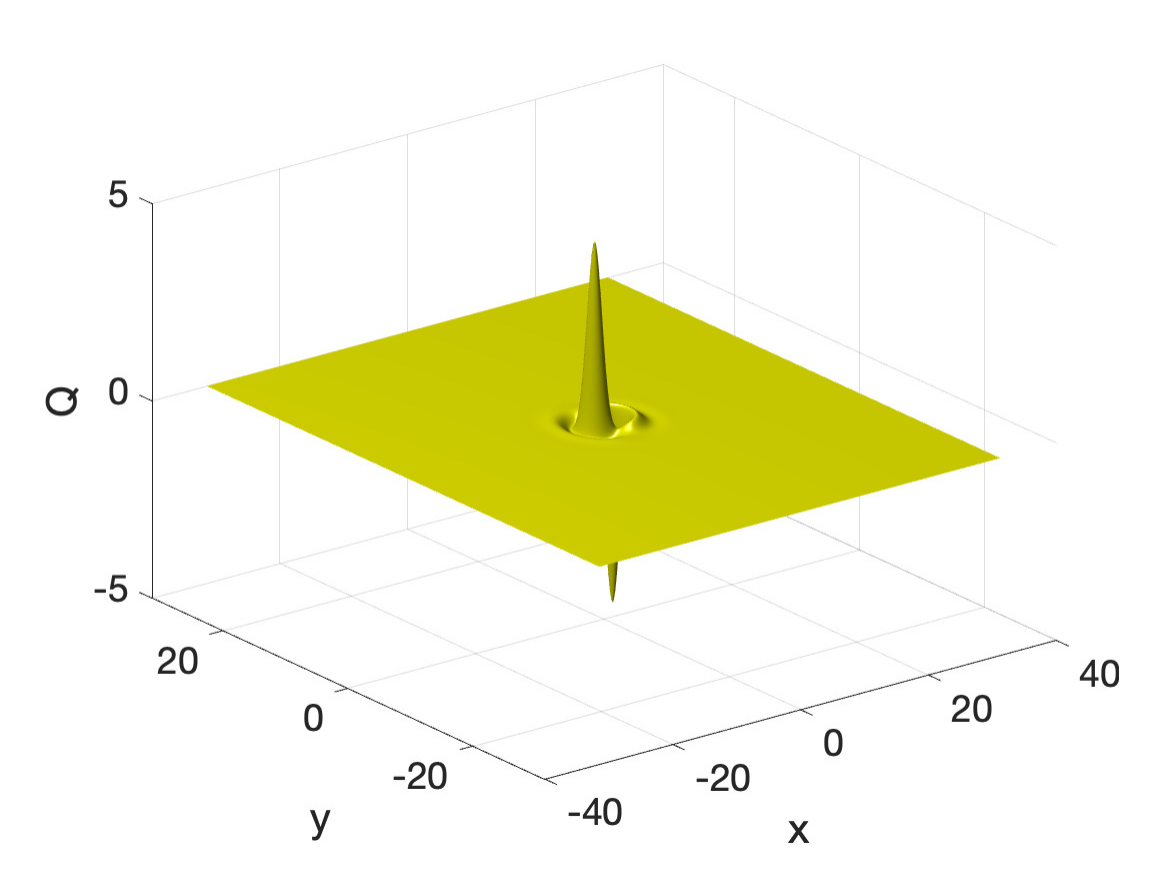}
\subcaption[]{{\footnotesize $b=10$}}
\end{subfigure}
\begin{subfigure}{.32\textwidth}
\includegraphics[width=1\linewidth,height=0.75\linewidth]{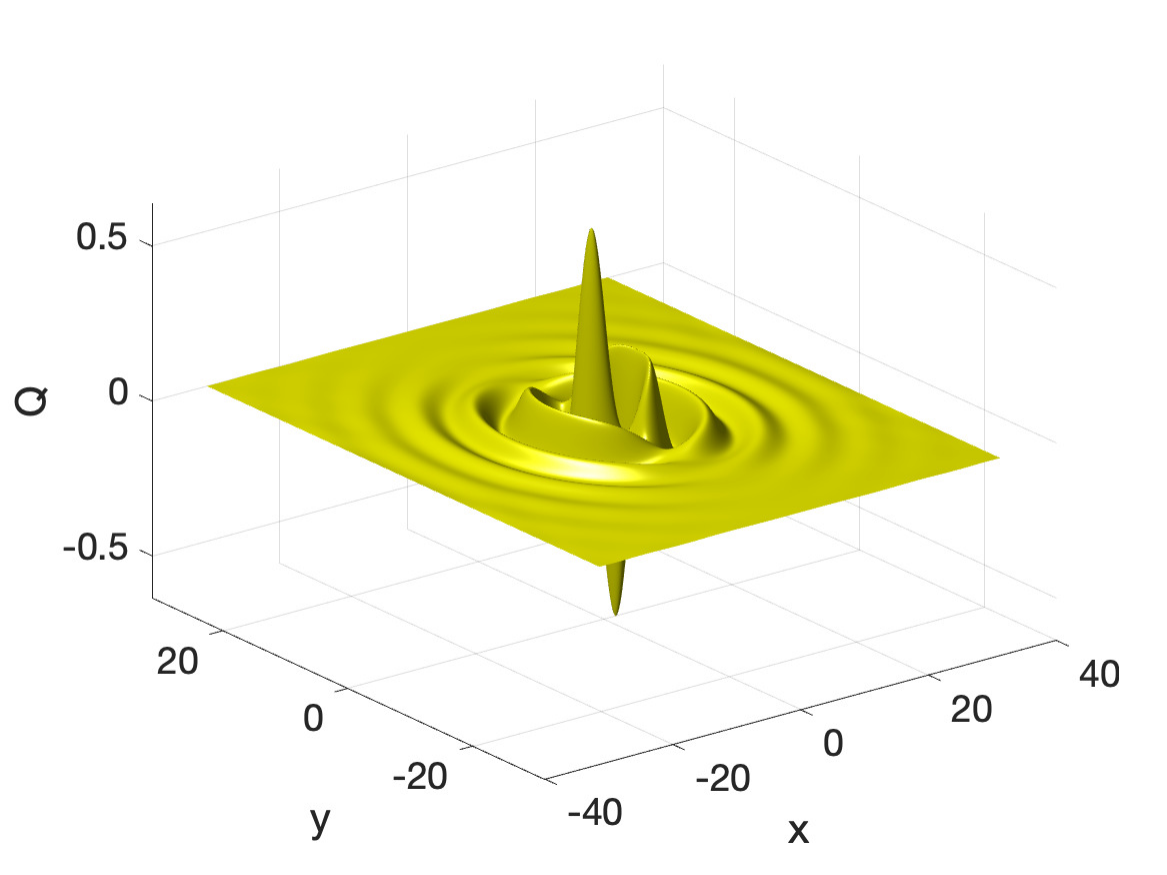}
\subcaption[]{{\footnotesize $b=1.1$}}
\end{subfigure}
\begin{subfigure}{.32\textwidth}
\includegraphics[width=1\linewidth,height=0.75\linewidth]{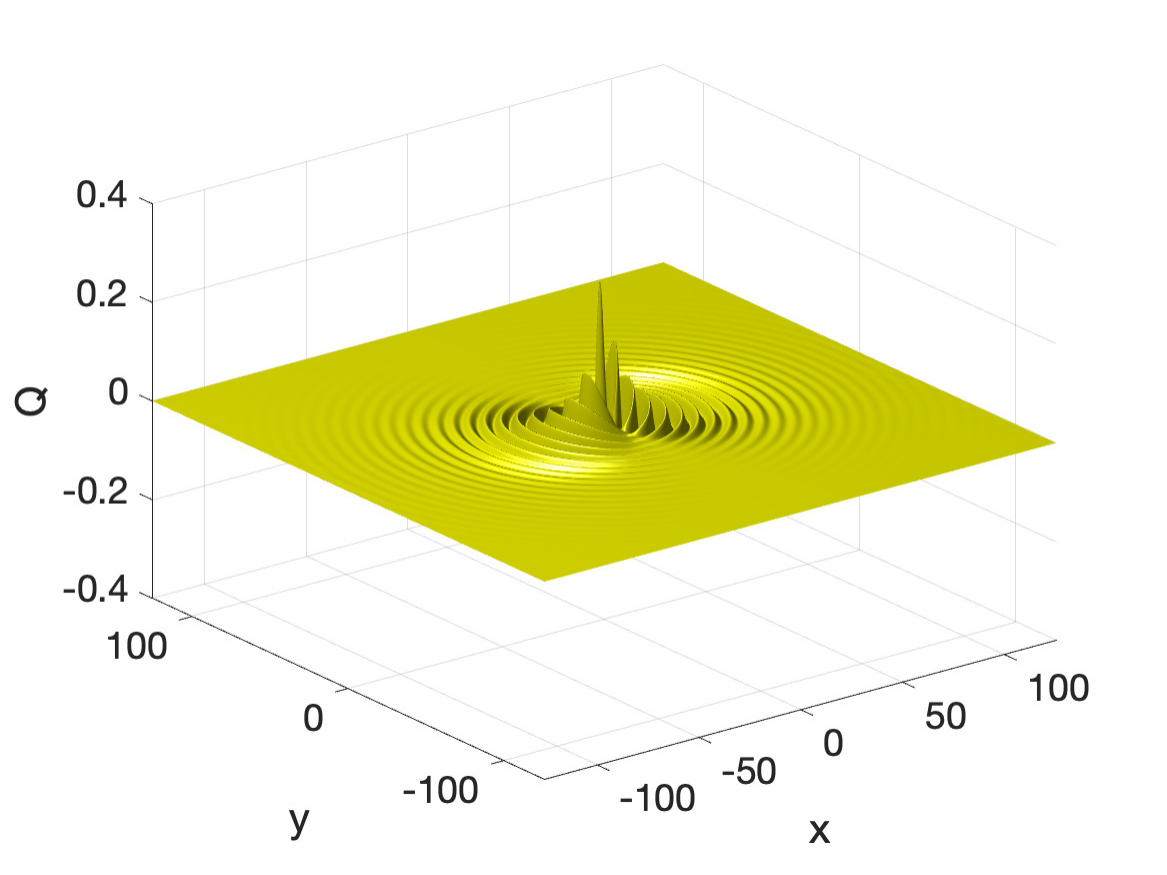}
\subcaption[]{{\footnotesize $b=1.01$}}
\end{subfigure}\\
\begin{subfigure}{.32\textwidth}
\includegraphics[width=1\linewidth,height=0.65\linewidth]{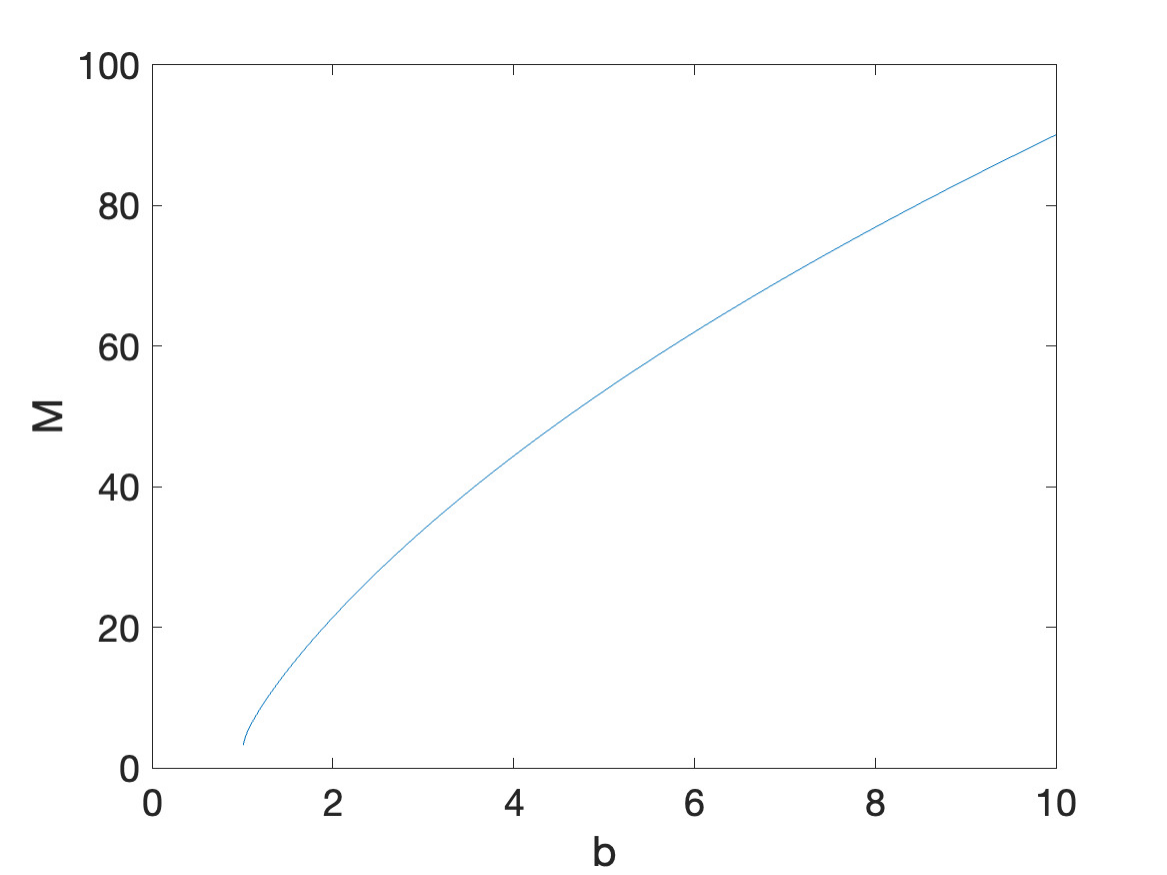}
\subcaption[]{{\footnotesize $M(Q) = M(b)$}}
\end{subfigure}
\begin{subfigure}{.32\textwidth}
\includegraphics[width=1\linewidth,height=0.65\linewidth]{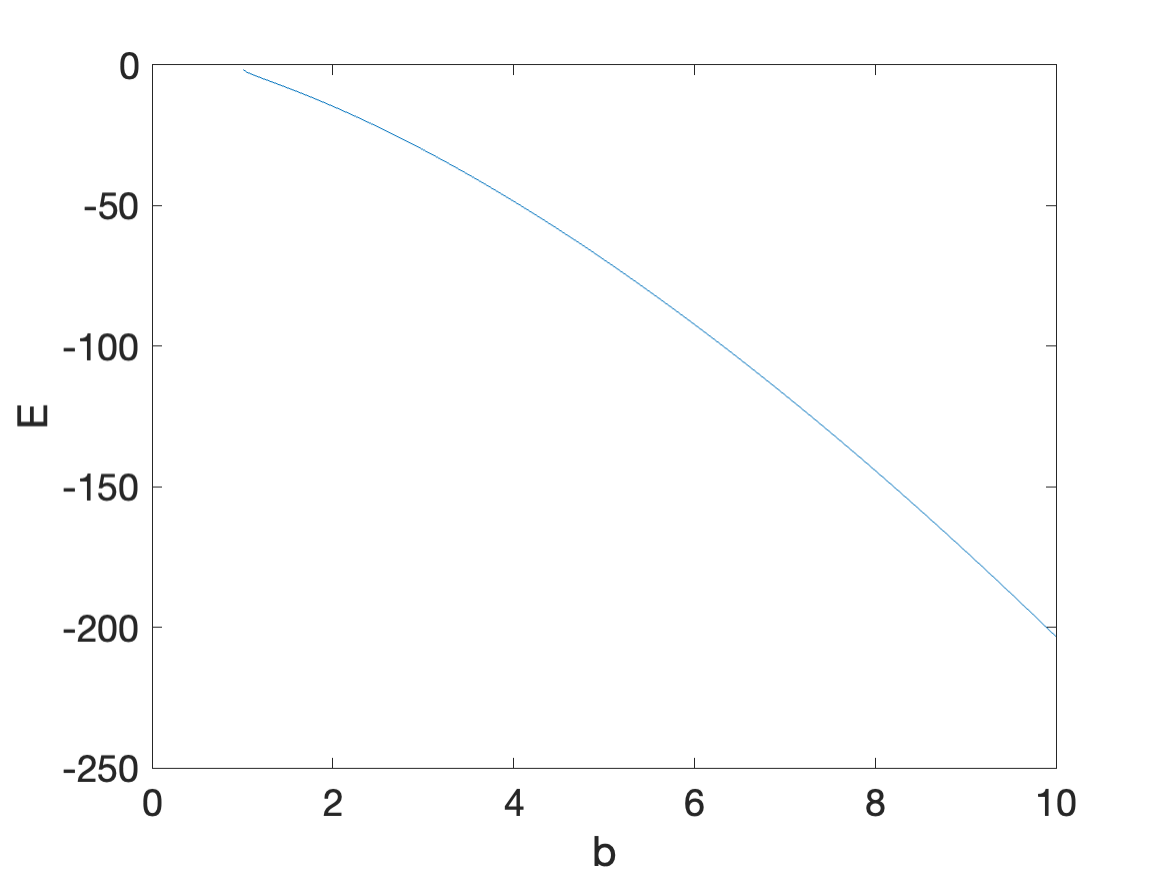}
\subcaption[]{{\footnotesize $E(Q)$ as function of $b$}}
\end{subfigure}
\begin{subfigure}{.32\textwidth}
\includegraphics[width=1\linewidth,height=0.65\linewidth]{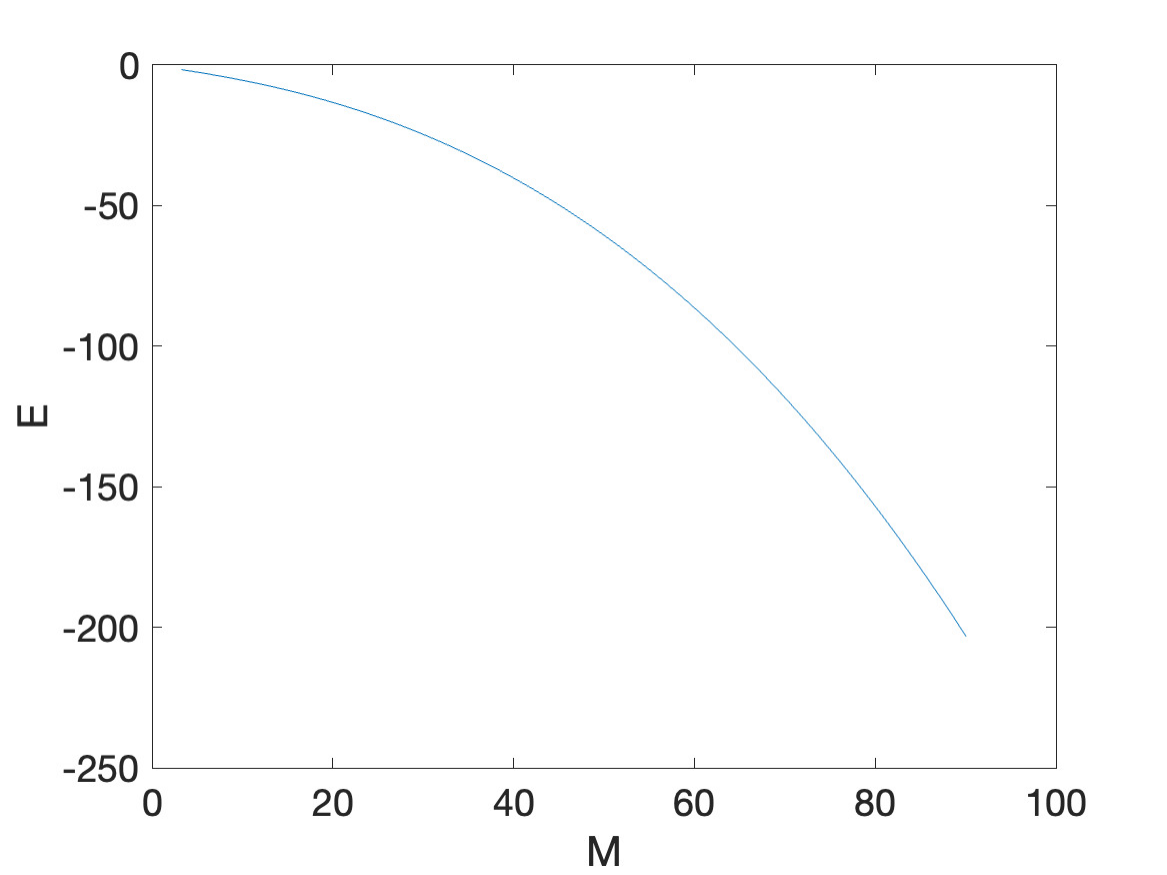}
\subcaption[]{{\footnotesize $E = E(M)$}}
\end{subfigure}
\caption{\footnotesize Solutions to the cubic equation \eqref{E:2dGS} with angular mode $m=1$ (odd symmetry): profiles for different $b$ (top row). Dependence of mass $M(Q)$ and energy $E(Q)$ on $b$ (bottom left, middle), and energy as a function of mass $E=E(M)$ (bottom right).}
\label{figm1}
\end{figure}

We show the mass and energy dependence on $b$ and the energy as the function of mass for these solutions in the bottom of Figure~\ref{figm0}. It can be seen that the mass is monotonically increasing with $b$, whereas the energy is decreasing. 
This implies that the energy in dependence of $M$ is a smooth 
monotone graph. We check these dependencies in each case of angular mode initialization to make sure that no turning points appear in any of these graphs, as this will be essential for other nonlinearities. 
Below, we also compare the rate of $M(Q_{\ep})$ in dependence with $\ep$, as derived in \S \ref{A:sub-rig-2d}.

\subsubsection{The case $m=1$} While this is not a ground state solution, it can be thought of the solution to \eqref{E:2dbiNLS} restricted to the odd symmetry subspace (similar to the 1D example we discussed at the end of \S \ref{S:1D-properties} and Figure~\ref{F:1D-odd}).
Thus, we take $m=1$, and use $\lambda=7$ in \eqref{init} for $b=10$. The 
solutions for $b=10$ and $b=1.1$ are shown in the top row of Figure~\ref{figm1}. 
To reach even smaller values of $b$, for instance, $b=1.01$, we proceed as in the radially symmetric case, refining and tracing the interval for smaller $b$ values.  
Similar to the case with radial symmetry, the solution becomes much more oscillatory as 
$b$ \scalebox{0.75}{ $\searrow$} $1$.

In the bottom row of Figure~\ref{figm1}, we show the mass and energy dependence on $b$ and between each other. It can be seen that the mass is again monotonically increasing with $b$ and that the energy is decreasing either depending on $b$ or on the mass $M(Q)$.

\subsubsection{The case $m=2$}
For $m=2$, we use $\lambda=9$ in \eqref{init} for $b=10$. The 
solution for $b=10$, $b=1.1$ and $b=1.01$ can be seen in the top row of Figure~\ref{figm2}. Once more, the solution becomes very oscillatory 
for $b\to1$, but it preserves the symmetry enforced by the initial 
guess. There are always two global maxima and two global minima in 
this case solution. 
\begin{figure}[!htb]
\begin{subfigure}{.32\textwidth}
\includegraphics[width=1\linewidth,height=0.75\linewidth]{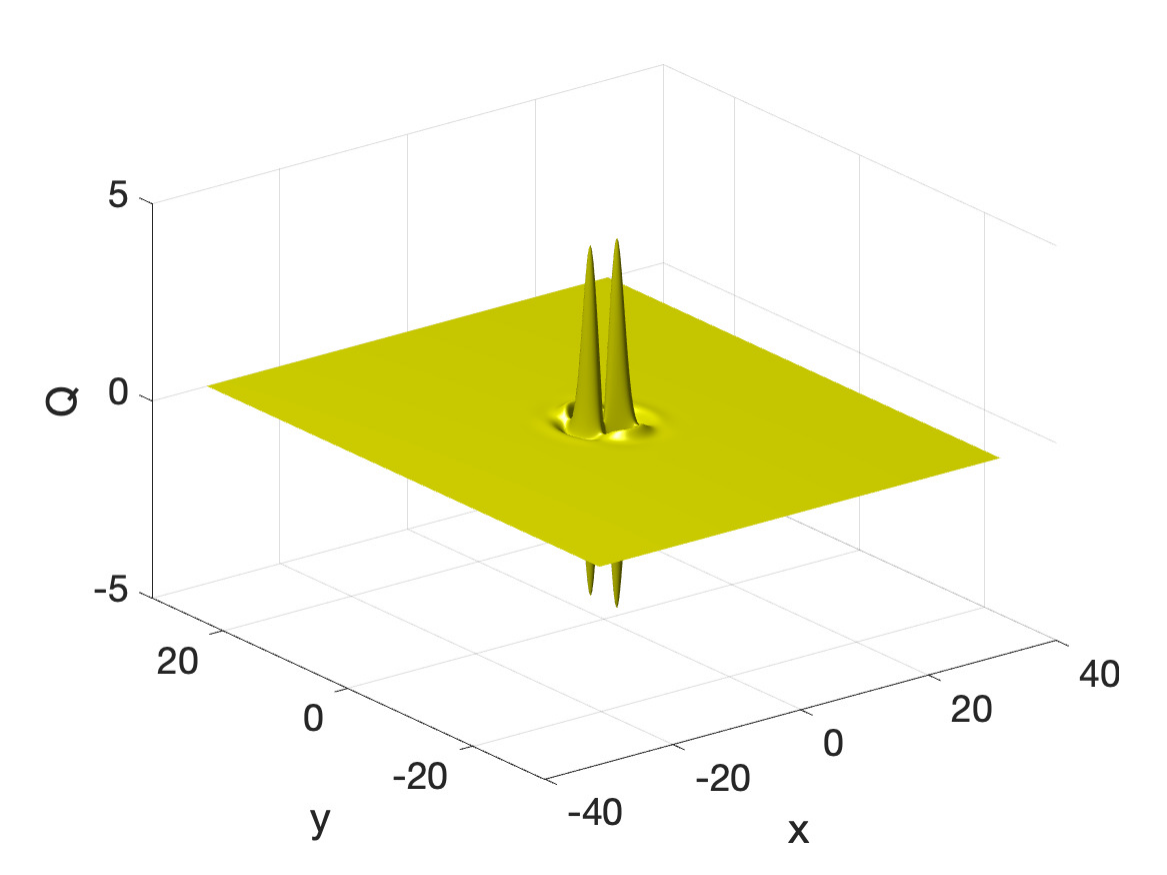}
\subcaption[]{{\footnotesize $b=10$.}}
\end{subfigure}
\begin{subfigure}{.32\textwidth}
\includegraphics[width=1\linewidth,height=0.75\linewidth]{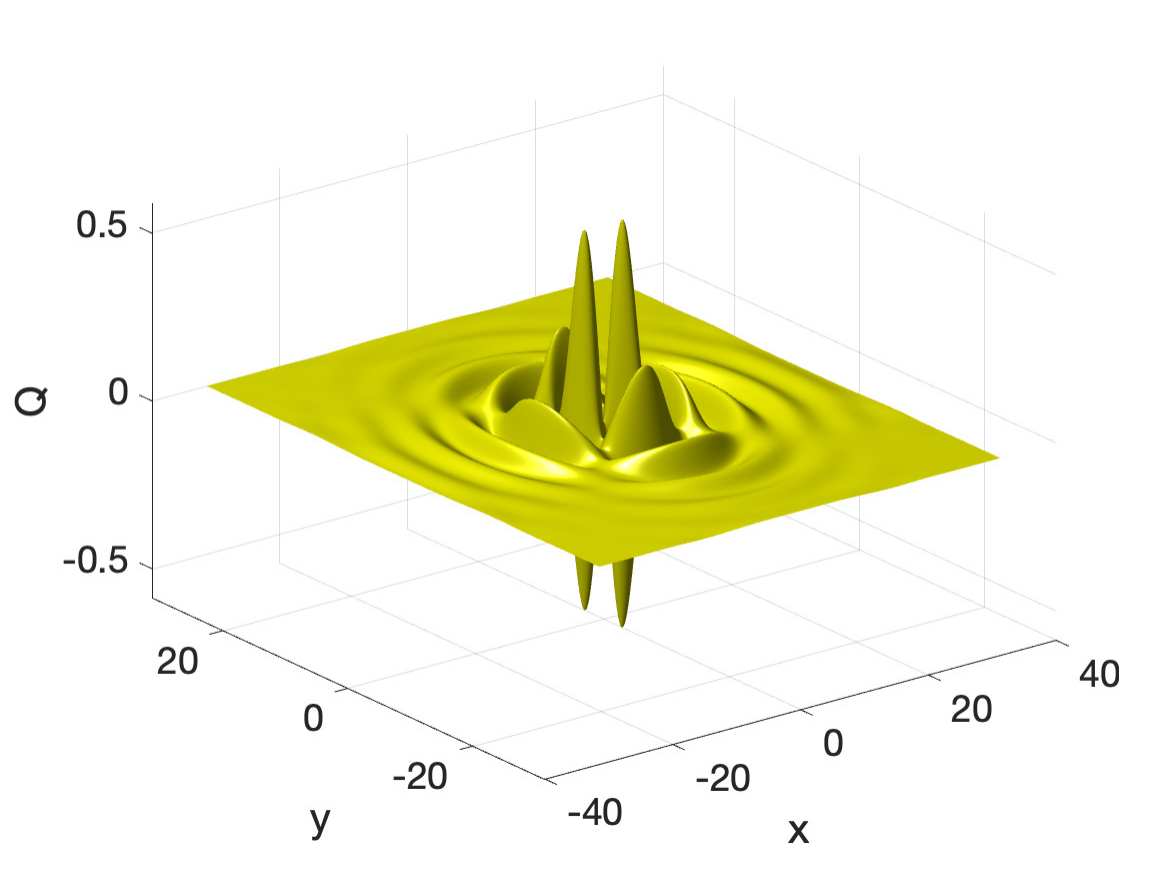}
\subcaption[]{{\footnotesize $b=1.1$.}}
\end{subfigure}
\begin{subfigure}{.32\textwidth}
\includegraphics[width=1\linewidth,height=0.75\linewidth]{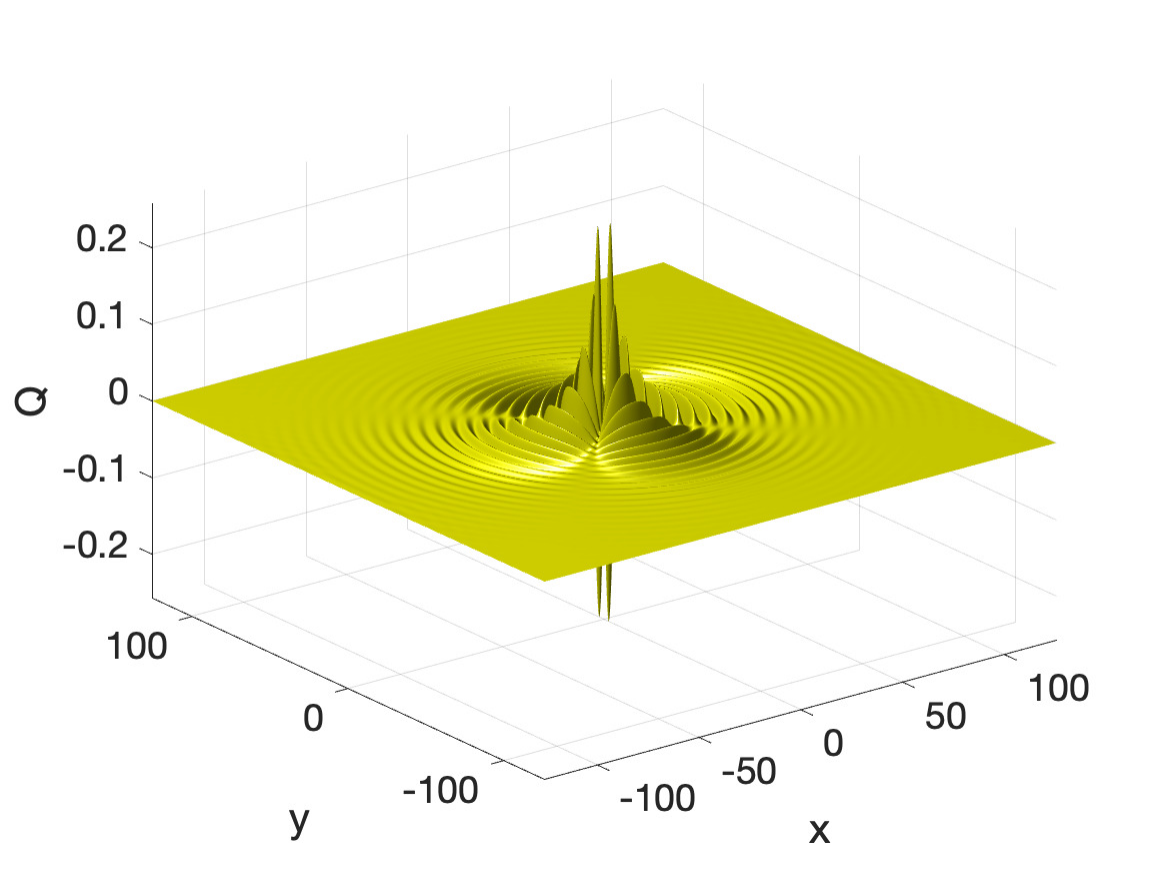}
\subcaption[]{{\footnotesize $b=1.01$.}}
\end{subfigure}\\
\begin{subfigure}{.32\textwidth}
\includegraphics[width=1\linewidth,height=0.65\linewidth]{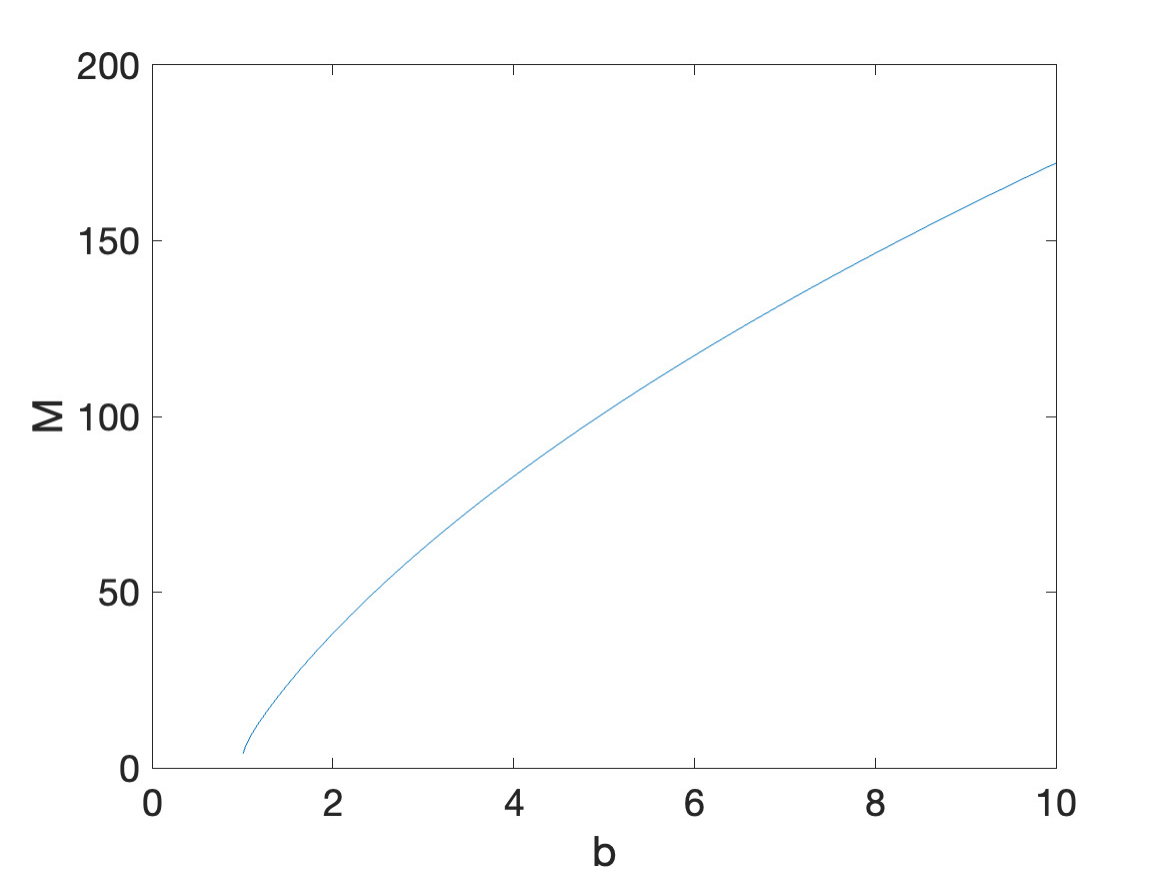}
\subcaption[]{{\footnotesize $M(Q) = M(b)$.}}
\end{subfigure}
\begin{subfigure}{.32\textwidth}
\includegraphics[width=1\linewidth,height=0.65\linewidth]{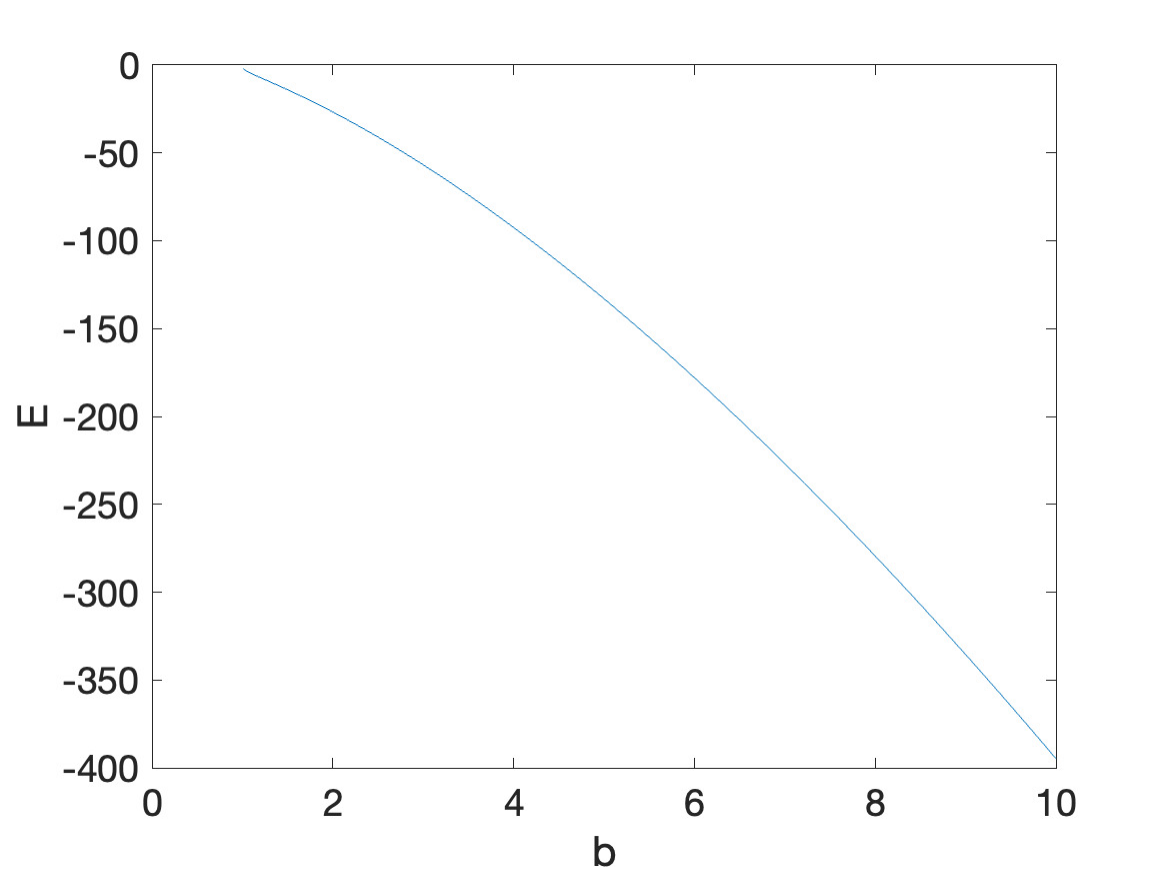}
\subcaption[]{{\footnotesize $E(Q)$ as function of $b$.}}
\end{subfigure}
\begin{subfigure}{.32\textwidth}
\includegraphics[width=1\linewidth,height=0.65\linewidth]{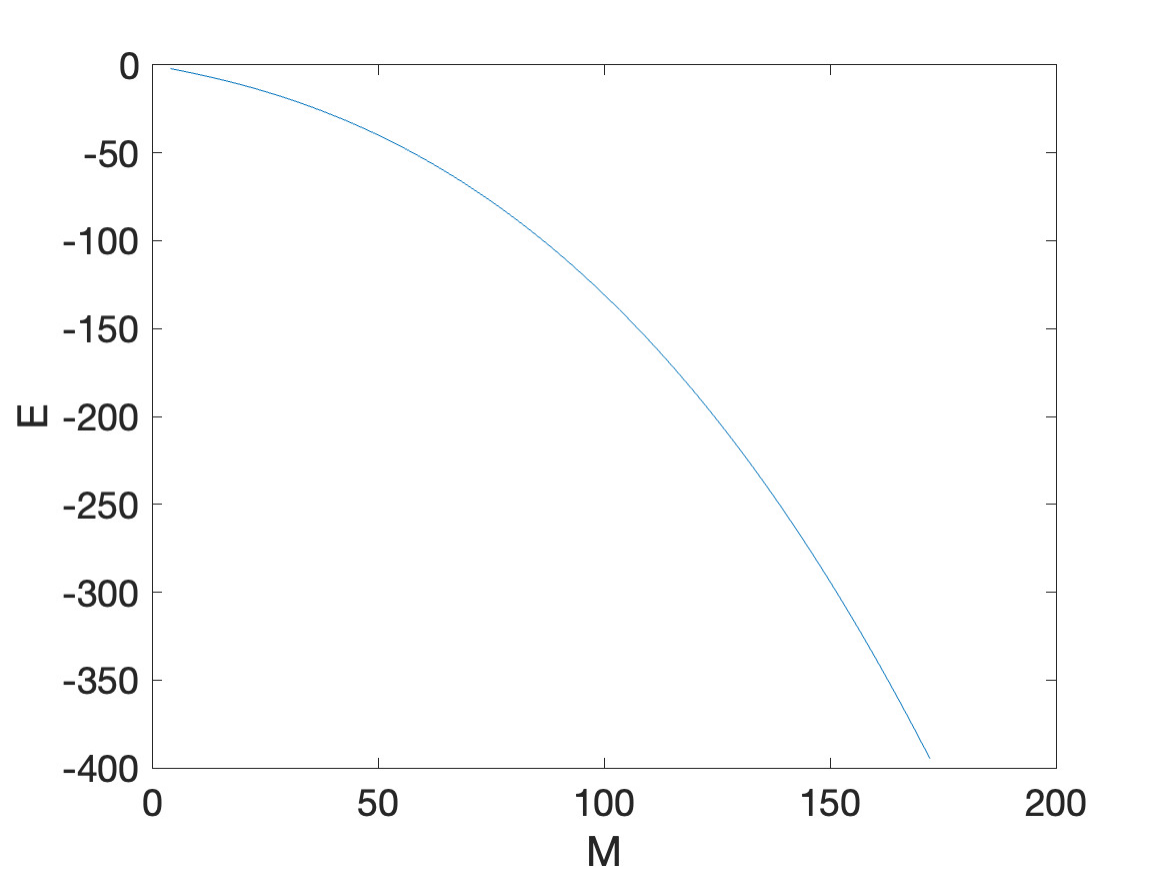}
\subcaption[]{{\footnotesize $E = E(M)$.}}
\end{subfigure}
\caption{\footnotesize Solutions to the cubic equation \eqref{E:2dGS} with angular mode $m=2$: profiles for different $b$ (top row). Dependence of mass $M(Q)$ and energy $E(Q)$ on $b$ (bottom left, middle), and energy as a function of mass $E=E(M)$ (bottom right).}
\label{figm2}
\end{figure}
In the bottom row of Figure~\ref{figm2}, we show the mass and energy dependence on $b$ and the energy as the function of mass; similarly, noting that mass is 
monotonically increasing with $b$ whereas the energy is decreasing.

\subsubsection{The case $m=3$}
For $m=3$, we use $\lambda=10.5$ in \eqref{init} for $b=10$. In this 
case we were not able to reach values of $b$ as small as in the 
previous cases. We start with values of  
$b \in [1.5,10]$. The solution for $b=10$ and $b=1.5$ can be seen in the top row of Figure~\ref{figm3}. The solutions now have a threefold degenerated global maximum and minimum. 

\begin{figure}[!htb]
\begin{subfigure}{.32\textwidth}
\includegraphics[width=1\linewidth,height=0.75\linewidth]{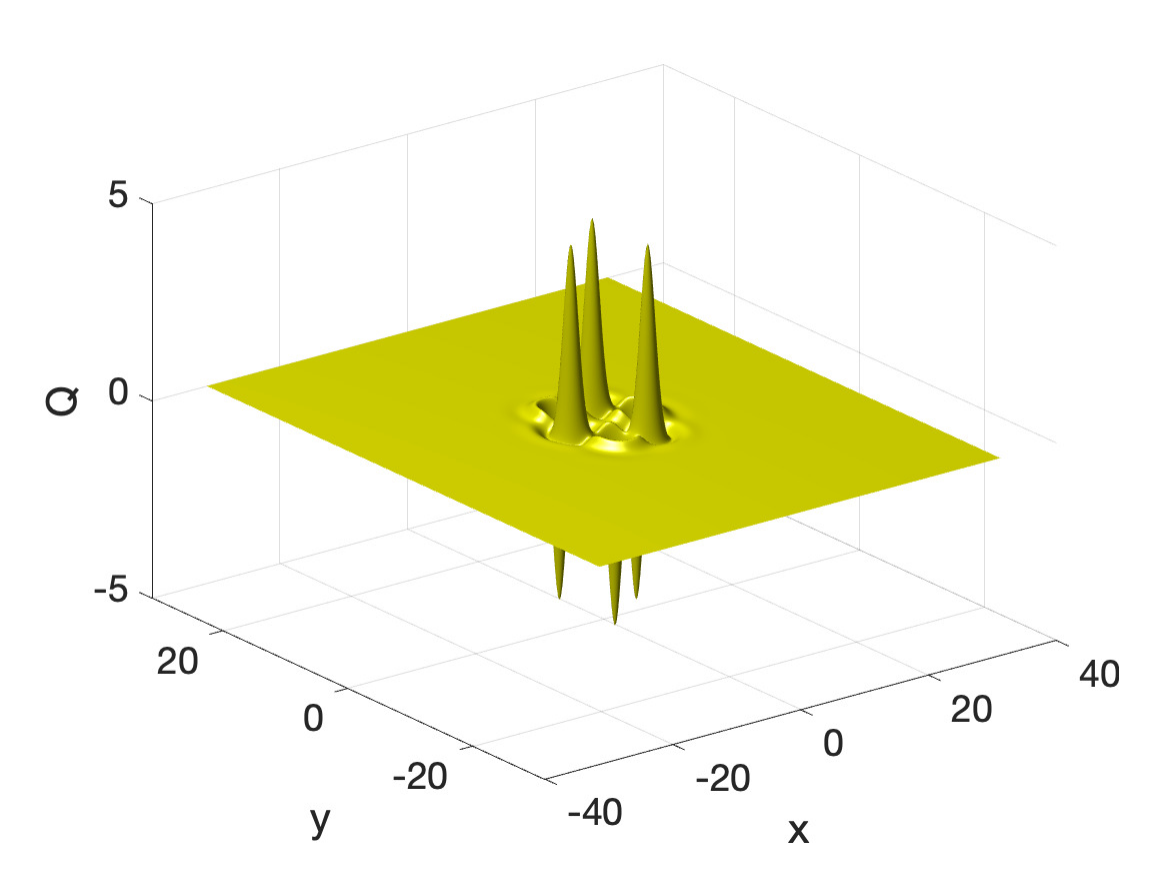}
\subcaption[]{{\footnotesize $b=10$.}}
\end{subfigure}
\begin{subfigure}{.32\textwidth}
\includegraphics[width=1\linewidth,height=0.75\linewidth]{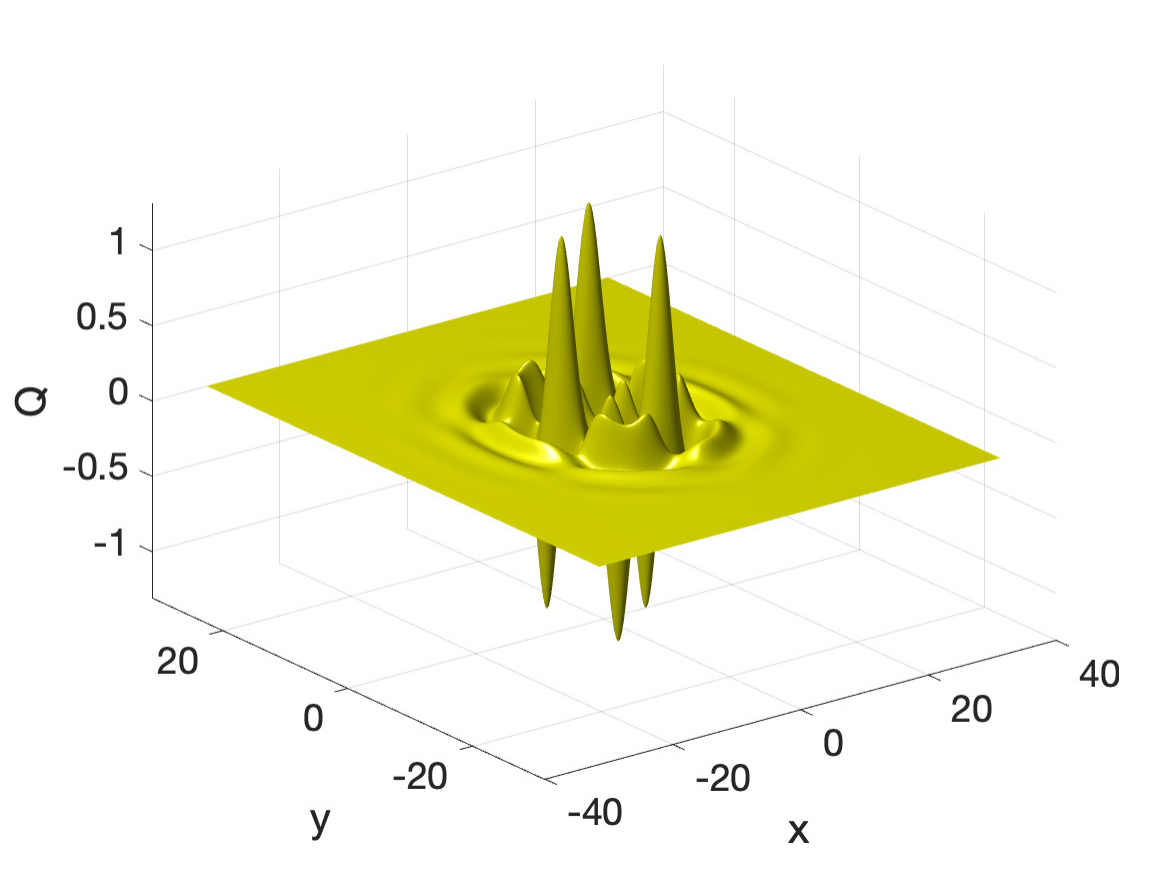}
\subcaption[]{{\footnotesize $b=1.5$.}}
\end{subfigure}
\begin{subfigure}{.32\textwidth}
\includegraphics[width=1\linewidth,height=0.75\linewidth]{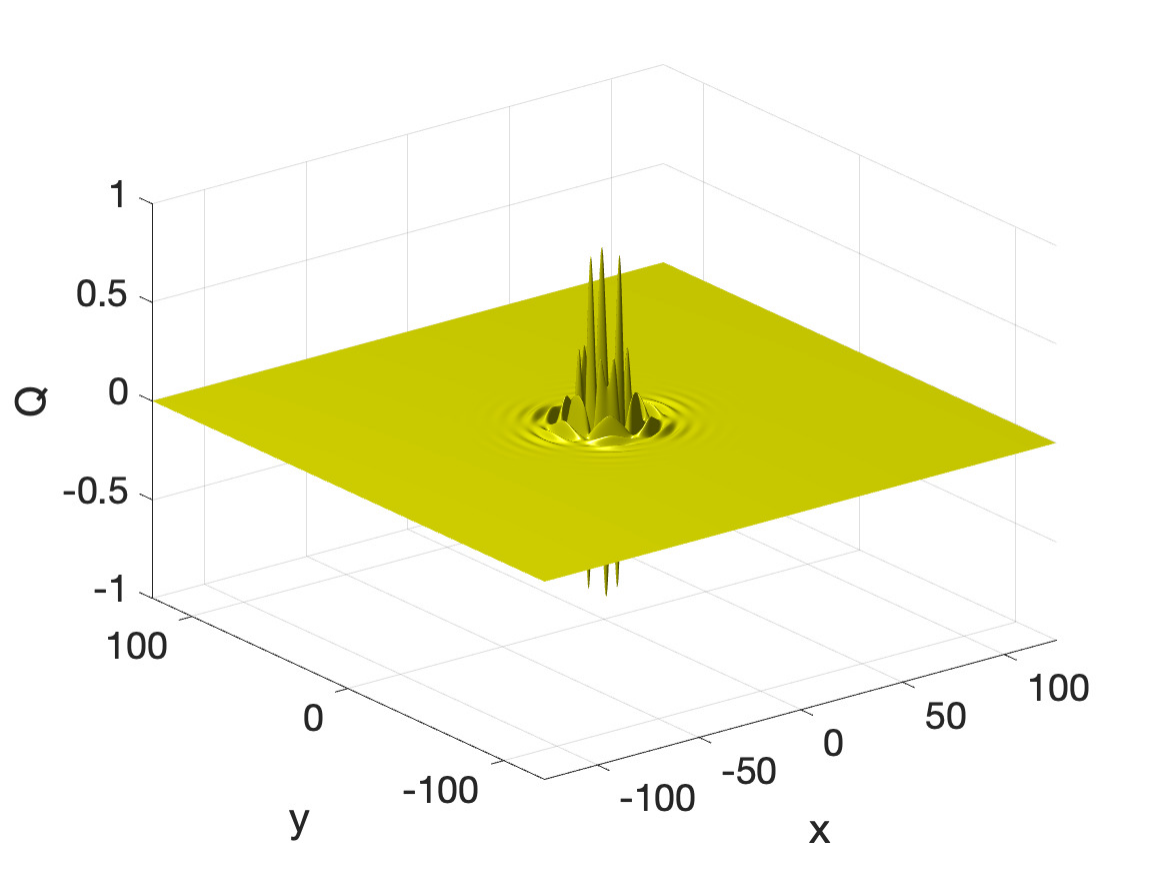}
\subcaption[]{{\footnotesize $b=1.12$.}}
\end{subfigure}\\
\begin{subfigure}{.32\textwidth}
\includegraphics[width=1\linewidth,height=0.65\linewidth]{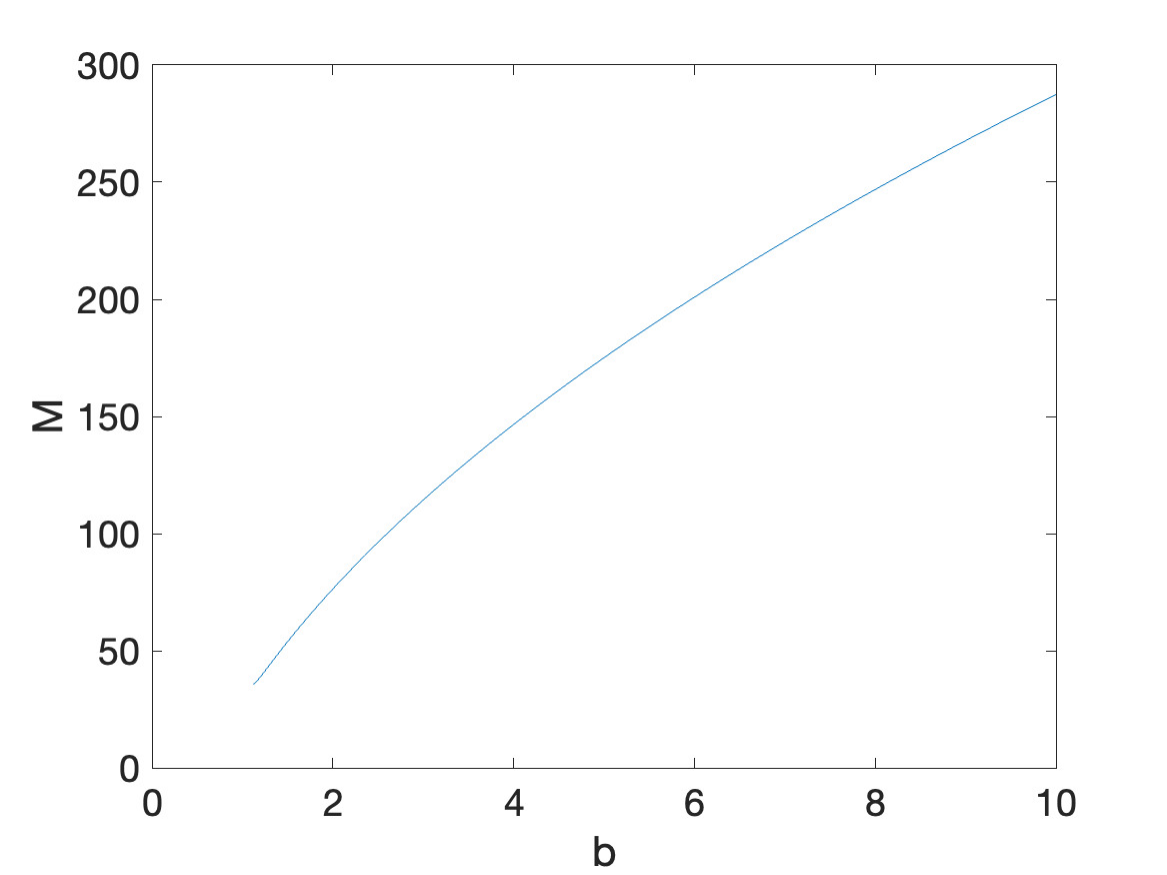}
\subcaption[]{{\footnotesize $M(Q) = M(b)$.}}
\end{subfigure}
\begin{subfigure}{.32\textwidth}
\includegraphics[width=1\linewidth,height=0.65\linewidth]{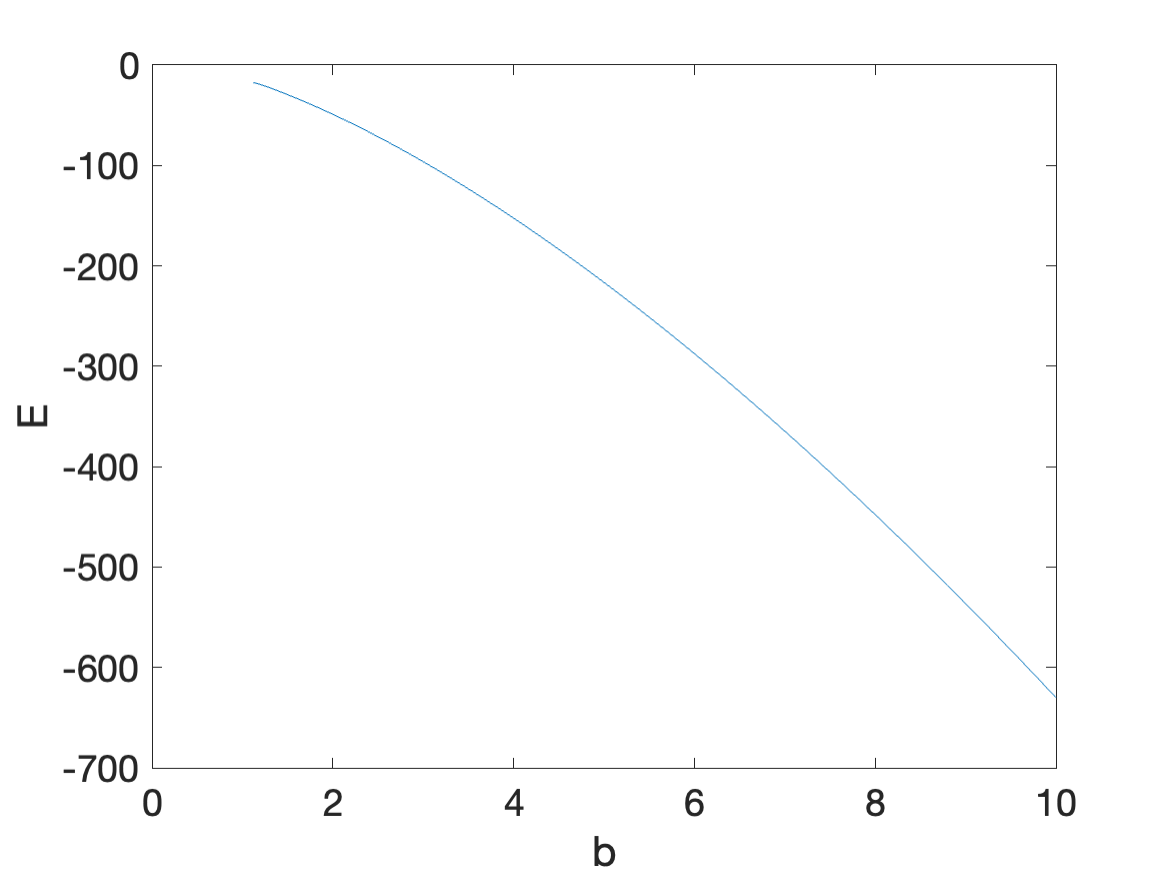}
\subcaption[]{{\footnotesize $E(Q)$ as function of $b$.}}
\end{subfigure}
\begin{subfigure}{.32\textwidth}
\includegraphics[width=1\linewidth,height=0.65\linewidth]{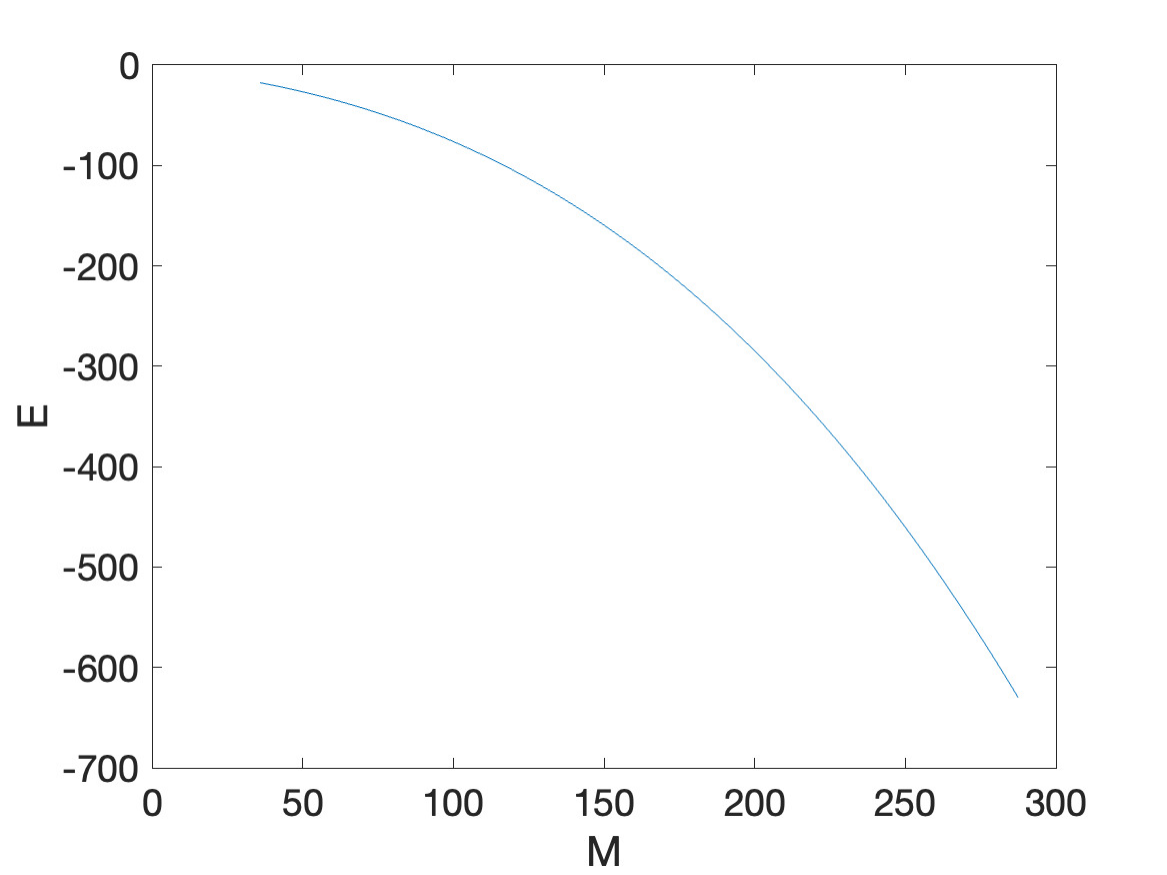}
\subcaption[]{{\footnotesize $E = E(M)$.}}
\end{subfigure}
\caption{\small Solutions to the cubic equation \eqref{E:2dGS} with angular mode $m=3$: profiles for different $b$ (top row). Dependence of mass $M(Q)$ and energy $E(Q)$ on $b$ (bottom left, middle), and energy as a function of mass $E=E(M)$ (bottom right).}
\label{figm3}
\end{figure}

For $b \to 1$, it becomes more and more difficult to obtain the 
solution with a given symmetry ($m=3$), since the solutions come very close to 
other branches, and the iteration converges to the corresponding 
solutions instead of the wanted $m=3$ branch. We consider the 
intervals $b \in [1.4,1.5]$ and $b \in [1.12,1.4]$. It was not possible 
to reach lower values for $b$ on this branch with our approach. 
The solution for $b=1.12$ can be seen in the top row of Figure~\ref{figm3}.

Similarly to the previous cases, the mass is increasing monotonically and the energy is decreasing.

\subsubsection{The case $m=4$}
For larger values of $m$, the numerical solutions becomes more and 
more problematic. The reason is that for the initial guess 
\eqref{init}, the iteration either converges to zero (which is an 
obvious solution) or does not converge at all. Thus, we have to use a 
relaxation ($Q_{new}=\mu Q_{new}+(1-\mu)Q_{old}$ with $0<\mu<1$) 
to stabilize the iteration. For $m=4$, we use $\mu=0.998$ and 
$\lambda=8.85$ in \eqref{init} for $b=10$.  The 
solution for $b=10$, $b=1.1$, $b=1.01$ can be seen in the top row of Figure~\ref{figm4}. 

\begin{figure}[!htb]
\begin{subfigure}{.32\textwidth}
\includegraphics[width=1\linewidth,height=0.75\linewidth]{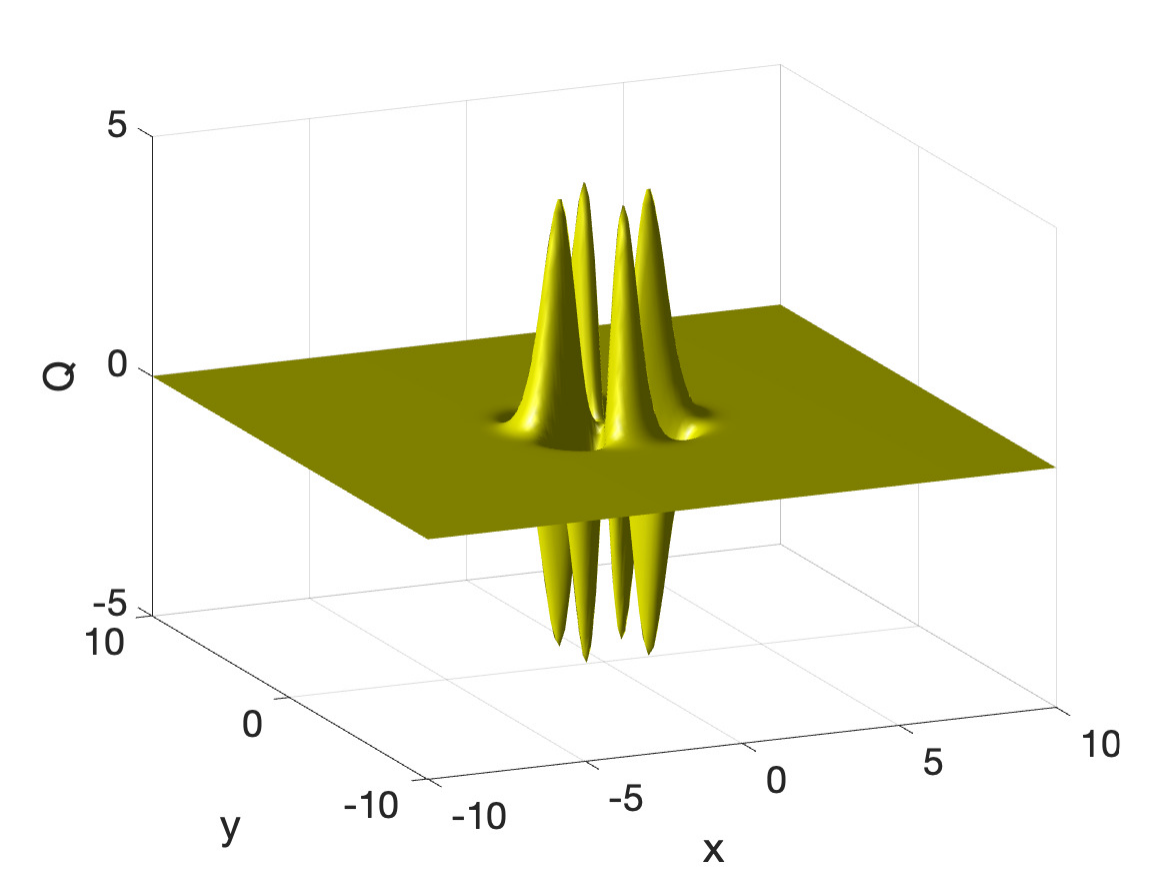}
\subcaption[]{{\footnotesize $b=10$.}}
\end{subfigure}
\begin{subfigure}{.32\textwidth}
\includegraphics[width=1\linewidth,height=0.75\linewidth]{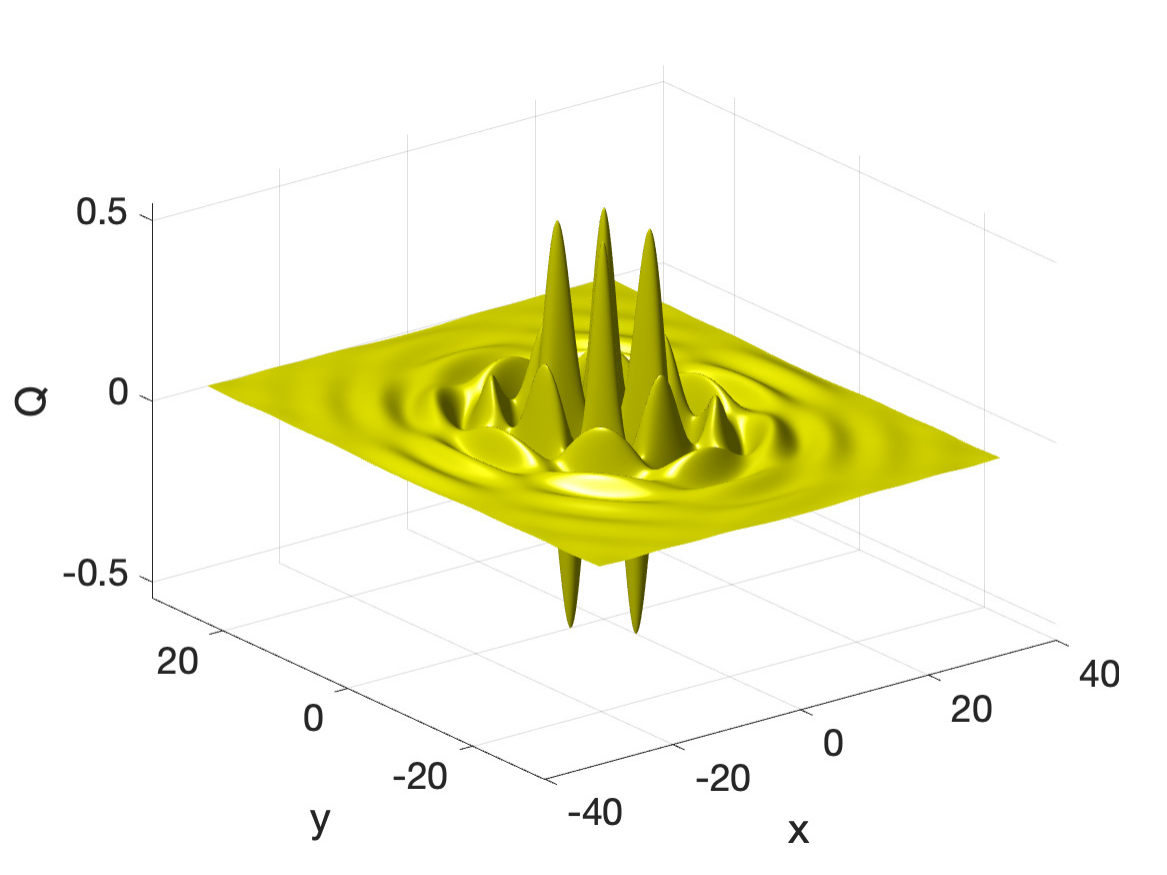}
\subcaption[]{{\footnotesize $b=1.1$.}}
\end{subfigure}
\begin{subfigure}{.32\textwidth}
\includegraphics[width=1\linewidth,height=0.75\linewidth]{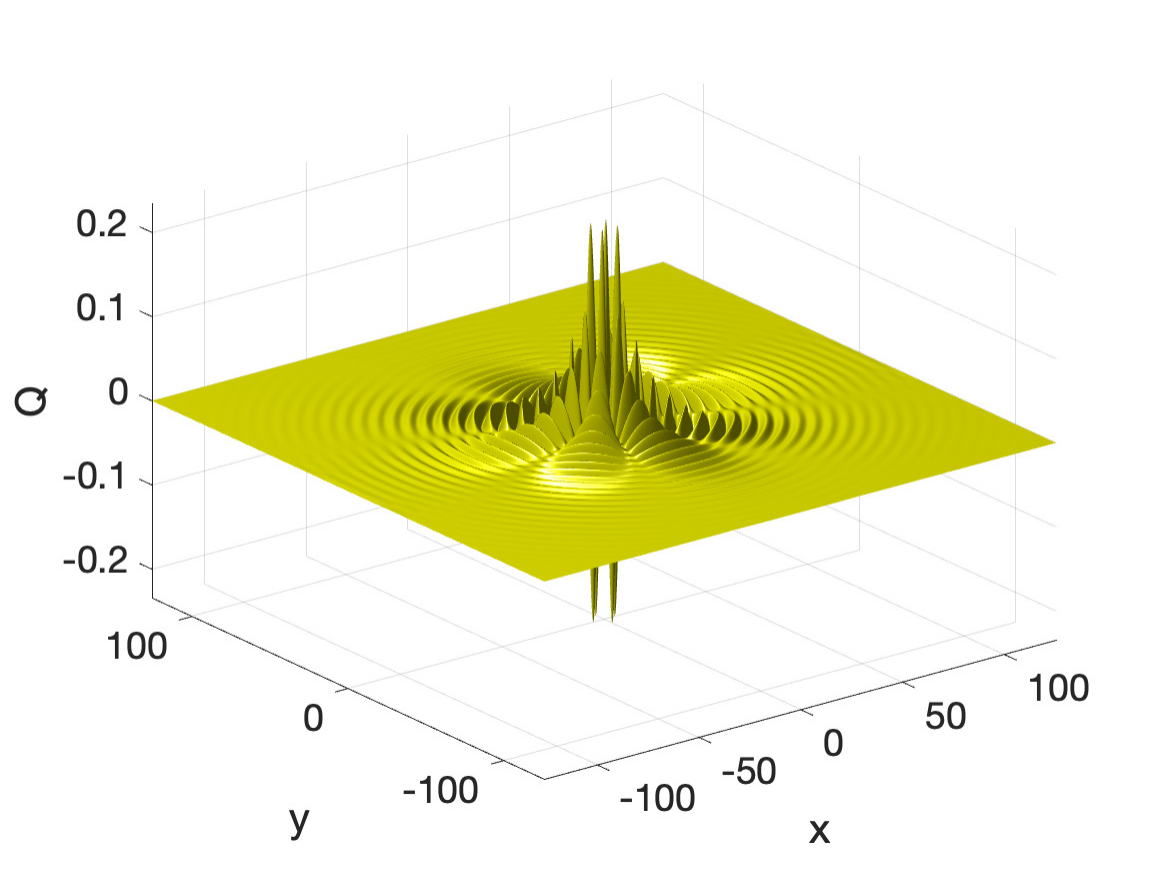}
\subcaption[]{{\footnotesize $b=1.01$.}}
\end{subfigure}\\
\begin{subfigure}{.32\textwidth}
\includegraphics[width=1\linewidth,height=0.65\linewidth]{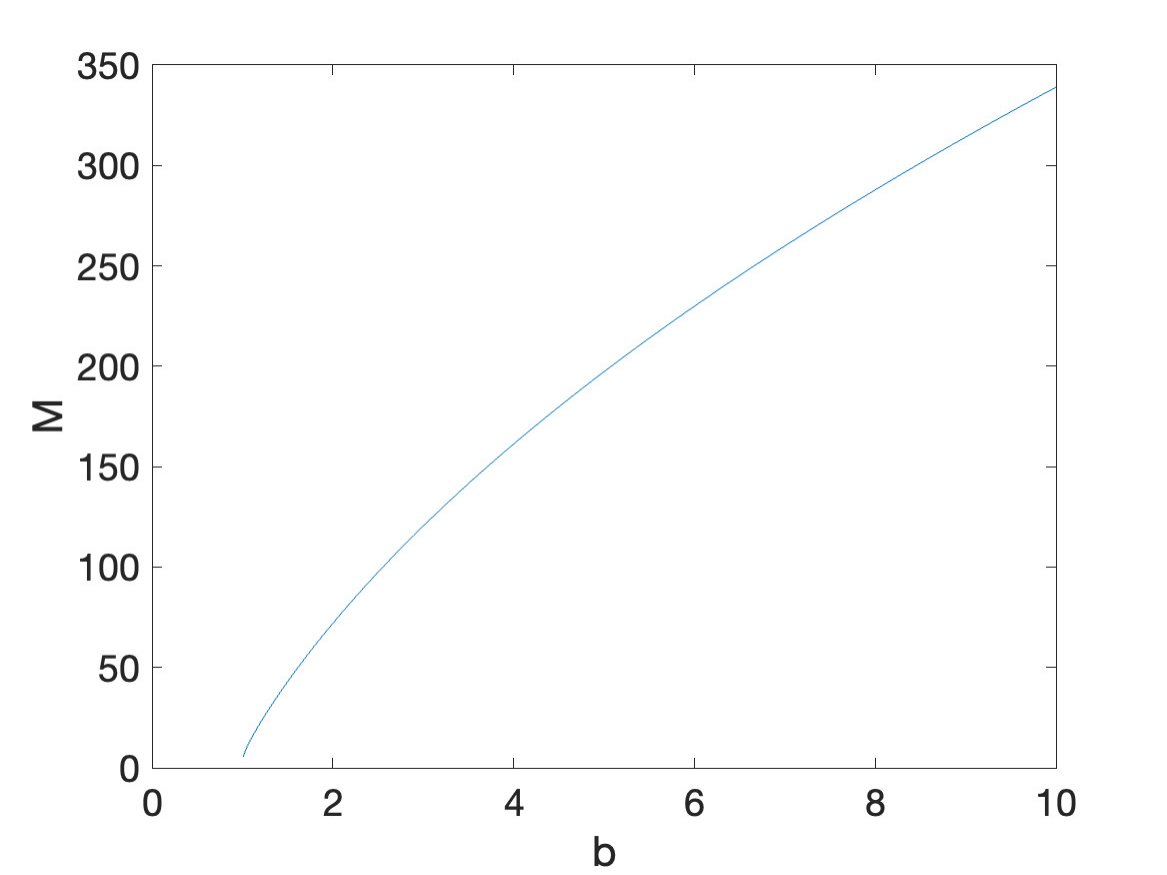}
\subcaption[]{{\footnotesize $M(Q) = M(b)$.}}
\end{subfigure}
\begin{subfigure}{.32\textwidth}
\includegraphics[width=1\linewidth,height=0.65\linewidth]{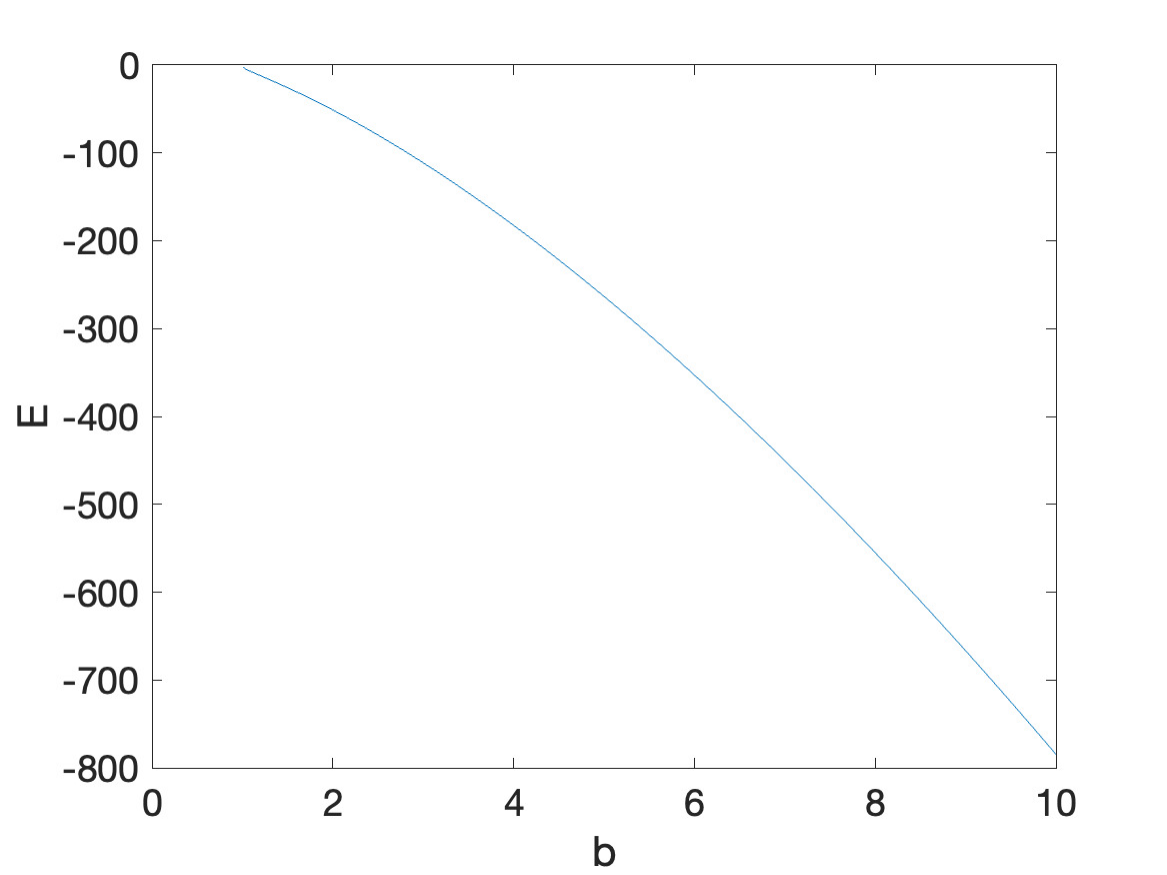}
\subcaption[]{{\footnotesize $E(Q)$ as function of $b$.}}
\end{subfigure}
\begin{subfigure}{.32\textwidth}
\includegraphics[width=1\linewidth,height=0.65\linewidth]{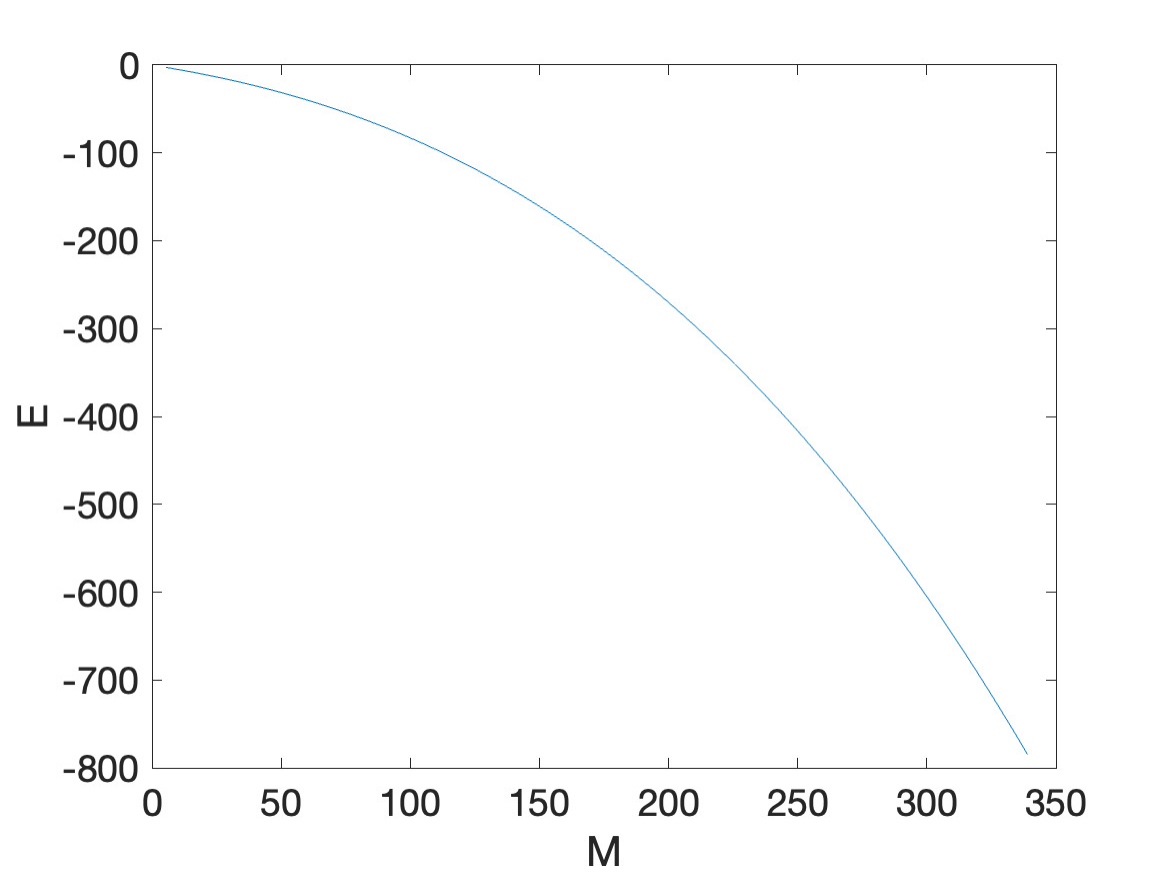}
\subcaption[]{{\footnotesize $E = E(M)$.}}
\end{subfigure}
\caption{\small Solutions to the cubic equation \eqref{E:2dGS} with angular mode $m=4$: profiles for different $b$ (top row). Dependence of mass $M(Q)$ and energy $E(Q)$ on $b$ (bottom left, middle), and energy as a function of mass $E=E(M)$ (bottom right).}
\label{figm4}
\end{figure}

In the bottom row of Figure~\ref{figm4}, we show the mass, the energy and their interdependence, observing once more the mass is monotonically increasing with $b$ whereas the energy is decreasing.

\begin{figure}[!htb]
\begin{subfigure}{.32\textwidth}
\includegraphics[width=1\linewidth,height=0.75\linewidth]{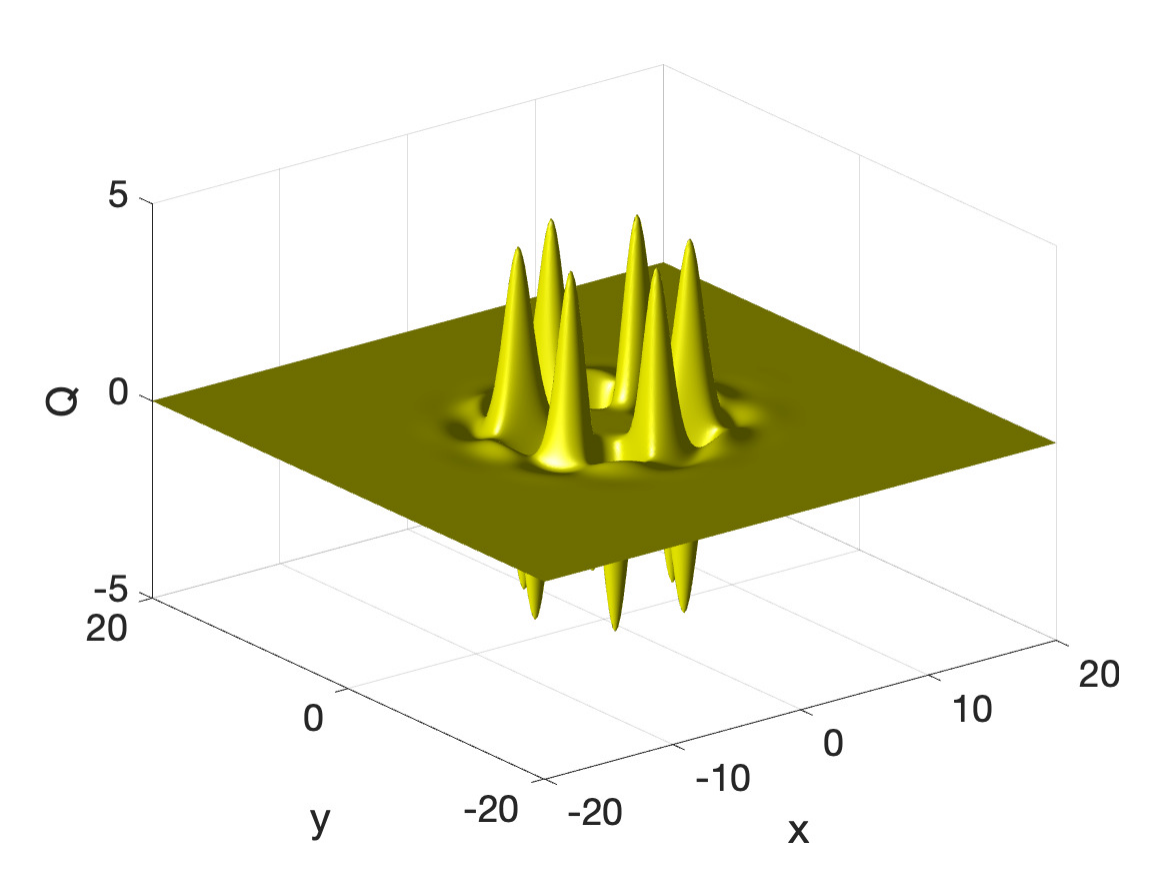}
\subcaption[]{{\footnotesize $b=10$.}}
\end{subfigure}
\begin{subfigure}{.32\textwidth}
\includegraphics[width=1\linewidth,height=0.75\linewidth]{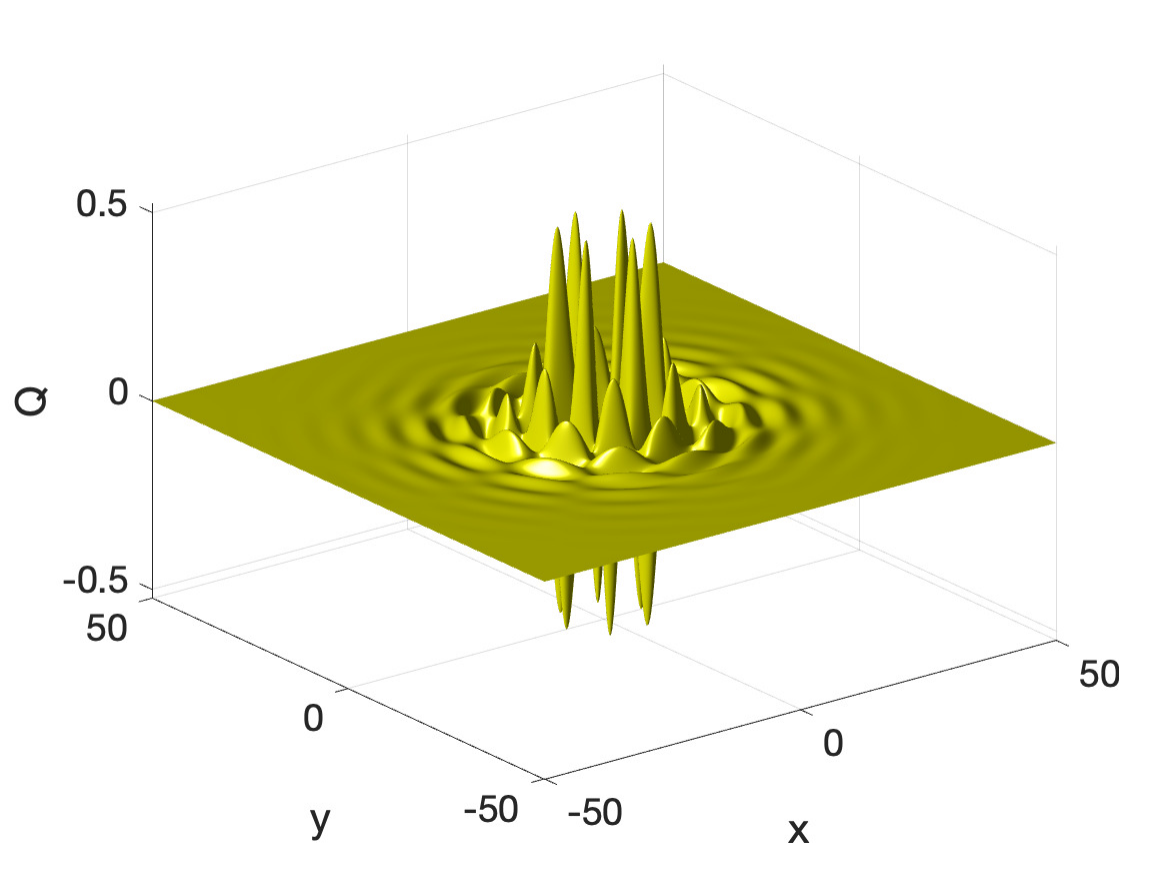}
\subcaption[]{{\footnotesize $b=1.1$.}}
\end{subfigure}
\begin{subfigure}{.32\textwidth}
\includegraphics[width=1\linewidth,height=0.75\linewidth]{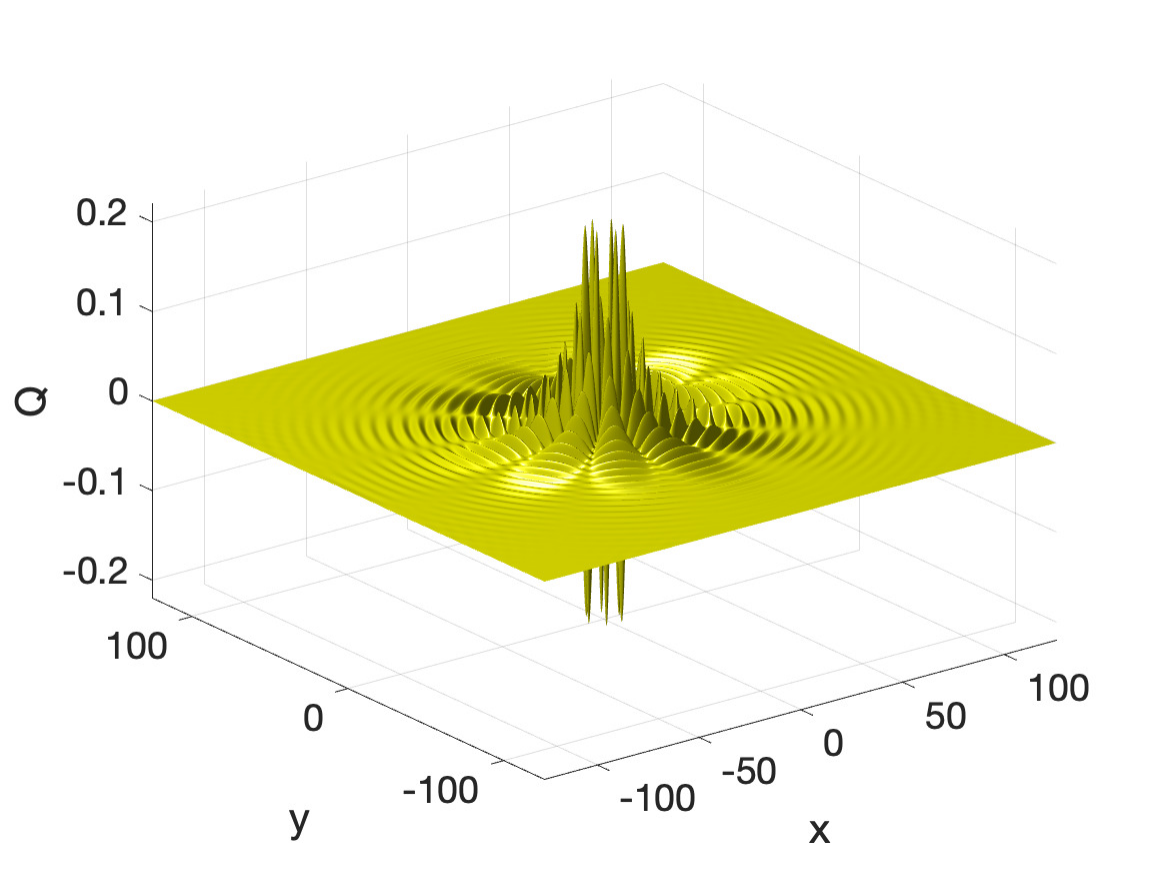}
\subcaption[]{{\footnotesize $b=1.01$.}}
\end{subfigure}\\
\begin{subfigure}{.32\textwidth}
\includegraphics[width=1\linewidth,height=0.65\linewidth]{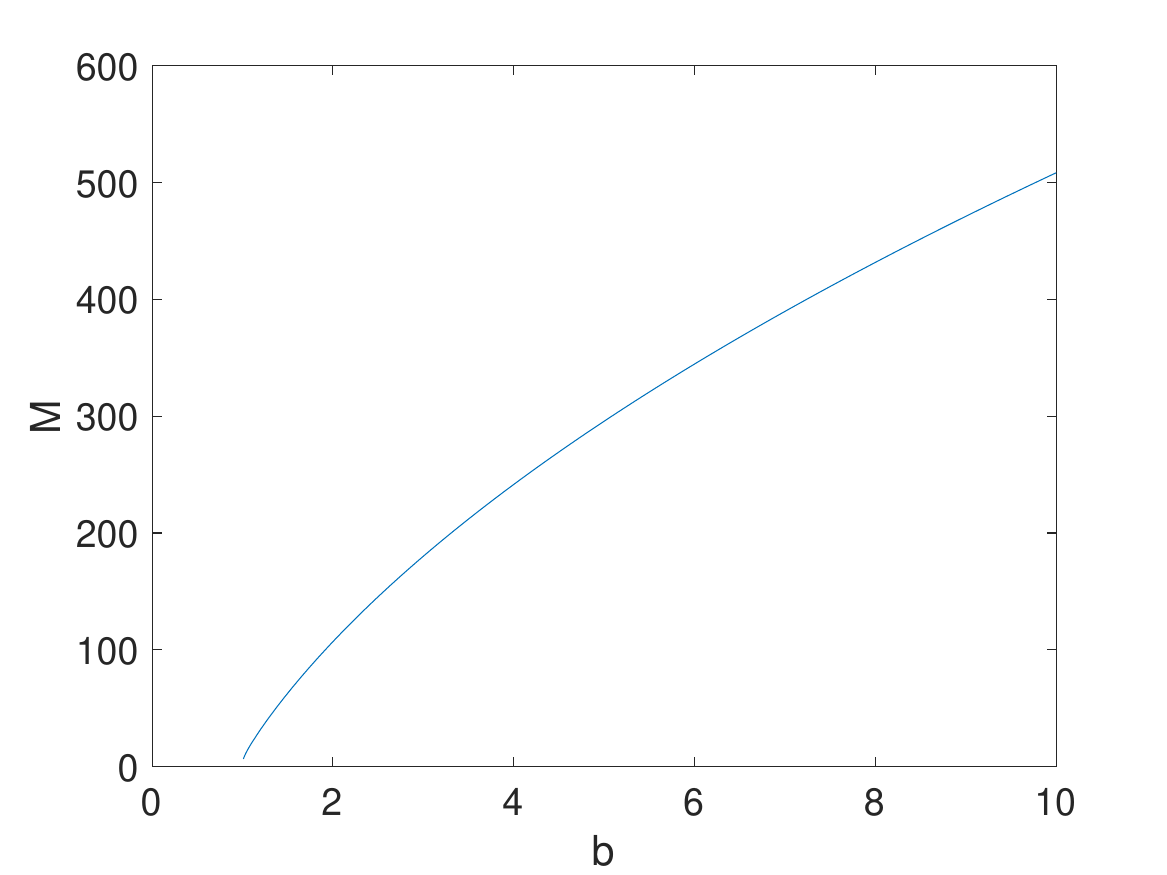}
\subcaption[]{{\footnotesize $M(Q) = M(b)$.}}
\end{subfigure}
\begin{subfigure}{.32\textwidth}
\includegraphics[width=1\linewidth,height=0.65\linewidth]{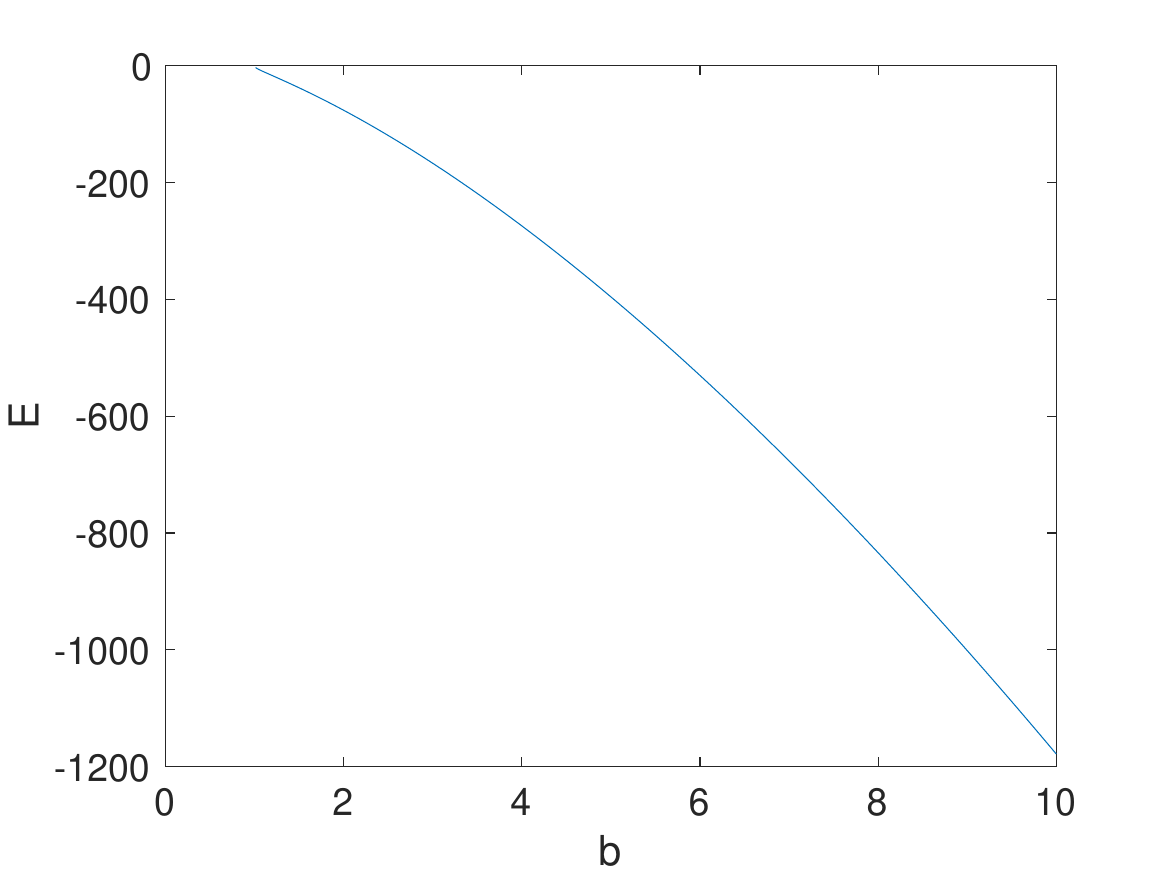}
\subcaption[]{{\footnotesize $E(Q)$ as function of $b$.}}
\end{subfigure}
\begin{subfigure}{.32\textwidth}
\includegraphics[width=1\linewidth,height=0.65\linewidth]{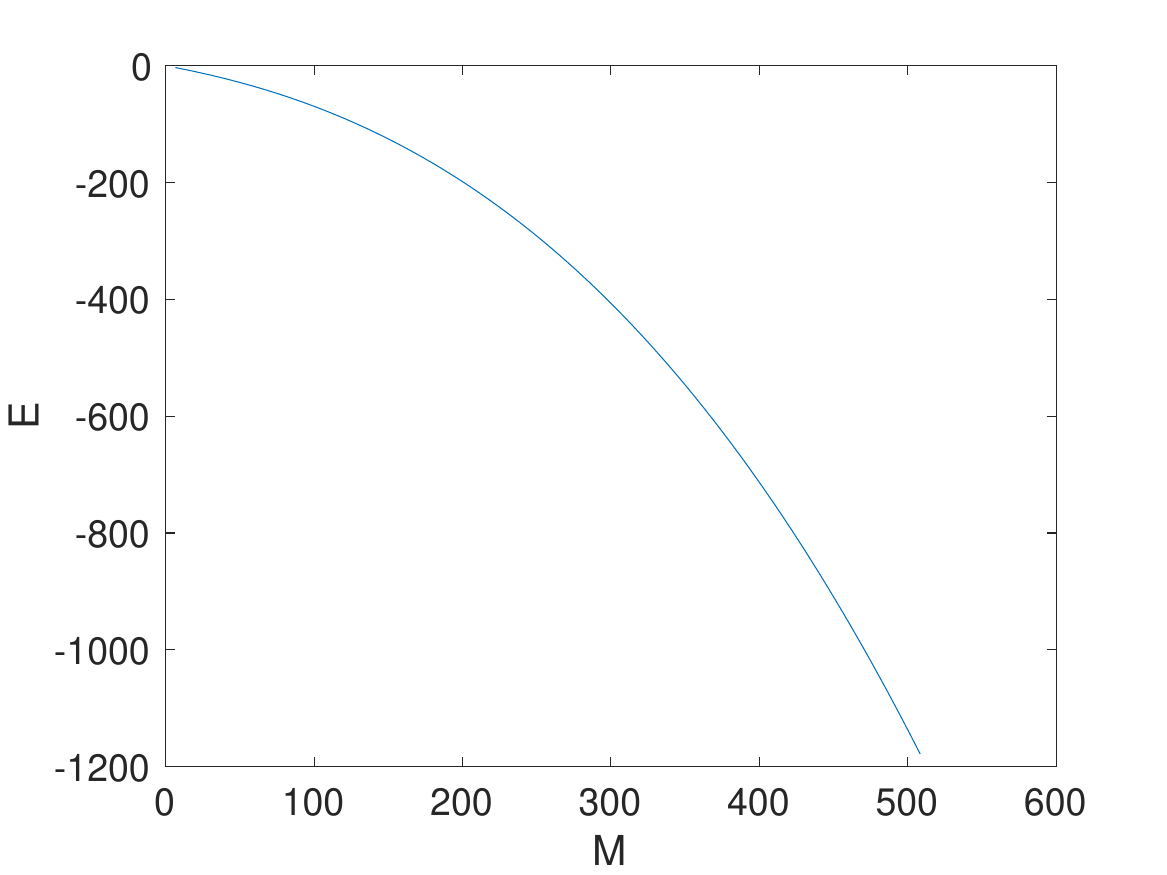}
\subcaption[]{{\footnotesize $E = E(M)$.}}
\end{subfigure}
\caption{\footnotesize Solutions to the cubic equation \eqref{E:2dGS} with angular mode $m=6$: profiles for different $b$ (top row). Dependence of mass $M(Q)$ and energy $E(Q)$ on $b$ (bottom left, middle), and energy as a function of mass $E=E(M)$ (bottom right).}
\label{figm6}
\end{figure}

\subsubsection{Higher values of $m$}
As mentioned before, it becomes increasingly difficult to numerically 
construct solutions of equation \eqref{E:2dGS} for larger values of 
the parameter $m$ with the initial condition \eqref{init} in order to investigate the solutions with a given symmetry. The reason is that there is always the trivial solution for small values of 
$\lambda$, whereas for larger values, there appear to be solutions 
with two or more annular arrangements of peaks with the imposed 
symmetry that could be interpreted as excited states (they have 
larger energy than the solution we are interested in). Even a 
relaxation of the iteration as above does not solve this problem.

\begin{figure}[!htb]
\begin{subfigure}{.32\textwidth}
\includegraphics[width=1\linewidth,height=0.75\linewidth]{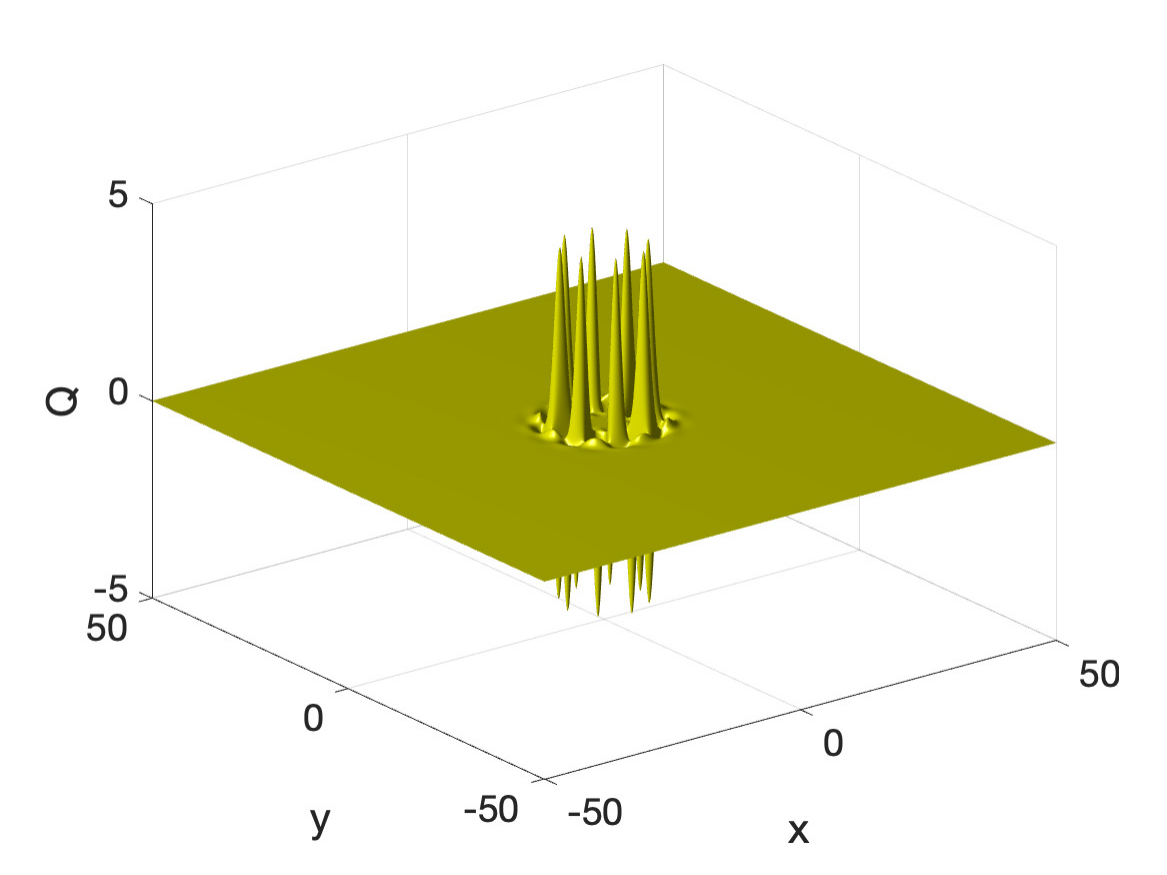}
\subcaption[]{{\footnotesize $b=10$}}
\end{subfigure}
\begin{subfigure}{.32\textwidth}
\includegraphics[width=1\linewidth,height=0.75\linewidth]{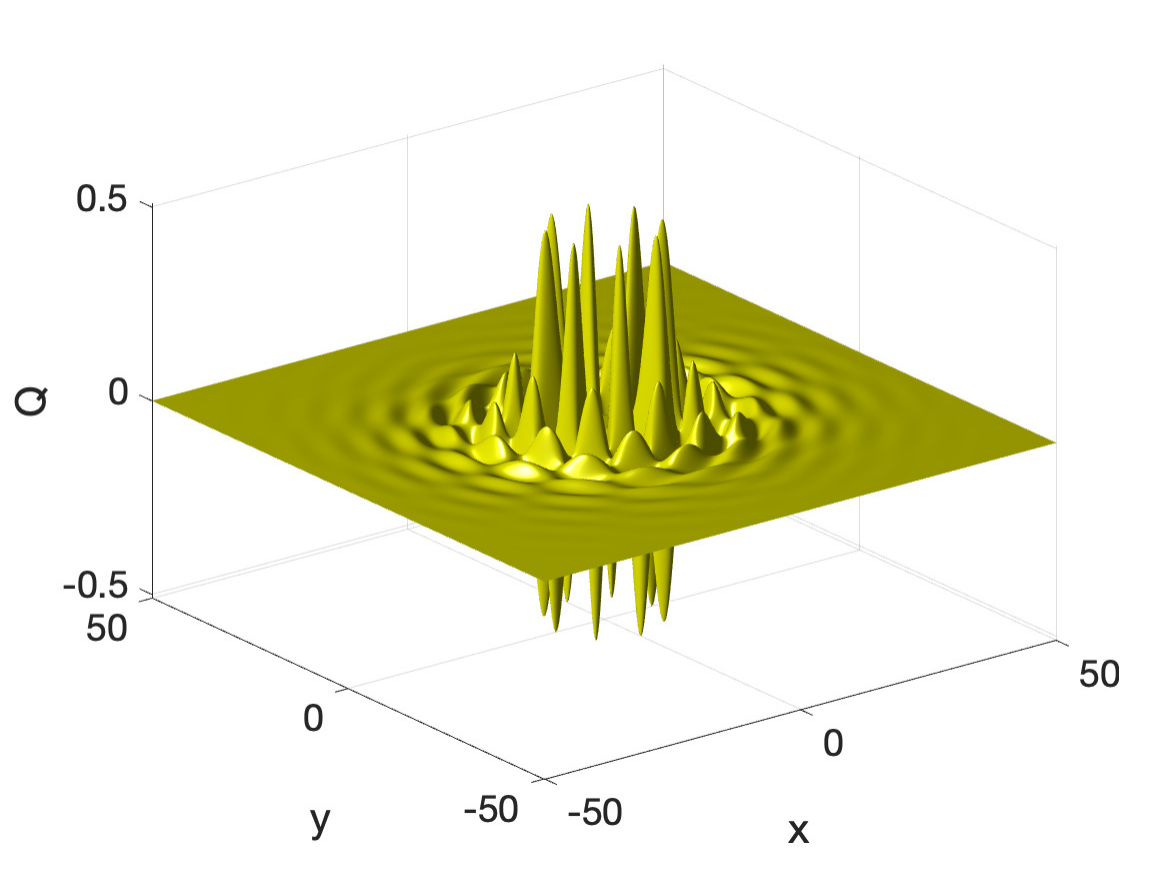}
\subcaption[]{{\footnotesize $b=1.1$}}
\end{subfigure}
\begin{subfigure}{.32\textwidth}
\includegraphics[width=1\linewidth,height=0.75\linewidth]{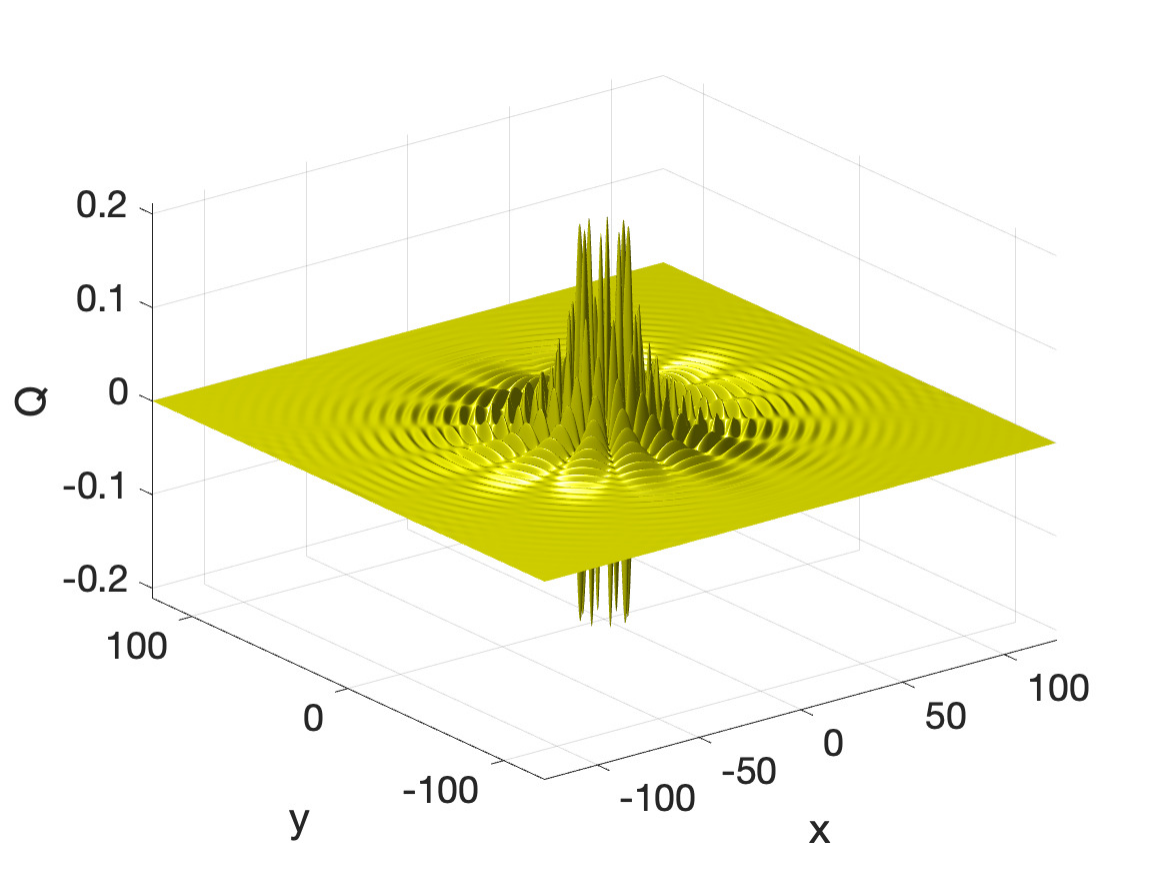}
\subcaption[]{{\footnotesize $b=1.01$}}
\end{subfigure}\\
\begin{subfigure}{.32\textwidth}
\includegraphics[width=1\linewidth,height=0.65\linewidth]{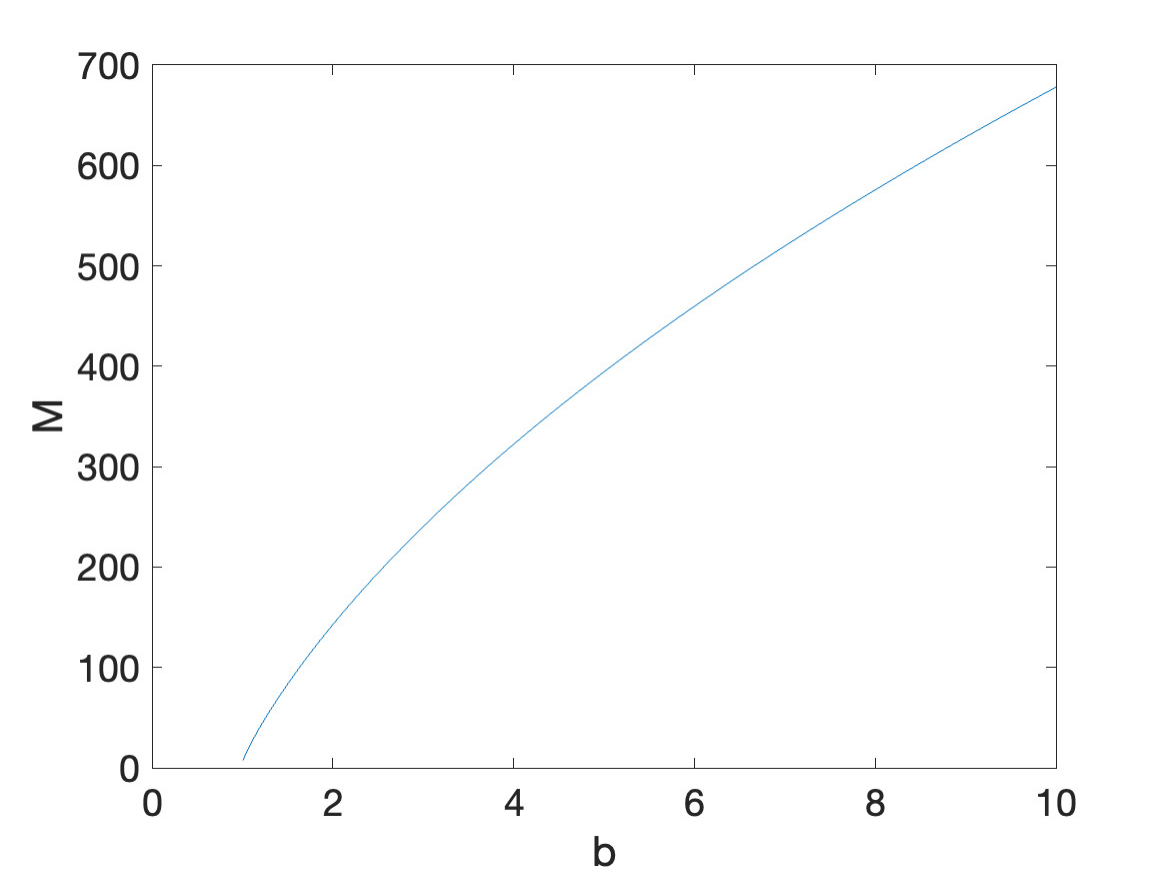}
\subcaption[]{{\footnotesize $M(Q) = M(b)$}}
\end{subfigure}
\begin{subfigure}{.32\textwidth}
\includegraphics[width=1\linewidth,height=0.65\linewidth]{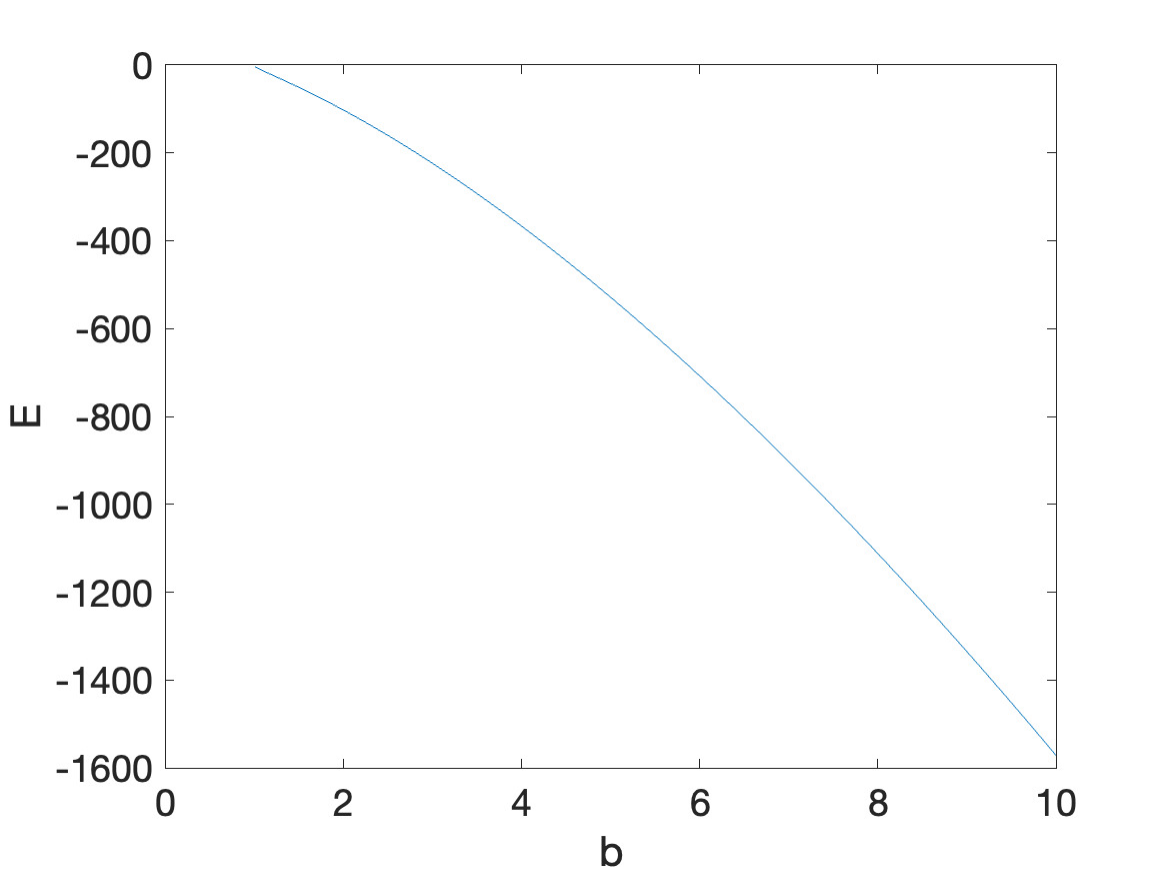}
\subcaption[]{{\footnotesize $E(Q)$ as function of $b$}}
\end{subfigure}
\begin{subfigure}{.32\textwidth}
\includegraphics[width=1\linewidth,height=0.65\linewidth]{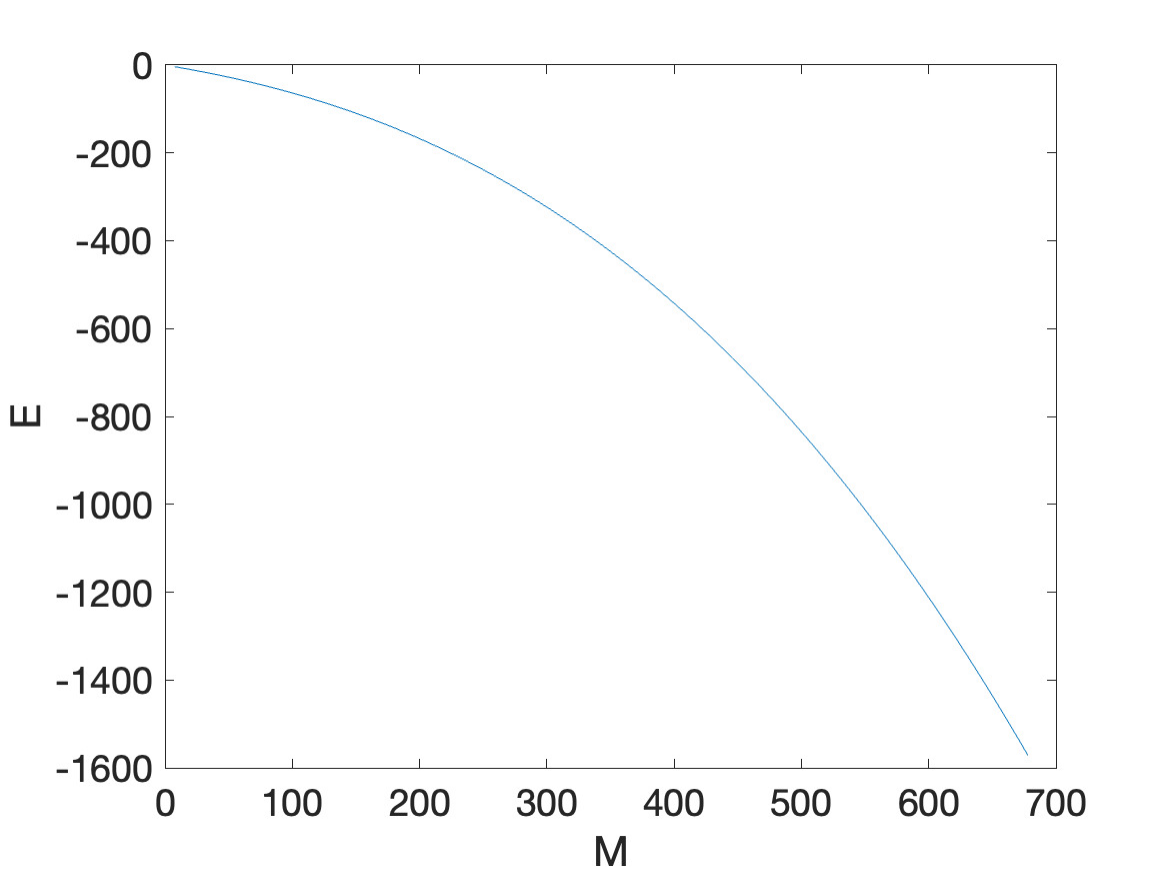}
\subcaption[]{{\footnotesize $E = E(M)$}}
\end{subfigure}
\caption{\footnotesize {Solution to \eqref{E:groundstate} with angular mode $m=8$: profiles for different $b$ (top row). Dependence of mass $M(Q)$ and energy $E(Q)$ on $b$ (bottom left, middle), and energy as a function of mass $E=E(M)$ (bottom right).}}
\label{figm8}
\end{figure}
For large values of $b$ as above, we did not find any solution for 
the cases $m=6,8$ with the initial guess \eqref{init}. Therefore, we 
start here with the value $b=1.01$ and the initial iterate 
\eqref{init} with $\lambda=2$. Then we use the tracing technique with 
100 steps up to $b=1.1$, and another 100 steps to $b=10$. 
The solutions for $b=10, b=1.1$ and $b=1.01$ can be seen in the top row of Figure~\ref{figm6}.
\smallskip

For $m=8$, we use the initial guess \eqref{init} with $\lambda=1.2$ 
for $b=1.01$. The resulting solution times a factor $1.2$ is then 
used as an initial iterate for $b=1.1$. The usual tracing procedure 
is applied to reach values of $b$ between $1.1$ and $10$ as well as 
between $1.1$ and $1.01$. The solutions for  $b=10$, $b=1.1$ and $b=1.01$ can be 
seen in the top row of Figure~\ref{figm8}.

In the bottom row of Figure~\ref{figm6} ($m=6$) and Figure~\ref{figm8} ($m=8$), the mass, the energy and the energy as a function mass are shown. Similar to the previous angular modes, the mass is monotonically increasing with $b$ whereas the energy is decreasing, and the energy as a function of mass is a monotone graph.

\subsubsection{Mass-energy relations}\label{S:ME-grand}
The interesting question is which of the above computed ground state branches has the smallest energy for a given mass. To illustrate this, we put the 
graphs for energy in dependence of mass into a single graph shown in 
Figure \ref{figME}.
\begin{figure}[!htb]
\centering
\includegraphics[width=1\hsize,height=.55\hsize]{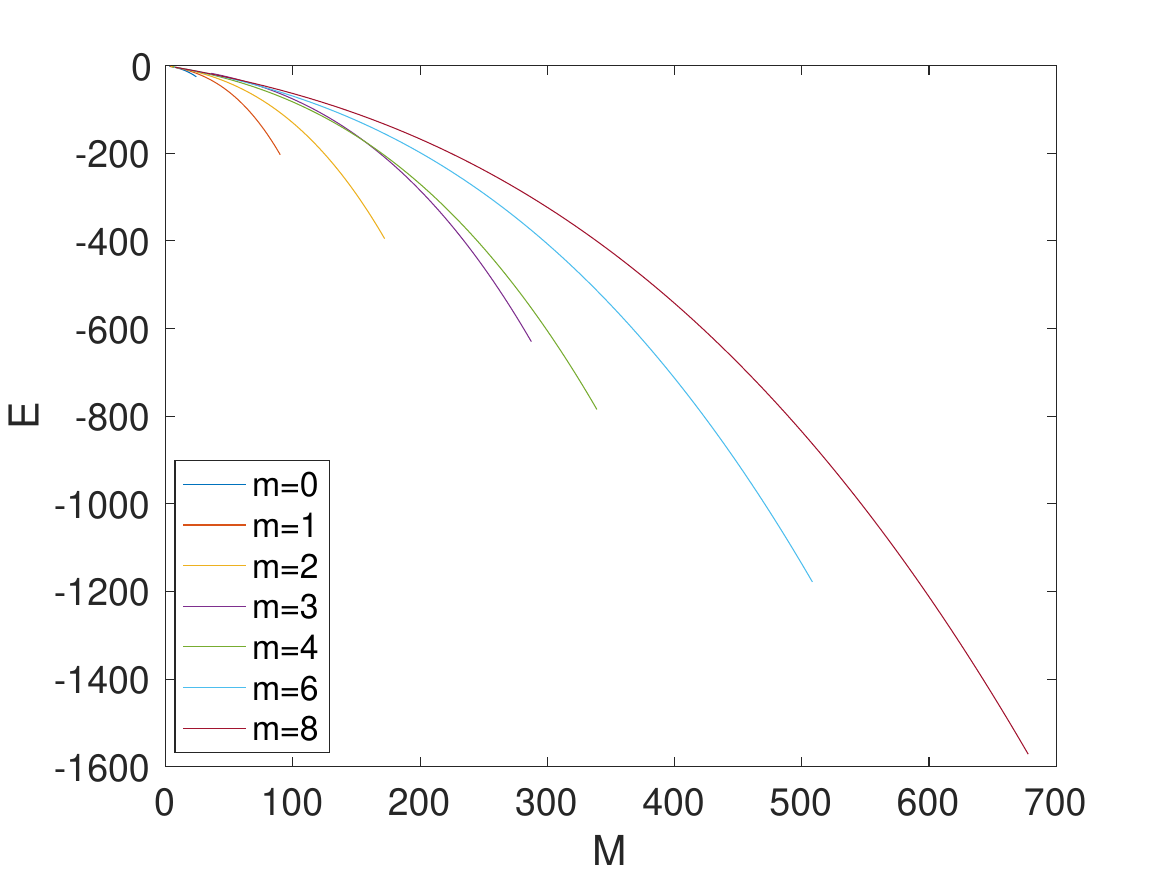}
\caption{\footnotesize Energy as a function of mass for 
solutions to the equation \eqref{E:2dGS} with various angular symmetries $m$.}
\label{figME}
\end{figure}
\begin{figure}[!htb]
\includegraphics[width=1\hsize,height=.5\hsize]{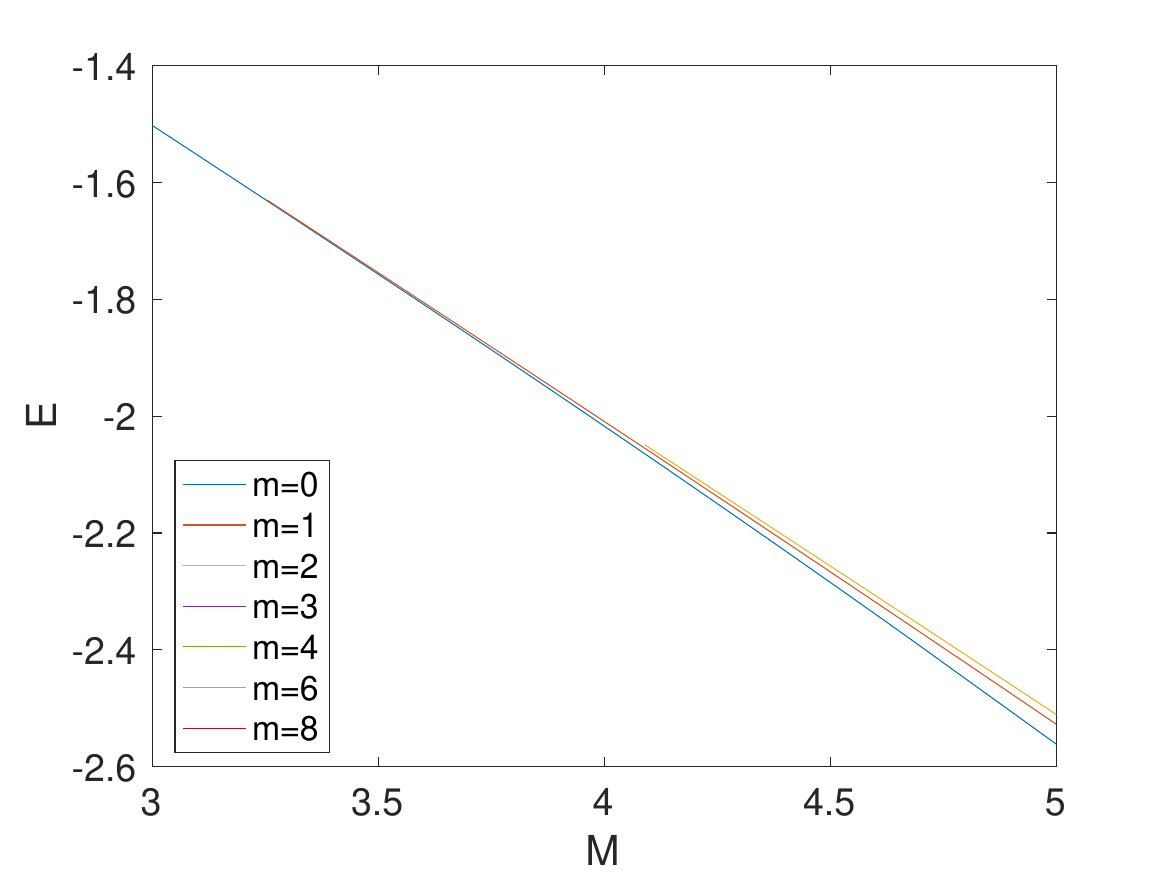}
\caption{\small Close-up of Figure~\ref{figME} for small values of mass.}
\label{figMEcu}
\end{figure}

In the top left corner it can be seen that for the considered values of $m$, the lowest energy is reached in the case with radial symmetry ($m=0$), we enlarge it since for small masses the curves get very close to each other. 
A close-up in Figure \ref{figMEcu} shows, the curves do not cross for the 
values of $b$ studied above.  Thus, the radial case appears to be the 
one of lowest energy in the class of solutions we investigated so far.

To explore even smaller values of $b$, one has to work on larger tori, 
since the solutions are less and less localized for $b\to 1$. Therefore, we 
consider $L_{x}=L_{y}=100$ with otherwise unchanged numerical 
parameters below to reach values of $b=1.001$ (the same tracing 
techniques are applied for 10 values of $b$ between $1.01$ and 
$1.001$ for the case $m=0,1,2,4$). The resulting relation of energy 
in dependence of mass is shown in Figure~\ref{figME1001} on the left. 
In a close-up on the right of the same figure (with even further zoom in of the top part and the bottom part to show the branches of the computed solutions), 
it can be seen that the solutions with radial symmetry always have 
the smallest energy among these. 

\begin{figure}[!htb]
\includegraphics[width=.49\hsize,height=.3\hsize]{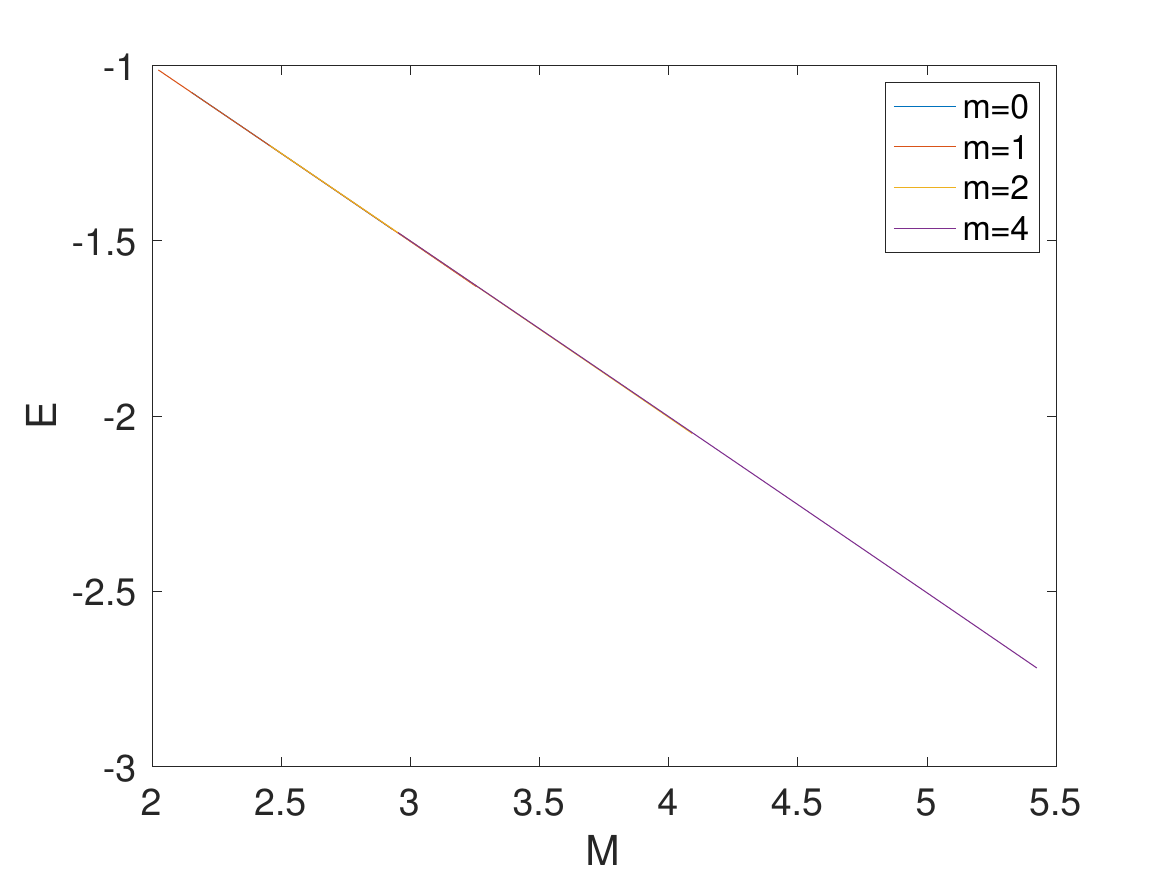}
\includegraphics[width=.49\hsize,height=.3\hsize]{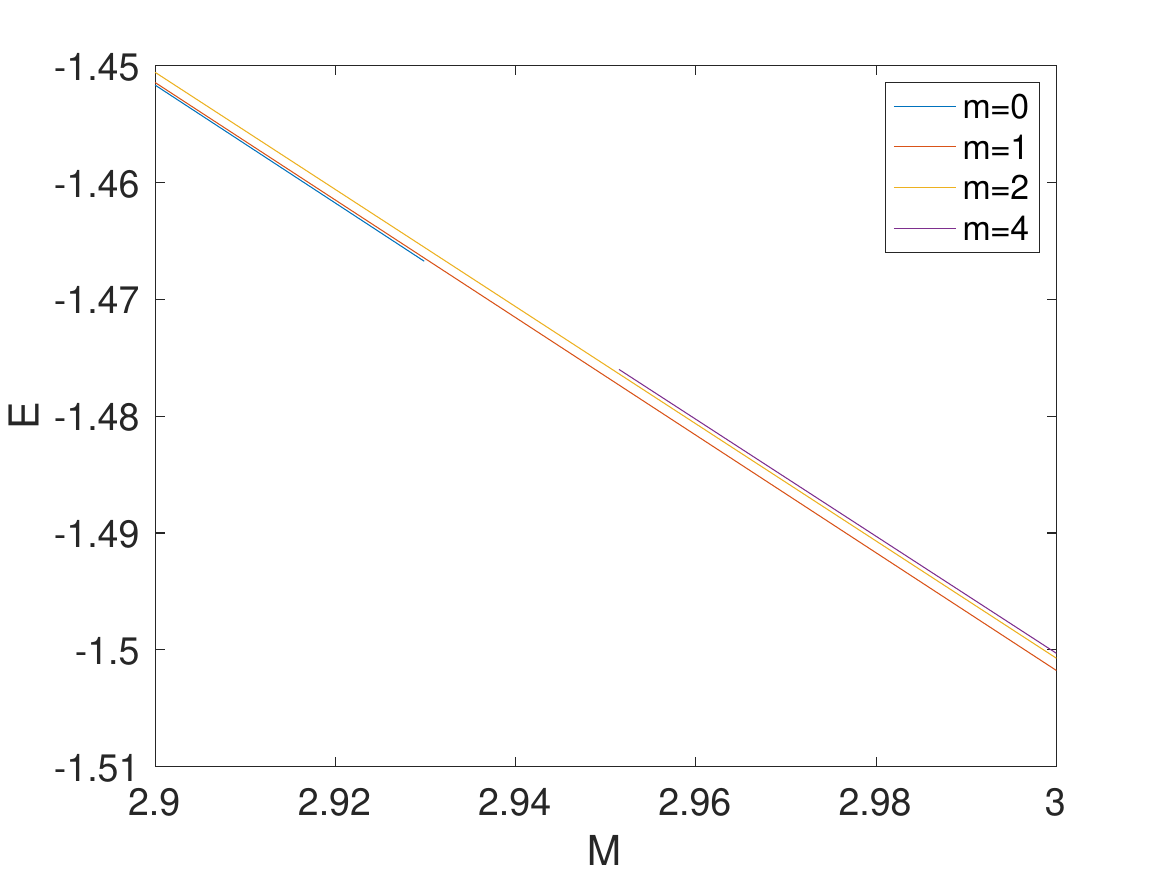}
\includegraphics[width=.49\hsize,height=.3\hsize]{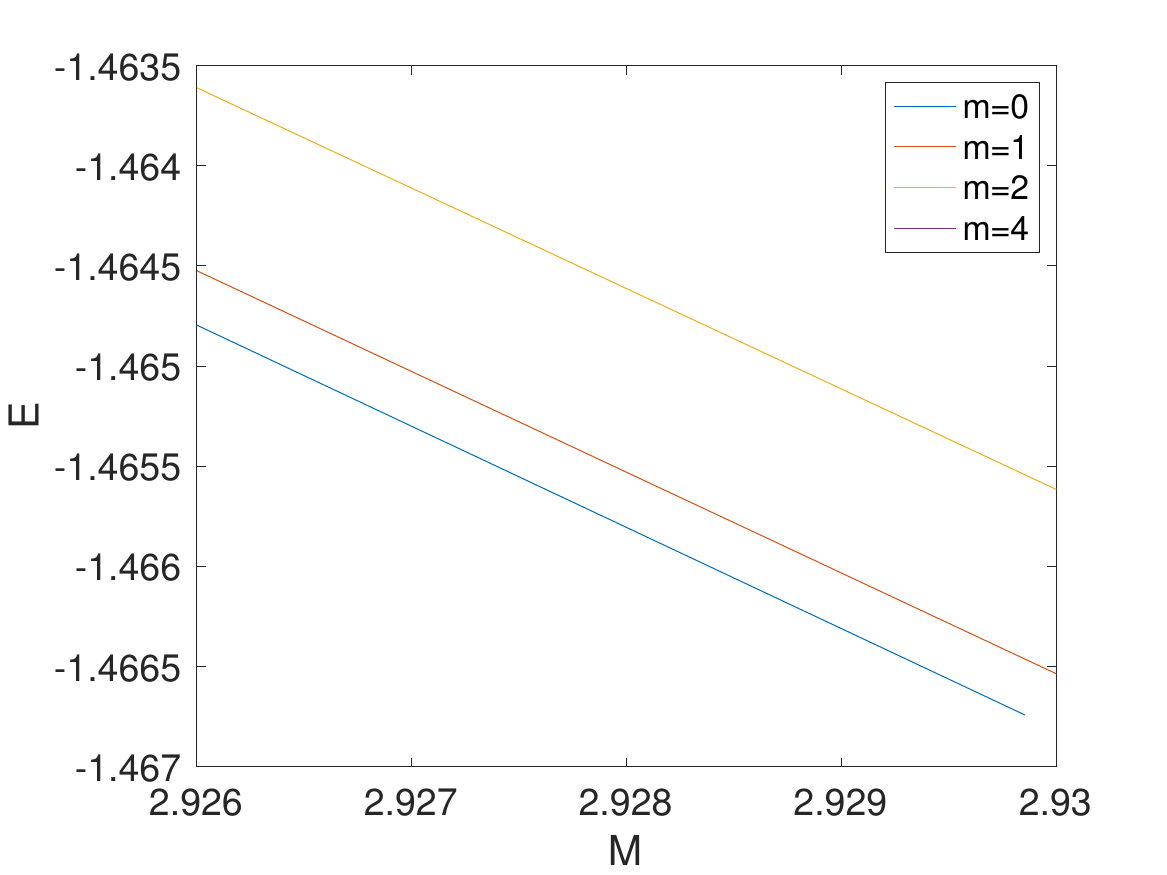}
\includegraphics[width=.49\hsize,height=.3\hsize]{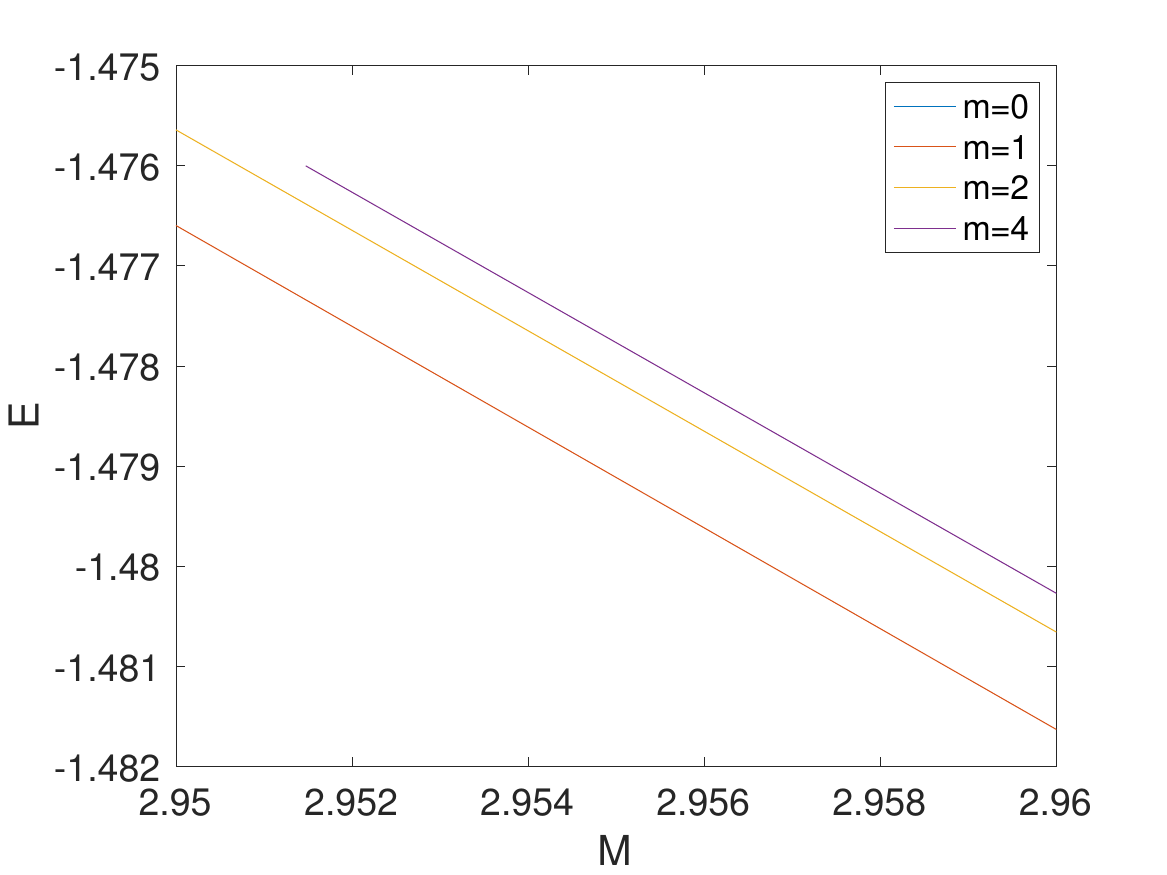}
\caption{\footnotesize Energy as a function of mass for solutions to \eqref{E:2dGS}, $\alpha=2$, with angular modes $m$ and $b \in [1.001,1.01]$ (top left); zoom-in for mass in $[2.9,3]$ (top right); further enlargements $M \in [2.926,2.93]$ (bottom left): blue, $m=0$, radial solution has the lowest energy for a given mass, followed by red, $m=1$, and then yellow, $m=2$; $M\in[2.95,2.96]$ (bottom right): the highest energy branch is $m=4$ (purple).}
\label{figME1001}
\end{figure}

The nearly parallel curves in all plots of Figure \ref{figME1001} could also be understood from the action $S_b$ and its characterization of ground state solutions. Differentiating \eqref{E:action} with respect to $b$ along a ground state branch and simplification, we get 
$$
\frac{d}{db}E(Q_b) = -\frac{b}2 \,\frac{d}{db}M(Q_b), 
$$
and hence, $dE/dM = -b/2$ (when the mass derivative is nonzero). Since $b$ is around $1.001 \sim 1.01$ in this figure, the slopes are basically $-1/2$ for all displayed branches. 
\smallskip

We therefore conclude that the solutions that are obtained from the angular modes initial guess \eqref{init} are candidates for higher energy branches (excited states) in this problem, or if restricted to the related symmetry subspace, then  symmetry-constrained ground states with initialized angular or mode $m$ symmetry. We note that the odd branches are more challenging to find, which is in agreement that the ground states (if not restricted to that specific subspace) are even.

\subsection{Nonradial ground state}\label{S:nonradial}
Since it does not appear that there are solutions in the above angular class 
that have lower energy for a given mass than the solution with radial 
symmetry, even as we take $\ep$ sufficiently small, we next study solutions obtained from an initial guess of the Knapp-type cap-concentration form \eqref{E:cap1}. 

We use $L_{x}=L_{y}=200$ and other numerical parameters as before. 
In addition, we take $\sigma_{x}=0.8$ and 
$\sigma_{y}=1$ as well as $\lambda\ep^{1/2}=0.05$  for 
$\epsilon=10^{-3}$. With this ansatz we find a solution, which is 
different from the ones discussed above as can be seen in 
Figure~\ref{figflat1001}. This solution is somewhat reminiscent of the solutions 
in 1D with just a minor extension of the axis of main support (here 
chosen to be the $x$-axis). 
\begin{figure}[!htb]
\includegraphics[width=1\hsize,height=0.62\hsize]{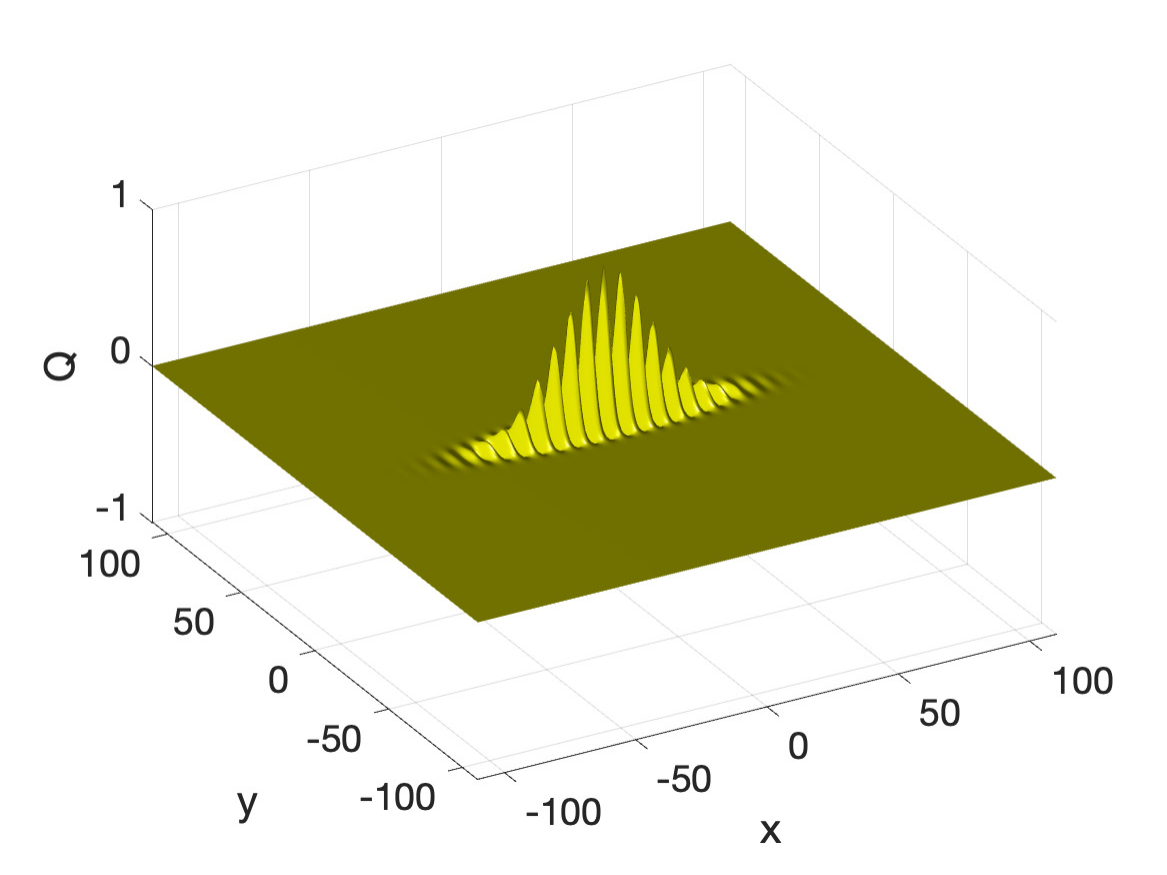}
\caption{\small Nonradial ground state solution to equation \eqref{E:2dGS}, $\alpha=2$, for $b=1.001$.}
\label{figflat1001}
\end{figure}
\begin{figure}[!htb]
\includegraphics[width=1\hsize,height=0.62\hsize]{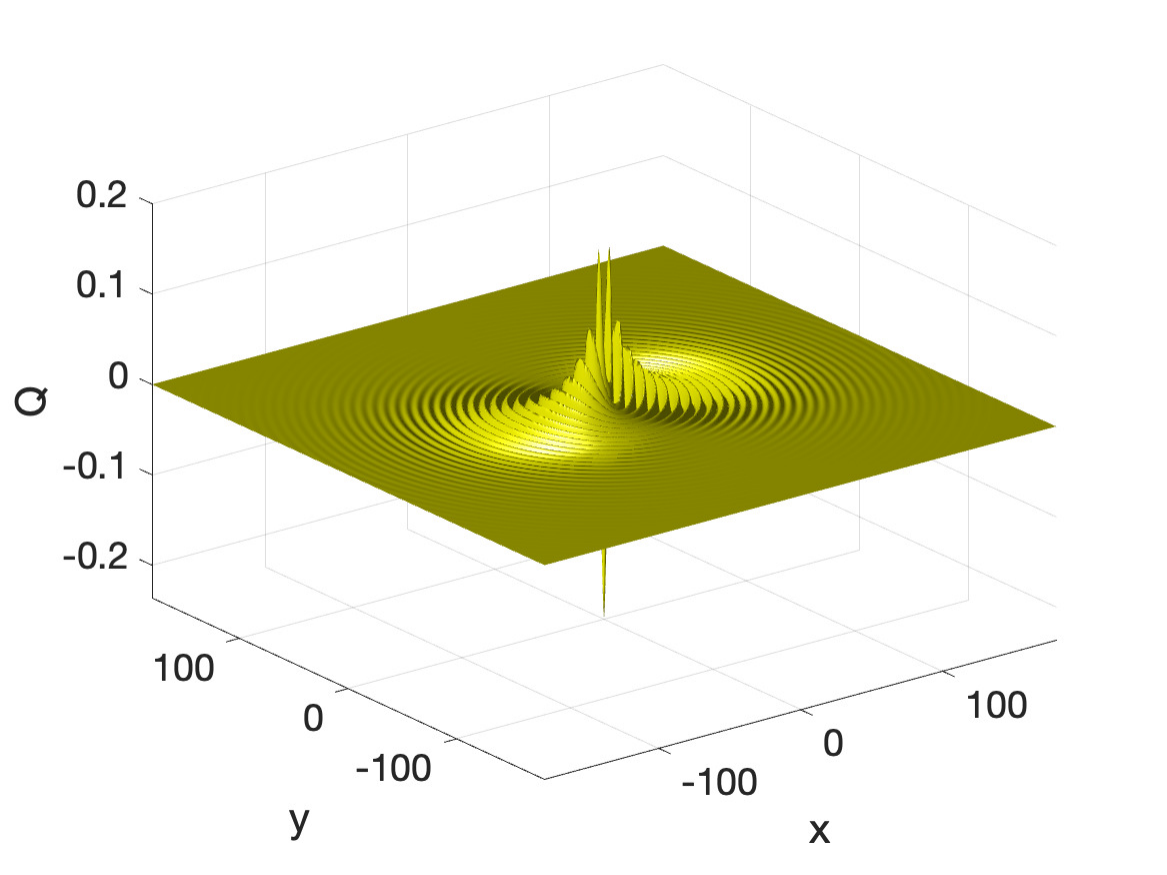}
\caption{\small Nonradial solution to equation \eqref{E:2dGS}, $\alpha=2$, for $b=1.005$.}
\label{figflat1005}
\end{figure}

Tracing this solution to larger values of $b$, we find that the 
iteration ceases to converge for values of $b>1.005$. The solution 
for the value $b=1.005$ can be seen in Figure~\ref{figflat1005}. It is visibly 
less localized around the $x$-axis compared to the solution for $b=1.001$, it has more prominent radial type oscillations, and approaches solutions of the radial/angular 
type studied in the previous subsection. This explains, to some 
extent, why we were not able to trace this branch of solutions for 
larger values of $b$ that we have considered in the previous section. 

In Figure~\ref{figflatME}, we show mass and energy in dependence of $b$, and also the energy as a function of mass $E = E(M)$. Similar to the angle mode solutions the mass is monotonically increasing with $b$ whereas the energy is decreasing as well as the $E(M)$ curve.  
\begin{figure}[!htb]
\includegraphics[width=0.32\hsize,height=.23\hsize]{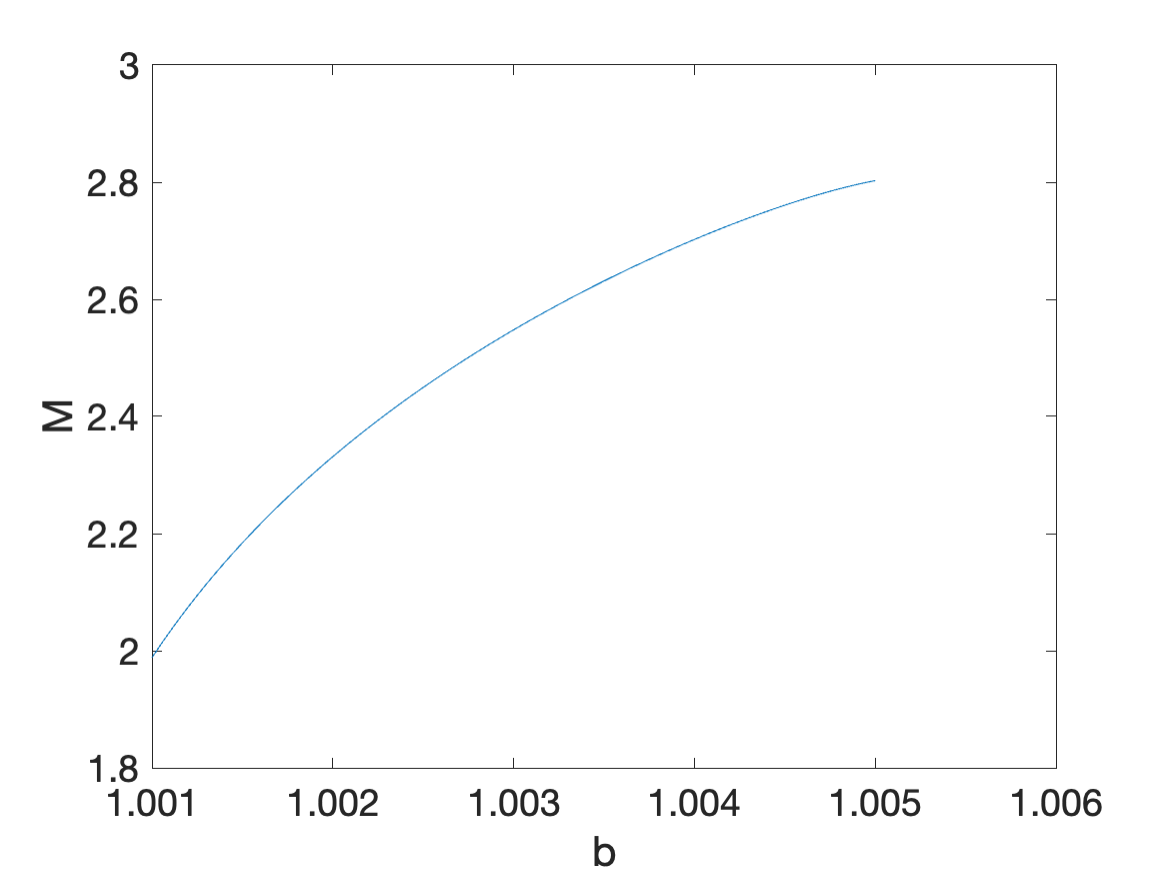}
\includegraphics[width=0.32\hsize,height=.23\hsize]{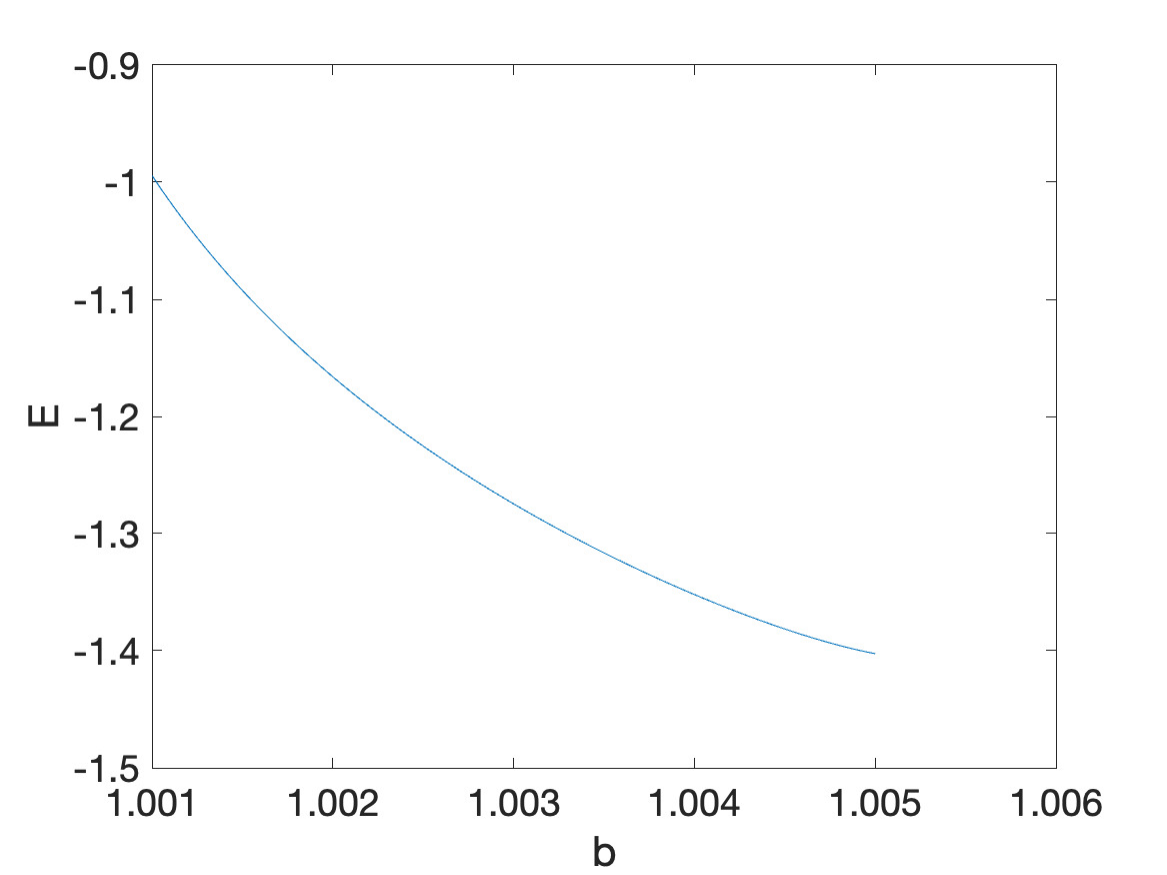}
\includegraphics[width=0.32\hsize,height=.23\hsize]{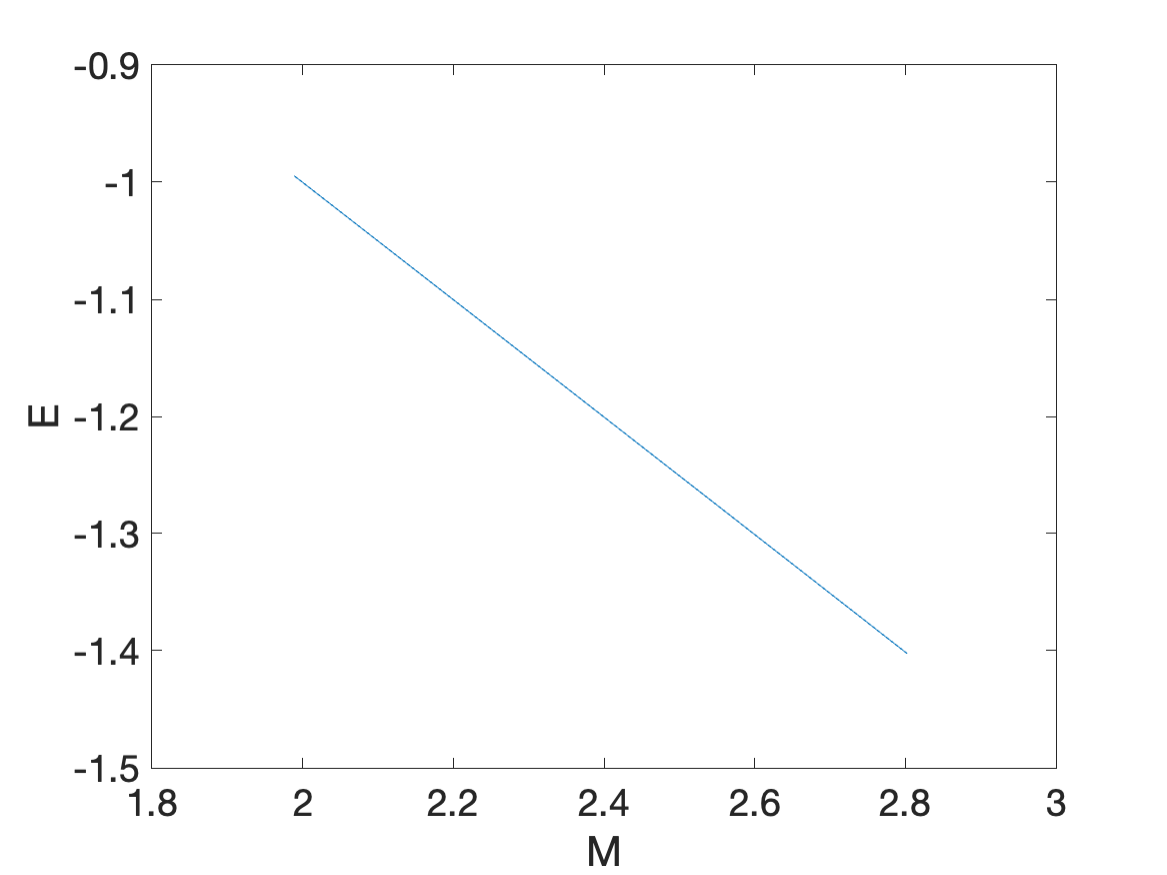}
\caption{\footnotesize Nonradial solutions to the equation \eqref{E:groundstate}, $\alpha=2$: mass and energy in dependence of $b \in [1.001,1.005]$ (left, middle) and energy as a function of mass $E=E(M)$ (right).}
\label{figflatME}
\end{figure}

The decisive question is whether the energy of these solutions is smaller than the energy of solutions with radial symmetry, therefore, we compare these nonradial solutions mass-energy interdependence with that of the radial solutions (of a similar mass). 
Since the mass of the above nonradial solutions is rather small, we 
compute solutions with radial symmetry to even smaller values (than in the previous section) of 
$\epsilon=10^{-4}$.

\begin{figure}[!htb]
\includegraphics[width=0.32\textwidth]{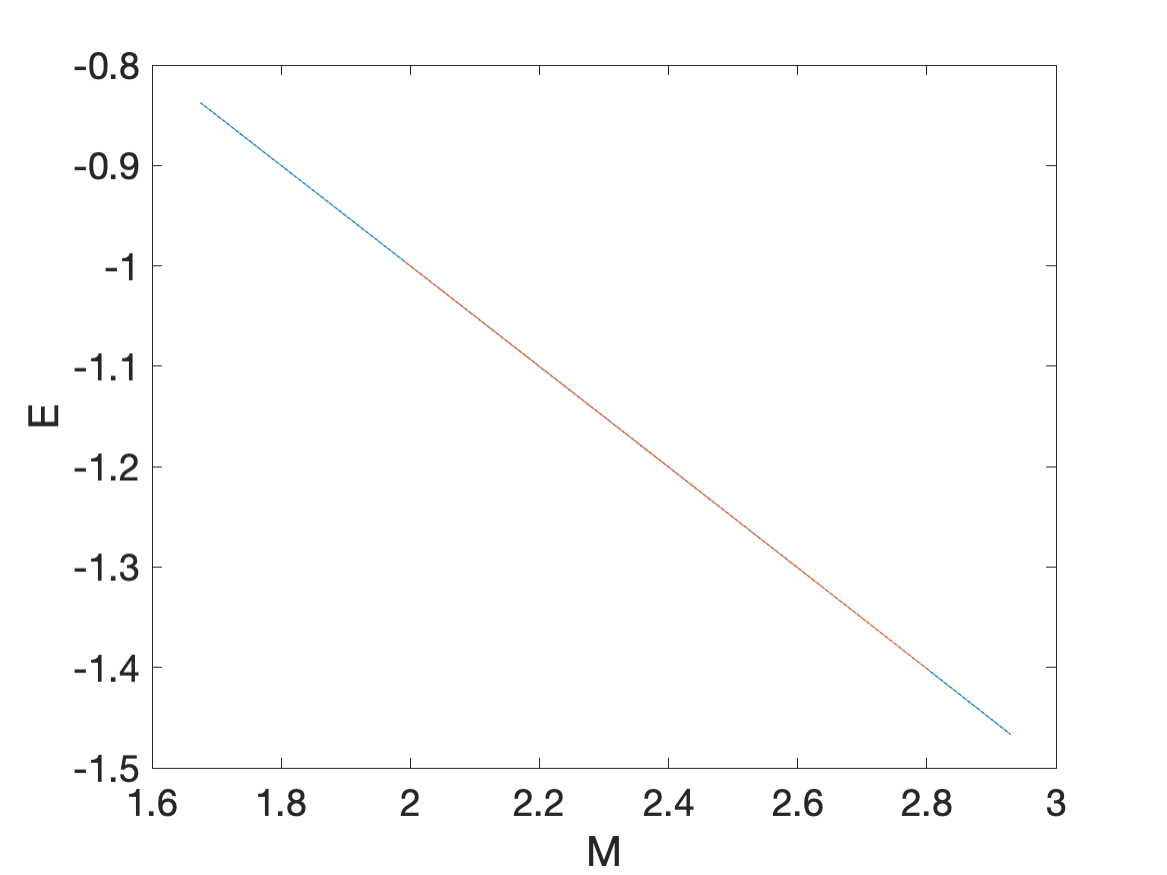}
\includegraphics[width=0.32\textwidth]{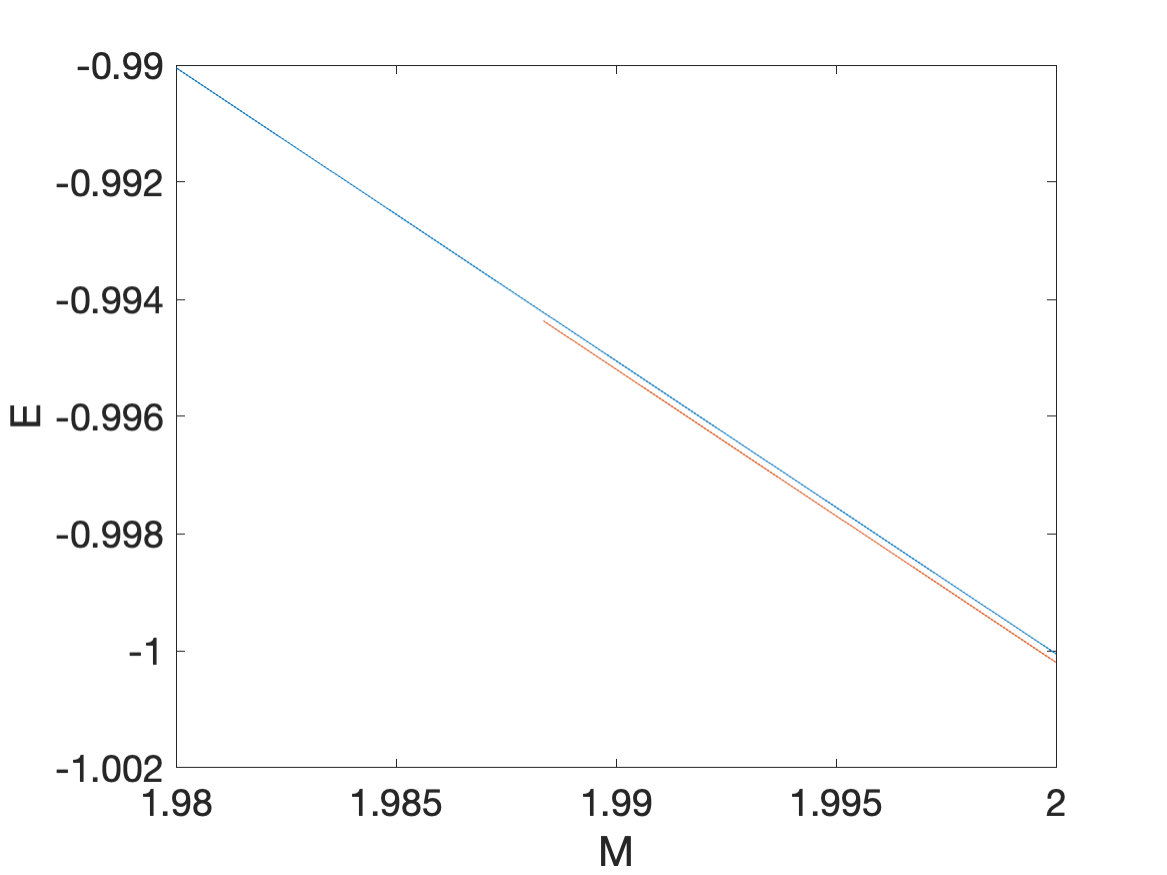}
\includegraphics[width=0.32\textwidth]{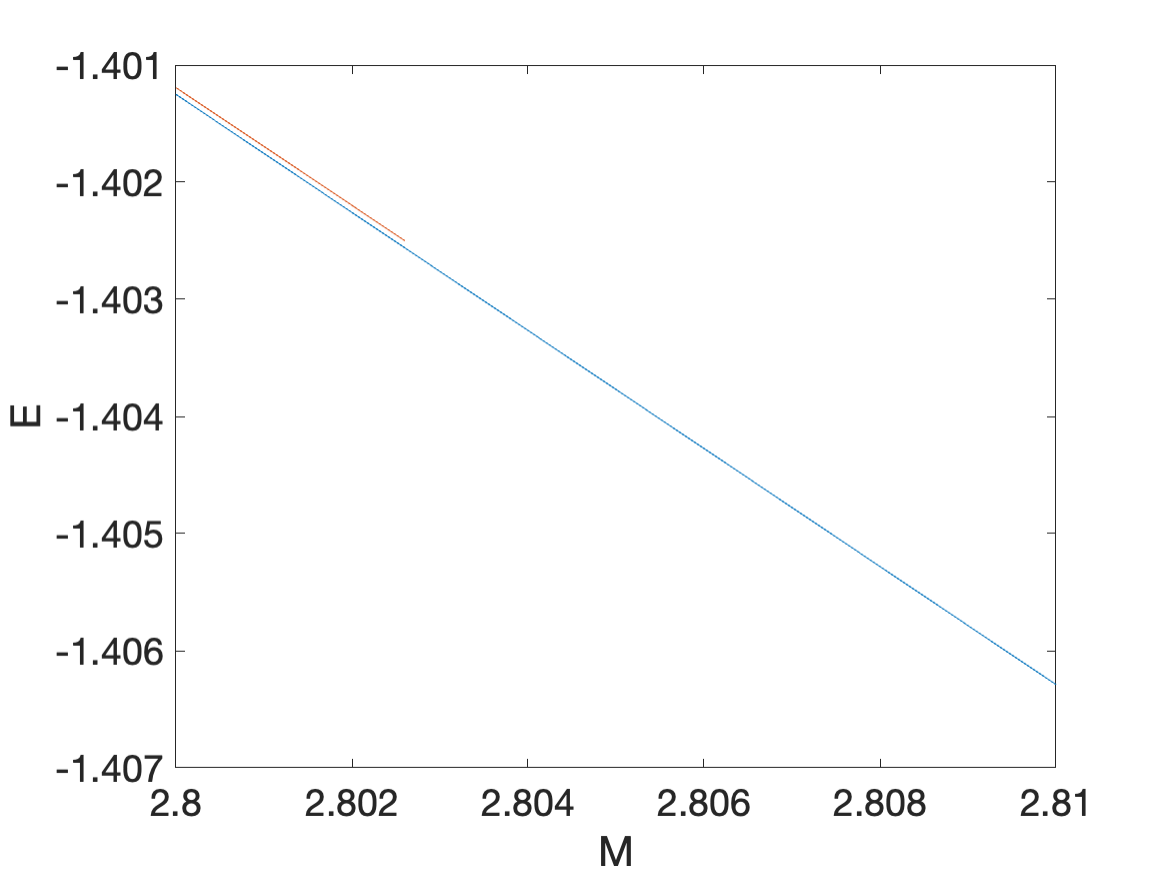}
\caption{\footnotesize Nonradial (red) vs. radial (blue) solutions to 
\eqref{E:2dGS}, $\alpha=2$. Left: energy as a function of mass. Middle: close-up for mass values in $[1.98,2]$: the red line (nonradial) is below the blue line (radial). Right: close-up for mass values in $[2.8,2.81]$:  the red line (nonradial) is above the blue line (radial). Thus, an intersection of branches occurs for some mass in the interval $[2,2.8]$.}  
\label{figflatradial}
\end{figure}
Figure~\ref{figflatradial} displays the mass-energy dependence for the radially symmetric solutions in blue and the nonradial solutions in red. The left plot shows that they merely coincide, but in a close-up near the smaller mass region, $M \in [1.98,2]$ (see the middle plot in the same figure), one can see that 
the energy of the nonradial solutions is in fact lower than the energy of the radial solution (red line is below the blue line). On the other hand, zooming-in into the mass values $M \in [2.8,2.81]$ shows that the nonradial solutions have higher energy than radial (red line above the blue line). Thus, an intersection for mass values between $2$ and $2.8$ has occurred, providing numerical evidence that for very small $\ep$ the ground state solutions are indeed nonradial. As $\ep$ becomes larger, the ground state solutions gain the radial symmetry. Thus, we conclude that the solution in Figure~\ref{figflat1001} belongs to the {\it nonradial ground state} solutions branch of the cubic equation \eqref{E:2dGS}, with $\ep=0.001$. On the other hand, for larger $\epsilon$ values (from about $0.005$, or $b=1.005$) ground state solutions are radial, for example, as shown in Figure~\ref{figm0} top row.    

We can also explain why different branches as $\ep \to 0^+$ (or $b \to 1^+$) can lie close in the mass-energy diagrams. For the unrestricted or radial
least-action families considered here, the Pokhozhaev identity
\eqref{EM-relation} gives
$$
\frac{E(Q_\ep)}{M(Q_\ep)} = -\frac{1+\ep}{2}
+\frac{\alpha}{2(\alpha+2)} \frac{\|Q_\ep\|_{L^{\alpha+2}}^{\alpha+2}}{M(Q_\ep)}.
$$
By \eqref{E:potential-R} and the quotient-to-mass relation in
Theorem \ref{T:quotient-mass-general}(ii), we have
$\|Q_\ep\|_{L^{\alpha+2}}^{\alpha+2}\approx\ep M(Q_\ep)$.
Therefore, the last term is $O(\ep)$, and hence,
$$
\frac{E(Q_\ep)}{M(Q_\ep)} = -\frac12+O(\ep) \qquad \mbox{as } \ep \to 0^+.
$$
Thus, these branches have the same leading energy-mass relation
$E\sim-M/2$, as can be clearly seen in Figure \ref{figflatradial}.

To further confirm the nonradial branch of ground state solutions as $\ep \to 0^+$, we study the asymptotic behavior in $\ep$ given in \eqref{E:nr-3} of the least-action functional $S_{1+\ep}(Q)$ and compare it with the radial branch asymptotics in \eqref{E:r-3}. This can be easily checked, since for solutions of \eqref{E:2dGS} the functional $S_{1+\ep}$ is equal to the quartic power of the potential $L^{4}$-norm $\|Q\|^4_{L^4}$ up to a constant $\frac14$, 
see \eqref{E:S-Potential}. 
\begin{figure}[!htb]
\includegraphics[width=.32\hsize]{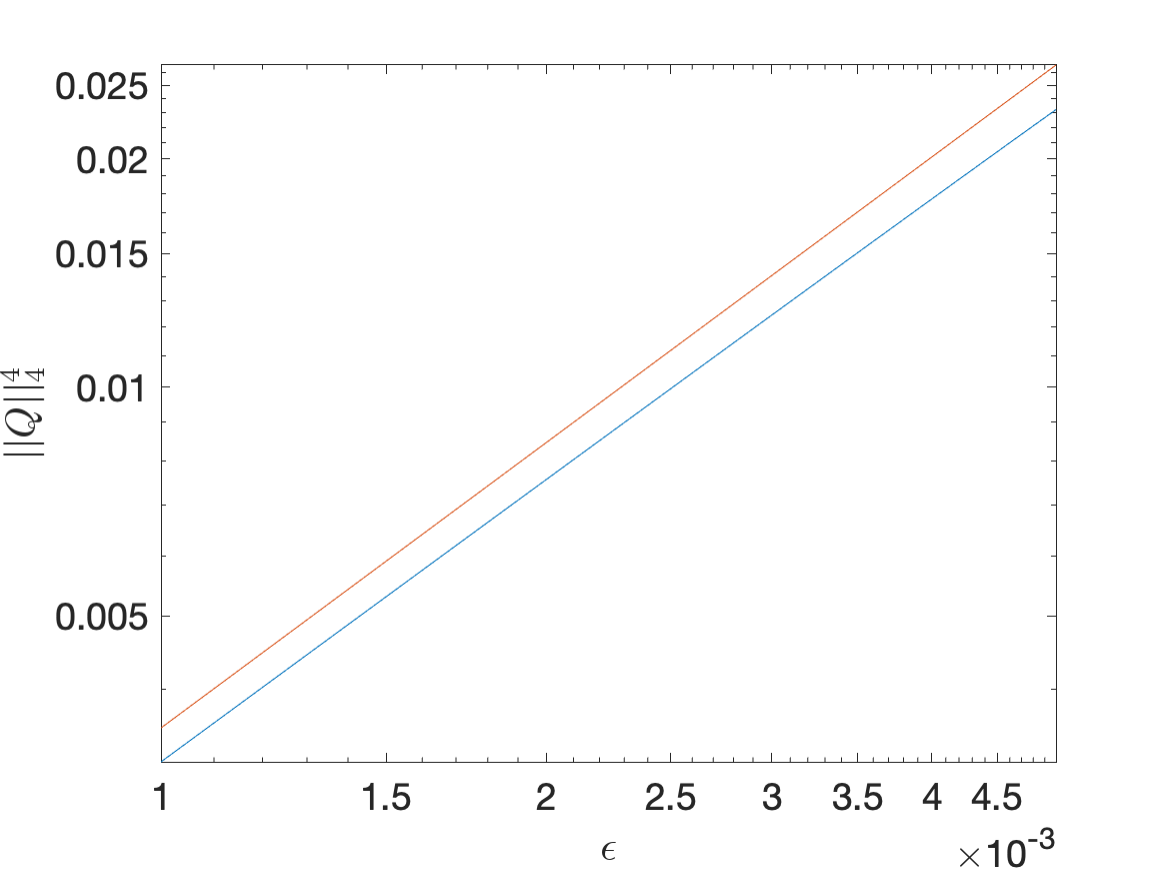}
\includegraphics[width=.32\hsize]{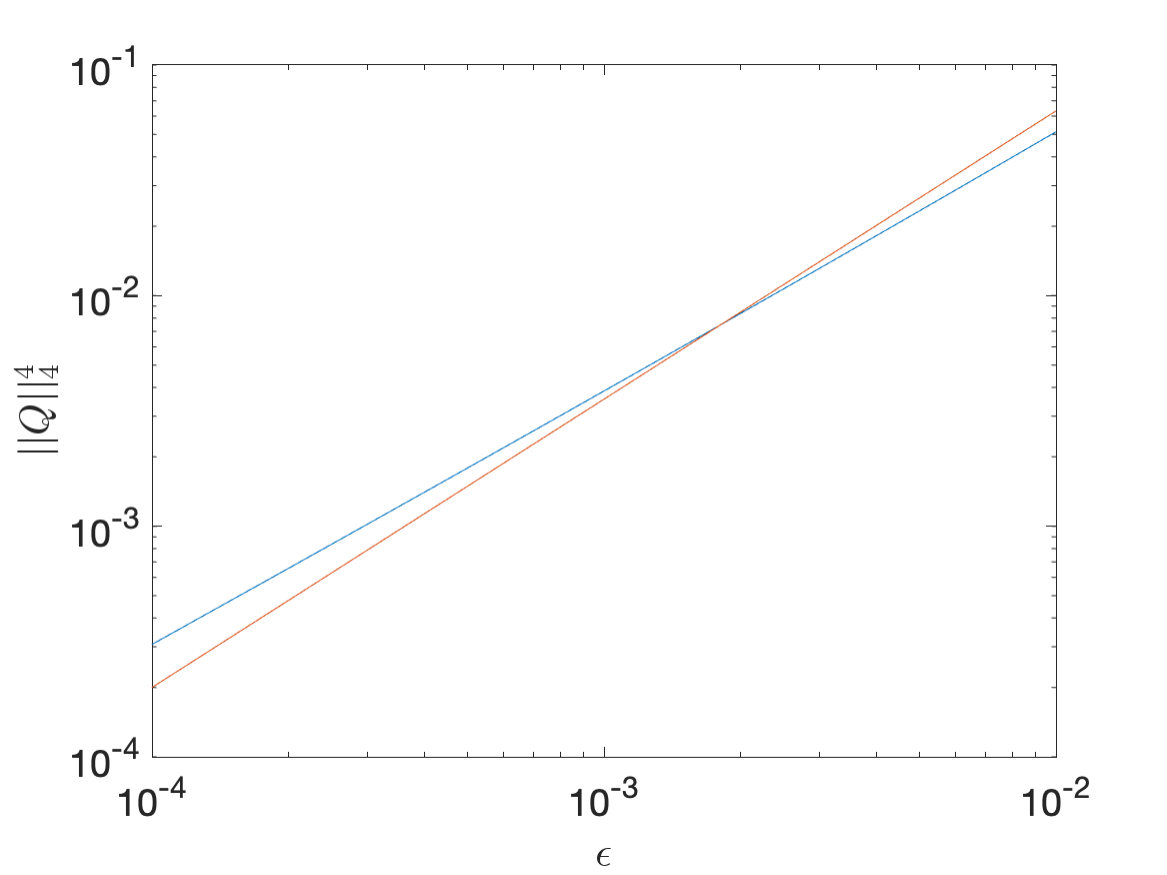}
\includegraphics[width=.32\hsize]{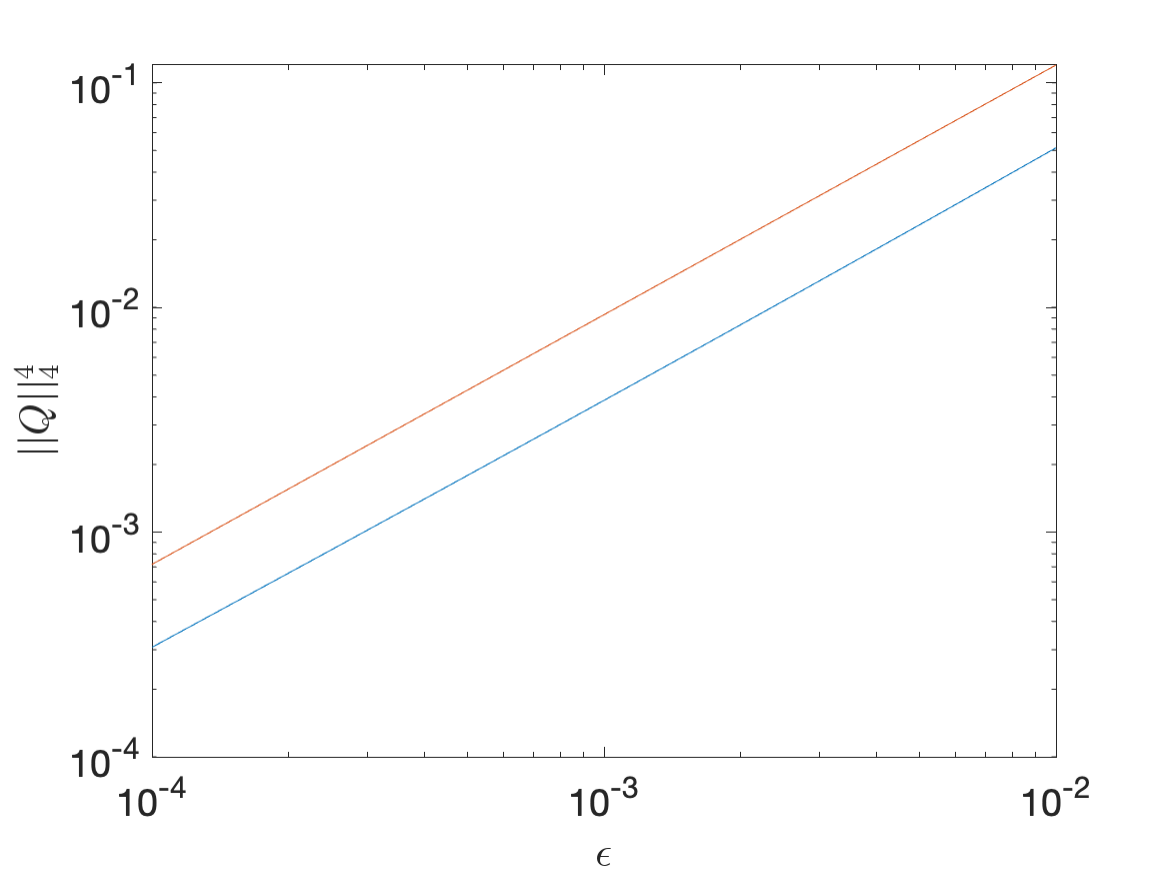}
\caption{\footnotesize Logarithmic scale comparison of $\|Q\|_{L^4}^{4}$ (blue) and $C\,\epsilon^{\frac54}$ (red) for the nonradial (left) and radial (middle) solutions to the equation \eqref{E:2dGS}, $\alpha=2$; comparison of $\|Q\|_{L^4}^{4}$ (blue) and $C\,\epsilon^{\sigma}$, $\sigma =1+\frac{1}{9}$, 
for the radial (red) solutions (right); here, $C=20$ is used for visualization.
}
\label{F:loglog-1}
\end{figure}

Figure~\ref{F:loglog-1} shows comparison on the logarithmic scale of the values $\|Q_\ep\|^4_{L^4} \equiv 4 S_{1+\ep}(Q_\ep)$ in dependence of $\ep$. On the left plot the action of the nonradial branch is compared to the asymptotic law $\ep^{\frac54}$ (for $\ep \to 0^+$) as proved in \cite[Theorem 1.3]{LW2021} and reformulated for the action $S$ in \eqref{E:nr-3}, there we take $\ep \in [0.001,0.005]$. The two lines appear to be parallel, indicating the correct power of the $\ep$-dependence, $\ep^{5/4}$. On the middle plot, the action of the radial branch is compared to the same asymptotic law $\ep^{\frac54}$ for $\ep \in [0.0001,0.01]$, where we had to take smaller $\ep$ to obtain the radial solutions of comparable action. Observe that the lines intersect, in other words, have different slopes, indicating that the asymptotic behavior of the radial branch for small $\ep$ is different from the nonradial rate with power $\frac54$. On the right plot of Figure \ref{F:loglog-1} we compare the radial branch action with $\ep^{1+\frac19}$, obtaining visually parallel lines, thus, indicating a different rate of convergence for radial solutions as $\ep \to 0^+$. We note that at these values of $\epsilon$, this rate is computationally comparable to a possible logarithmic correction rate, 
see Remark \ref{R:rates} and Appendix \ref{A:sub-formal-rad}, especially \eqref{E:S-asym-rad-2d}.

Thus, for $\alpha=2$, the two diagnostics agree: the mass-energy diagram comparison selects the nonradial branch only in the small-$\epsilon$ regime, while the action (equivalently, the potential norm) follows the predicted nonradial $\epsilon^{5/4}$ rate. 


\section{Quartic nonlinearity}
\label{S:quartic}

We continue the numerical construction of ground states in 2D to the 
equation \eqref{E:2dGS}, and in this section we consider the quartic nonlinearity, $\alpha=3$. This power also belongs to the range covered by  Theorem \ref{Thm1} (or \cite[Theorem 1.2]{LW2021}), which predicts nonradial least-action ground states for sufficiently small $\epsilon>0$. 
In our construction of the ground states, as in the cubic case, we track energy vs. mass to identify the lower energy branches (and connect that with least-action solutions). In this case of quartic nonlinearity we find that there can be {\it unstable} and {\it stable} branches in the mass-energy diagrams. 
This is consistent with the appearance of branching observed numerically in 1D in our previous work \cite{KPRS} for $\alpha=4, 6, 8$, where two branches of ground states appear in the mass-energy plots; see also the related discussion of normalized ground states for the bi-NLS in \cite{SP2020, BCGJ2019, FJMM2022}. This in fact influences what should be termed by a ground state solution, which we discuss in more detail in \S \ref{S:alpha3-nonradial}.

As in the cubic case, we first investigate angular mode solutions as it gives a testbed for our computations and then try to find nonradial solutions to the equation \eqref{E:2dGS} by testing Knapp-type examples or using the known solutions for the $\alpha=2$ nonlinearity as an initial guess or continuously tracing the nonradial branch from $\alpha=2$.

\subsection{Radial and angular mode solutions}\label{S:quartic-radial}
To compute radial solutions and other angular mode solutions, we use the ansatz \eqref{init} for the initial guess.

\begin{figure}[!htb]
\begin{subfigure}{.32\textwidth}
\includegraphics[width=1\linewidth]{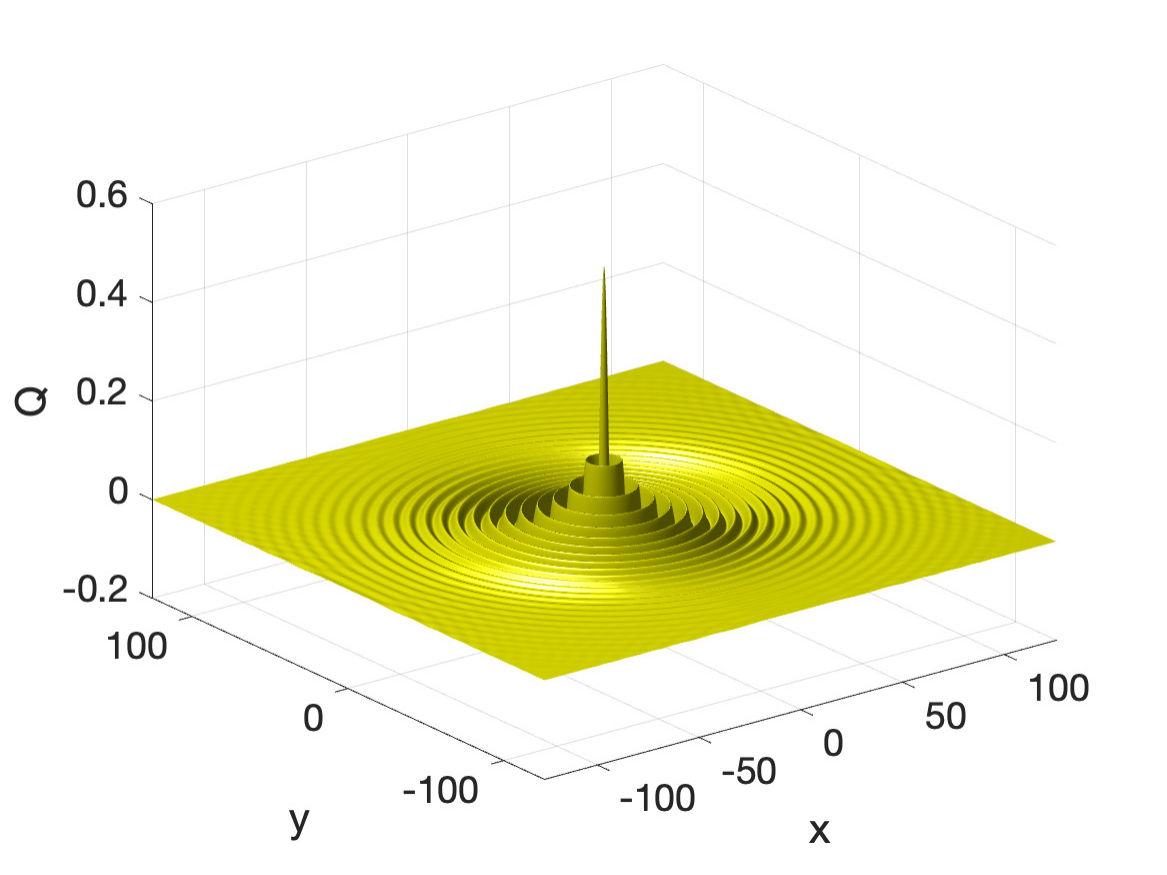}
\subcaption[]{{\footnotesize $m=0$, $b=1.005$.}}
\end{subfigure}
\begin{subfigure}{.32\textwidth}
\includegraphics[width=1\linewidth]{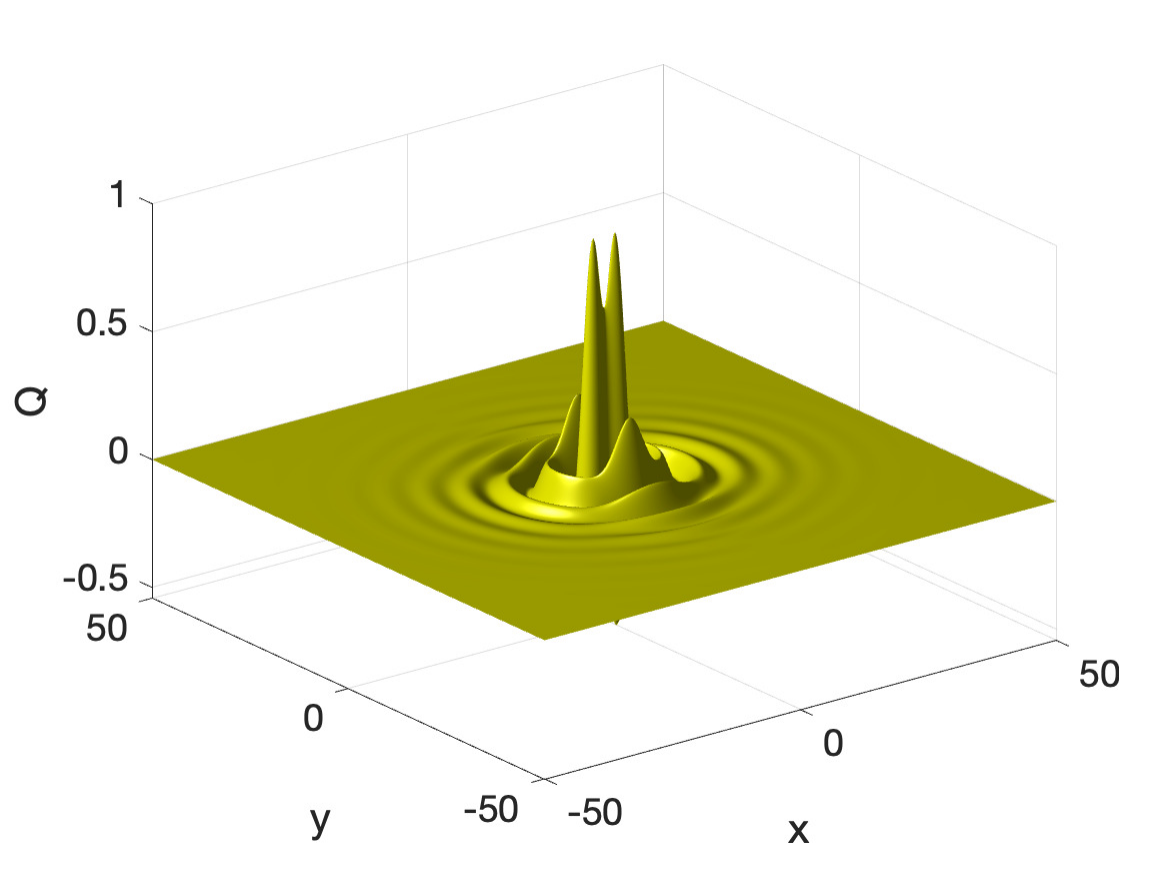}
\subcaption[]{{\footnotesize $m=2$, $b=1.02$.}}
\end{subfigure}
\begin{subfigure}{.32\textwidth}
\includegraphics[width=1\linewidth]{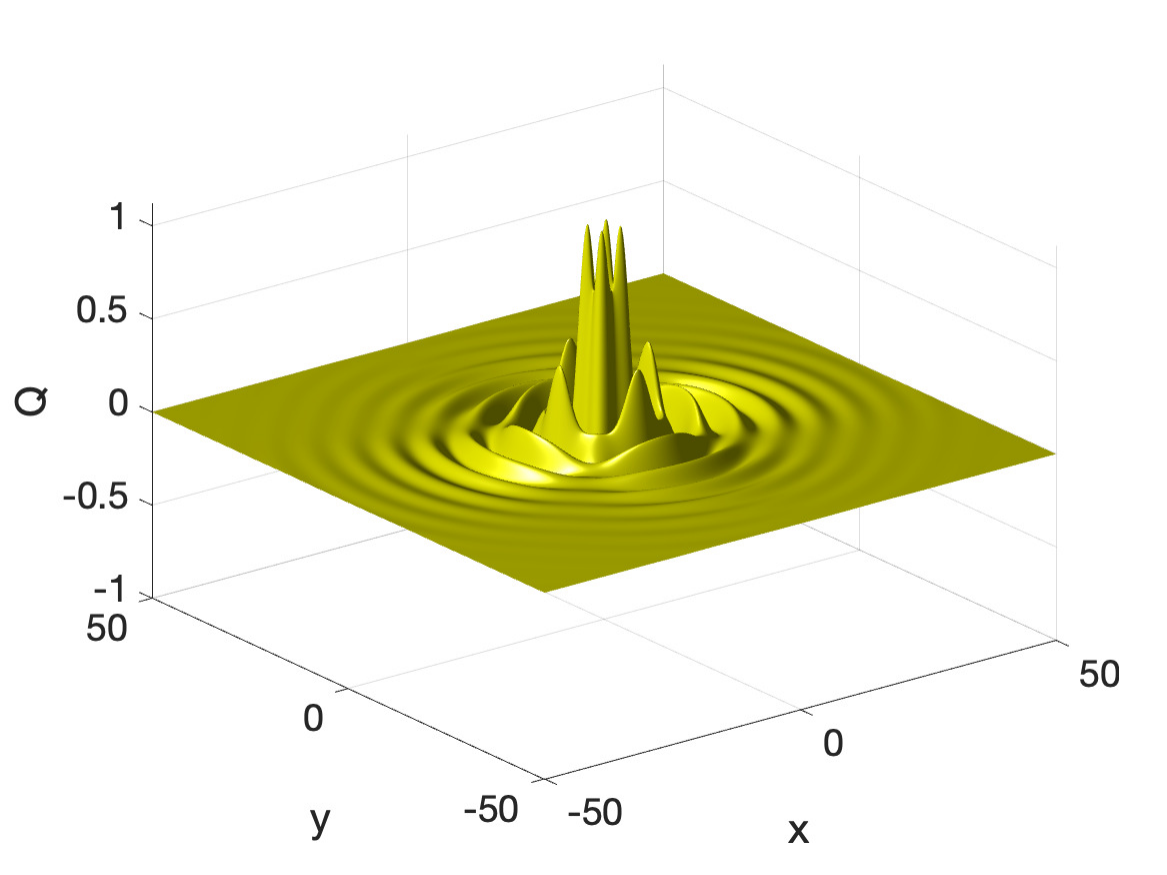}
\subcaption[]{{\footnotesize $m=4$, $b=1.05$.}}
\end{subfigure}
\caption{\footnotesize {Solutions to the equation \eqref{E:2dGS}, $\alpha=3$, with angular modes $m=0, 2, 4$.}}
\label{F:4ab1005}
\end{figure}
\begin{figure}[!htb]
\begin{subfigure}{.32\textwidth}
\includegraphics[width=1\linewidth,height=0.67\linewidth]{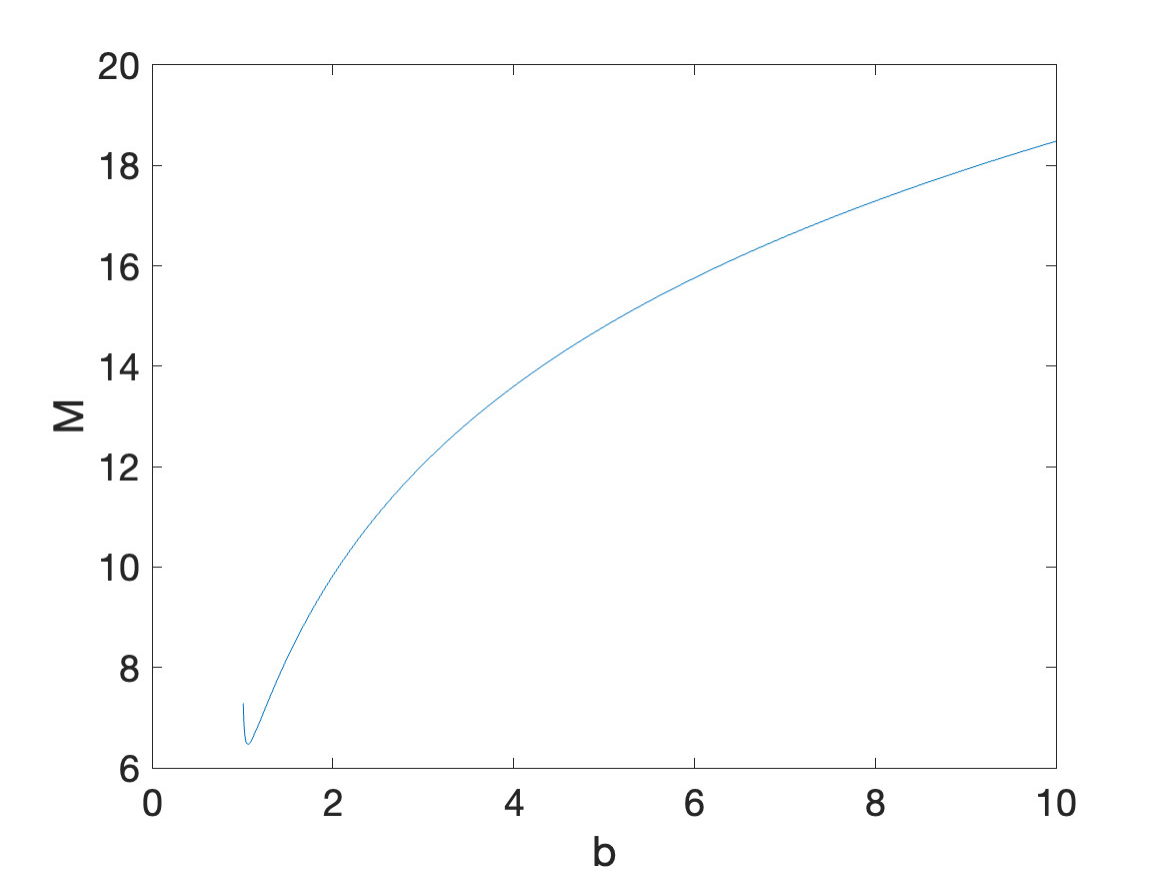}
\subcaption[]{{\footnotesize {$m=0$, $M(Q_\ep) = M(b)$}}}
\end{subfigure}
\begin{subfigure}{.32\textwidth}
\includegraphics[width=1\linewidth,height=0.67\linewidth]{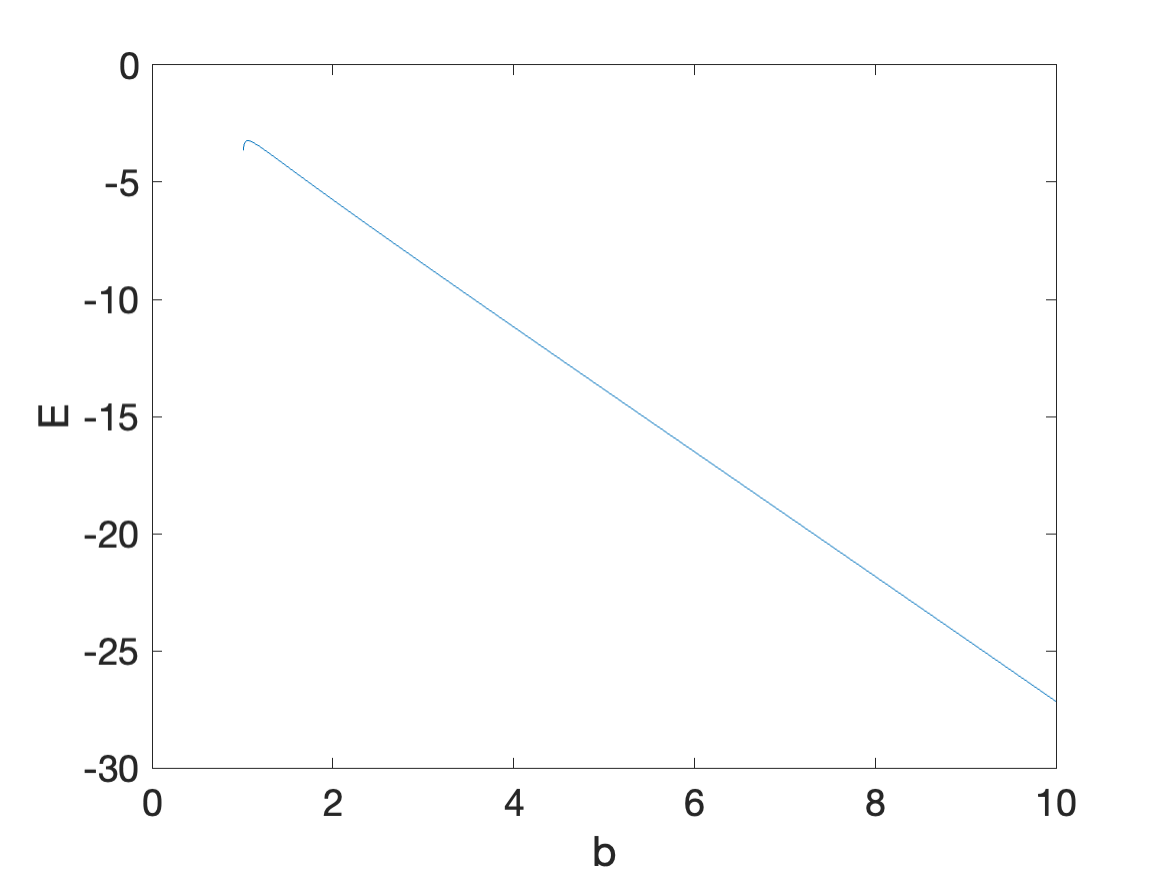}
\subcaption[]{{\footnotesize $E(Q_\ep)$ as function of $b$}}
\end{subfigure}
\begin{subfigure}{.32\textwidth}
\includegraphics[width=1\linewidth,height=0.67\linewidth]{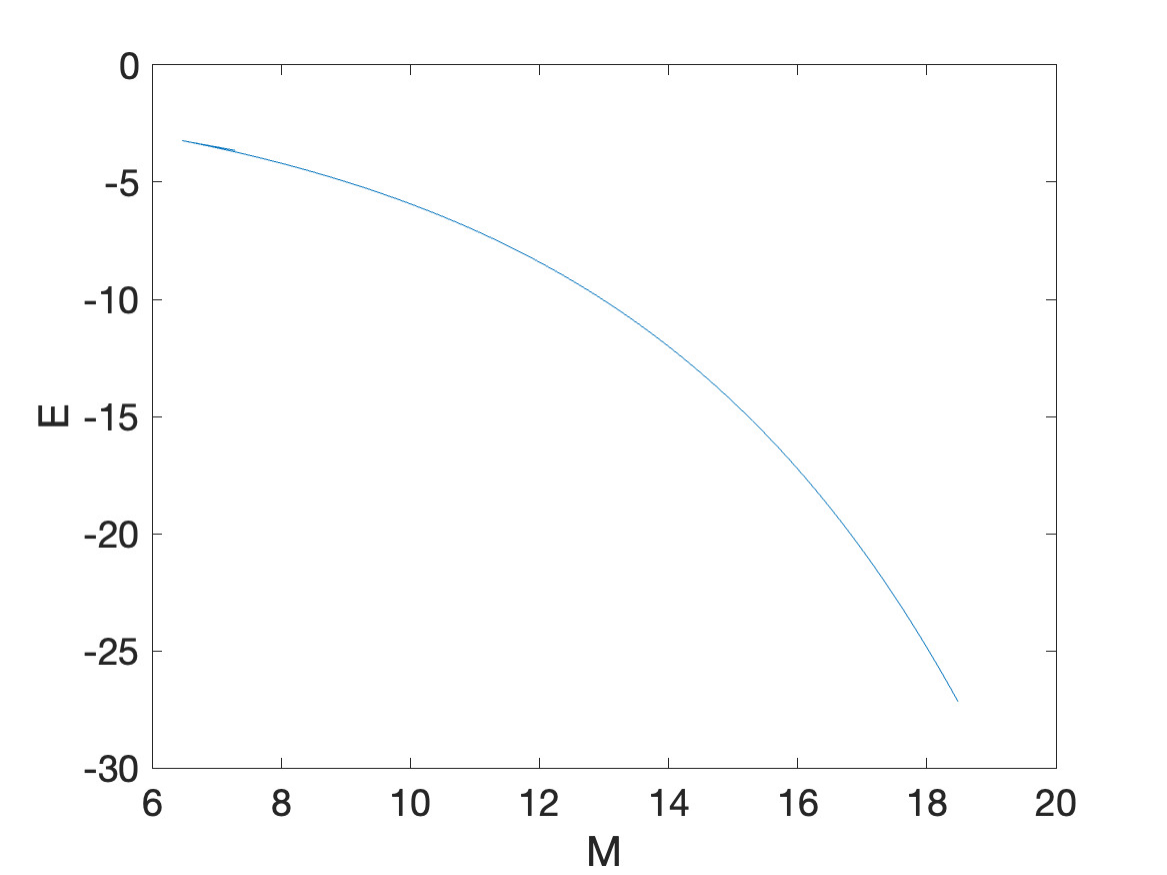}
\subcaption[]{{\footnotesize $E = E(M)$}}
\end{subfigure}\\
\begin{subfigure}{.32\textwidth}
\includegraphics[width=1\linewidth,height=0.67\linewidth]{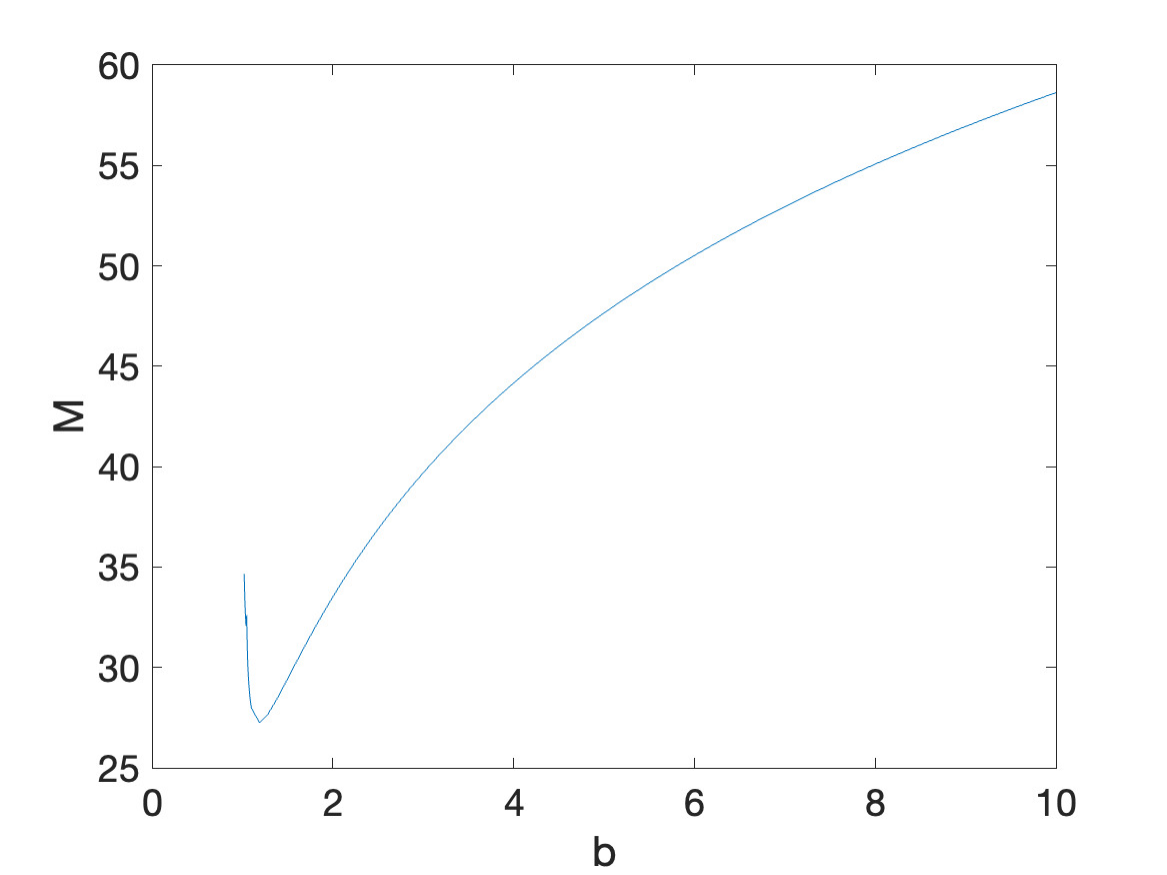}
\subcaption[]{{\footnotesize {$m=2$, $M(Q_\ep) = M(b)$}}}
\end{subfigure}
\begin{subfigure}{.32\textwidth}
\includegraphics[width=1\linewidth,height=0.67\linewidth]{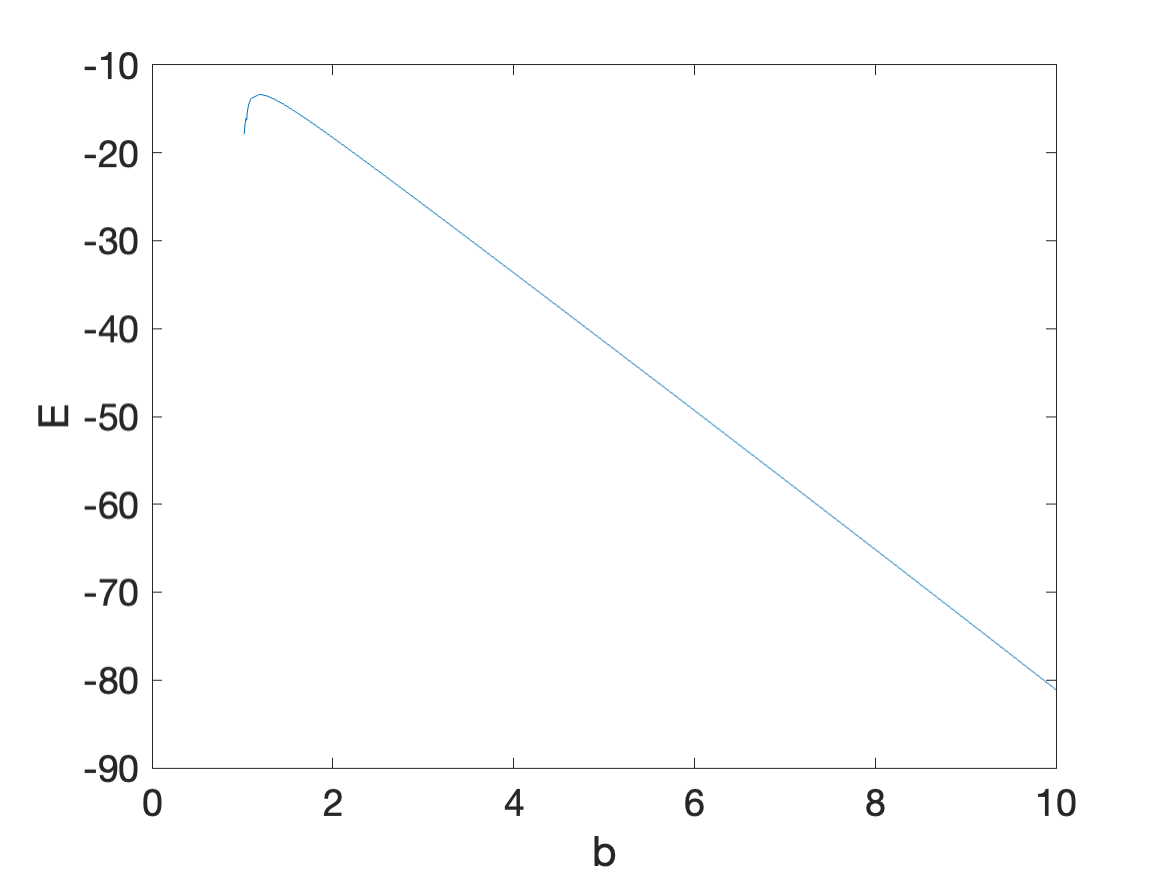}
\subcaption[]{{\footnotesize $E(Q_\ep)$ as function of $b$}}
\end{subfigure}
\begin{subfigure}{.32\textwidth}
\includegraphics[width=1\linewidth,height=0.67\linewidth]{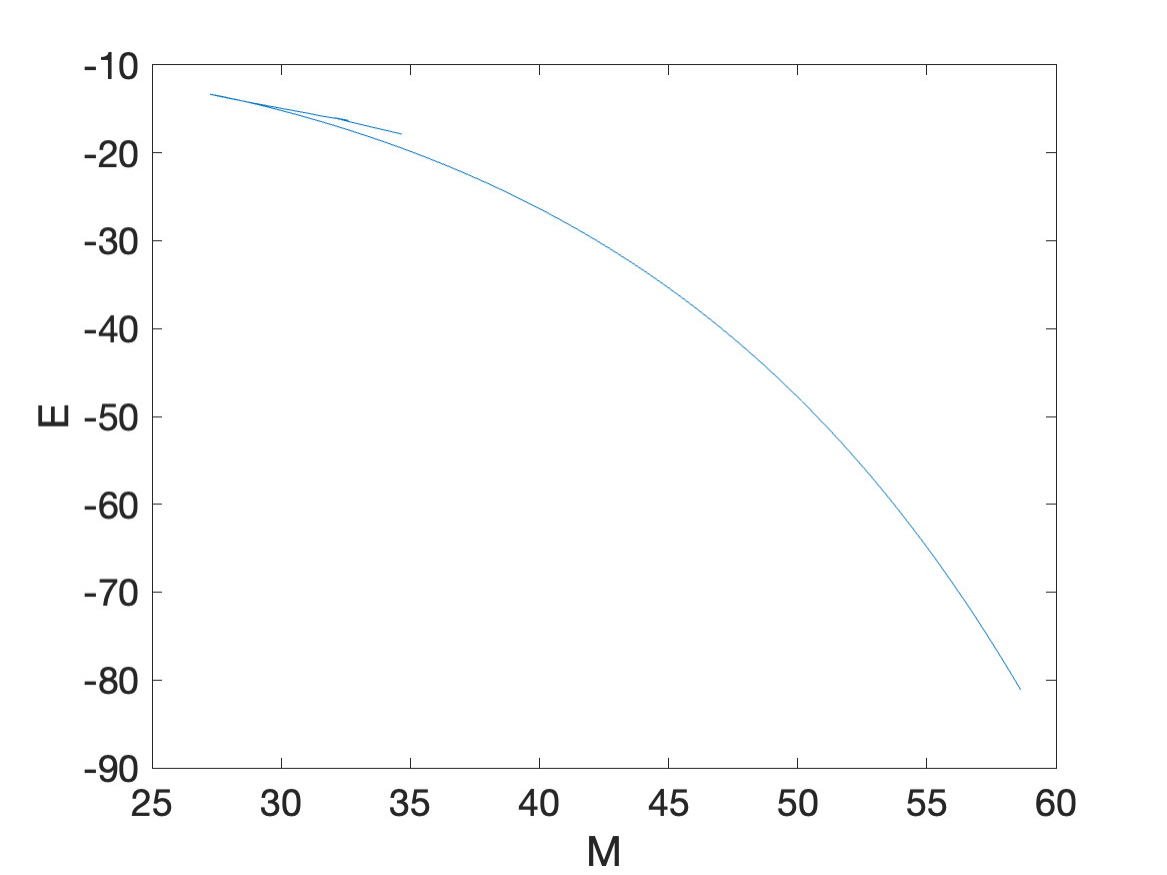}
\subcaption[]{{\footnotesize $E = E(M)$}}
\end{subfigure}\\
\begin{subfigure}{.32\textwidth}
\includegraphics[width=1\linewidth,height=0.67\linewidth]{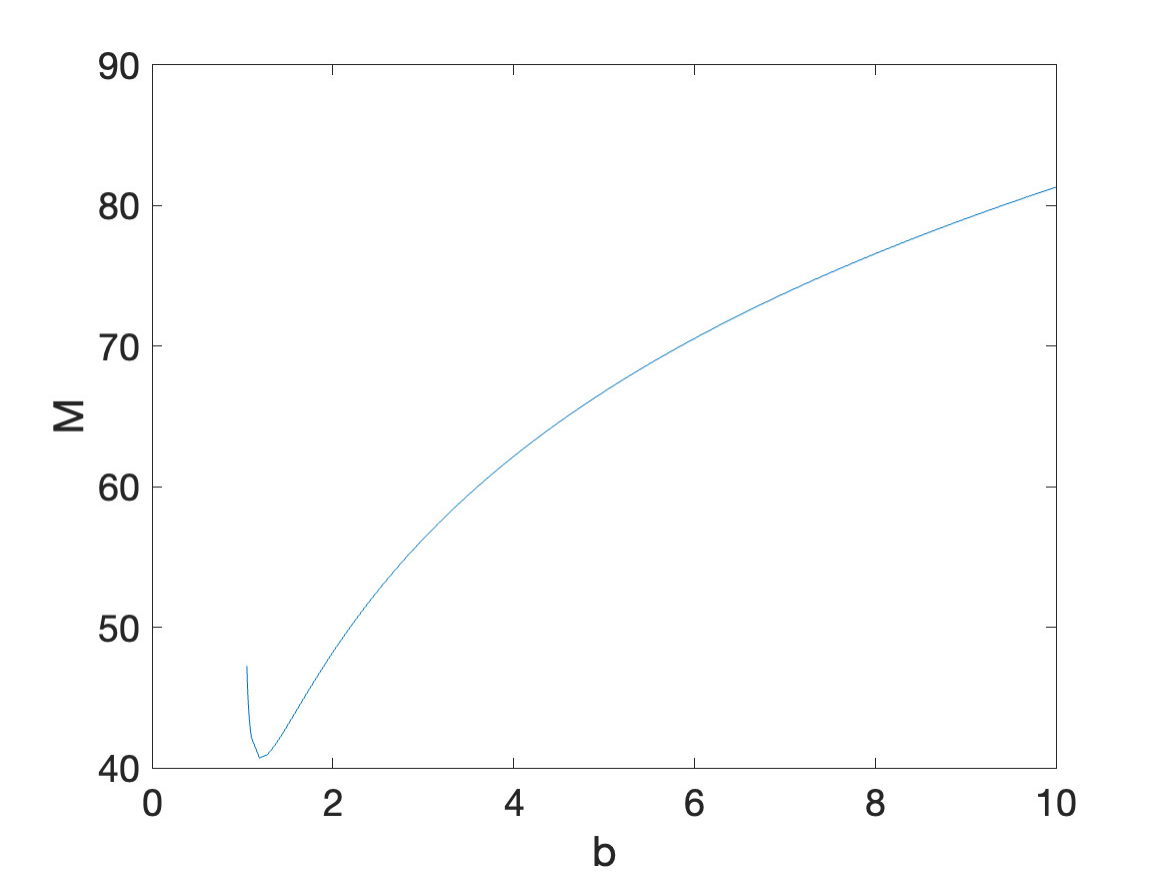}
\subcaption[]{{\footnotesize {$m=4$, $M(Q_\ep) = M(b)$}}}
\end{subfigure}
\begin{subfigure}{.32\textwidth}
\includegraphics[width=1\linewidth,height=0.67\linewidth]{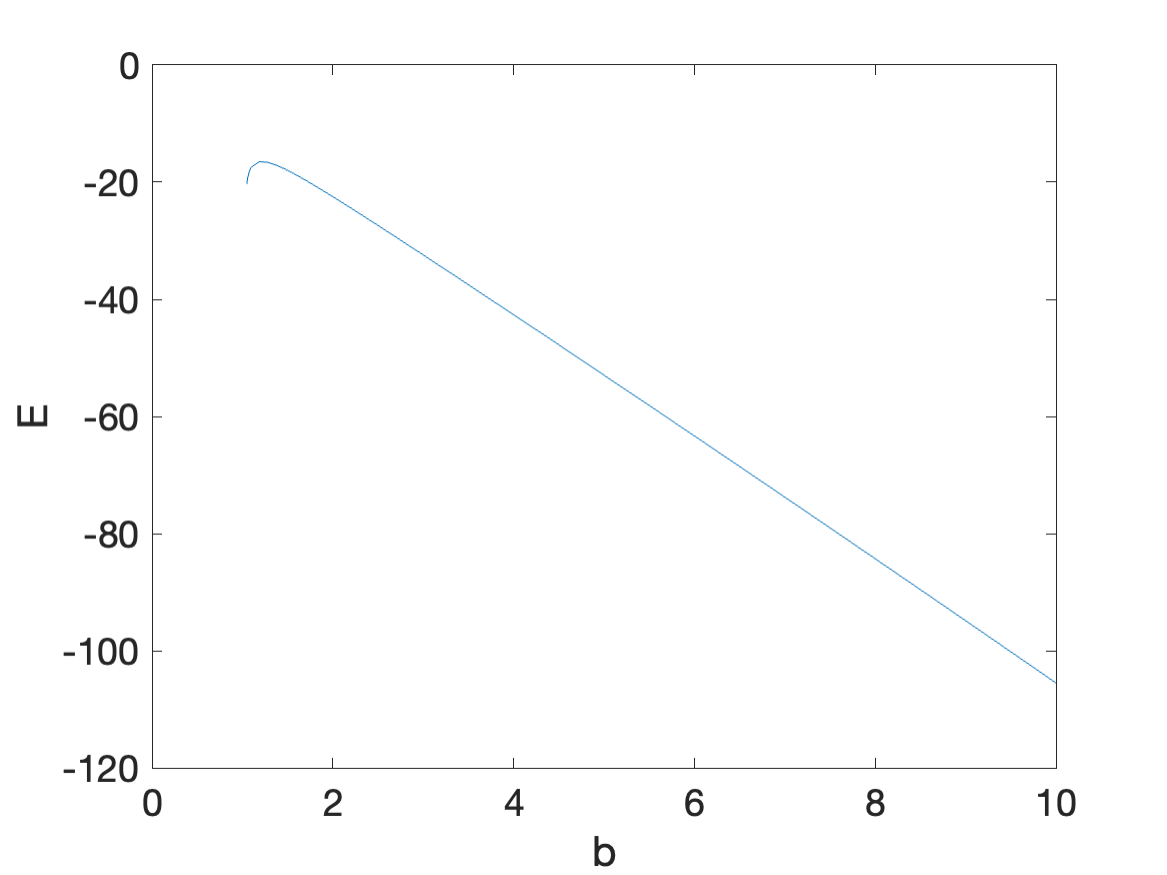}
\subcaption[]{{\footnotesize $E(Q_\ep)$ as function of $b$}}
\end{subfigure}
\begin{subfigure}{.32\textwidth}
\includegraphics[width=1\linewidth,height=0.67\linewidth]{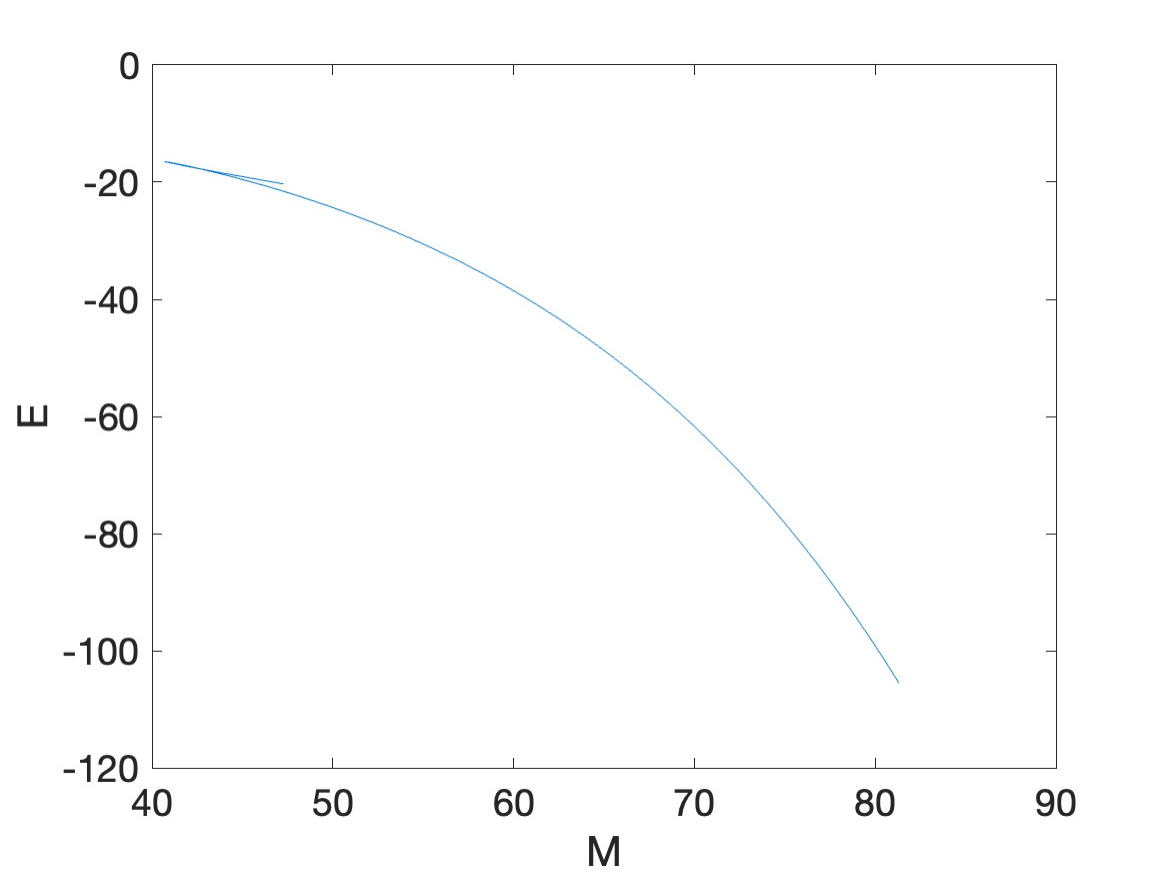}
\subcaption[]{{\footnotesize $E = E(M)$}}
\end{subfigure}
\caption{\footnotesize {Dependence in the quartic case $\alpha=3$ for $m=0$ (top), $m=2$ (middle), $m=4$ (bottom), of $M(Q_\ep)$ and $E(Q_\ep)$ on the parameter $b$ (left and middle columns). Dependence of energy as a function of mass, $E = E(M)$ (right column).}}
\label{F:4a}
\end{figure}

In the radially symmetric case, $m=0$, we show the solution for $b=1.005$ on the left of Figure \ref{F:4ab1005}. In the case $m=2$, we can reach values of $b=1.02$, this solution is shown in the middle of the same figure.
The solution with the angular mode $m=4$ for $b=1.05$ is shown on the right of Figure \ref{F:4ab1005}. 

The mass, energy and their interdependence of the angular solutions are shown in Figure \ref{F:4a}. 
In all of these cases due to non-monotonicity of the mass (and also energy),  
there is a stable and an unstable branch in the mass-energy dependence as can be seen in the right plots of Figure \ref{F:4a}.

\smallskip

\begin{figure}[!htb]
\includegraphics[width=0.95\hsize,height=.55\linewidth]{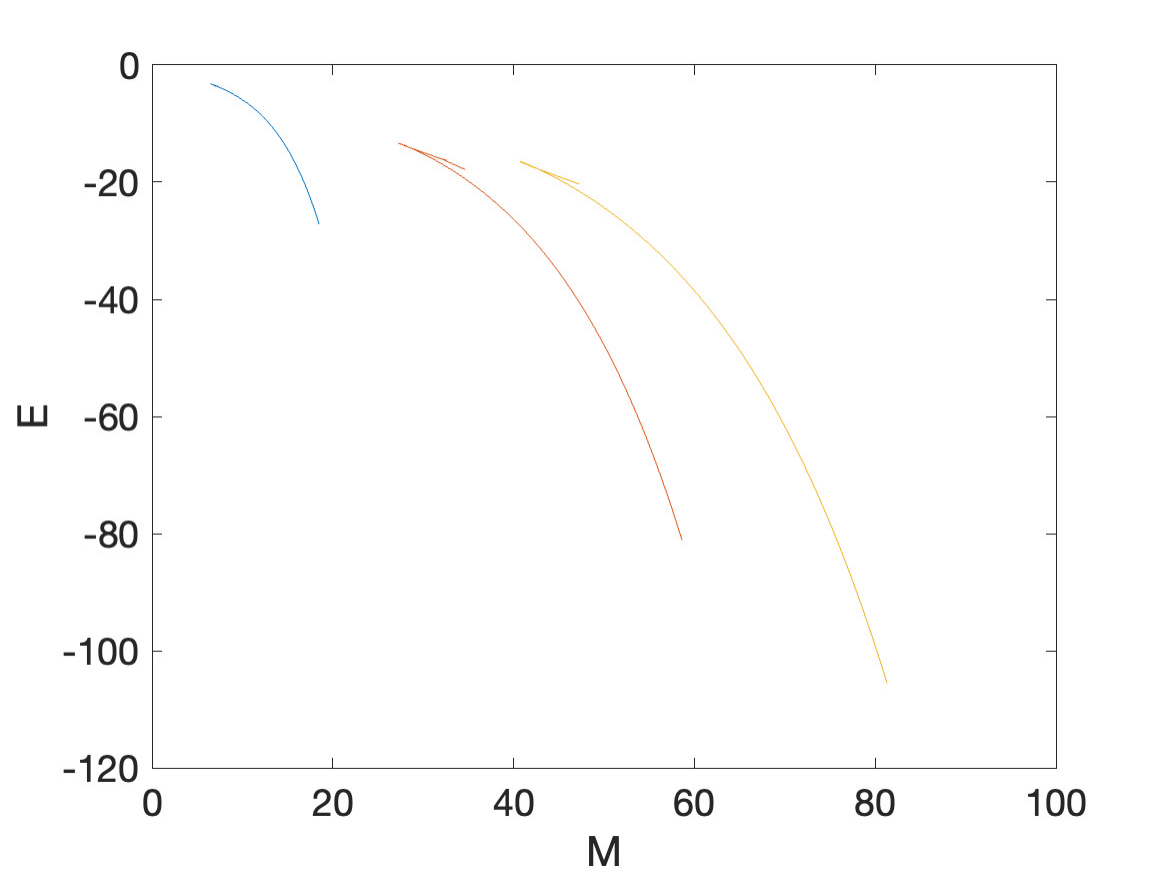}
\caption{\footnotesize Dependence of energy on mass for angular mode solutions to \eqref{E:2dGS}, $\alpha=3$: angular modes: $m=0$ radial (blue), $m=2$ (red), $m=4$ (yellow).} 
\label{F:alpha3angularME}
\end{figure}

In Figure \ref{F:alpha3angularME} for a better comparison, we show the dependence of the energy on the mass for the cases $m=0,2,4$ from Figure \ref{F:4a} (C), (F), (I) in a single plot. Observe that each curve has a lower- and an upper-energy branch. By analogy with the one-dimensional dynamics we call these a stable and an unstable branch (though we do not study the stability of branches in this work as it is beyond the scope of this paper). Among the angular-mode branches, the radial solutions seem to have the lowest energy in the computed range. 


\begin{figure}[!htb]
\begin{subfigure}{.32\textwidth}
\includegraphics[width=1\linewidth,height=0.72\linewidth]{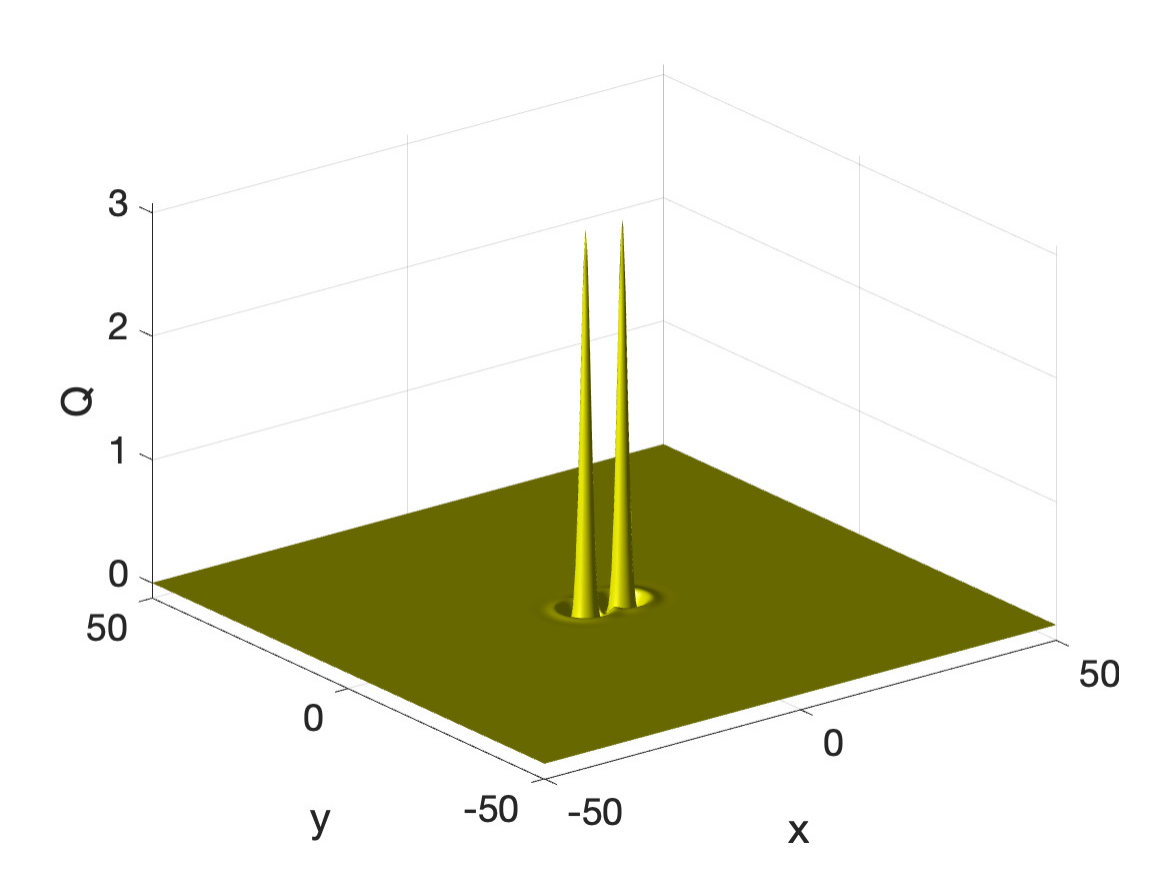}
\subcaption[]{{\footnotesize $b=10$}}
\end{subfigure}
\begin{subfigure}{.32\textwidth}
\includegraphics[width=1\linewidth,height=0.72\linewidth]{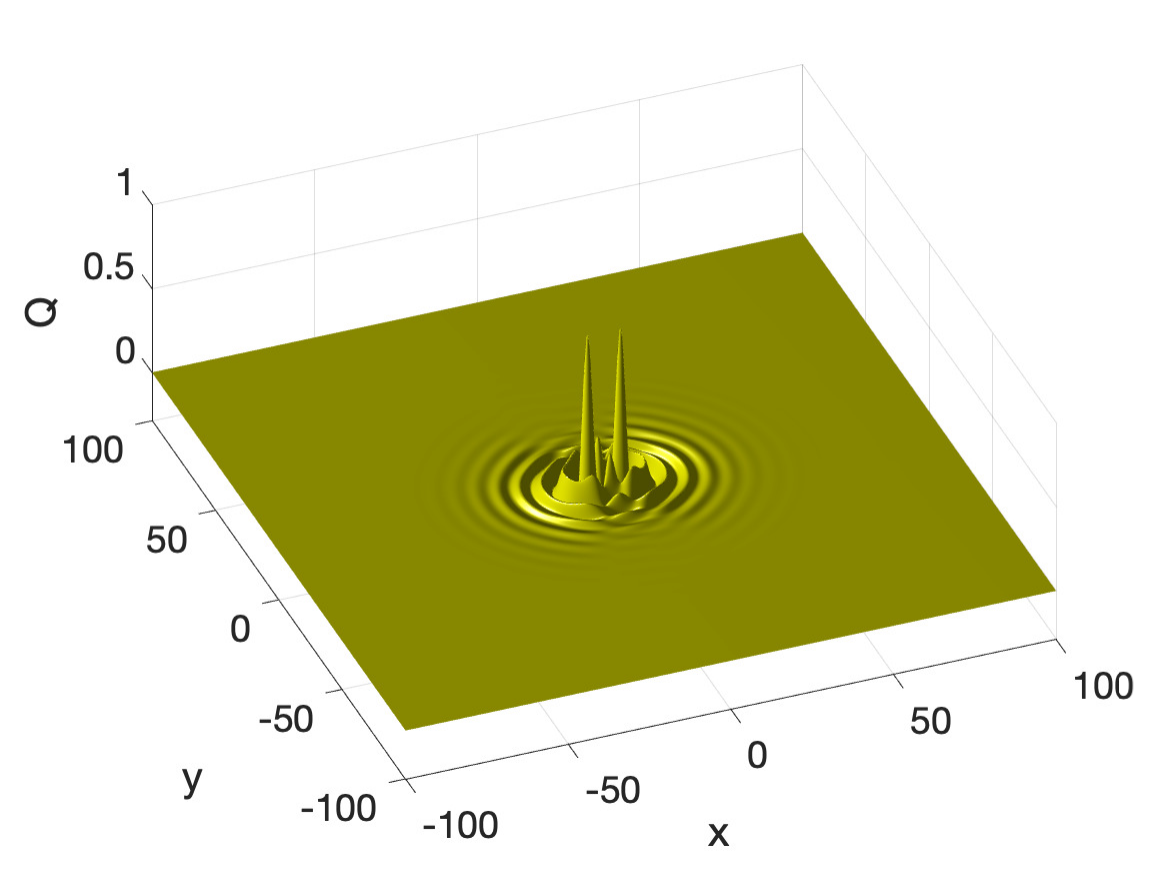}
\subcaption[]{{\footnotesize $b=1.1$}}
\end{subfigure}
\begin{subfigure}{.32\textwidth}
\includegraphics[width=1\linewidth,height=0.72\linewidth]{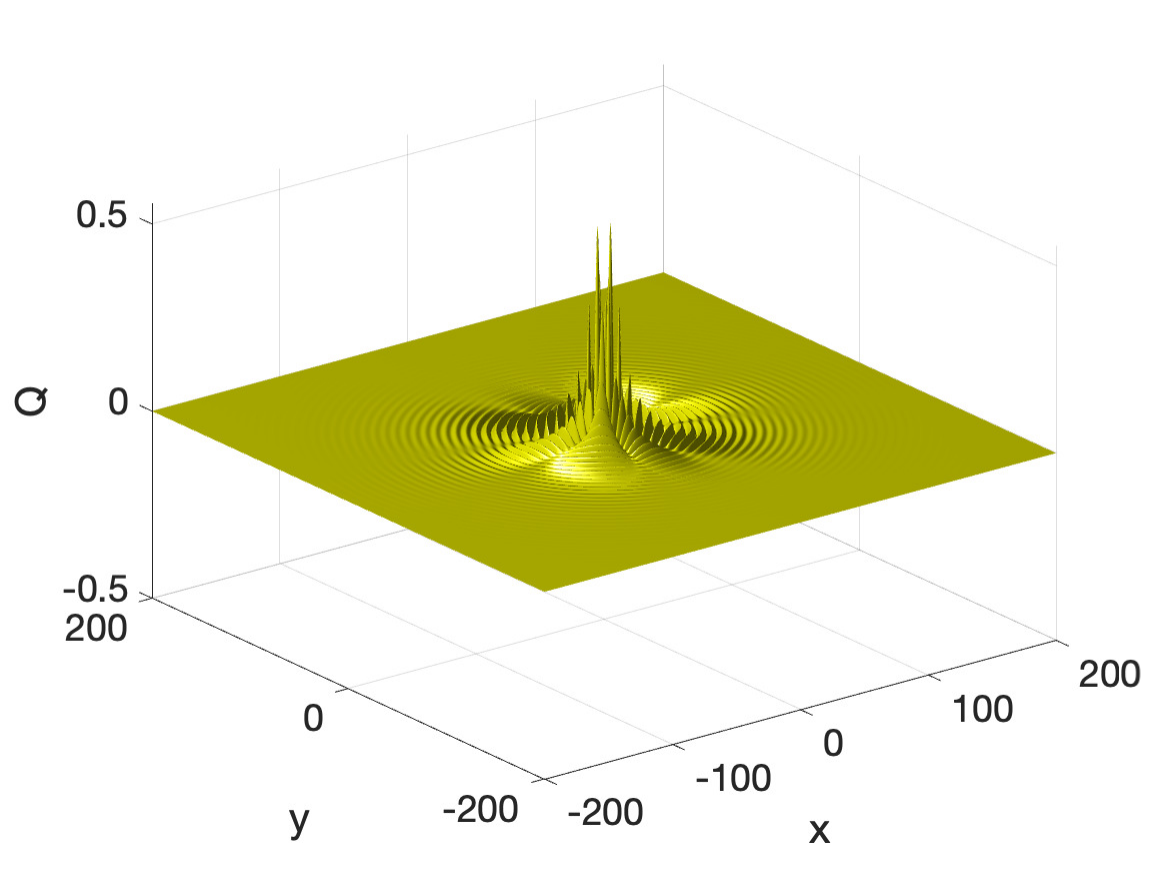}
\subcaption[]{{\footnotesize $b=1.007$.}}
\end{subfigure}\\
\begin{subfigure}{.32\textwidth}
\includegraphics[width=1\linewidth,height=0.72\linewidth]{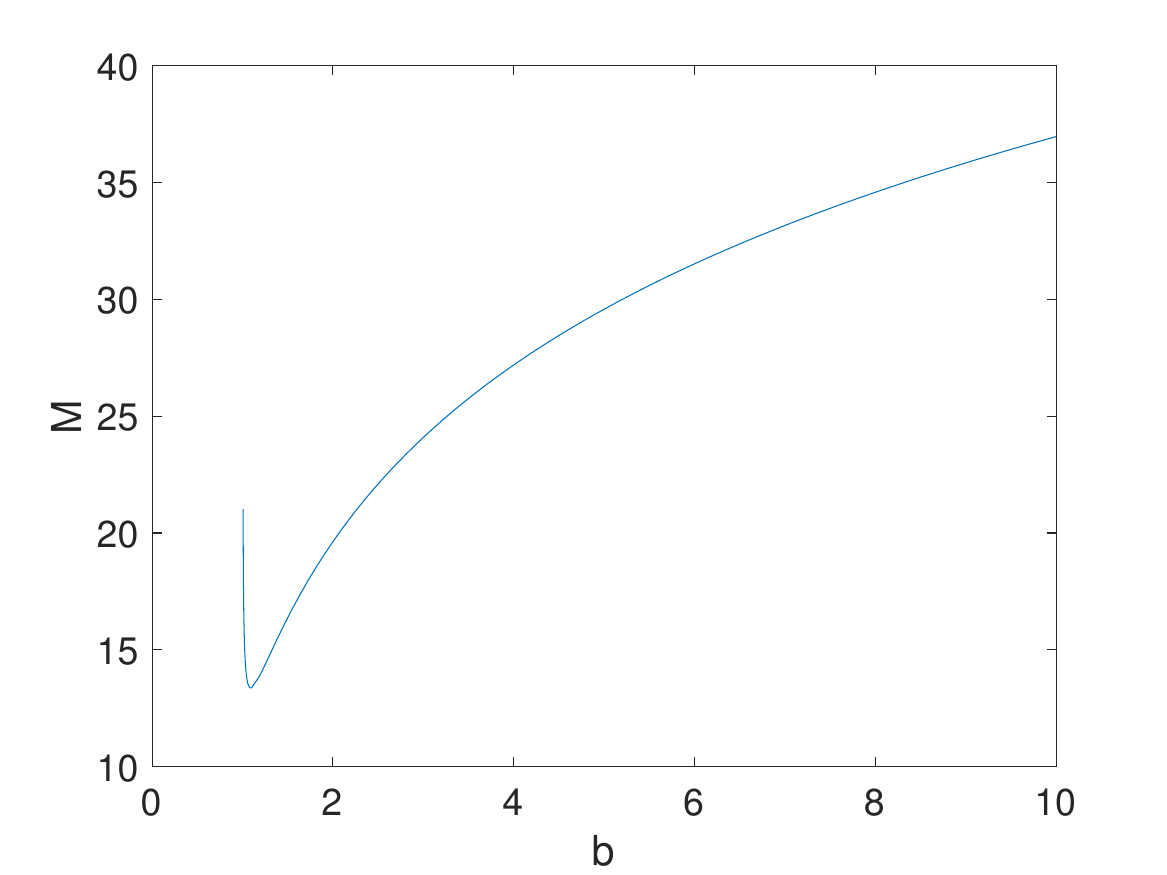}
\subcaption[]{{\footnotesize $M(Q_\ep) = M(b)$.}}
\end{subfigure}
\begin{subfigure}{.32\textwidth}
\includegraphics[width=1\linewidth,height=0.72\linewidth]{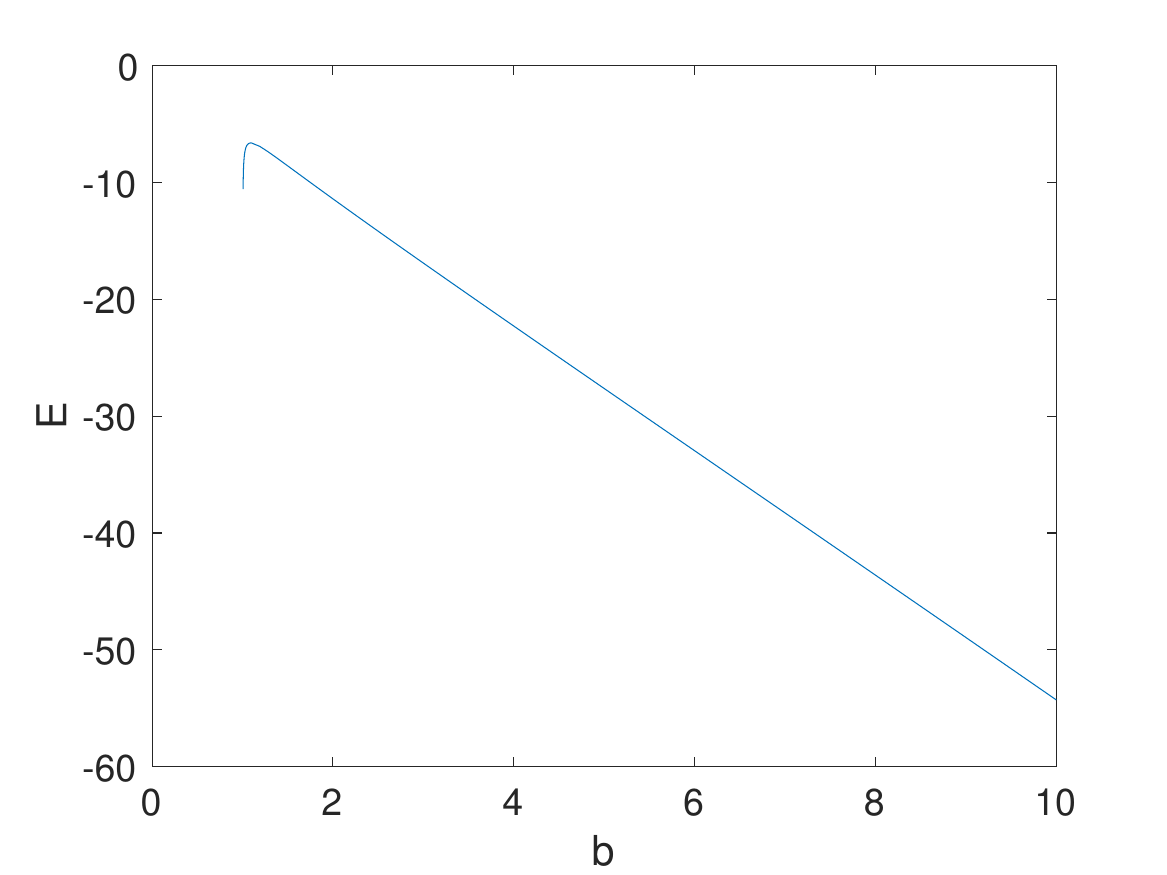}
\subcaption[]{{\footnotesize $E(Q_\ep)$ as function of $b$.}}
\end{subfigure}
\begin{subfigure}{.32\textwidth}
\includegraphics[width=1\linewidth,height=0.72\linewidth]{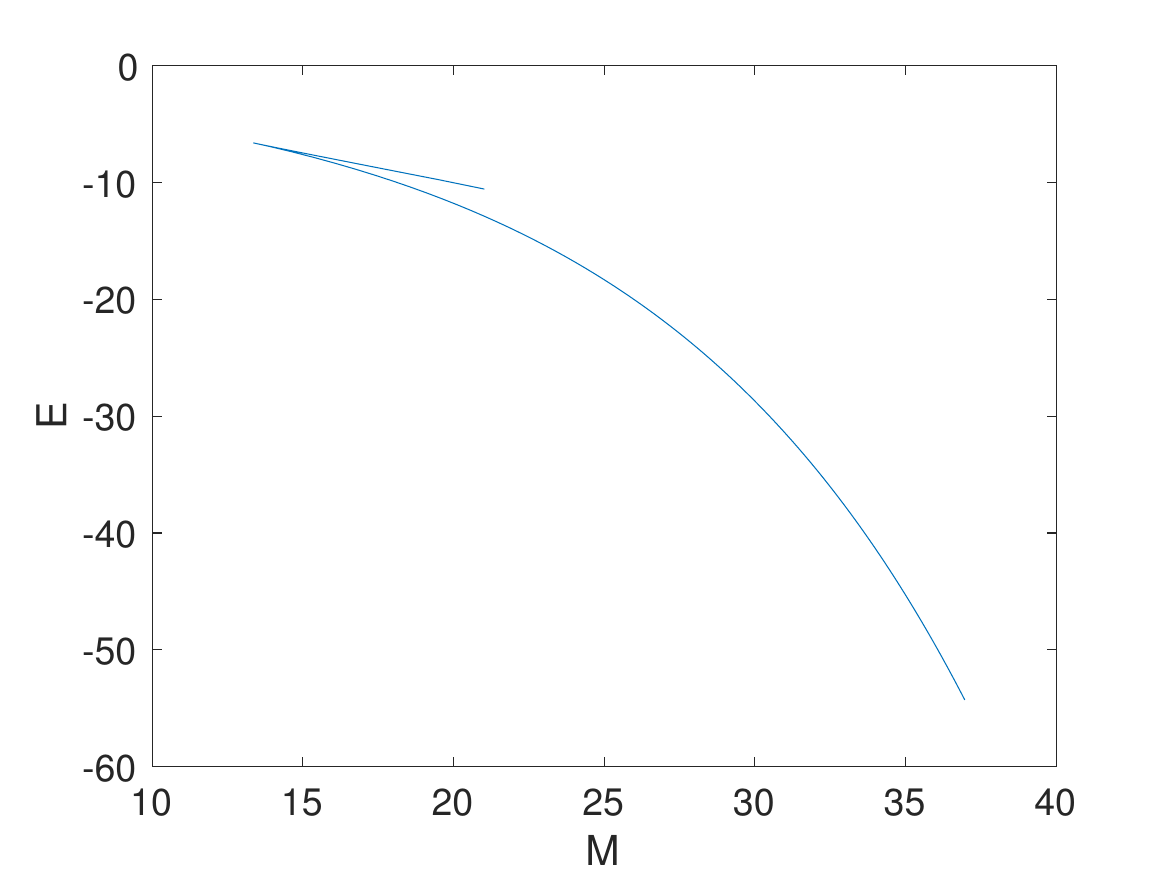}
\subcaption[]{{\footnotesize $E = E(M)$.}}
\end{subfigure}
\caption{\footnotesize {Nonradial two-peak solutions to the equation \eqref{E:2dGS} with $\alpha=3$: 
profiles for different $b$ 
(top row). Dependence of mass $M(Q_\ep)$ and energy $E(Q_\ep)$ on $b \in [1.007,10]$ (bottom left, middle), and energy as a function of mass $E=E(M)$ (bottom right).}}
\label{F:alpha3-nonradial}
\end{figure}

\smallskip

The top row of Figure ~\ref{F:alpha3-nonradial} shows the two-peak 
nonradial solutions for $b=10, 1.1$ and $1.007$. One can observe that 
the smaller $\ep$ becomes, the more features of radial symmetry it acquires (the profile develops more pronounced annular oscillations).

\subsection{Nonradial solutions}\label{S:alpha3-nonradial}
To construct possible nonradial solutions, we need to have a good 
initial guess. For this we either take a nonradial solution from the 
$\alpha=2$ case for sufficiently small $\ep$ (for instance, $b=1.1$ 
or $1.01$) and use it (after multiplication with a factor smaller 
than 1) as an initial iterate, or we consider a Knapp-type example. 

\subsubsection{Two-peak nonradial solutions}\label{S:2-peak}
We take $b=1.01$ and use the nonradial solution (which is an excited 
state, but provides a sufficiently close to the searched branch 
initial start) from the cubic $\alpha=2$ case as an initial iterate. 
The standard tracing technique towards smaller values of $b$ applied 
before does not really work in this example. Therefore we have to 
compute the solution for different values of $b$ from the 
corresponding solution for $\alpha=2$ as an initial iterate (after 
multiplication with a factor that essentially has to be chosen by 
hand). 

The bottom row of Figure ~\ref{F:alpha3-nonradial} shows the behavior of mass and energy as they depend on $b$ and also energy as a function of mass. Observe that monotonicity is gone (unlike the cubic cases in Section \ref{S:2d-cubic}). In fact, the mass as $\ep \to 0^+$ seems to grow unboundedly, which is the case for the quartic nonlinearity (see \eqref{E:rates-2D-nr} and Section \ref{S:rate-consequences}). This non-monotonicity creates  minimum or maximum points in the graphs of mass or energy, respectively, and a turning point in the $E(M)$ graph, thus, creating two branches of ground states. We term them as a stable (lower energy) and an unstable (higher energy) branches, consistent with our 1D findings in \cite{KPRS}. 

\smallskip

In order to determine if these nonradial solutions of the quartic equation \eqref{E:2dGS} are indeed the lowest-energy ground states, we compare them with radial solutions. 
\begin{figure}[!htb]
\includegraphics[width=0.49\hsize,height=.4\linewidth]{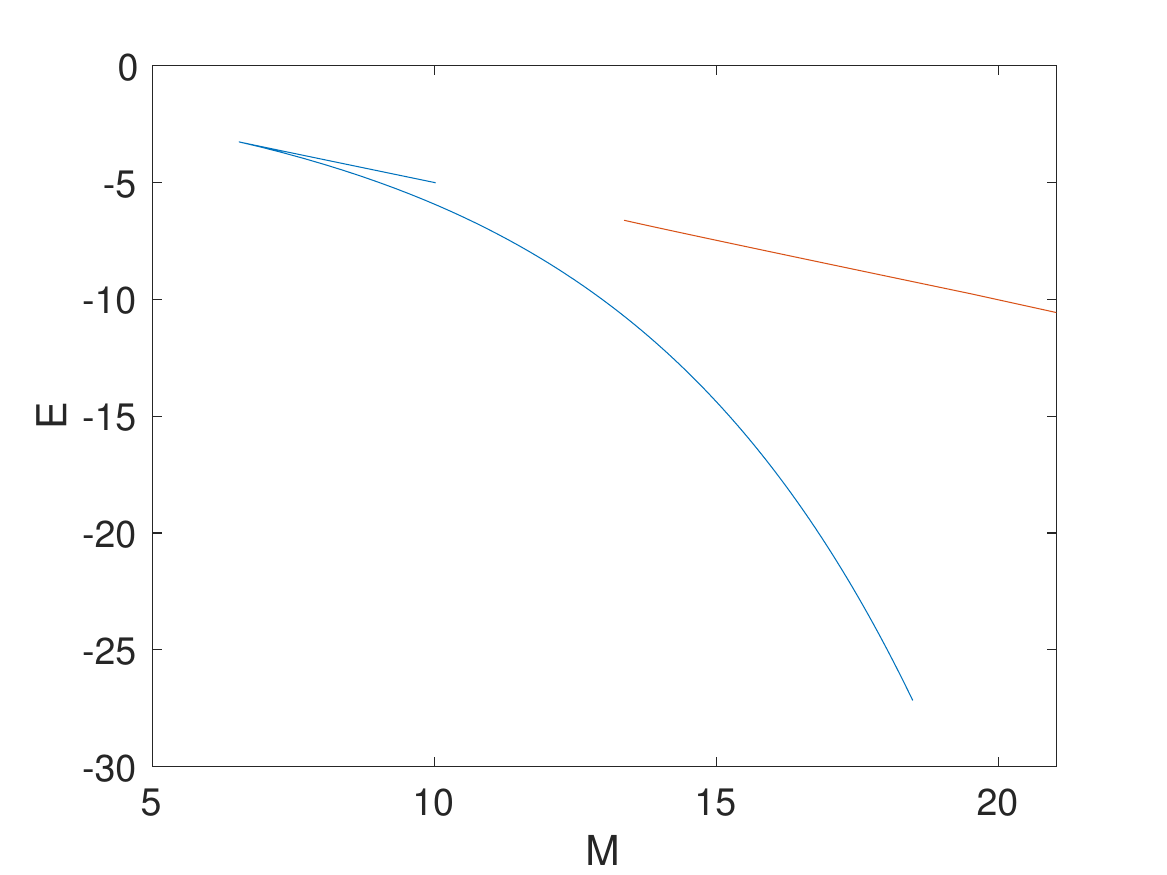}
\includegraphics[width=0.50\hsize,height=.4\linewidth]{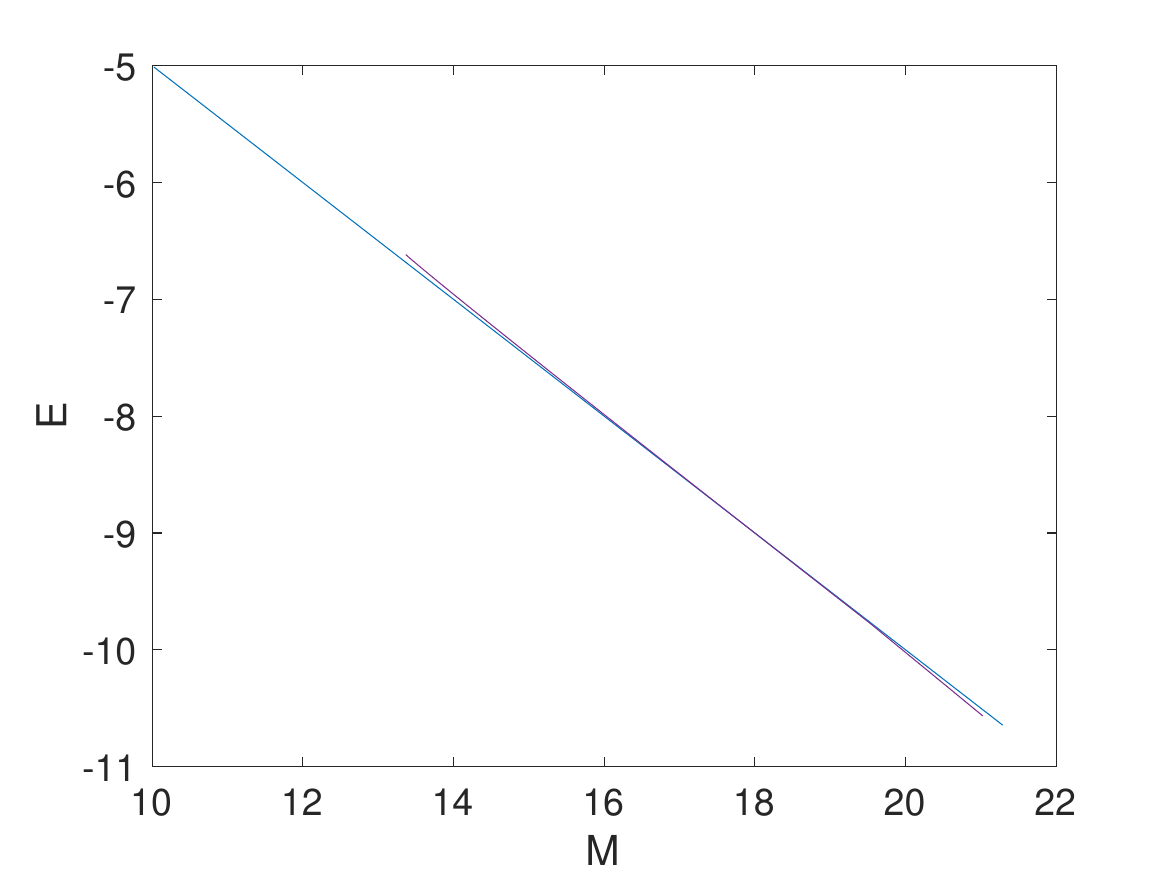}
\caption{\footnotesize Dependence of energy on mass: comparison of radial vs. two-peak nonradial solutions to \eqref{E:2dGS}, $\alpha=3$. Left: radial (blue), part of curve in (C) in Figure \ref{F:4a},  
and two-peak nonradial (red), part of curve in (F) in Figure \ref{F:alpha3-nonradial}. 
Right: Zoom-in with continuation of (unstable branch) radial 
solutions (blue) $b \in [1.00001,1.001]$ vs. nonradial (red) for $b 
\in [1.007,2]$. Note the intersection of lines:  as $\ep \to 0^+$ the 
red curve (nonradial solutions) has eventually lower energy than the 
radial solutions. Comparing with the {\it stable} branch on the left (part of the blue curve below the turning point), there are radial solutions with the same mass and lower energy (but larger $b$).} 
\label{F:alpha3-ME}
\end{figure}

On the left of Figure ~\ref{F:alpha3-ME} we plot the energy vs. mass dependence of the radial solution that we described in Section \ref{S:quartic-radial}, namely, plot (C) of Figure \ref{F:4a}, and a part of energy vs. mass dependence (relevant for the considered values of mass) of the two-peak solution from the graph in (F) of Figure \ref{F:alpha3-nonradial}. Observe the following:
\begin{itemize}
\item
The energy of the stable branch of radial solutions (blue curve in the left plot) are lower for values of mass (from around 13 to 21) than the energy of nonradial solutions (red line) with corresponding mass values.
We note that the $b$-values are significantly different on the 
displayed blue and red curves: the blue curve has $\ep (\equiv b-1)$ 
of order 1 (more precisely, it approaches 10 as the energy decreases 
down) and the red line has $\ep$ on the order of $10^{-1} \div 10^{-3}$. Nevertheless, for a fixed mass, the radial solutions from the stable branch have lower energy.  

\item
To continue the unstable branch of radial solutions (top part of the blue curve in the left plot), radial solutions had to be computed for much smaller values of $\epsilon$ to see the disposition of the blue and red curves, so that they exist close to each other. 
Thus, to compare, we compute the radial solutions down to $\ep= 
10^{-5}$ and plot them on the right of Figure \ref{F:alpha3-ME}: the 
blue line representing the unstable branch of radial solutions and 
the red line representing the nonradial two-peak solutions on its lower branch (since that curve also has lower and upper branches as can be seen in the right bottom plot (F) of Figure \ref{F:alpha3-nonradial}).  

\item 
Observe that the lines intersect in the right plot of Figure~\ref{F:alpha3-ME} with the red line becoming lower than the blue line (around the mass $M \sim 18$ and higher), implying that when $\ep \to 0^+$ the nonradial solutions have lower energy than the radial solutions on the unstable branch. This is in correspondence with the statement of Theorem \ref{Thm1} (or \cite[Theorem 1.2]{LW2021}) in the regime when $\ep \to 0^+$. 
On the other hand, 
in a global view, 
we also have radial solutions of the same mass (but different, much higher, values of $b$) with 
the lower energy, on the {\it stable} branch of radial solutions, the lower part of blue line on the left plot of Figure \ref{F:alpha3-ME}, which is decreasing down.   
Thus, this is quite an intricate situation, indicating that the definition of a {\it ground state} has to be considered carefully with the underlying context in mind. We discuss it in the next point. 

\item
For the considered quartic case $\alpha=3$, one has to distinguish (at least) two comparisons: by {\it action} and by {\it mass normalization}. 
A least-action comparison is at a fixed $b=1+\epsilon$ and the relevant quantity is the action
$$
S_b(Q)=E(Q)+\frac{b}2 M(Q).
$$
On the other hand, the normalized problem fixes the mass $M(Q)$ and compares the energy
$E(Q)$. Therefore, when two states have the same mass but correspond to
different values of $b$, the appropriate comparison is by energy, instead of by their respective actions. 

In our computations, shown in Figure \ref{F:alpha3-ME}, the two-peak nonradial
branch lies below the small-$\ep$ radial unstable branch as can be seen on the right plot for $\ep$ approaching $0^+$, thus, it is appropriate to compare actions $S_b(Q^{r})$ and $S_b(Q^{nr})$.

We note that since the mass appears to be in a similar range (near the intersection of curves) and $b = 1+ \ep \approx 1$, the energy values indicate that the least-action ground state may be nonradial (consistent with the result in \cite{LW2021}).
However, the stable radial branch, which occurs at larger values of $b$,
has lower energy at the same mass. 
\end{itemize}

Thus, our numerical evidence supports two distinct conclusions: the
{\it least-action} ground states for sufficiently small $\epsilon$ could be {\it nonradial} (supporting the results of Lenzmann-Weth \cite{LW2021}), 
whereas the normalized comparison selects the lower radial branch at the masses for which it exists, and hence, a {\it normalized} ground state at the corresponding mass (if exists) is {\it radial}.

\subsubsection{Four-peak nonradial solutions}  
The Knapp-type profile \eqref{E:Knapp-1}, which adjusted to $\alpha=3$, has the form 
\begin{equation}\label{E:Knapp-3}
Q_\epsilon(x,y)
\approx \epsilon^{1/3} W(\epsilon^{1/2} x,\epsilon^{1/4} y )\cos x,
\end{equation}
where the localized profile function can be taken to be Gaussian as before, and thus, a possible initial guess  
\begin{equation}\label{E:Knapp-3gauss}
Q^{(0)}_\epsilon (x,y) =\lambda \, \epsilon^{1/3} 
e^{ -\frac12 \left( \frac{(\epsilon^{1/2} x)^{2}}{\sigma_x^{2}} 
+\frac{(\epsilon^{1/4}y)^{2}}{\sigma_y^{2}}
\right)} 
\cos x.  
\end{equation}

For these simulations we take $L=100$ and the parameters $\sigma_x =0.5 $, $\sigma_y= 1$, $\lambda = 1$.
\begin{figure}[!htb]
\begin{subfigure}{.32\textwidth}
\includegraphics[width=1\linewidth]{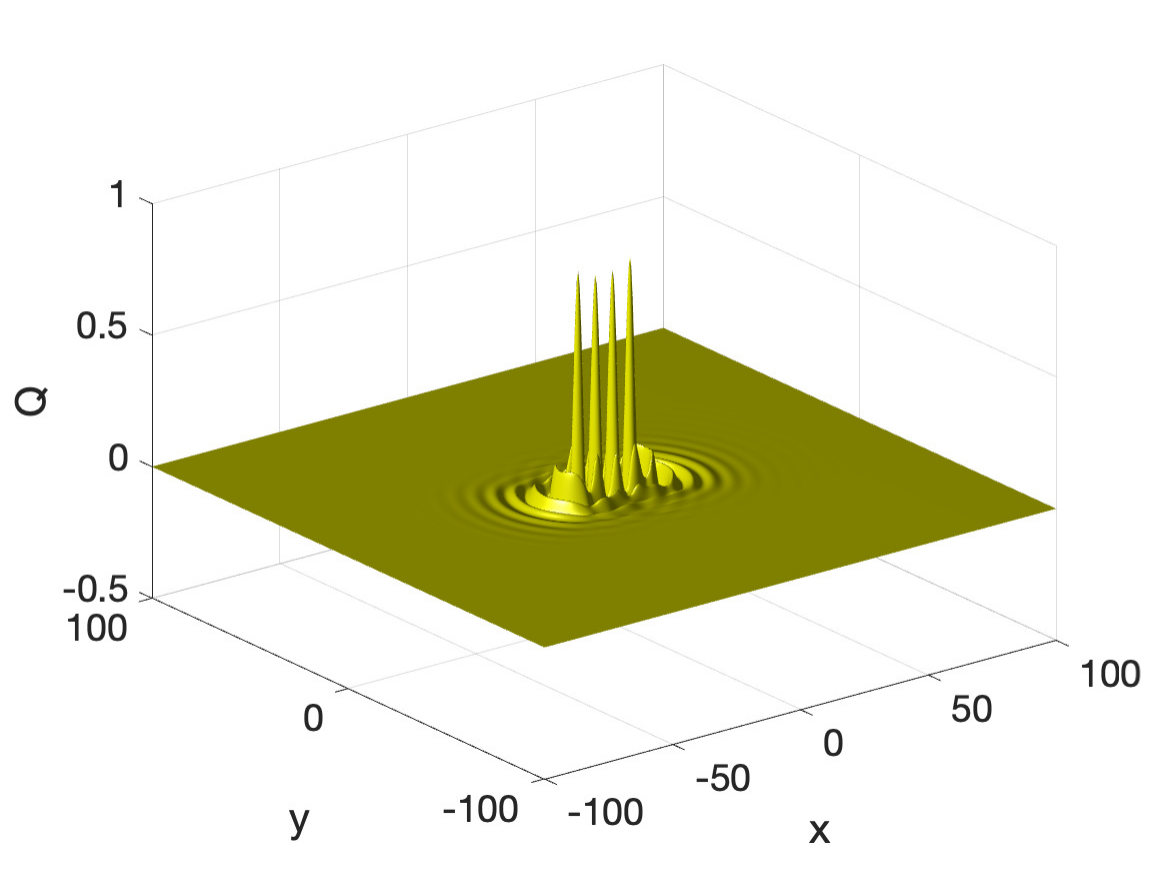}
\subcaption[]{{\footnotesize $b=1.1$.}}
\end{subfigure}
\begin{subfigure}{.32\textwidth}
\includegraphics[width=1\linewidth]{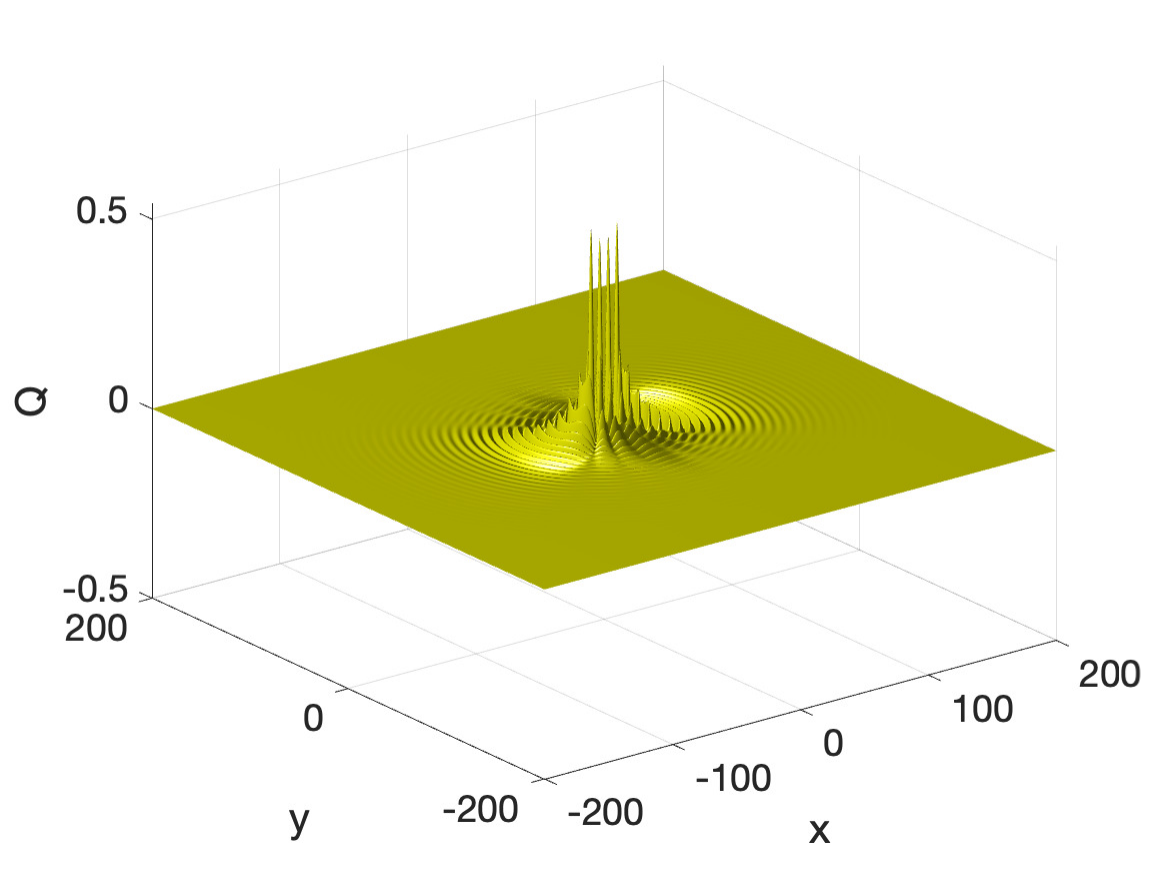}
\subcaption[]{{\footnotesize $b=1.01$.}}
\end{subfigure}
\begin{subfigure}{.32\textwidth}
\includegraphics[width=1\linewidth]{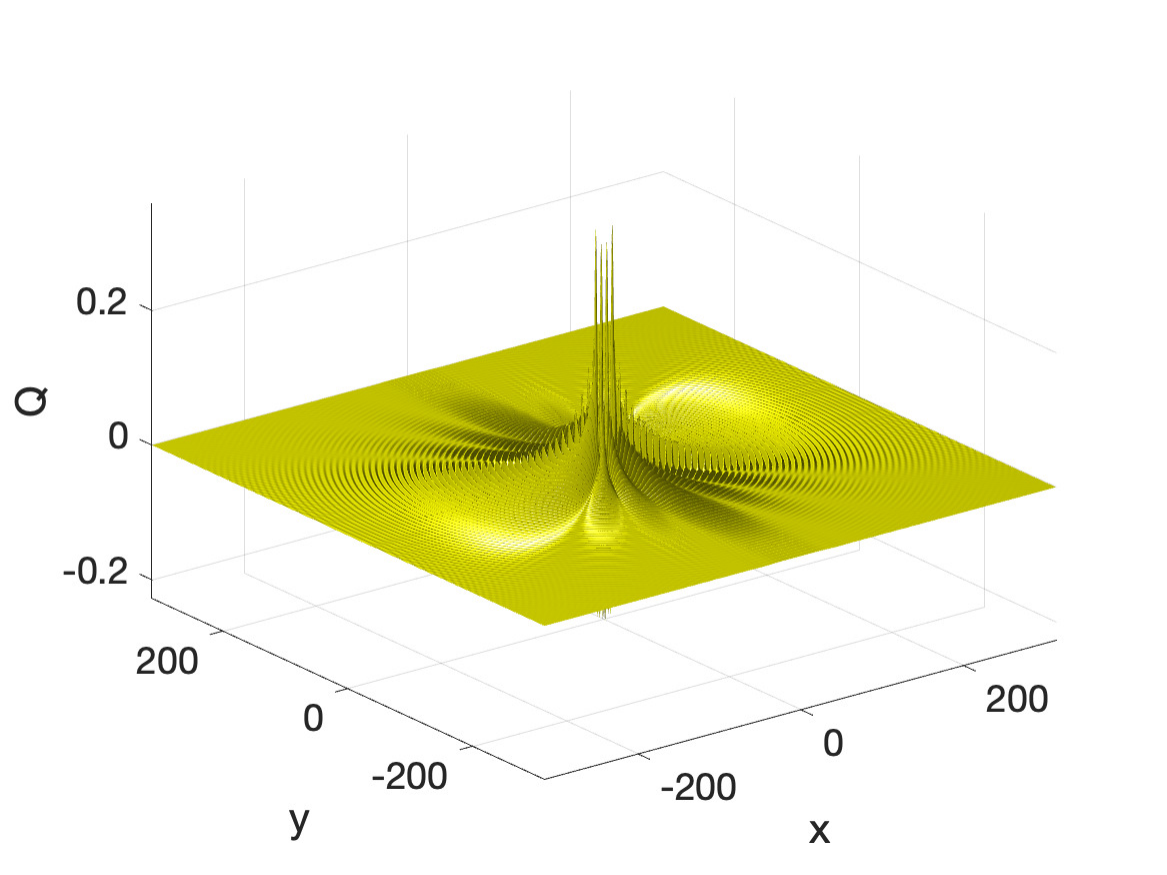}
\subcaption[]{{\footnotesize $b=1.001$.}}
\end{subfigure}\\
\begin{subfigure}{.32\textwidth}
\includegraphics[width=1\linewidth,height=0.65\linewidth]{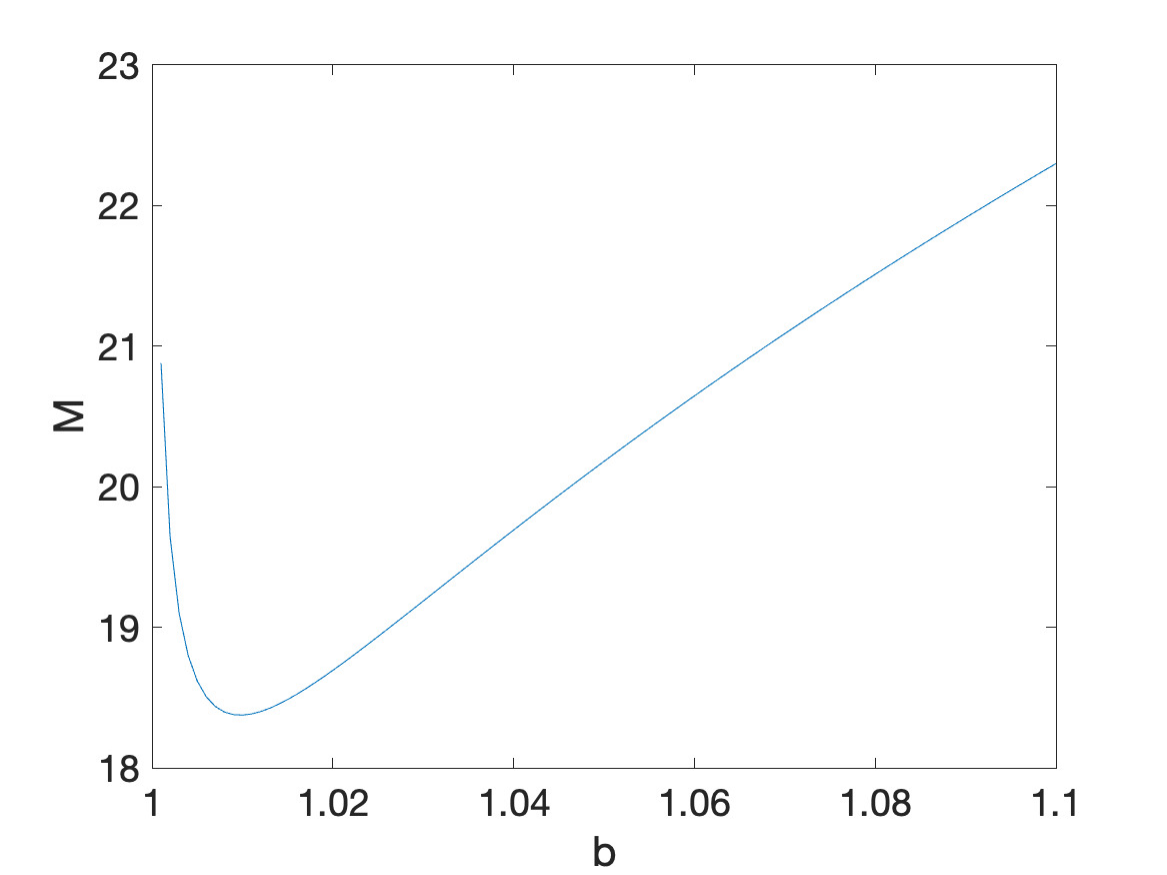}
\subcaption[]{{\footnotesize $M(Q_\ep) = M(b)$.}}
\end{subfigure}
\begin{subfigure}{.32\textwidth}
\includegraphics[width=1\linewidth,height=0.65\linewidth]{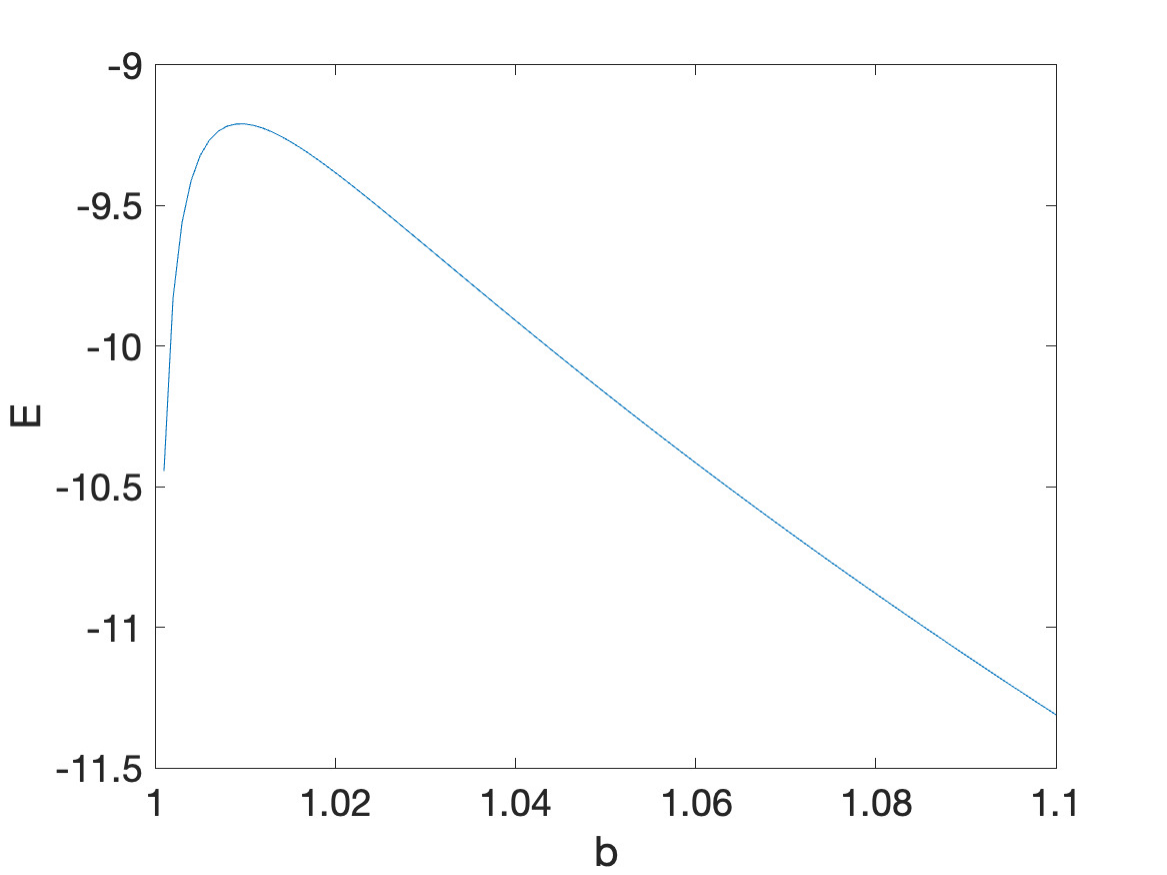}
\subcaption[]{{\footnotesize $E(Q_\ep)$ as function of $b$.}}
\end{subfigure}
\begin{subfigure}{.32\textwidth}
\includegraphics[width=1\linewidth,height=0.65\linewidth]{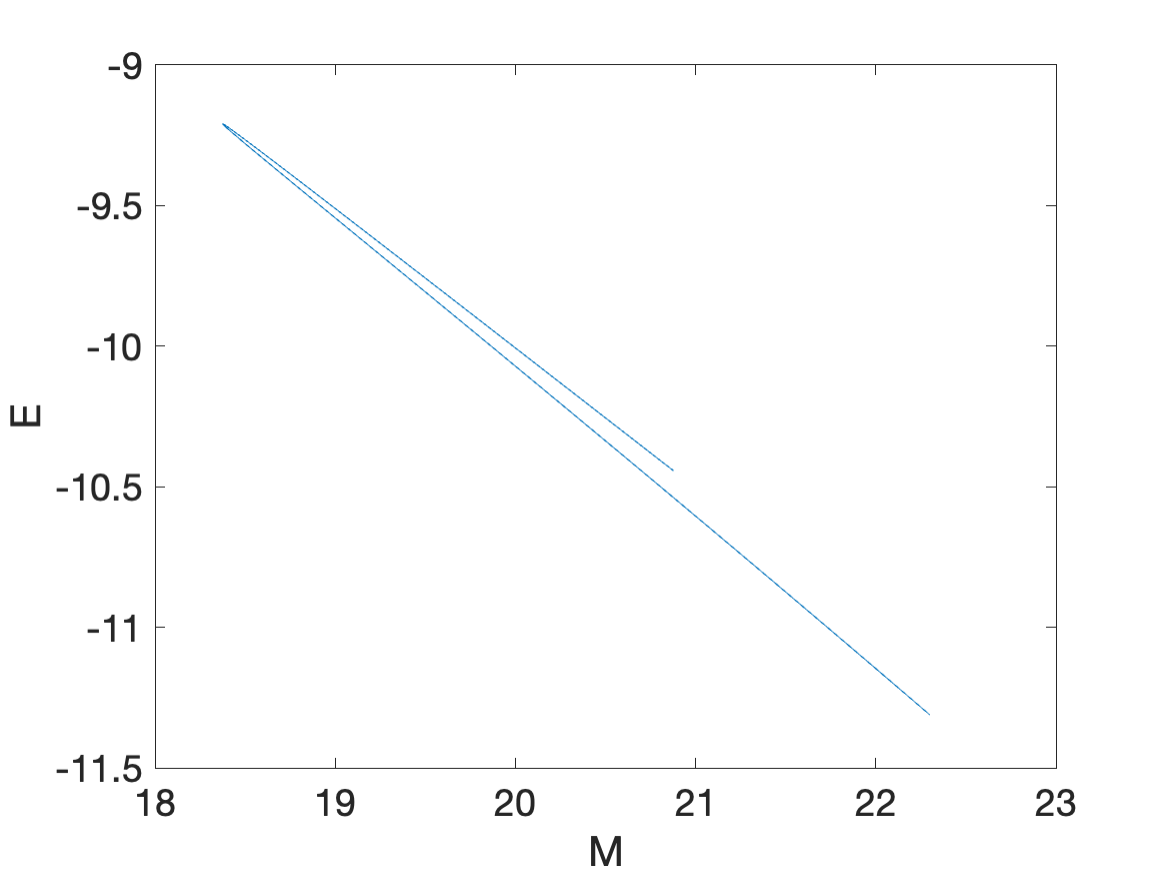}
\subcaption[]{{\footnotesize $E = E(M)$.}}
\end{subfigure}
\caption{\footnotesize {Nonradial four-peak (excited state) solutions to \eqref{E:2dGS} with $\alpha=3$:  
profiles for different $b$ (top row). Dependence of mass $M(Q_\ep)$ and energy $E(Q_\ep)$ on $b \in [1.001,1.1]$ (bottom left, middle), and energy as a function of mass $E=E(M)$ (bottom right).}}
\label{F:alpha3-4peak}
\end{figure}

\smallskip

In Figure \ref{F:alpha3-4peak}
we show computed, from this initial guess, nonradial solutions  
for $b=1.1$, $b=1.01$, and $b=1.001$ (we computed this branch down to $b=1.0001$). 
From this initial guess, with  relaxation methods in our numerical scheme, we obtain the nonradial 4-peak solutions (compare this to the two-peak ones in the previous section).  
We have compared the energy vs. mass curves for the 2-peak and 4-peak nonradial solutions and found that the 4-peak energy is higher than that of the 2-peak ones, 
which indicates that these solutions are candidates for the nonradial excited states in the quartic equation.

\section{Higher nonlinearities}\label{S:higher}

For nonlinearities with $\alpha \geq 4$, Theorem \ref{Thm1} (or \cite[Theorem 1.2]{LW2021}) is inconclusive (meaning that there is no longer separation of radial and nonradial branches via asymptotic behavior of the quotient) and one may expect that the ground states are radially symmetric. 
More precisely, by Corollary \ref{C:sharp-mass} the radial and unrestricted masses have the same rate $\ep^{1/\alpha-1/2}$ in this range, and their leading constants coincide if and only if constant functions optimize the Stein - Tomas inequality on $\mathbb S^1$, see Remark~\ref{R:constants-radiality}. For $\alpha=4$ this is conjectured, and a strict inequality would force ground states to be nonradial in this power range.

We tried angular-mode initial iterates of the form \eqref{init} and 
variations thereof, as well as Knapp-type cap concentration initial 
iterates of the form 
\eqref{E:Knapp-1}.  In the cases $\alpha=4, 5, 6$ considered here, the 
iterations (if they converged at all) either converged to radial profiles or did not converge to a nontrivial nonradial branch.
Since we use an iterative approach, this does not mean that nonradial solutions cannot  exist and be numerically constructed with more general initial 
iterates  (perhaps as excited states). Our results are compatible with the absence of such solutions for higher nonlinearities. In the view of Remark \ref{R:constants-radiality}, they are compatible with equality of the leading mass constants in Corollary \ref{C:sharp-mass}.

The computed ground state solutions, which are radially symmetric, for the quintic, sextic and septic nonlinearities ($\alpha=4,5,6$) for $b=1.01$ are shown in Figure~\ref{figm0Q101}.  
\begin{figure}[!htb]
\includegraphics[width=0.32\hsize]{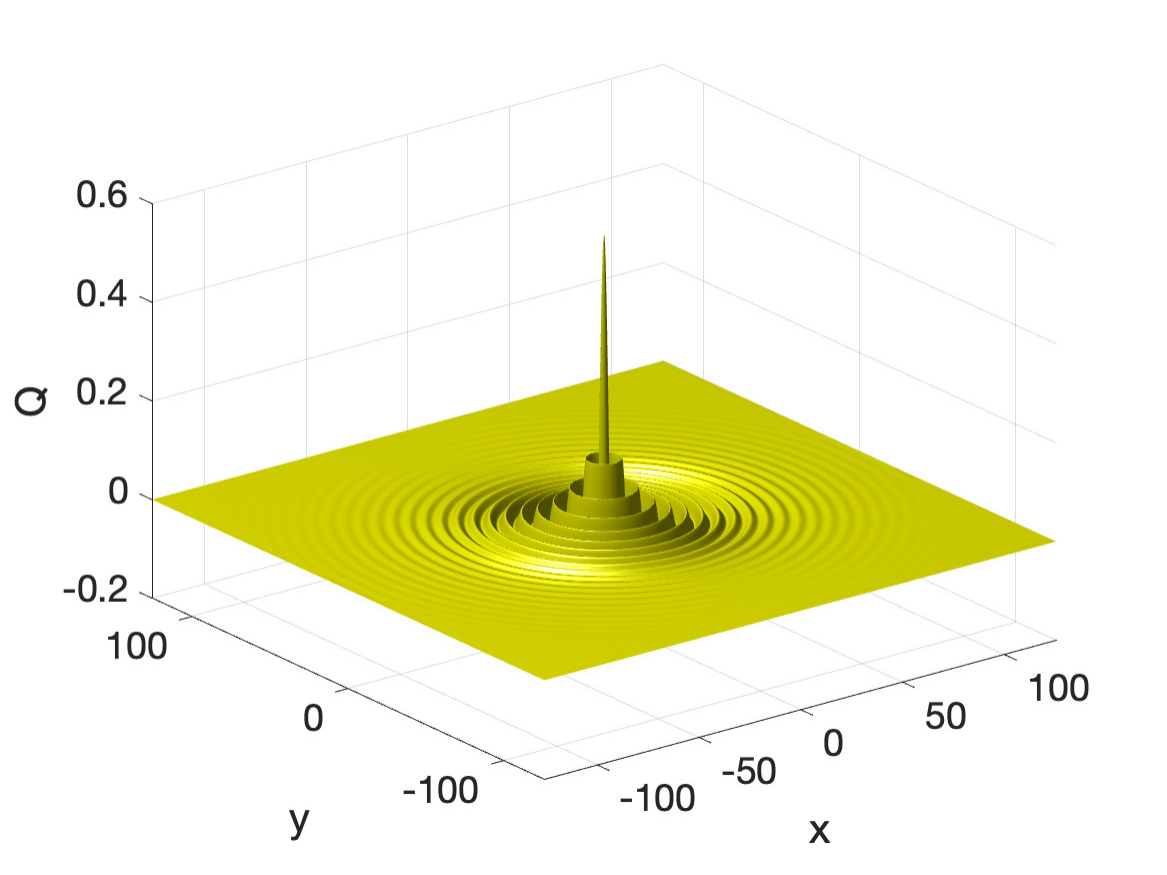}
\includegraphics[width=0.32\hsize]{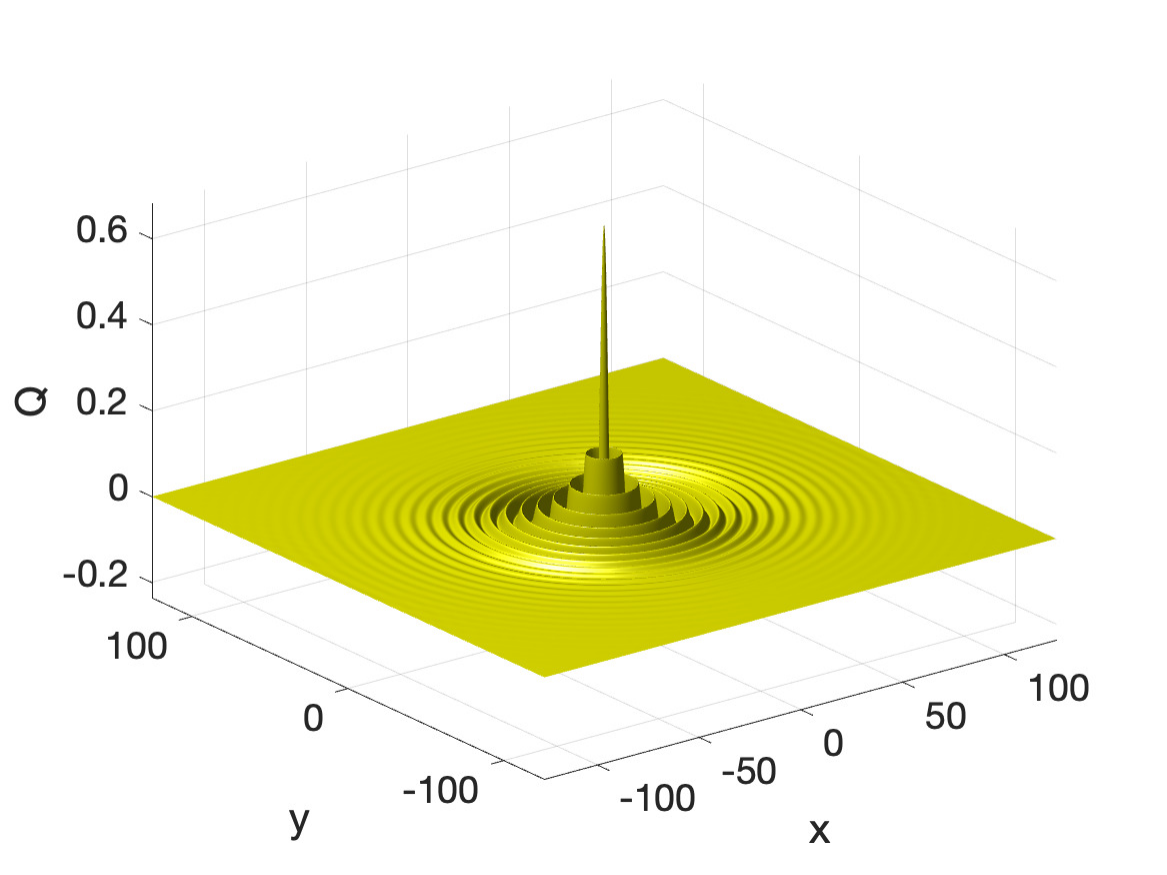}
\includegraphics[width=0.32\hsize]{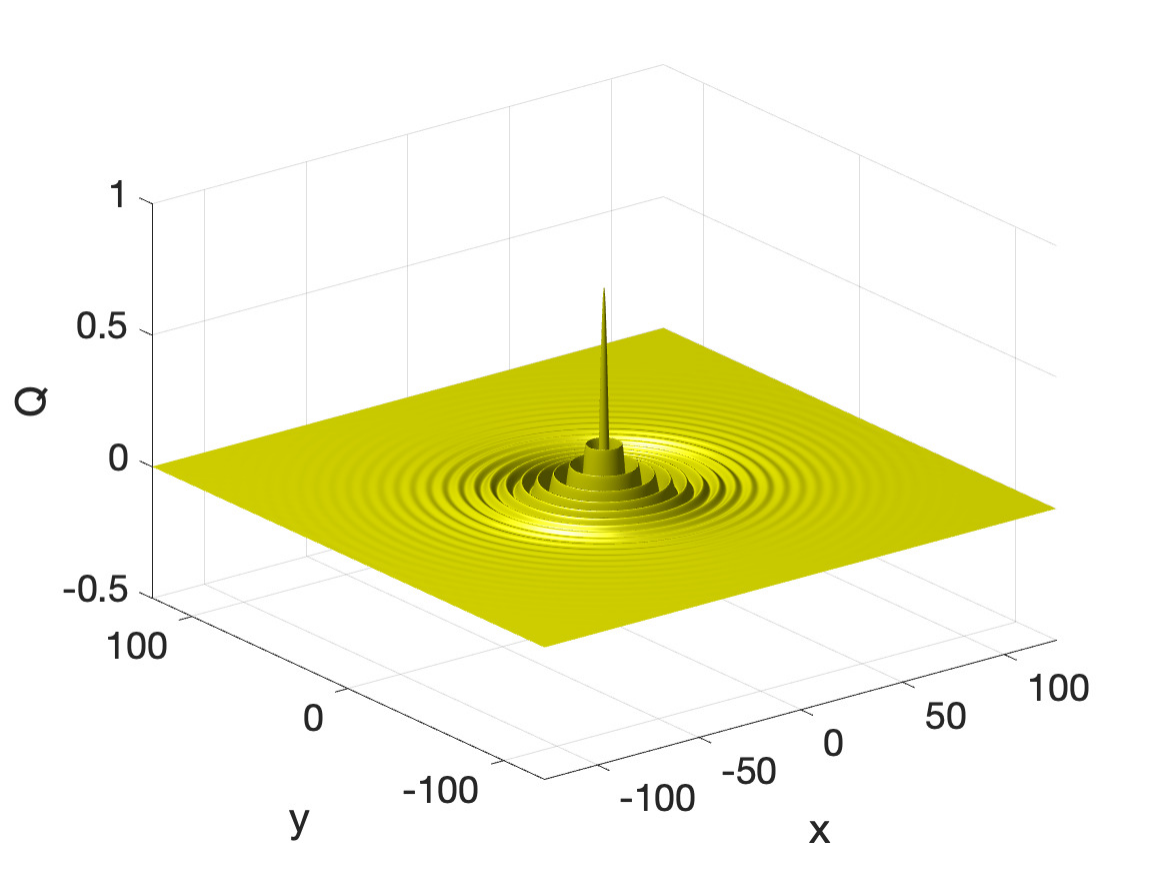}
\caption{Solution to the equation \eqref{E:2dGS} with radial 
symmetry for $b=1.01$ and $\alpha=4,5,6$.}
\label{figm0Q101}
\end{figure}

\begin{figure}[!htb]
\includegraphics[width=0.32\hsize]{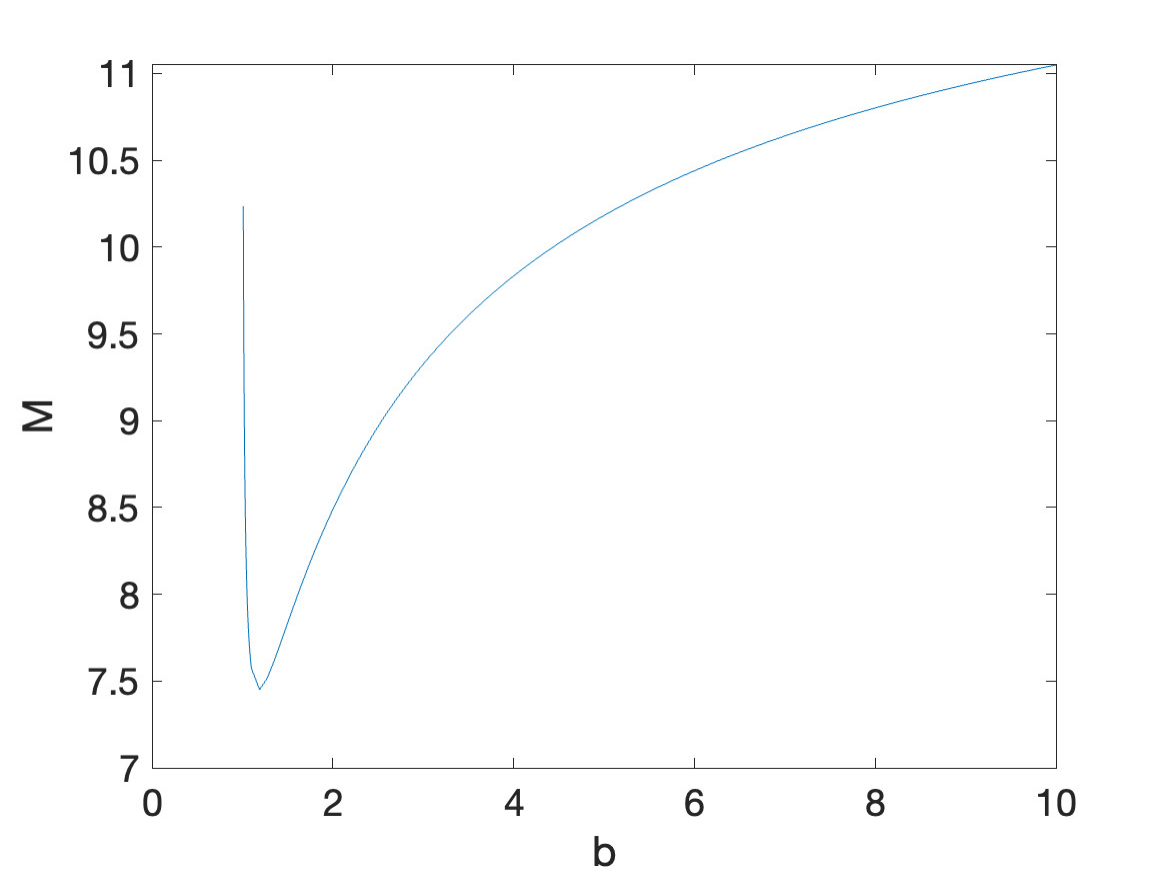}
\includegraphics[width=0.32\hsize]{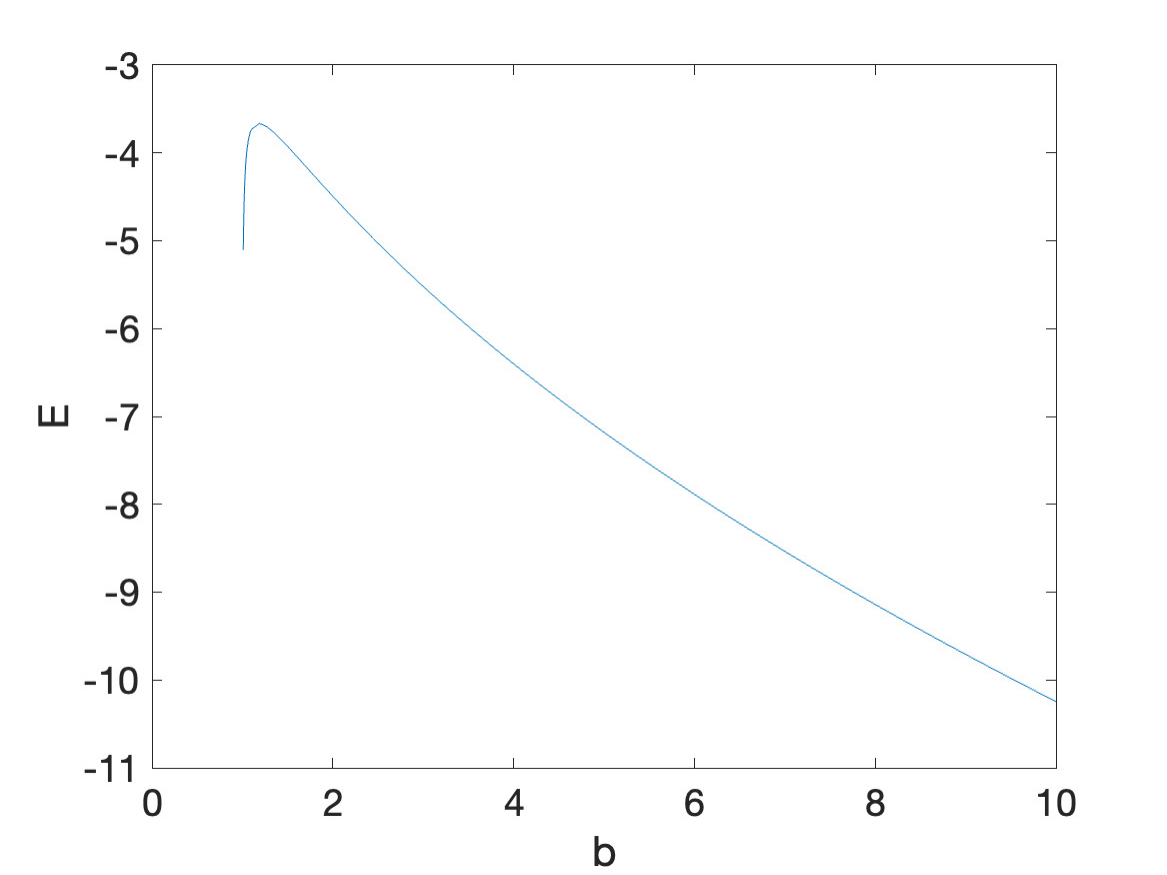}
\includegraphics[width=0.32\hsize]{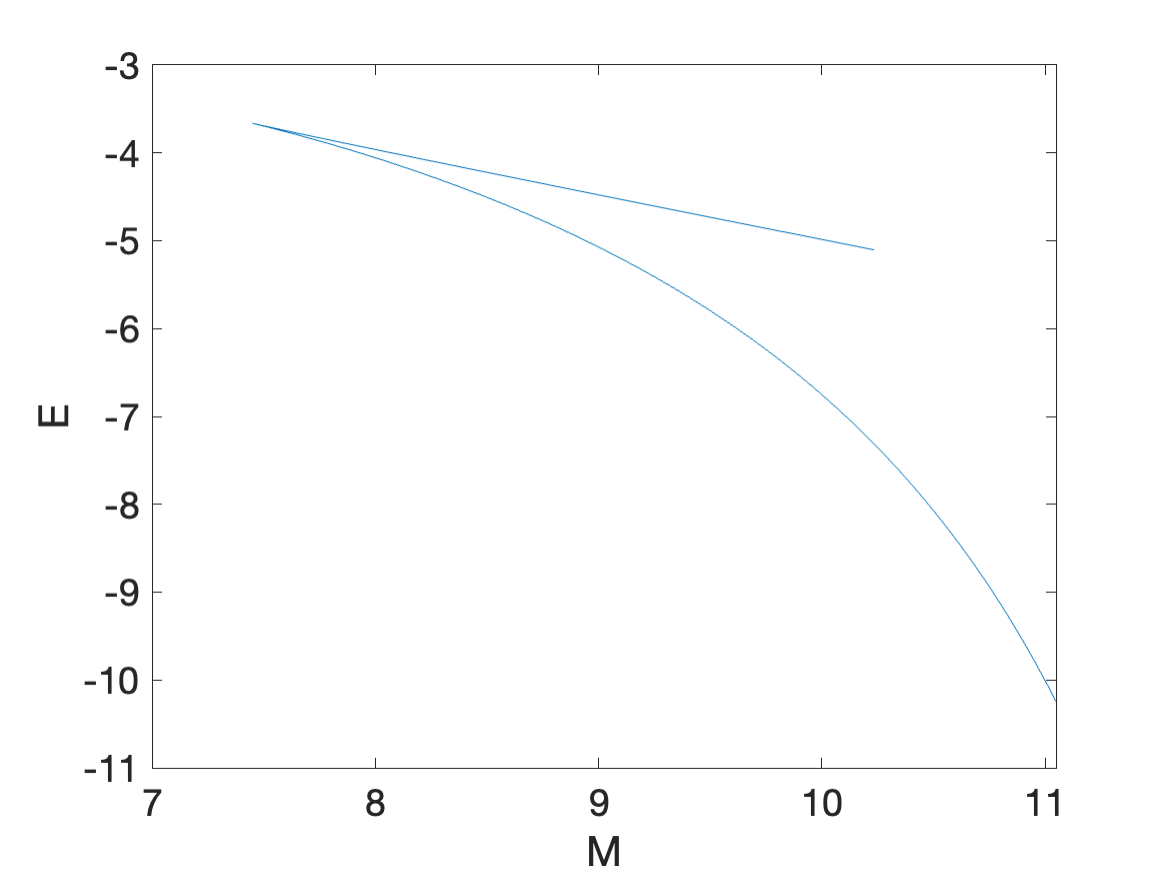}
\caption{\small Radial solutions to the quintic equation \eqref{E:2dGS}, $\alpha=4$, in dependence of $b$: mass $M(Q_\ep)$, energy $E(Q_\ep)$, and $E=E(M)$.} 
\label{F:5a}
\end{figure}

The same behavior for mass and energy as in the radially symmetric solutions in the quartic case can be observed in the higher nonlinearities, Figure~\ref{F:5a} shows the quintic case $\alpha=4$, where on the right plot one can see that the branching occurs. 
A similar situation occurs in other computed nonlinearities.

\section{Branching and normalized ground states}
\label{S:branching}

We next relate the branching (or the turning points in a sense) observed in the numerical mass-energy $E(M)$ diagrams to the normalized ground-state problem from Definition \ref{D:norm-GS}. 

As we pointed out in Section \ref{S:2-peak}, two comparisons must be kept separate.  The {\it least-action} problem fixes $b$ and compares $S_b$ among critical points at that same value of $b$.  The {\it normalized} problem fixes the mass and compares the Hamiltonian energy $E$, the corresponding value of $b$ is then the Lagrange multiplier.  Hence, when several solutions have the same mass, the normalized minimizer would be selected from the lower branch in the mass-energy $E=E(M)$ diagram.

This distinction becomes essential once the mass curve changes slope, or equivalently, develops a minimum or a turning point. If $b \mapsto M(Q_b)$ is strictly monotone (for $b>1$), then the unrestricted least-action branch of ground states can be parametrized by mass, and, provided normalized ground states exist at the corresponding mass, the least-action and normalized ground states correspond to the same family under this parametriation by \cite[Theorem 1.3]{FJMM2022}. If $M(Q_b)$ is not monotone anymore, for example, first decreases and then increases, two (or more) stationary states can have the same mass. The upper branch in the $E=E(M)$ diagram is then the one, which we call {\it unstable} branch, while the lower branch (as it has the lower energy) provides candidates for the normalized minimizers. This is exactly what we observed in the one-dimensional numerical study of the bi-harmonic NLS in \cite{KPRS}.

The one-dimensional connection between the mass threshold
$\alpha_{K}=4$ and normalized ground states was discussed in
Remark \ref{R:normGS}. We next consider the corresponding
two-dimensional question, using the mass asymptotic
\eqref{E:rates-2D-nr}.

\begin{remark}\label{R:Knapp-threshold}
Our 2D Knapp mass threshold
$\alpha_{\rm K}=\frac83$ 
is related to the normalized minimization problem studied in \cite{FJMM2022}. 
In our notation the result of their 
\cite[Theorem~1.2]{FJMM2022} 
shows that in 2D normalized minimizers exist for every prescribed mass 
when ~$0 < \alpha < \frac83$, whereas for
~$\frac83 \leq \alpha<4$~ a positive mass threshold appears. Thus, although the two statements study different properties - our $\alpha_K$ describes the small-$\ep$ mass behavior of a near-threshold least-action branch, while the second value from \cite[Theorem~1.2]{FJMM2022} concerns global minimization of the energy at fixed mass  - the same exponent $\frac83$ appears in both problems. This agreement is natural, since both results are related to the same Fourier cap-concentration, or Knapp, geometry.

The connection with the shape of the mass curve can be understood as
follows. For $\alpha<\frac83$, the mass of the near-threshold
nonradial branch tends to zero as $\ep \to 0^+$, whereas at
$\alpha = \frac83$ it remains of order one, and for
$\frac83 < \alpha < 4$ it diverges. In the latter range, if the mass
depends continuously on $b$ and also diverges as $b\to\infty$ (which we discuss below), this behavior implies the presence of a minimum along the branch. This
branch interpretation is consistent with 
the global fixed-mass variational result of
\cite{FJMM2022}.

If we consider the mass asymptotics restricted to the {\it radial} 
branch, then a different threshold appears. Our quotient-to-mass relation in \S \ref{S:M-to-R} combined with the radial quotient asymptotics of \cite{LW2021,MO2023} yields the rates \eqref{E:rates-2D-r}, and hence, the transition between vanishing and diverging radial mass occurs at $\alpha_{\rm rad}=2$: 
the radial mass vanishes for $0<\alpha<2$, vanishes
logarithmically at $\alpha=2$, and diverges for $\alpha>2$.

We note that in \cite{STCK2022} a threshold was also formulated for the same
(isotropic) normalized problem and in \cite[Theorem 1]{STCK2022} the authors reported $p_*(2) \in (3,1+8/3)$ with the numerical value $p_* \in (3.2,3.4)$, possibly in the vicinity $p_*(2)\approx 3.3$, which corresponds in our notation ($p=\alpha+1$) to  $\alpha \in (2.2,2.4)$ and $\alpha \approx 2.3$. 
These values are between the unrestricted threshold  $\alpha=\frac83 \approx 2.66$ obtained in \cite{FJMM2022} and the radial threshold $\alpha=2$ as shown above. 
\end{remark}

\smallskip

{\it Large $b$.} The behavior at large $b$ can clarify why the {\it lower-energy}  branch occurs at much higher values of $b$ (whenever branching does take place).  Recall the scaling \eqref{Qscaling} and set
$$
Q_b(x)=b^{1/\alpha}P_b(b^{1/4}x),
$$
then substituting into \eqref{E:groundstate} with $a=1$ gives
$$
\Delta^2 P_b + 2b^{-1/2} \Delta P_b +P_b -|P_b|^\alpha P_b=0.
$$
Note on a heuristic level that  as $b \to \infty$, the second term vanishes and $P_b$ approaches a pure biharmonic profile $P$ (which would be independent of $b$, also Pokhozhaev gives $M(P_b) \approx 1$).  
Computing the mass of $Q_b$ in two dimensions, we obtain $M(Q_b) = b^{\frac2{\alpha}-\frac12} M(P_b)$, which at least heuristically (as $P_b$ approaches $P$) yields that for $b \to \infty$
\begin{equation}\label{E:large-b-mass}
M(Q_b)\approx b^{\frac2{\alpha} - \frac12},
\end{equation}
see also \cite[Corollary 3.15]{FJMM2022}.

For $\alpha=3$, the mass diverges on both ends: as $b \to 1+$, by \eqref{E:rates-2D-nr} or \eqref{E:rates-2D-r}, and as $b \to \infty$, by \eqref{E:large-b-mass}. Thus, the mass curve must have a minimum (provided that it is continuous), as well as (at least) two values of $b$ corresponding to the same mass, as one can see, for example, in the left plot ($M = M(b)$) of Figure \ref{F:4a}(A). Then the solutions on the larger-$b$ branch may have lower energy at that fixed mass without contradicting the nonradial least-action result near $b=1$ (because these are different variational comparisons) as one can see in the left plot of Figure \ref{F:alpha3-ME}.

\smallskip

The radial two-dimensional near-threshold branch has a different mass scale. In the range $2<\alpha<4$, we prove via Lemma~\ref{L:quotient-to-mass} (and Section \ref{S:rate-consequences}) that
$M(Q_\epsilon^{\rm rad})\approx \epsilon^{\frac1{\alpha} - \frac12}$.
For $\alpha=3$, this is $M(Q_\epsilon^{\rm rad}) \approx \epsilon^{-1/6}$, which diverges faster than the nonradial law $\epsilon^{-1/12}$.  This explains why, at the same value of $b$, radial and nonradial branches may have substantially different masses. 

The comparison (as we discussed in Section \ref{S:2-peak}) should be then of $S_b$ at fixed $b$ for least action minimizers and of the energy $E$ at fixed mass for normalized minimization.
(We note that in our quartic computations, the two-peak nonradial branch lies below the four-peak branch on the overlapping mass interval, so the four-peak branch should be interpreted as an excited nonradial branch.  However, both these nonradial (small-$\epsilon$) branches may still lie above the lower-energy (stable) branch occurring at larger $b$ but with the same mass.)

\smallskip

This also explains the difference in solutions behavior between the cubic case $\alpha=2$ and the quartic case $\alpha=3$ (and other higher $\alpha$).  For $\alpha=2$, the nonradial branch has monotone behavior and has vanishing mass near $b=1$ (see left plot in Figure \ref{figflatME}), thus, no branching is observed in our computed nonradial solutions.  For $\alpha=3$, the asymptotic law gives
$M(Q_\epsilon^{\rm nr}) \approx \epsilon^{-1/12}$,
so the mass increases as $\epsilon \to 0^+$.  This is consistent with mass behavior on the left plot of Figure~\ref{F:alpha3-nonradial}, and the branching can be seen in the mass-energy diagram on the right of the same figure.

We summarize pointing out that the obtained rates for mass asymptotics explain when branching could appear, while the mass-energy $E = E(M)$ diagrams help determine which branch is selected by the normalized variational problem (and possibly the dynamics of those states as we have seen in one dimensional setting in \cite{KPRS}).


\appendix
\section{Small-$\epsilon$ asymptotics}\label{A:formal}

This appendix contains two parts: in 
Section \ref{S:appendix-2} we show the quotient upper bound calculations via trial-functions (Knapp-type wave-packet in 1D, Knapps caps in 2D and a radial Bessel construction) that explain how the concentration geometry produces the powers of $\ep$ and the radial endpoint logarithm. We summarize the rates in Table \ref{T:rates-summary}.

We mention that the trial-function method below is analogous in spirit to asymptotic profile arguments for vanishing parameters, as in \cite{MorozMuratov2014}, but here the minimum set is $|\xi|=1$ and the cap geometry is specific to the mixed fourth-order symbol. 

\subsection{Upper bounds via trial-functions} 
\label{S:appendix-2} 

As in Section \ref{S:positive}, we set $a=1$, $b=1+\epsilon$ with $0<\epsilon\ll 1$ and consider a trial function that is on the Fourier side  mostly concentrated near the minimum set of the Fourier multiplier $m_{1,1+\epsilon}(\xi)=(|\xi|^2-1)^2+\epsilon$, which is  the unit sphere $\{|\xi|=1\}$ as we previously discussed. It reduces in 1D to the two points $\xi=\pm 1$ and in 2D to the full unit circle.

In 1D we take the trial function to be real and even, which is the natural choice given that the minimum set $\{\pm 1\}$ is symmetric, so we consider only one case. In 2D, there is more flexibility for location of the minimum set, for example, trial functions could concentrate (on the Fourier side) on a cap (non-radial) or on the full annulus (radial), so we consider these two cases (it is possible to concentrate on other sets, but to have real-valued examples, we have to choose the symmetric geometry of the two opposite caps).

\subsubsection{One-dimensional trial-function}\label{A:sub-formal-1d}
Since in 1D the minimum set is $\{ \xi=\pm 1\}$ (two points), we make the following trial-function ansatz $\tilde Q_\ep$ for the (even) ground state:
\begin{equation}\label{E:ansatz-1d}
\tilde{Q_\ep}(x) = A \, W(\epsilon^{1/2} x)\,\cos(x),
\end{equation}
where $A = A(\ep)$ is an amplitude to be determined and $W \in \mathcal{S}(\R)$ is a real, even, slowly varying profile function (we find it later). The Fourier transform $\widehat{\tilde Q_\ep}(\xi) = \frac{A}{2} \big( \widehat W_\epsilon(\xi - 1)+ \widehat W_\epsilon(\xi+1) \big)$ with 
$\widehat W_\epsilon(\xi) =\epsilon^{-1/2} \widehat W(\epsilon^{-1/2}\xi)$ is concentrated in two intervals of width $\epsilon^{1/2}$ around $\xi = \pm 1$ (the cross-term is negligible due to separation). In a sense, this is a 1D analog of the Knapp two-cap concentration example, since in 1D the ``caps'' are just two points.
\smallskip

We estimate the mass $M(\tilde{Q_\ep})$, the potential norm $\|\tilde Q_\ep\|_{L^{\alpha+2}(\R)}^{\alpha+2}$, and the quadratic form $q_{1,1+\ep}(\tilde Q_\ep)$. Then we choose the amplitude $A$ so  that the trial function  satisfies the first Pokhozhaev identity and after that we discuss what  $W$ is. 
\smallskip

To compute the mass of $\tilde Q_\ep$, 
we need to deal with the integrand $(W(\epsilon^{1/2}x)\cos(x))^2$. 
Applying  $\cos^2 x = \tfrac12 + \tfrac12\cos(2x)$, we split the integral into two parts. 
The first (constant) term gives the main contribution, 
the second term has the fast-oscillating factor $\cos (2x)$ multiplied by the slowly-varying function $(W(\epsilon^{1/2}x))^2$. 
We substitute $y = \ep^{1/2}x$ in the second term, which turns the fast-oscillating factor into a Fourier coefficient of $W^2$ at the large frequency $2 \epsilon^{-1/2}$. Since $W \in \mathcal S(\R)$, so is $W^2$, and so we get $\widehat{W^2}(2\epsilon^{-1/2})=O(\epsilon^n)$ for any $n \in \mathbb N$.
Hence, we obtain
\begin{align*}
M(\tilde Q_\ep) = \|\tilde Q_\epsilon\|_{L^2(\R)}^2
& = A^2 \int_\R \big(W(\epsilon^{1/2}x)\cos(x) \big)^2 \,dx
 = \frac{A^2}{2} \int_\R \big( W(\ep^{1/2}x) \big)^2\,dx \big(1 + O(\epsilon^{n})\big)\\
& = \frac{A^2}{2} \,\epsilon^{-1/2}\,\|W\|_{L^2(\R)}^{2}\,\big(1+O(\ep^{n})\big),
\quad \mbox{for ~~any} ~~ n\in\mathbb N.
\end{align*}
In particular, the oscillatory term does not affect the power of $\epsilon$, so we can write
the leading-order behavior as 
$$
M(\tilde Q_\ep)=\frac12 A^2\epsilon^{-1/2} \|W\|_{L^2}^2(1 +o(1)).
$$

To compute the potential norm, we have to deal with the term $|\cos x|^{\alpha+2}$, which we expand into a cosine Fourier series $|\cos x|^{\alpha+2} = c_{\alpha} + \sum_{n=1}^{\infty} c_n(\alpha) \cos(2nx)$,
and, in a similar fashion as for the $\cos(2x)$ above, integrate out the rest of the terms except for the first one (as the leading order). 
The first term is an average over a full period, we denote it by 
\begin{equation}\label{E:c-alpha}
c_{\alpha}:= \frac{1}{\pi}\int_0^{\pi}|\cos x |^{\alpha+2} \,dx = {\Gamma\big(\tfrac{\alpha+3}{2}\big)}/({\sqrt\pi\,\Gamma\big(\tfrac{\alpha+4}{2}\big)})>0
\end{equation}
(this constant is also used in computing the limiting value in \eqref{E:C4-constant}),
and hence, we have 
$$
\|\tilde Q_\epsilon\|_{L^{\alpha+2}(\R)}^{\alpha+2} = c_{\alpha} \, A^{\alpha+2}\,\epsilon^{-1/2}\,\|W\|_{L^{\alpha+2}(\R)}^{\alpha+2} (1+o(1)).
$$

Next, we estimate the quadratic form $q_{1,1+\ep}(\tilde Q_\ep)$ and use that $\tilde Q_\ep$ has concentration near $\xi = \pm 1$. So, on the Fourier side (to be near $\xi = 1$ of width $\ep^{1/2})$, we set ~$\xi=1+\eta$~ and write $(|\xi|^2-1)^2 = (2\eta+\eta^2)^2 = 4\eta^2 (1+O(\eta))$ for $\eta \to 0$. Substituting $\zeta = \epsilon^{-1/2}\eta$, so that $\eta = O(\ep^{1/2})$ on the bump and the relative error becomes $O(\ep^{1/2})$, we obtain
$$
\int_\R (|\xi|^2-1)^2|\widehat{\tilde Q_\ep}|^2\,d\xi =  2 A^2 \ep^{1/2} \int_\R \zeta^2 |\widehat{W}(\zeta)|^2 \, d\zeta = 
2 A^2\,\epsilon^{1/2}\,\|W'\|_{L^2(\R)}^2 (1+O(\ep^{1/2})),
$$
by using Plancherel for the derivative (the constant 2 comes from the same contribution from $\xi=-1$ by symmetry).
From the above mass estimate, we have $\epsilon\|\tilde Q_\epsilon\|_{L^2}^2 = \frac{A^2}{2}\epsilon^{1/2} \|W\|_{L^2}^2$.  Putting two terms together, we deduce
$$
q_{1,1+\epsilon}(\tilde Q_\ep) = \int_\R ( (|\xi|^2-1)^2 + \ep )|\widehat{\tilde Q_\ep}|^2\, d\xi = A^2 \,\ep^{1/2}\, 
(2\|W'\|_{L^2(\R)}^2 + \tfrac{1}{2}\|W\|_{L^2(\R)}^2) (1+ O(\ep^{1/2})).
$$
Since the error terms in the mass, the potential norm and the quadratic form expansion above are all $o(1)$ relative to the corresponding leading terms, we simply write $(1+o(1))$.

Recalling the first Pokhozhaev identity \eqref{E:Pokh1-d}, we have $q_{1,1+\epsilon}(\tilde Q_\epsilon) = \|\tilde Q_\epsilon\|_{L^{\alpha+2}}^{\alpha+2}$, and substituting our estimates above into each side, we get 
$$
A^2\,\epsilon^{1/2}\, \big(2\|W'\|_{L^2(\R)}^2 + \tfrac{1}{2}\|W\|_{L^2(\R)}^2 \big) = c_{\alpha}\,A^{\alpha+2}\,\epsilon^{-1/2}\,\|W\|_{L^{\alpha+2}(\R)}^{\alpha+2}(1+o(1)),
$$
which implies 
\begin{equation}\label{E:1d-A}
A^\alpha = \frac{2\|W'\|_{L^2(\R)}^2 + \tfrac{1}{2}\|W\|_{L^2(\R)}^2}{c_{\alpha}\,\|W\|_{L^{\alpha+2}(\R)}^{\alpha+2}} \, \ep\, (1+o(1)) 
\quad \Rightarrow \quad A =C_*(\alpha)\,\epsilon^{1/\alpha}(1+o(1)),
\end{equation}
which justifies the amplitude scaling we wrote in \eqref{E:Knapp-1}.

Substituting $A$ into the mass, we obtain the following asymptotic rate 
\begin{equation}\label{E:M-1d-app}
M(\tilde Q_\epsilon) = \frac{A^2}{2}\, \epsilon^{-1/2} \,\|W\|_{L^2(\R)}^2  (1+o(1)) =  C(\alpha)\, \epsilon^{\frac2{\alpha} - \frac12} (1+o(1)),
\end{equation}
where $C(\alpha)$ is a constant, which depends on $\alpha$ and a localized profile function $W$.  
This is consistent with \eqref{E:1d-rates} and is proved in \S \ref{S:rate-consequences}.  For now we note that the true ground state $Q_\ep$ (that verifies Pokhozhaev) has the upper bound, since 
$\|Q_\ep\|_{L^{\alpha+2}}^\alpha = \mathcal R_\ep(\alpha) 
\leq ({q_{1,1+\ep}(\tilde Q_\ep)})/({\|\tilde Q_\ep\|_{L^{\alpha+2}}^2}) 
= \|\tilde Q_\ep\|_{L^{\alpha+2}}^\alpha$, 
and together with $\ep M(Q_\ep) \leq \|Q_\ep\|_{L^{\alpha+2}}^{\alpha+2}$ 
yields 
$$
M(Q_\ep) \leq \ep^{-1}\|\tilde Q_\ep\|_{L^{\alpha+2}}^{\alpha+2} \lesssim \ep^{\frac2\alpha-\frac12}.
$$
(To obtain the matching lower bound, we need the quotient-to-mass Lemma~\ref{L:quotient-to-mass}.)
\smallskip

We remark that the optimal localized profile function $W$ minimizes the ratio 
\begin{equation}\label{E:W-ratio}
\frac{2\|W'\|_{L^2(\R)}^2 + \tfrac{1}{2}\|W\|_{L^2(\R)}^2}{\|W\|_{L^{\alpha+2}(\R)}^2}.
\end{equation}
This is a (rescaled) Weinstein quotient for the second-order NLS, and thus, its Euler-Lagrange equation, after rescaling, reduces to the well-known ground state equation for the NLS, ~$-U'' + U = |U|^\alpha U$, which has the unique (up to translation) positive solution $U(x) = (\frac{\alpha+2}{2})^{1/\alpha} \sech^{2/\alpha}(\frac{\alpha}2 x)$. Note that the Euler-Lagrange equation of the above ratio in \eqref{E:W-ratio} is $-4W''+W = c_1 |W|^\alpha W$, and setting $W(x) = c_2 U(x/2)$ reduces it to the $U$ equation. So $W$, up to constants, is a dilated $\sech^{2/\alpha}$ profile.

Finally, since the quotient is invariant under multiplication by the amplitude $A$, the above trial function together with the estimates for the quadratic form and the potential norm, yield the rigorous upper bound 
\begin{equation}\label{E:R-upper-1d}
\mathcal R_\epsilon(\alpha) \leq \frac{q_{1,1+\epsilon}(\tilde Q_\epsilon)}{\|\tilde Q_\epsilon\|_{L^{\alpha+2}(\R)}^2} \lesssim \epsilon^{\frac{\alpha+4}{2(\alpha+2)}}. 
\end{equation}

Note that together with the matching lower bound proved in Section  \ref{A:sub-rigorous} (with either Theorem \ref{Thm:B1} or Theorem \ref{Thm:GN-lower}), this gives the rate formula \eqref{E:R_sharp}.

\subsubsection{Two-dimensional case: nonradial Knapp ansatz}\label{A:sub-formal-2d-nr}

In 2D, as we discussed, the minimum set of $m_{1,1+\epsilon}$ is the unit circle $|\xi|=1$. 
A Knapp-type nonradial ground state on the Fourier side concentrates on two opposite caps, say, near $\xi=(\pm 1,0)$. The cap has the width $\epsilon^{1/2}$ in the radial direction (along $\xi_1$) and $\epsilon^{1/4}$ in the tangential direction (along $\xi_2$). As we mentioned in \eqref{E:Knapp-1}, the corresponding form in the physical space is
\begin{equation}\label{E:ansatz-2d}
Q_\epsilon^{\rm nr}(x,y) = A \,W(\epsilon^{1/2} x, \epsilon^{1/4} y) \,\cos(x),
\end{equation}
with $W \in \mathcal S(\R^2)$ being a real, even, localized profile function.
As in 1D, the two caps are separated by a distance $\approx 2$, so the cross-term between them is negligible. 

First, we compute the mass (but now with $X = \ep^{1/2}x$ and $Y=\ep^{1/4}y$). Similar to the 1D case, writing $\cos^2 x = \frac12 + \frac12 \cos(2x)$ and substituting $X$ and $Y$, we integrate the (second) oscillatory part in $x$-variable, which produces the Fourier coefficient of $W^2$ in its first variable at the large frequency $2\ep^{-1/2}$, which is $O(\ep^n)$ for any $n\in \mathbb N$ uniformly in the second variable, since $W \in \mathcal{S}(\R^2)$ (and so is $W^2$). Thus, we have
$$
\|Q_\epsilon^{\rm nr}\|_{L^2(\R^2)}^2 = \frac{A^2}{2}\, \epsilon^{-1/2} \epsilon^{-1/4}\,\|W\|_{L^2(\R^2)}^2 (1+O(\ep^n)) = \frac{A^2}{2}\,\epsilon^{-3/4}\,\|W\|_{L^2(\R^2)}^2(1+O(\ep^n)).
$$

In a similar fashion, we obtain the potential norm estimate
$$
\|Q_\epsilon^{\rm nr}\|_{L^{\alpha+2}(\R^2)}^{\alpha+2} = c_{\alpha} \,A^{\alpha+2} \,\epsilon^{-3/4}\,\|W \|_{L^{\alpha+2}(\R^2) }^{\alpha+2}(1+o(1)).
$$

To estimate the quadratic form, we use the caps concentration near $(1,0)$ and set $\xi=(1+\eta_1, \eta_2)$, writing  
$$
|\xi|^2-1 = 2 \eta_1 + \eta_1^2 + \eta_2^2.
$$
Note that $|\widehat{Q_\ep}|^2 = \frac{A^2}4 \ep^{-3/2} | \widehat{W}(\ep^{-1/2}\eta_1,\ep^{-1/4}\eta_2)|^2$. A similar contribution will come from the concentration near $(-1,0)$.
For $\eta_1 \sim \epsilon^{1/2}$ and 
$\eta_2 \sim \epsilon^{1/4}$, 
we have $\eta_1 \sim \eta_2^2 \sim \ep^{1/2}$ and $\eta_1^2 \sim \epsilon \ll \epsilon^{1/2}$, so $|\xi|^2-1 \approx 2\eta_1 + \eta_2^2 = O(\epsilon^{1/2})$, 
giving $(|\xi|^2-1)^2 = O(\epsilon)$, 
the same order as the second term $\ep$ in $m_{1,1+\ep}$. 
More precisely, after the change of variables 
$\zeta_1=\epsilon^{-1/2}\eta_1$, $\zeta_2 = \epsilon^{-1/4}\eta_2$, we have
$$
|\xi|^2-1 = \ep^{1/2}(2\zeta_1+\zeta_2^2+\ep^{1/2}\zeta_1^2), \qquad \mbox{so} \quad (|\xi|^2-1)^2 = \ep [(2\zeta_1 + \zeta_2^2)^2 + O(\ep^{1/2}) ].
$$
Then
\begin{align*}
q_{1,1+\epsilon}(Q_\epsilon^{\rm nr}) 
&=  \int_{\R^2} \big((|\xi|^2-1)^2+\epsilon\big)| \widehat Q_\epsilon|^2\, d\xi \\
& = 2\cdot\frac{A^2}{4} \, \ep^{-3/2} \cdot \ep \cdot 
  \epsilon^{3/4}\, \int_{\R^2}  \big[(2\zeta_1 +\zeta_2^2)^2 +1 + O(\epsilon^{1/2})\big]\,
  |\widehat W (\zeta_1,\zeta_2)|^2\,d\zeta_1\, d\zeta_2 \\ 
&= \frac{A^2}{2}\,\epsilon^{1/4}\, \Lambda(W)\, \big(1 + O(\ep^{1/2}) \big),
\end{align*}
where $\Lambda(W)= \int_{\R^2}((2\zeta_1+\zeta_2^2)^2+1)|\widehat W(\zeta_1,\zeta_2)|^2\,d\zeta_1 d\zeta_2>0$. 

As in the one-dimensional case, 
since the error terms in the mass, the potential norm and the quadratic form expansion above are all $o(1)$ relative to their respective leading terms, we just write $(1+o(1))$ below. 

Similarly to the 1D case, from the first Pokhozhaev identity \eqref{E:q-equals-Lp-new}, $q_\ep(Q) = \|Q\|_{L^{\alpha+2}(\R^2)}^{\alpha+2}$,
we estimate the amplitude
$$
\frac12 A^2\,\epsilon^{1/4}\,\Lambda(W) = c_{\alpha}\,A^{\alpha+2}\,\epsilon^{-3/4}\,\|W\|_{L^{\alpha+2}(\R^2)}^{\alpha+2}(1+o(1)),
$$
which gives 
$$
A^\alpha = \frac{\Lambda(W)}{2c_{\alpha}\,\|W\|_{L^{\alpha+2}(\R^2)}^{\alpha+2}}   
\epsilon (1+o(1)), \quad \mbox{and ~~ so} \quad  
A  = C(\alpha)\,\epsilon^{1/\alpha}(1+o(1)).
$$
Finally, substituting the amplitude in terms of $\ep$ into the mass, we obtain
\begin{equation}\label{E:M-2d-nr-app}
M(Q_\epsilon^{\rm nr}) = \frac{A^2}{2}\,\epsilon^{-3/4}\,\| W\|_{L^2(\R^2)}^2 (1+o(1)) = C(\alpha)\,\epsilon^{2/\alpha-3/4} (1+o(1)),
\end{equation}
which is the rate claimed in \eqref{E:rates-2D-nr} in 2D for the nonradial ground state mass. As in 1D, this is the mass of the trial function \eqref{E:ansatz-2d}, the corresponding statement for the true ground state is obtained in \S\ref{S:rate-consequences}.
 
For the localized profile $W$, a fixed nonzero Schwartz function, we obtain the quotient upper bound
$$
\mathcal R_\epsilon(\alpha) 
\leq \frac{q_{1,1+\ep}(Q_\epsilon^{\rm nr})}{\|Q_\epsilon^{\rm nr}\|_{L^{\alpha+2}(\R^2)}^2} = \frac{A^2\epsilon^{1/4}}{A^2\,\epsilon^{-3/(2(\alpha+2))}} C(\alpha,W) (1+o(1)) 
=C(\alpha,W)\epsilon^{\frac{8+\alpha}{4(\alpha+2)}} (1+o(1)).
$$
For $0 < \alpha \leq 4$, the matching lower bound follows from Theorem \ref{Thm:GN-lower} (and agrees with \cite[Theorem 1.3]{LW2021}, so our trial-function geometry recovers precisely the two-sided quotient  exponent $\gamma_{2,\alpha}$ from \eqref{E:R-asym}.

\subsubsection{Two-dimensional case: radial ansatz}\label{A:sub-formal-rad}

For  radial ground states in 2D, the Fourier mass should be spread evenly on the unit circle (uniformly in angle). 
Before constructing the trial function, we point out one elementary observation that explains why the radial case is governed by the quotient-to-mass relation. 
Let $Q_\epsilon^{\rm rad}$ be a radial ground state with $\widehat {Q_\epsilon^{\rm rad}}$ concentrated on the annulus $\mathcal A_\epsilon = \{\xi \in \R^2 : ||\xi|-1| \leq \epsilon^{1/2}\}$. Then on $\mathcal A_\ep$, we have 
$(|\xi|^2-1)^2 + \epsilon = O(\epsilon)$, and so
\begin{equation}\label{E:disp-vs-mass-rad}
q_{1,1+\epsilon}(Q_\epsilon^{\rm rad}) = \int_{\R^2}\big((|\xi|^2-1)^2+\epsilon\big)\,|\widehat Q_\epsilon^{\rm rad}|^2\,d\xi 
\approx  \epsilon \, M(Q_\epsilon^{\rm rad}).
\end{equation}
Together with the first Pokhozhaev identity 
$q_{1,1+\ep}(Q_\ep^{\rm rad}) =\|Q_\ep^{\rm rad} \|_{L^{\alpha+2}}^{\alpha+2}$, \eqref{E:disp-vs-mass-rad} is precisely the heuristic of the quotient-to-mass Lemma~\ref{L:quotient-to-mass}, which is proved in \S\ref{S:M-to-R} (the same linear relation $q \sim \epsilon M$ holds in the nonradial case, where $M \sim A^2\epsilon^{-3/4}$ 
and $q \sim A^2 \ep^{1/4}$). 
We therefore do not repeat that argument here, and instead construct an explicit radial trial function, which yields an upper bound on $\mathcal R_\ep^{\rm rad}$ and, in addition, explains where the three regimes and the endpoint logarithm come from.

\subsubsection*{Bessel-function approach}

Here we use the Bessel function representation of the Fourier transform of a radial function and develop more precise asymptotic rates in $\ep$.

Since $Q$ and $\hat Q$ are both radial ($r=|x|$), the Fourier inversion reduces to 
the zeroth-order Hankel transform with the Bessel function $J_0$ (see \cite[Appendix~B.5]{Grafakos2014}) in our normalization:
$$
Q(r) = \int_0^\infty \hat Q(\rho)\,J_0(\rho r)\,\rho\,d\rho,
$$
and the large-$r$ behavior of $Q(r)$ is governed by the Bessel asymptotic (see \cite[Appendix~B.8]{Grafakos2014}) $J_0(r) = (\frac2{\pi r})^{1/2}\cos(r-\pi/4) + O(r^{-3/2})$.

If the ground state $Q_\ep^{\rm rad}$ concentrates on the Fourier side near $\rho=1$, then writing $\rho = 1+\eta$ with $|\eta| \lesssim \sqrt \ep$, we heuristically expect 
$$
Q(r) \approx \int_{|\eta| \lesssim \sqrt \ep} \widehat{Q}(1+\eta)\, J_0((1+\eta)r)\, d\eta \approx
J_0(r) \int_{|\eta| \lesssim \sqrt \ep} \widehat{Q}(1+\eta) \cos(\eta r)\, d\eta  
=: J_0(r)\, F_\ep(r),
$$
where the sine term is dropped by taking $\hat Q$ symmetric around $\rho=1$ and $\rho\,d\rho \approx d\eta$.
Then, with $\widehat{Q}(1+\eta)$ having width $\sqrt \ep$ in $\eta$, so $\cos(\eta r)$ varies on the reciprocal scale $r \sim (\sqrt \ep)^{-1}$, we have that $F_\ep$ depends on $r$ only via the slow variable $\sqrt \ep\, r$, so we write $F_\ep(r) = A_\ep\, F(\sqrt \ep\, r)$. This gives us a motivation to define the radial trial function as
\begin{equation}\label{E:Bessel}
Q_\epsilon^{\rm rad}(r) := A_\epsilon\, F(\sqrt\epsilon\, r)\, J_0(r), \qquad \mbox{with} \quad F\in\mathcal S(\R)~~ \mbox{even} \quad \mbox{and} ~~ F(0) \neq 0,
\end{equation}
so even $F$ gives $F'(0) = 0$, which ensures the required regularity at the origin and $Q^{\rm rad}_\ep \in H^2(\R^2)$, and $F(0) \neq 0$ is needed for  integration for $\alpha \geq 2$ later.
(Recall that in 1D, the support near $\pm 1$ with width $\sqrt \ep$ produced $Q(x) \sim W(\sqrt \ep x) \cos x$.)

To estimate the mass $M(Q_\ep^{\rm rad})$, we use the $J_0$ large-radius asymptotic $J_0(r)^2 = \frac{2}{\pi r}\, \cos^2(r-\frac{\pi}{4}) + O(r^{-2})$, with the localized function $F$ cutting off at $r\sim\epsilon^{-1/2}$. 
Multiplying by the Jacobian $r$ (in polar coordinates) and using $\cos^2(r-\tfrac{\pi}4)=\tfrac12 + \tfrac12 \cos(2r - \tfrac\pi2)$ gives $|J_0(r)|^2r=\tfrac1{\pi}+\tfrac1{\pi}\cos(2r-\tfrac{\pi}2) + O(r^{-1})$. 
Integrating $|F(\sqrt\ep\,r)|^2$ with the oscillatory ($\cos$) middle term in the region $r \geq 1$ gives $O(1)$, and integrating $O(r^{-1})$ term contributes $O(|\ln\ep|)$, the region $0<r<1$ also gives only a constant, thus, all terms are of lower order than the main term $\ep^{-1/2}$. Therefore, we deduce
\begin{align}
M(Q_\ep^{\rm rad}) = 2\pi A_\ep^2 \int_0^\infty |F(\sqrt \ep\, r)|^2\, |J_0(r)|^2\, r\,dr 
&=2 A_\ep^2 \int_0^\infty |F(\sqrt\ep\,r)|^2\,dr \, (1+o(1)) \notag
\\
&= 2 A_\ep^2\, \ep^{-1/2} \|F\|^2_{L^2(0,\infty)} (1+o(1)), \label{E:M-F}
\end{align}
where the factor $2\pi$ from the angular integration and the averaged value $\tfrac1\pi$ of $|J_0(r)|^2r$ combine into the constant $2$.

Since $\widehat{Q_\ep^{\rm rad}}$ is concentrated in the annulus $\big||\xi|-1\big|\lesssim\sqrt\ep$,
we have (with $|\xi|=1+\sqrt\ep\,\zeta$) that
$(|\xi|^2-1)^2+\ep = \ep (1+4\zeta^2+O(\sqrt\ep) )$, so
$$
q_{1,1+\ep}(Q_\epsilon^{\rm rad}) = \int_{\R^2}\big((|\xi|^2-1)^2+\ep\big)\,|\widehat{Q_\ep^{\rm rad}}(\xi)|^2\,d\xi
= C(F)\,\ep\,M(Q_\ep^{\rm rad})\,(1+o(1)),
$$
where by Plancherel 
$$
C(F)=\frac{\int_{\R^2} (1+4\zeta^2 )\,|\widehat{Q_\ep^{\rm rad}}(\xi)|^2\,d\xi}
{\int_{\R^2}|\widehat{Q_\ep^{\rm rad}}(\xi)|^2\,d\xi}
=1+ \frac{4\|F'\|_{L^2(0,\infty)}^2}{\|F\|_{L^2(0,\infty)}^2} + o(1),
$$
and using \eqref{E:M-F}, we get
\begin{equation}\label{E:q-rad}
q_{1,1+\ep}(Q_\ep^{\rm rad}) = 2\,A_\ep^2\,\ep^{1/2} 
(4 \|F'\|_{L^2(0,\infty)}^2 + \|F\|^2_{L^2(0,\infty)}) (1+o(1)).
\end{equation}

In a similar manner as we estimated the potential norm before, using the $J_0$ large-$r$ decay together with the period average $c_\alpha=\frac1\pi\int_0^\pi|\cos x|^{\alpha+2}dx$ of $|\cos(r-\tfrac\pi4)|^{\alpha+2}$, 
we obtain
\begin{align}
\|Q_\epsilon^{\rm rad}\|_{L^{\alpha+2}(\mathbb R^2)}^{\alpha+2}
&= 2\pi\Big(\tfrac{2}{\pi}\Big)^{\frac{\alpha+2}{2}} c_\alpha \, A_\epsilon^{\alpha+2}
   \int_1^\infty |F(\sqrt\ep\,r)|^{\alpha+2}\, r^{-\alpha/2}\,dr\,(1+o(1)) \label{E:C-I-a} \\
&= C(F, \alpha) \, A_\epsilon^{\alpha+2} \int_1^{\ep^{-1/2}} r^{-\alpha/2}\,dr (1+o(1)). \label{E:C-I}
\end{align}

The reason for separating the integral on $r$ in the last line \eqref{E:C-I} is because 
it is exactly where the power of $\alpha$ plays the key role. For $\alpha=2$ the integral evaluates exactly to $\int_1^{\ep^{-1/2}} r^{-1}\,dr = -\tfrac{1}{2}\ln \ep$, thus, giving a $\log$ correction in this case. The integral 
diverges for $0<\alpha<2$ and remains bounded for $\alpha>2$. 
To be precise, 
\begin{equation}\label{E:I-alpha}
I(\alpha,\ep) := \int_1^{\ep^{-1/2}} r^{-\alpha/2}\,dr =
\left\{
\begin{array}{ll}
\frac{2}{2-\alpha}\, (\ep^{-(2-\alpha)/4}-1), & 0<\alpha<2,\\
\tfrac{1}{2}\, |\ln \ep|, & ~\alpha=2,\\
\frac{2}{\alpha-2} (1-\ep^{(\alpha-2)/4}), & \alpha>2.
\end{array}
\right.
\end{equation}
Coming back to \eqref{E:C-I-a}, we note that  (i) when $0<\alpha<2$ the integral is dominated by large $r$, so the large-$r$ asymptotic of $J_0$ applies, (ii) when $\alpha=2$ all the scales $1\lesssim r\lesssim \ep^{-1/2}$ contribute equally (which produces the logarithm), and on that range
$F(\sqrt\ep\,r)\approx F(0)$, so its coefficient is $|F(0)|^{4}$, and thus, we need $F(0)\neq0$, 
and (iii) when $\alpha>2$ the integral converges at the upper limit, so it is concentrated at bounded $r$, so in a sense the large-$r$ asymptotic is inaccurate (but we wrote it for the clarity of explanation), so in the last case the constant is not really $\tfrac{2}{\alpha-2}$, however, it is still a positive constant independent of $\ep$, which is all that we need 
($\|Q_\ep\|^{\alpha+2}_{L^{\alpha+2}} \to 2\pi \, A_{\ep}^{\alpha+2} |F(0)|^{\alpha+2} \int_0^\infty |J_0(r)|^{\alpha+2}\,r\, dr$.) 
Collecting \eqref{E:C-I} and \eqref{E:I-alpha}, we obtain
\begin{equation}\label{E:pot-rad}
\|Q_\epsilon^{\rm rad}\|_{L^{\alpha+2}(\R^2)}^{\alpha+2}
= C(F,\alpha)\, A_\epsilon^{\alpha+2}
\left\{
\begin{array}{ll}
\epsilon^{-\frac{2-\alpha}{4}}\,(1+o(1)), & 0<\alpha<2,\\
|\ln\epsilon|\,(1+o(1)), & ~\alpha=2,\\
1+o(1), & ~\alpha>2,
\end{array}
\right.
\end{equation}
where the constant factors of \eqref{E:I-alpha} have been absorbed into $C(F,\alpha)$.

Choosing $A_\ep$, so that our trial function \eqref{E:Bessel} satisfies the first Pokhozhaev identity  gives 
$q_{1,1+\ep}(Q_\ep^{\rm rad}) = \|Q_\ep^{\rm rad}\|_{L^{\alpha+2}(\R^2)}^{\alpha+2}$. Substituting \eqref{E:q-rad} on the left and \eqref{E:pot-rad}
on the right, we obtain
$$
A_\ep^2\, \ep^{1/2} = C\, A_\ep^{\alpha+2}\, I(\alpha,\ep)(1+o(1)),
$$
and thus, 
\begin{equation}\label{E:A-eps}
A_\ep^\alpha = C \, \frac{\ep^{1/2}}{I(\alpha,\ep)}(1+o(1)).
\end{equation}
Substituting \eqref{E:I-alpha}, we have 
\begin{equation}\label{E:A-alpha}
A_\ep^\alpha =
\left\{
\begin{array}{ll}
C_\alpha\, \ep^{(4-\alpha)/4}(1+o(1)), & 0<\alpha<2,\\
C\, \ep^{1/2}\, |\ln \ep|^{-1}(1+o(1)), & \alpha=2,\\
C_\alpha\, \ep^{1/2}(1+o(1)), & ~\alpha>2.
\end{array}
\right. ,
\end{equation}
and inserting \eqref{E:A-alpha} into $M(Q_\ep^{\rm rad}) = C\, A_\ep^2\, \ep^{-1/2}(1+o(1))$, 
gives the radial mass rate
\begin{equation}\label{E:M-rad-Bessel}
M(Q_\epsilon^{\rm rad}) = \left\{
\begin{array}{ll}
C_\alpha \,\epsilon^{\frac2\alpha-1}(1+o(1)), & 0<\alpha<2,\\
C \, |\ln \epsilon|^{-1} (1+o(1)), & ~\alpha=2,\\
C_\alpha \, \epsilon^{\frac1\alpha-\frac12}(1+o(1)), & ~\alpha>2.
\end{array}
\right.
\end{equation}

As in the previous cases, we point out that \eqref{E:M-rad-Bessel} is the mass of the trial function \eqref{E:Bessel} and it agrees with the sharp radial rate \eqref{E:R-asym-rad-MO} of \cite{MO2023}. The corresponding two-sided statement for the mass of the true radial ground state follows from \eqref{E:R-asym-rad-MO} and Lemma \ref{L:quotient-to-mass}, see \S \ref{A:sub-rig-2d}.
The estimates for the corresponding potential norm and quotient follow from $\|Q_\ep^{\rm rad}\|_{L^{\alpha+2}(\R^2)}^{\alpha+2}= C(F) \,\ep\, M(Q_\ep^{\rm rad})(1+o(1))$ and $\mathcal R_\ep^{\rm rad}(\alpha) \leq q_{1,1+\ep}(Q_\ep^{\rm rad})/\|Q_\ep^{\rm rad}\|_{L^{\alpha+2}(\R^2)}^{2}$, and agree with \eqref{E:R-asym-rad-MO} and \eqref{E:S-asym-rad-2d}.

\begin{remark}
An angular mode $J_m(r)\cos(m\theta)$ also solves $(\Delta+1)u=0$ and has the same large-$r$ Bessel decay $r^{-1/2}$, so for the angular mode trial functions $A_\ep F(\sqrt \ep r) J_m(r) \cos(m\theta)$ with a fixed $m$, the powers in $\ep$ asymptotics are unchanged. 
\end{remark}

\begin{remark}[Higher dimensions]\label{R:Bessel-d}
The above computation can be easily generalized to $d >2 $. In $\R^d$ the radial solution of
$(\Delta+1)u=0$ that is regular at the origin is (e.g., as in \cite{MO2023})
$$
\mathcal J_d(r):=r^{-\frac{d-2}{2}} \, J_{\frac{d-2}{2}}(r),
\qquad \mbox{so that} \quad \mathcal J_2=J_0,
\quad \mathcal J_3(r) = \sqrt{\frac{2}{\pi}} \, \frac{\sin r}{r},
$$
and, the Bessel asymptotic gives 
$\mathcal J_d(r) = \big(\tfrac2{\pi}\big)^{1/2} r^{-\frac{d-1}{2}} \cos \big(r-\tfrac{(d-1)\pi}{4} \big) + O\big(r^{-\frac{d+1}{2}}\big)$. 
Taking
$Q_\ep^{\rm rad}(r):=A_\ep F(\sqrt\ep\,r)\,\mathcal J_d(r)$ with $F$ as in \eqref{E:Bessel},
every step above goes through verbatim, with $r^{d-1}dr$ in place of $r\,dr$ and the surface measure $|\mathbb S^{d-1}|$ instead of $2\pi$.

When computing mass, the mass integrand has
$|\mathcal J_d(r)|^2r^{d-1} \sim r^{-(d-1)}r^{d-1} = O(1)$, so the mass
estimate does not change. But when integrating the potential norm, that integrand hass
$$
|\mathcal J_d(r)|^{\alpha+2}\,r^{d-1} \sim r^{-\frac{(d-1)\alpha}{2}},
$$
so that \eqref{E:I-alpha} is replaced by
$I(\alpha,\ep)=\int_1^{\ep^{-1/2}}r^{-\frac{(d-1)\alpha}{2}}\,dr$. 
Depending on the exponent $\frac{(d-1)\alpha}{2}$ this integral diverges, or logarithmic, or converges with the threshold value $\alpha=\frac{2}{d-1}$ (which is $\alpha=2$ when $d=2$). Proceeding with the rest in the same manner, we obtain the upper bounds
$$
\mathcal R_\ep^{\rm rad}(\alpha) \lesssim
\left\{
\begin{array}{ll}
\ep^{r_{d,\alpha}}, & 0 < \alpha < \frac{2}{d-1}, \\
\ep^{\frac12}\,|\ln\ep|^{-\frac{d-1}{d}}, & \alpha=\frac{2}{d-1},\\
\ep^{\frac12}, & \alpha>\frac{2}{d-1},
\end{array}
\right.
$$
with $r_{d,\alpha}$ as in \eqref{E:R-asym-rad}, and the corresponding mass will be
$\ep^{\frac2{\alpha} - \frac{d}2}$, $\ep^{\frac{d-2}{2}}|\ln \ep|^{-(d-1)}$ and
$\ep^{\frac1{\alpha}-\frac12}$, which agrees with the sharp radial rates
\eqref{E:M-rad-d} in \S \ref{S:asym}. 

\end{remark}

\begin{remark}[about the trial functions]\label{R:trial-scope}
In each computation of this subsection we showed two types of results: (1) the proof of the upper bound $\mathcal R_\ep\lesssim\ep^{\gamma}$, since the quotient is scale-invariant and the trial function is an admissible function in the infimum, (2) the mass that was obtained in each computation is the mass \emph{of the trial function}: it predicts the rate for the {\it true} ground state mass (that result  is obtained only after combining two-sided quotient bounds with Lemma~\ref{L:quotient-to-mass}). We felt that providing explicit trial functions examples helps one to see where the asymptotic rates are coming from.
\end{remark}

\subsubsection{Rate table}\label{A:table}

We summarize all the small-$\ep$ rates established in the paper in Table \ref{T:rates-summary}. While the rates for $d=2$ are derived from the same formulas as for $d \geq 3$ (in the $H^2$-subcritical range), we write them separately for the purpose of this paper. 

\begin{table}[h]
\centering
\renewcommand{\arraystretch}{1.5}
\begin{tabular}{l|c|c|c|c}
& range of $\alpha$ & $\gamma$ (and $L$)
& $\mathcal R_\ep$ & $M(Q_\ep)$\\
\hline \hline
$d=1$ & $\alpha>0$ & $\frac{\alpha+4}{2(\alpha+2)}$
& $\ds \ep^{\frac{\alpha+4}{2(\alpha+2)}}$ & $\ep^{\frac2\alpha-\frac12}$\\
\hline
$d=2$, unrestricted & $0<\alpha \leq 4$ & $\frac{8+\alpha}{4(\alpha+2)}$
& $\ds \ep^{\frac{8+\alpha}{4(\alpha+2)}}$ & $\ds\ep^{\frac2\alpha-\frac34}$\\
$d=2$, unrestricted & $\alpha>4$ & $\frac12$ & $\ds\ep^{\frac12}$
& $\ds \ep^{\frac1\alpha-\frac12}$ \\
$d=2$, radial & $0<\alpha<2$ & $\frac{2}{\alpha+2}$
& $\ds \ep^{\frac{2}{\alpha+2}}$ & $\ds\ep^{\frac2\alpha-1}$\\
$d=2$, radial & $\alpha=2$ & ~~$\frac12$, \, $L=|\ln\ep|^{-\frac12}$
& $\ds \ep^{\frac12}|\ln\ep|^{-\frac12}$ & $\ds|\ln\ep|^{-1}$\\
$d=2$, radial & $\alpha>2$ & $\frac12$
& $\ds \ep^{\frac12}$ & $\ds \ep^{\frac1\alpha-\frac12}$\\
\hline
$d \geq 3$, unrestricted & $0 < \alpha \leq \frac4{d-1}$
& $\frac{8+(3-d)\alpha}{4(\alpha+2)}$
& $\ds \ep^{\frac{8+(3-d)\alpha}{4(\alpha+2)}}$
& $\ds \ep^{\frac2\alpha -\frac{d+1}{4}}$\\
$d \geq 3$, unrestricted & $\alpha>\frac4{d-1}$ & $\frac12$ & 
$\ds\ep^{\frac12}$ & $\ds \ep^{\frac1\alpha-\frac12}$ \\
$d\geq3$, radial & $0<\alpha<\frac{2}{d-1}$
& $\frac{4+(2-d)\alpha}{2(\alpha+2)}$
& $\ds \ep^{\frac{4+(2-d)\alpha}{2(\alpha+2)}}$
& $\ds \ep^{\frac2\alpha-\frac d2}$ \\
$d\geq3$, radial & $\alpha=\frac{2}{d-1}$
& ~~$\frac12$, $L=|\ln\ep|^{-\frac{d-1}{d}}$
& $\ds\ep^{\frac12}|\ln\ep|^{-\frac{d-1}{d}}$
& $\ds\ep^{\frac{d-2}{2}}|\ln\ep|^{-(d-1)}$\\
$d\geq3$, radial & $\alpha>\frac{2}{d-1}$ & $\frac12$
& $\ds\ep^{\frac12}$ & $\ds\ep^{\frac1\alpha-\frac12}$\\
\end{tabular}
\caption{Small-$\ep$ rates for $a=1$, $b=1+\ep$. Here, $\mathcal R_\ep \approx \ep^\gamma L(\ep)$, with $L \equiv 1$ except in the radial threshold rows.
Rows 1, 2 and 7 follow from Part 2 of Theorem \ref{ThmA}. 
Rows 3 and 8 follow from \eqref{E:R-unrestr-sharp}
and Corollary \ref{C:sharp-mass}.
Rows 4-6, 9-11 use Part 1 of Theorem \ref{ThmA} and 
\eqref{E:R-asym-rad-MO} and then more generally \eqref{E:R-radial-d}  from
\cite{LW2021,MO2023}.}
\label{T:rates-summary}
\end{table}

\smallskip



\bibliographystyle{abbrv}
\bibliography{references}

\end{document}